\documentclass[11pt,a4paper]{article}
\usepackage{graphicx}
\usepackage{latexsym}
\usepackage{amsmath,mathtools,amsfonts,amssymb, amsthm, euscript,makeidx,color,mathrsfs}
\usepackage{enumitem}
\usepackage{url}
\usepackage[utf8]{inputenc}
\usepackage[T1]{fontenc}

\usepackage[colorlinks,linkcolor=blue,anchorcolor=green,citecolor=red]{hyperref}

\usepackage[margin=1.1in]{geometry}

\def\sqr#1#2{{\vcenter{\vbox{\hrule height.#2pt
              \hbox{\vrule width.#2pt height#1pt \kern#1pt \vrule width.#2pt}
              \hrule height.#2pt}}}}
\def\dbH{{\mathbb{H}}}

\def\dbN{{\mathbb{N}}}

\def\dbP{{\mathbb{P}}}

\def\dbR{{\mathbb{R}}}

\def\R{{\bf R}}

\def\L{\Lambda}

\def\O{\Omega}
\def\no{\noindent}

\def\bs{\bigskip}

\def\limsup{\mathop{\overline{\rm lim}}}
\def\liminf{\mathop{\underline{\rm lim}}}

\def\resp{{\it resp. }}
\DeclareMathOperator{\supp}{supp}
\DeclareMathOperator{\sgn}{sgn}

\newcommand{\dd}{\,\mathrm d}

\DeclareMathOperator{\Div}{div}

\newcommand{\s}{\stackrel{\circ}{\subset}}

\newcommand{\Ls}[1]{\mathcal{L}_{s}^{#1}}

\theoremstyle{plain}
\newtheorem{theorem}{Theorem}[section]
\newtheorem{lemma}[theorem]{Lemma}
\newtheorem{corollary}[theorem]{Corollary}

\newtheorem{proposition}[theorem]{Proposition}

\theoremstyle{definition}
\newtheorem{definition}[theorem]{Definition}
\newtheorem{remark}[theorem]{Remark}
\newtheorem{example}[theorem]{Example}
\newtheorem{problem}[theorem]{Problem}

\newtheorem{assumption}[theorem]{Assumption}

\makeatletter
   
   \@addtoreset{equation}{section}
\makeatother

\usepackage{microtype}

\begin{document}

\title{Sobolev spaces in infinite dimensions}

\author{Zhouzhe Wang\thanks{
School of Mathematics,
Sichuan University, Chengdu 610064, China.  {\small\it
e-mail:} {\small\tt wangzhouzhe@pku.edu.cn}.},~ Jiayang Yu\thanks{
School of Mathematics,
Sichuan University, Chengdu 610064, China.  {\small\it
e-mail:} {\small\tt jiayangyu@scu.edu.cn}.},~ Feinuo Zhang\thanks{
Shanghai Center for Mathematical Sciences,
Fudan University, Shanghai 200438, China.  {\small\it
e-mail:} {\small\tt fnzhang17@fudan.edu.cn}.},~ Xu Zhang\thanks{
School of Mathematics,
Sichuan University, Chengdu 610064, China.  {\small\it
e-mail:} {\small\tt zhang\_xu@scu.edu.cn}.},~ and~ Shiliang Zhao\thanks{
School of Mathematics,
Sichuan University, Chengdu 610064, China.  {\small\it
e-mail:} {\small\tt zhaoshiliang@scu.edu.cn}.}}

\date{}
\maketitle

\begin{abstract}
   The classical theory of Sobolev spaces in finite dimensions is well established.  Because infinite-dimensional spaces possess inherent analytical and topological complexities, developing a theory of Sobolev spaces for functions of infinitely many variables—rather than merely extending the classical finite-dimensional theory—is far from routine and has led to long-standing open problems. In this paper, we establish such a theory and systematically determine the extent to which these classical results can be carried over to this setting. By carefully adapting the tools introduced in our recent works, we establish a series of infinite-dimensional counterparts of the classical theorems and, in the process, reveal new phenomena with no finite-dimensional analogue. The methods and concepts developed here, together with the results obtained, furnish a robust framework for further study of Sobolev spaces and related problems in infinite-dimensional analysis.
\end{abstract}

\bs

\no{\bf 2020 Mathematics Subject Classification.} 46G05, 46E35, 46A22, 26E15, 28C20.

\bs

\no{\bf Key Words.} Sobolev spaces, infinite-dimensional analysis, approximation, extension operators, trace operators, Sobolev inequalities.

\maketitle

\tableofcontents

\section{Introduction}
The theory of Sobolev spaces  on finite-dimensional (Euclidean) spaces has been one of the cornerstones of modern analysis (e.g., \cite{AF,Leo17,Mazya2011}). By bringing together differentiability and integrability within a unified framework, they provide a fundamental setting for studying numerous problems in the theory of partial differential equations, approximation theory, and many other areas of pure and applied mathematics. In particular, these spaces appear in linear and nonlinear partial differential equations originating from fields such as differential geometry, harmonic analysis, engineering, control theory, mechanics, and physics.

Sobolev spaces in infinite dimensions constitute a relatively unexplored but quite important field, which plays an indispensable role in many areas of mathematics, including probability theory and stochastic analysis (e.g., \cite{Nualart06}), infinite-dimensional real and complex analysis  (e.g., \cite{Bog, WYZ1}), infinite-dimensional differential geometry (e.g., \cite{Elworthy-Li, Li2003}), statistics and information science (e.g., \cite{Newton}), and control theory for infinite-dimensional systems (e.g., \cite{FGS17, LY95}), just to mention a few. 

Generally speaking, the study of many important analysis problems in infinite dimensions will encounter huge challenges:
\begin{itemize}
    \item In sharp contrast to the finite-dimensional setting, unfortunately, there is no nontrivial locally finite translation-invariant Borel measure on any infinite-dimensional separable Banach space;

   \item Any useful topology in an infinite-dimensional space is usually much more complicated than its finite-dimensional counterpart, and in infinite dimensions the gap between local and
global is much wider than that in finite dimensions;

		\item As remarked in our paper \cite[p. 520]{YZ}, another essential difficulty in infinite dimensions is how to treat the differential (especially the higher order differential) of (possibly vector-valued) functions. In this respect, the most popular notions are  Fr\'echet's derivative and G\^ateaux's derivative, which coincide whenever one of them is continuous. However, higher order (including the second order) Fr\'echet's derivatives are multilinear mappings, which are usually not convenient to handle analytically.
\end{itemize}
Due to the above inherent analytical and topological obstacles, many basic tools---such as convolutions, mollifiers, and standard covering arguments---are not directly available in the infinite-dimensional setting. Meanwhile, many infinite sums inevitably arise in infinite dimensions. Therefore, attempts to establish an infinite-dimensional theory of Sobolev spaces have led to several long-standing open problems, to be mentioned later.

In this paper, we shall develop a general theory of Sobolev spaces for functions on general open subsets of the infinite-dimensional Hilbert space $\ell^2$. Note that, $\ell^2$ is, in some sense, the simplest infinite-dimensional space and the closest to Euclidean space; nevertheless, the essential obstacles mentioned above are by no means reduced.
Let $P$ be the nondegenerate centered Gaussian measure on $\ell^2$ whose $i$-th coordinate has variance $a_i^2>0$, with $\sum_{i=1}^{\infty}a_i^2<\infty$. For any nonempty open subset ${O}$ of $\ell^2$, we consider the Sobolev spaces $W^{m,p}(O,P)$ and $\mathscr{W}^{m,p}(O,P)$ (see Definitions~\ref{def:Sob1} and \ref{def:Sob2}) defined via two different Sobolev norms, so as to cover both possibilities that may arise in applications.

Motivated by central topics in the classical theory of Sobolev spaces (see, e.g., \cite{AF,Leo17}), we study the properties of $W^{m,p}(O,P)$ and $\mathscr{W}^{m,p}(O,P)$. Already at the level of problem formulation, passing to infinite dimensions is far from straightforward: identifying the appropriate analogues of finite-dimensional questions requires additional care and, in some cases, new notions. In the course of our study, we examine which of the classical results continue to hold in our setting, what assumptions are needed, and what genuinely new phenomena occur in infinite dimensions. In particular, we provide several examples with unexpected behaviour, which reveal new phenomena that do not appear in finite dimensions. These examples all rely on the existence of infinitely many uniformly separated vectors on a sphere. In fact,
Riesz's Lemma implies that such vectors can be found in every infinite-dimensional normed spaces, which is one of the key differences between finite and infinite dimensions.

Our main results can be described as follows, organized according to the principal themes of this paper.

\subsection{$H=W$}
A famous and fundamental result regarding Sobolev spaces in finite dimensions is the so-called ``$H = W$" theorem (\cite{MS}), proved by Meyers and Serrin in 1964. This theorem resolves a critical question about how these spaces are defined, showing that two different historical definitions of these Sobolev spaces actually describe exactly the same one. More precisely, for any $m,n\in\mathbb{N}$, $p\in[1,\infty)$ and every nonempty open subset $\Omega$ of $\mathbb{R}^n$, one has (see Section~\ref{20250704sec1} for notations to be used below)
\begin{enumerate}
\item[1)] The ``$W$" space $W^{m,p}(\Omega)$, defined as the Banach space of all complex valued functions $u$ on $\Omega$ with
weak derivatives up to order $m$ in $L^p(\Omega)$, with norm $\|u\|_{W^{m,p}(\Omega)}=\sum_{|\alpha| \le m}\|\partial^\alpha u\|_{L^p(\Omega)}$;

\item[2)]  The ``$H$" space $H^{m,p}(\Omega)$, defined as the closure of $C^\infty(\Omega)\cap W^{m,p}(\Omega)$ in $W^{m,p}(\Omega)$.
\end{enumerate}
Clearly, $H^{m,p}(\Omega)\subset W^{m,p}(\Omega)$. However, it is not so obvious that $W^{m,p}(\Omega)\subset H^{m,p}(\Omega)$.
If $\partial \Omega$ is smooth, people knew the answer is yes. But what if $\Omega$ has a highly irregular, fractal, or infinitely jagged boundary?
In \cite{MS}, the ``$H = W$" theorem, i.e., $H^{m,p}(\Omega) = W^{m,p}(\Omega)$ was proven for any open set $\Omega$ (no matter how jagged or irregular its boundary is)!

The above theorem makes solving many problems (particularly those related to partial differential equations) much easier. Indeed, when people want to prove a property about a rough, weakly differentiable function, the ``$H = W$" theorem allows to first prove it for perfectly smooth, easy-to-work-with functions, and then simply pass to the limit to prove it holds for the rough function too. This theorem is famous not only for its mathematical importance, but also for having one of the shortest and most direct titles in mathematical history.
Also the paper \cite{MS} is {\bf only about two pages long}, and Meyers and Serrin proved the above theorem {\bf within half a page} using a remarkably clever but elementary technique. Actually, they used only the Friedrichs mollifiers combined with a carefully constructed partition of unity for $\Omega$ subordinate to the collection of open subsets $\{\Omega_{k+1}\setminus\overline{\Omega_{k-1}}\}_{k=1}^\infty$, where
\begin{eqnarray}\label{20260518e1}
		\Omega_k\triangleq\{x \in \Omega:\; \|x\|_{\mathbb{R}^n} < k \text{ and } \operatorname{dist}(x, \partial \Omega) > 1/k\},\qquad \forall\, k \in\mathbb{N},
\end{eqnarray}
and $\Omega_0=\Omega_{-1}=\varnothing$. Using this partition of unity, they shifted the smoothing process slightly inward away from the boundary of $\Omega$, completely bypassing the need for the boundary to be well-behaved.
Before \cite{MS}, people actively used both $W$ and $H$ notations because they represented potentially different spaces. After Meyers and Serrin's seminal paper, the distinction between $H^{m,p}(\Omega)$ and $W^{m,p}(\Omega)$ has essentially vanished.
As Adams and Fournier remarked in \cite[p.~60]{AF}, this result ``{\it ended much confusion about the relationship of these spaces that existed in the literature before that time. It is surprising that this elementary result remained undiscovered for so long}".

The corresponding ``$H=W$" property in infinite dimensions (or its variants) still makes sense, and it is sometimes crucial for solving certain problems.
Indeed, such a property (if it is true) may reduce some problems for Sobolev-type spaces to simpler/smaller spaces of smooth functions. For example, a longstanding open problem in infinite dimensional complex analysis, i.e., $L^2$ estimates for $\overline{\partial}$-operators in infinite dimensions, was recently solved in our previous works \cite{WYZ1, YZ}, for which one of the key ingredients is to establish a ``$H=W$" type property in \cite[Theorem 4.1, p. 37]{WYZ1}. Note that the ``$H=W$" property is also useful for the study of domains of Dirichlet forms in abstract Wiener spaces (see \cite{Hino03, Hino11}), as well as other issues (e.g., \cite{AKR, BDT, DL10, DL13, FU}).

In \cite[pp.~211--217]{Bog98}, V.I. Bogachev gave four definitions of Sobolev spaces on any locally convex space equipped with a centered Radon Gaussian measure and proved that three of them coincide (see \cite[p. 217]{Bog98}). This can be viewed as an infinite-dimensional version of ``$H=W$" in the whole space. To the best of our knowledge, so far the only nontrivial results for the ``$H=W$" property in general infinite-dimensional domains  were obtained by M.~Hino (\cite{Hino03, Hino11}), who showed that smooth cylindrical functions are dense in the Sobolev space $W^{1,2}(G)$ for each convex set $G$ with nonempty interior, or $H$-convex and $H$-open set $G$ in any abstract Wiener space. It is worth noting that the definition of $W^{1,2}(G)$ (\cite{Hino03, Hino11}), in terms of Dirichlet forms, seems not generalizable to arbitrary choices of $m$ and $p$. Also, Hino remarked in \cite[p. 143]{Hino03} that ``{\it the convexity assumption in this study is rather technical and we expect that the claim of the main theorem is true for more general sets, like the set $\{ \varphi > 0\}$ with a smooth function $\varphi$}".

By combining some tools introduced in our recent works \cite{WYZ1, WYZ, YZ}, we obtain one of our main results in this paper, i.e., the following infinite-dimensional version of ``$H=W$":

\begin{theorem}\label{zhuyaodingyi}
It holds that
	$C_{\ell^{2}}^{\infty}( {O}) \bigcap W^{m,p}( {O} ,P ) \bigcap C_{\mathscr{F}}^{\infty}( {O} ,loc )$ is dense in $W^{m,p}( {O},P)$.
\end{theorem}

Compared to the classical ``$H=W$" theorem (\cite{MS}), in order to prove Theorem \ref{zhuyaodingyi}, the following specific difficulties arise:
\begin{enumerate}
\item[1)] The following type of estimates frequently arise
\[\begin{aligned}\label{20260427for3}
 \left\| \varphi\cdot f \right\|_{W^{m,p}\left( {O} ,P\right)}^{p}
&\le C(m,p)\cdot\sup\limits_{\mathbf{x}\in {O}}\left\{\sum\limits_{\left| \alpha \right| \le m}   \left|  \mathbf{a}^{\alpha} \partial^{\alpha}\varphi (\mathbf{x}) \right|^{p}\right\}  \cdot \left\| f \right\|_{W^{m,p}\left( {O} ,P \right)}^{p},
\end{aligned}\]
where $\varphi$ is a smooth function on ${O}$, $f$ is a Sobolev function in $W^{m,p}\left( {O} ,P \right)$ and $C(m,p)$ is a positive constant depending only on $m$ and $p$. We emphasize that, in the setting of infinite dimensions the quantity below
\begin{equation}\label{20260427for2}
 \sum\limits_{\left| \alpha \right| \le m}   \left|  \mathbf{a}^{\alpha} \partial^{\alpha}\varphi (\mathbf{x})\right|^{p}
\end{equation}
is an infinite sum, whose convergence must be carefully addressed. In the case of finite dimensions, this type of estimates also arises, but the corresponding quantity in \eqref{20260427for3} is a finite sum, and hence a trivial smoothness condition guarantees that its supremum on ${O}$ is finite.

\item[2)]  The usual Friedrichs mollifier procedure used in the proof of the classical ``$H=W$" theorem does not work well in the present setting of infinite dimensions.

\item[3)] As indicated by Theorem \ref{gross moguang}, the Schatten-$p$-type estimate for the Gross convolution is sharp for $p=2$, which is of crucial importance for constructing the desired compactly truncated sequence. However, proceeding in this way does not guarantee that the resulting compactly truncated sequence converges in the sense described above. Moreover, the function $\operatorname{dist}(\cdot,\partial{O})$ (which corresponds to $\operatorname{dist}(\cdot, \partial \Omega)$ in \eqref{20260518e1}) fails to satisfy this convergence requirement. Hence, new tools are needed to partition the open set ${O}$.
\end{enumerate}
Our strategies for addressing the above difficulties one by one are as follows:
\begin{enumerate}
\item[1)]  We shall restrict our attention to smooth functions (on ${O}$) for which the sum in \eqref{20260427for2} converges and its supremum over ${O}$ is finite, ensuring that $\varphi\cdot f \in W^{m,p}\left( {O} ,P\right).$ It will be shown in this paper that such functions are dense in the relevant function spaces. In what follows, we shall add the subscript ``$\mathscr{F}$" to the notation for each class of smooth functions satisfying this property.

\item[2)]  We introduce a class of functions (on ${O}$) that are Borel measurable, continuous on every finite-dimensional slice, and for which all partial derivatives exist. Such functions will be referred to as $\mathcal{F}$-continuous functions. As an alternative to the Friedrichs mollification procedure, we employ a modified dimension-reduction approach developed in \cite{WYZ1, YZ}, together with compact truncation and a partition-of-unity argument, to establish the key approximation result.  

\item[3)] We construct a compactly truncated sequence by a new method that is considerably simpler than Goodman's construction in \cite{Goo}, which is based on the Gross convolution. It is worth emphasizing that the truncated functions obtained by our method enjoy stronger regularity properties than those constructed via the Gross convolution. More precisely, the former satisfy \eqref{20260427for2} for every $p\in[1,\infty)$, whereas the latter satisfy \eqref{20260427for2} only for $p\in[2,\infty)$. Combining this construction with a partition-of-unity argument, we can carry out the approximation locally and therefore avoid the need to use a single function to partition the entire domain.
\end{enumerate}

This compact-truncation and partition-of-unity argument can also be used to establish Gauss--Green- or Stokes-type theorems in infinite-dimensional settings; see \cite{WYZ}. 

In the infinite-dimensional setting, it is quite natural, and even more important, to ask whether the set of smooth cylindrical functions is dense in  $W^{m,p}( {O} ,P)$. Indeed, if this is true, then many problems (especially those related to $W^{m,p}( {O} ,P)$) can be reduced to the corresponding ones in finite dimensions, because every cylindrical function is essentially a function of finitely many variables. Unfortunately, as  indicated by Example \ref{examp7.1}, this is not always possible for every open subset ${O}$ of $\ell^2$. A classical result (for the Euclidean space $\mathbb{R}^n$) states that if $\Omega$ satisfies the segment condition, then $C_c^{\infty}( \mathbb{R}^n)$ is dense in $W^{m,p}(\Omega)$ (see \cite[Theorem~3.2, p.~68]{AF}). Stimulated by this, we obtain another main
result in this paper, i.e., the following approximation theorem for any open set satisfying the strong segment condition in Definition \ref{20260427def1}.

\begin{theorem}\label{zhuhanshuguanghuabijin}
Suppose that ${O}$ satisfies the strong segment condition. Then, $\operatorname{span}\mathscr{C}^{\infty}_{c}$ is dense in $W^{m,p}( {O} ,P)$.
\end{theorem}

Our definition of the strong segment condition differs from the classical finite-dimensional definition. More precisely, we use nonzero vectors in $\widetilde{H}$ instead of $\ell^2$, where $\widetilde{H}$ is a proper subset of $\ell^2$. We shall show that the strong segment condition is more general than the regular boundary, $H$-convex and Lipschitz boundary in $\ell^2$, and therefore it is enough for many applications, such as in probability theory. Moreover, we shall provide an example of an open subset ${O}$ of $\ell^2$ satisfying the segment condition but not the strong segment condition, such that $\operatorname{span}\mathscr{C}^{\infty}_{c}$ is not dense in $W^{m,p}({O},P)$.

\subsection{Extensions of Sobolev functions}
When functions in a Sobolev space are defined on a proper subset, a fundamental question is whether they can be extended to the whole space while preserving their regularity and controlling their norms. This question is important because Sobolev spaces on the whole space are usually better understood; the existence of a suitable extension allows one to apply a wide range of established results for the whole space to problems posed on subsets.

With $\Omega(\subset \mathbb{R}^n)$ as above, a linear operator $E_{\O}:W^{m,p}(\O)\to W^{m,p}(\dbR^n)$ is called an extension operator if it is bounded and
 $$
 (E_{\O}f)|_{\Omega} = f\text{ almost everywhere on $\O$},  \quad \forall\; f \in W^{m,p}(\Omega).
 $$
When an extension operator exists, $\Omega$ is called a $W^{m,p}$-extension domain.
Much effort has been made to give criteria for $\Omega$ to be a $W^{m,p}$-extension domain. 
In \cite{Cal61}, Calder\'on proved that Lipschitz domains are $W^{m,p}$-extension domains for $1<p<\infty$ and $ m\ge1$. In \cite[Chapter VI]{Ste70}, Stein showed that Lipschitz domains are $W^{m,p}$-extension domains for $p=1$ and $\infty$. In \cite{Jon81},  Jones introduced the $(\varepsilon, \delta)$-domains and proved that they are $W^{m,p}$-extension domains for $1\le p \le \infty$ and $m\ge1$. The study of the extension problem and its variants in finite dimensions has remained active over the years (e.g., \cite{FIL14,Isr13,KRZ25,KO21}).

To show that $\Omega$ is a $W^{m,p}$-extension domain, a natural approach is to  explicitly construct a bounded extension operator. Indeed, Calder\'on in \cite{Cal61} defined the extension operator locally and using the partition of unity, constructed  the bounded extension operator. Similarly, Stein's extension operator relies on Whitney's decomposition, and he also employed a partition of unity to assemble the global extension operator from locally defined extensions (\cite[Chapter VI]{Ste70}). Jones also adopted Whitney's decomposition in \cite{Jon81}, and by incorporating polynomial fitting on the cubes, he constructed a bounded extension operator with the aid of a partition of unity.

These results naturally raise the question of whether analogous statements continue to hold in an infinite-dimensional setting. However, passing from finite to infinite dimensions introduces substantial obstacles. With the Gaussian measure and infinitely many variables, standard techniques including the change of variables formula and the localization argument are much more delicate than in the finite-dimensional setting. Moreover, the norms of classical extension operators on $\Omega$ generally depend on the dimension $n$; therefore, in order to argue by taking limits in finite dimensions, one needs to seek estimates that are independent of the dimension.

As noted in \cite{BPS14}, the above problem had already been considered by Bogachev, first jointly with Da Prato and Malliavin in the 1990s and later with Lunardi. Apart from the elementary reflection argument for half-spaces (e.g., \cite[Example 2.1, pp.~1028--1029]{BPS14}), only some partial positive results on extensions of $H$-Lipschitzian mappings were obtained (e.g., \cite{Bog99,Bog10,BPS14,Hino11,Kob98}). In \cite{BPS14}, Bogachev, Pilipenko, and Shaposhnikov exhibited a convex domain for which no extension operator exists. However, the existence of extension operators for natural domains such as balls remained open.

In this paper, we also study the following extension problem for Sobolev functions on the open set
$O\subset\ell^2$ given above.
\begin{problem}\label{prob:ext}
	Find conditions for $O$ to be a $W^{1,p}$-extension domain or a $\mathscr{W}^{1,p}$-extension domain. Here $O$ is called a $W^{m,p}$-extension domain if there exists a bounded linear operator
	$$E:W^{m,p}(O,P)\longrightarrow W^{m,p}(\ell^2,P)$$
	such that
	$$(Ef)|_O=f\text{ almost everywhere on $O$},\quad\forall\; f\in W^{m,p}(O,P),$$
	and a $\mathscr{W}^{m,p}$-extension domain is defined similarly by replacing $W^{m,p}$ with $\mathscr{W}^{m,p}$.
\end{problem}

Since the localization and boundary-flattening argument is a standard approach to proving the classical extension theorem, we first generalize such an argument to our infinite-dimensional setting and obtain the following result.
\begin{theorem}\label{thm:extlfb}
If $O$ satisfies the uniform $C^1$-regularity condition in Definition~\ref{def:smbdry}, then it is both a $W^{1,p}$- and a $\mathscr{W}^{1,p}$-extension domain.
\end{theorem}

However, the uniform $C^1$-regularity condition depends heavily on the local finiteness of the coordinate charts that flatten $\partial O$, and is not even satisfied by the unit ball. In fact, a simple generalization of the local reflection argument in finite dimensions fails to solve Problem~\ref{prob:ext} for the unit ball, which indicates that new ideas are necessary.

Consequently, we turn to developing methods that are less dependent on the boundary charts. By introducing a new type of reflection, we obtain the next result.
\begin{theorem}\label{thm:ext}
   If $O$ is a Sobolev reflection domain as in Definition~\ref{def:reflection}, then it is a $\mathscr{W}^{1,p}$-extension domain.
\end{theorem}
As an application of Theorem~\ref{thm:ext} and Proposition~\ref{prop:ellipsoid}, we prove that all open ellipsoids are $\mathscr{W}^{1,p}$-extension domains for every $p\in[1,\infty)$. In particular, this solves an open problem posed in \cite{AMM23,BPS14,CL14}.

An interesting example is given by the real version of the following Hilbert polydisc
\begin{equation}\label{eq:Hilbert polydisc}
\mathbb{D}_2^{\infty}\triangleq\left\{ \mathbf{x}=(x_i)_{i\in\mathbb{N}} \in \ell^{2} :\sup_{i\in \mathbb{N}} \left| x_{i} \right| <1 \right\}.
\end{equation}
In Section~\ref{subsec:Hilbert polydisc}, we will show that $\mathbb{D}_2^{\infty}$ is not a $\mathscr{W}^{m,p}$-extension domain for any $m\in\dbN$ and $p\in[1,\infty)$. We will also prove that $\mathbb{D}_2^{\infty}$ is never a ${W}^{m,p}$-extension domain for each $m\in\dbN$ and $p\in(1,\infty)$. Nevertheless, it is a ${W}^{1,1}$-extension domain. This striking discrepancy, arising from differences in the definitions of Sobolev norms, has no analogue in finite dimensions.


\subsection{Traces of Sobolev functions}

For functions in a Sobolev space defined on a domain with suitably regular boundary, another fundamental question is whether their boundary values can be defined in a meaningful way, whose roots can be tracked back at least to Hadamard's paper \cite{Had} in 1906 on the Dirichlet principle. This question arises naturally from a basic need to formulate boundary value problems, and motivates the study of the trace problem.

The most basic trace problem in finite dimensions concerns the existence of a bounded linear operator
\begin{equation*}
  \operatorname{Tr}_{\O}: W^{1,p}(\Omega) \longrightarrow L^{p}(\partial \Omega)
\end{equation*}
extending the restriction map $u\mapsto u|_{\partial\O}$ for $u\in W^{1,p}(\Omega)\cap C^\infty(\overline{\O})$, where $L^p(\partial\O)$ is defined with respect to the (Lebesgue) surface measure on $\partial\O$. When $C^\infty(\overline{\O})$ is dense in $W^{1,p}(\O)$, such a bounded linear operator is unique whenever it exists and is referred to as the trace operator. Further problems include the characterization of the trace space $\operatorname{Tr}_{\O}(W^{1,p}(\O))$. If $\Omega$ is the half-space $\dbH_n=\{(x_1,\dots,x_n)\in\dbR^n:x_1>0\}$, the existence of the trace operator for $1\le p<\infty$ was proved using the fundamental theorem of calculus and Young's/H\"older's inequality (cf. \cite[Theorem~18.1, p.~592]{Leo17}). When $p=2$, $\operatorname{Tr}_{\dbH_n}(W^{1,2}(\dbH_n))$ can be identified with $H^{1/2}(\dbR^{n-1})$ via Fourier transform (e.g., \cite{Nec67}). When $p$ is not necessarily $2$, the characterization of the trace space relies on Gagliardo's method \cite{Gag57} that identified $\operatorname{Tr}_{\dbH_n}(W^{1,p}(\dbH_n))$ for $1<p<\infty$ with the Besov space $B^{1-1/p}_{p,p}(\dbR^{n-1})$, which is equivalent to the fractional Sobolev space $W^{1-1/p,p}(\dbR^{n-1})$ (see also \cite[Chapter~7]{AF}). These results were generalized to bounded domains with Lipschitz boundaries by flattening the boundary and localization.
See \cite{LP20} for more references and a recent approach. This line of research and related topics have received sustained attention (e.g., \cite{AC17,BGM25,LZ97,LZ98}).

Once again, one encounters numerous difficulties when studying the trace problem in infinite dimensions. The obstacles related to the change-of-variables formula, the localization argument, and dimension-dependent constants also arise in the extension problem considered above. Furthermore, the need to choose a surface measure when formulating the trace problem introduces additional subtleties.


The trace problem in infinite dimensions was considered in \cite{Kre77} for Sobolev-type spaces which are analogues of $H^s(\mathbb{R}^n)$, and in \cite{BDPT09,BDPT11,DL10,DL13} from the perspectives of Dirichlet or Neumann boundary conditions for differential equations. In \cite{CL14}, Celada and Lunardi showed the existence of a bounded trace operator from $W^{1,p}(\mathscr{O},\gamma)$ to $L^q(\partial\mathscr{O},\mu)$ for $1<p<\infty$ and $1\le q<p$, where $\mathscr{O}$ is a sublevel set on a Banach space $X$ with some regularity assumptions on the boundary $\partial\mathscr{O}$, $\gamma$ is a nondegenerate centered Gaussian measure on $X$, and $\mu$ is the surface measure of Feyel and de~La~Pradelle \cite{FdLP92}. As pointed out in \cite{CL14}, {\it it was unknown whether the trace operator maps $W^{1,p}(\mathscr{O},\gamma)$ into $L^p(\partial\mathscr{O},\mu)$ when $\mathscr{O}$ is the unit ball in a Hilbert space}, and this was related to the extension problem for Sobolev functions on the unit ball, which was also open at that time. See also \cite{Fer19} for subsequent progress. More recently, Addona, Menegatti, and Miranda characterized Sobolev functions with zero boundary trace in terms of their trivial extensions to the ambient Wiener space in \cite{AMM23}, where they remarked that ``{\it nowadays there is no other result about the definition of a trace operator from $W^{1,p}(\mathscr{O},\gamma)$ onto $L^p(\mathscr{O},\mu)$ for more general subsets $\mathscr{O}$, and also the case $p=1$ is not reached}''.


In this paper, we study the trace problem for functions in $W^{1,p}(O,P)$ and $\mathscr{W}^{1,p}(O,P)$. In the cases of our interest (including many commonly encountered situations), $O$ satisfies the strong segment condition in Definition \ref{20260427def1} and thus the set $\mathscr{C}_b^\infty$ of bounded cylindrical functions is dense in both $W^{1,p}(O,P)$ and $\mathscr{W}^{1,p}(O,P)$ by Theorem~\ref{zhuhanshuguanghuabijin}. When $\partial O$ is a $C^1$-differentiable surface of $\ell^2$ with codimension $1$ as in Definition~\ref{20241121def1}, we denote by $\mu_{\partial O}$ the surface measure on $\partial O$ as in Section~\ref{subsec:surfmes}, which is a Gaussian infinite-dimensional version of the ordinary surface measure in finite dimensions. Our choice of surface measure is different from that in \cite{CL14} (see Remark~\ref{rmk:surfmes} for a discussion on the difference). Denote by $L^p(\partial O, \mu_{\partial O})$ the space of (equivalence classes of) $L^p$-integrable functions on $\partial O$ with respect to $\mu_{\partial O}$.
The corresponding trace problem can be formulated as follows.
\begin{problem}\label{prob:tr}
   Suppose that $\partial O$ is a $C^1$-differentiable surface of $\ell^2$ with codimension $1$. Find conditions on $O$ under which there exist bounded linear operators
\[\begin{aligned}
	\operatorname{Tr}:W^{1,p}(O,P)&\longrightarrow L^{p}(\partial O,\mu_{\partial O}),\\
	\mathscr{T}\!\mathrm{r}:\mathscr{W}^{1,p}(O,P)&\longrightarrow L^{p}(\partial O,\mu_{\partial O}),
\end{aligned}\]
such that
\[\operatorname{Tr}(u)=\mathscr{T}\!\mathrm{r}(u)=u|_{\partial O},\quad\forall\; u\in\mathscr{C}_b^\infty.\]
\end{problem}

Note that every nonzero constant function belongs to both $W^{1,p}(O,P)$ and $\mathscr{W}^{1,p}(O,P)$. Thus any domain $O$ for which either of the operators in Problem~\ref{prob:tr} exists for some $p\ge 1$ satisfies $\mu_{\partial O}(\partial O)<\infty$. In Section~\ref{subsec:infsurf}, we provide some examples of bounded open subsets $O$ (of $\ell^2$) with smooth boundaries and $\mu_{\partial O}(\partial O)=\infty$, so that neither of the operators in Problem~\ref{prob:tr} exists for such subsets.

As in the extension problem, we first generalize the classical localization and boundary-flattening argument to our infinite-dimensional setting and obtain the following result.

\begin{theorem}\label{thm:trlfb}
   If $\partial O$ is uniformly of class $C^1$ as in Definition~\ref{def:unc1}, then the operators $\operatorname{Tr}$ and $\mathscr{T}\!\mathrm{r}$ in Problem~\ref{prob:tr} exist.
\end{theorem}

The condition for $\partial O$ to be uniformly of class $C^1$ also depends heavily on the local finiteness of the coordinate charts that flatten $\partial O$. It is closely related to, but different from the uniform $C^1$-regularity condition.
Therefore, we need once more a different approach to prove the trace theorems for open balls. It is a natural and classical strategy for trace theorems to relate interior and boundary integrals by the fundamental theorem of calculus or its higher-dimensional generalizations, as is also adopted in \cite{CL14}. We use an infinite-dimensional Gauss--Green type theorem (cf.~\cite{Goo,WYZ}) to prove the next result.

\begin{theorem}\label{thm:tr}
	If $O$ is $(W^{1,p},L^p)$-admissible (\resp $(\mathscr{W}^{1,p},L^p)$-admissible) as in Definition~\ref{def:adm}, then
	the operator $\operatorname{Tr}$ (\resp $\mathscr{T}\!\mathrm{r}$) in Problem~\ref{prob:tr} exists.
\end{theorem}
A distinctive feature of our approach to Theorem~\ref{thm:tr} is the introduction of an auxiliary vector field, whose choice provides additional flexibility in deriving the estimates.

As an application of Theorem~\ref{thm:tr} and Proposition~\ref{prop:conv}, for a class of convex open sets including all open ellipsoids, we prove the existence of the operators in Problem~\ref{prob:tr} for every $p\in[2,\infty)$. When $1\le p<2$, we also show in Section~\ref{sec:ball} that neither of the operators in Problem~\ref{prob:tr} exists for open balls, which exhibits the complexity of trace problems in infinite dimensions.

\subsection{Norm equivalence}

The next question we address is whether the full Sobolev norm, which involves all derivatives up to order $m$, can be replaced by the graph norm involving only the function itself and its highest derivative.

This question is classical in finite dimensions. For $u\in W^{m,p}(\Omega)$, set
\[
\|u\|_{W^{m,p}(\Omega)}\triangleq\sum_{j=0}^{m}\left\lVert D^j u\right\rVert_{L^p(\Omega)},
\qquad
\|u\|_{G^{m,p}(\Omega)}\triangleq\|u\|_{L^p(\Omega)}+\left\lVert D^m u\right\rVert_{L^p(\Omega)},
\]
where $D^j u\triangleq(\partial^\alpha u)_{|\alpha|=j}$ denotes the collection of all order $j$ weak derivatives of $u$ and $\left\lVert D^j u\right\rVert_{L^p(\Omega)}\triangleq\left\lVert\,\left| D^j u\right|_*\right\rVert_{L^p(\Omega)}$ is defined via any fixed norm $\left|\cdot\right|_*$ on the corresponding finite-dimensional space. A standard form of the classical norm equivalence theorem is the following; see, for example, \cite[\S 1.1.16, p.~26]{Mazya2011}, where the general equivalent-norm theorem is attributed to Sobolev.

\begin{theorem}\label{thm:norm classical}
If $\Omega\subset\mathbb R^n$ is a bounded domain with the cone property (see \cite[\S 1.1.9, Definition 1, p.~15]{Mazya2011}), then there exists $C=C(\Omega,n,m,p)\ge1$ such that
\[
\|u\|_{G^{m,p}(\Omega)}
\le
\|u\|_{W^{m,p}(\Omega)}
\le
C\,\|u\|_{G^{m,p}(\Omega)},
\qquad \forall u\in W^{m,p}(\Omega).
\]
More generally, if $F$ is a continuous seminorm on $W^{m,p}(\Omega)$ satisfying
$F(Q)>0$ for every nonzero polynomial $Q$ on $\mathbb R^n$ of degree at most $m-1$,
then the norm
\[
u\longmapsto F(u)+\left\lVert D^m u\right\rVert_{L^p(\Omega)}
\]
is equivalent to $\lVert\cdot\rVert_{W^{m,p}(\Omega)}$.  
\end{theorem}

Roughly speaking, the key point of Theorem~\ref{thm:norm classical} is that $D^m u$ determines $u$ only modulo a polynomial of degree at most $m-1$, and the zeroth-order term removes precisely this finite-dimensional ambiguity. The classical theory realized this principle via two ingredients: generalized Poincar\'e inequalities control $u$ modulo such polynomials, while the Gagliardo--Nirenberg interpolation inequalities \cite{Gagliardo1958,Nirenberg1959} control each intermediate derivative between the zeroth and highest orders. Together with standard localization or extension arguments on regular domains, these inequalities yield the equivalence in Theorem~\ref{thm:norm classical} (see also \cite{AF,Mazya2011}).  The equivalence is important because it simplifies {\it a priori} estimates and closedness arguments for differential operators, and allows one to recover all intermediate derivatives from only the two endpoint quantities.

For a finite-dimensional Gaussian measure, the same result remains valid, but dimension dependence becomes important.  Addona, Muratori, and Rossi \cite{AMR22} proved the Gaussian Sobolev norm equivalence on $\mathbb R^n$ for every $p\in[1,\infty)$ and every differentiation order, using Gaussian integration by parts and Euclidean Sobolev inequalities.  The constants in \cite[Theorem~2.4]{AMR22}, however, grow with $n$.  Thus the fixed-dimensional theorem does not pass directly to an infinite-dimensional space by cylindrical approximation.

The infinite-dimensional whole-space problem is naturally formulated in the framework of Malliavin.
The Sobolev theory of Wiener functionals was developed in a series of foundational works, including \cite{Shigekawa1980,Stroock1981,SugitaChar1985,Sugita1985,Watanabe1984}. For $1<p<\infty$, the equivalence between the full Malliavin--Sobolev norm $\lVert\cdot\rVert_{\mathcal D(k,p)}$ and the corresponding graph norm $\lVert\cdot\rVert_{\mathcal G(k,p)}$ was established in the 1980s and follows from Meyer's inequalities for the Ornstein--Uhlenbeck operator \cite{Meyer1982,Meyer1984}; see also \cite{Gundy1986,Pisier1988,Maas2010,PronkVeraar} for alternative approaches and extensions. The endpoint $p=1$, however, lies outside the scope of the classical Meyer--Riesz transform method and has remained open for decades. Addona, Muratori and Rossi proved the case $p=1$, $k=2$ in infinite dimensions and left the higher-order cases $k\ge3$ open \cite{AMR22}. Related weighted Gaussian results for $1<p<\infty$ were obtained in \cite{AddonaWeighted2025}.

A separate line of research concerns Sobolev spaces on subsets of infinite-dimensional Gaussian spaces (e.g., \cite{BPS14,Hino03,Hino11,DaPratoLunardi2014,DaPratoLunardi2015}). However, to the best of our knowledge, existing works do not yield an all-order equivalence between the full Sobolev norm and the graph norm on general proper open convex subsets.

The obstruction for $p=1$ reflects the difference between finite and infinite dimensions.  In finite dimensions, the norm equivalence can be obtained abstractly from compact Sobolev embeddings, or quantitatively from interpolation estimates controlling each intermediate derivative; the constants in such arguments are allowed to depend on the dimension of the domain.  In infinite-dimensional Gaussian analysis, by contrast, one needs estimates that are uniform over all finite-dimensional projections.  For $1<p<\infty$, Meyer's inequalities provide exactly such dimension-free control.  At $p=1$ this mechanism breaks down, while the direct finite-dimensional estimates of \cite{AMR22} are not dimension free.  

In this paper, we develop a new method that gives dimension-free finite-dimensional estimates. Let $\Omega\subset\mathbb R^n$ be nonempty, open and convex, let $\gamma_n$ be the standard Gaussian measure, and set
\[
\mathcal{J}_0(f)=\|f\|_{L^p(\Omega,\gamma_n)},
\qquad
\mathcal{J}_k(f)=\bigl\|\,\|D^kf\|_k\,\bigr\|_{L^p(\Omega,\gamma_n)},\qquad k\in\dbN
\]
(see Section~\ref{sec:norm} for the notations involved).
We first obtain a dimension-free vector-valued Poincar\'e inequality (Theorem~\ref{poincare}) on the normalized restriction of $\gamma_n$ to $\Omega$ by Caffarelli's contraction theorem \cite[Theorem~11]{Caf00}. Then we use an iterative argument to decompose $f$ into a polynomial $Q$ of degree at most $m-1$ and control the derivatives of $f-Q$ by $D^m f$ (see Theorem~\ref{bijin}). To control $Q$, we combine the Hermite expansion and hypercontractivity with a small-ball estimate of Carbery--Wright \cite{CW} (see Theorem~\ref{weishuguanguji}). As a consequence, we obtain the following dimension-free result.

\begin{theorem}\label{thm:intro-dimension-free}
Let $p\in[1,\infty)$, $m\ge2$, and $\eta\in(0,1]$.  There exists a constant $C_{p,m,\eta}>0$ independent of $n$, such that for every $n\in\mathbb N$, every nonempty open convex set $\Omega\subset\mathbb R^n$ satisfying $\gamma_n(\Omega)\ge\eta$, and every polynomial $f$ on $\mathbb R^n$,
\[
\mathcal{J}_0(f)+\mathcal{J}_m(f)\le\sum_{j=0}^{m}\mathcal{J}_j(f)\le C_{p,m,\eta}\bigl(\mathcal{J}_0(f)+\mathcal{J}_m(f)\bigr).
\]
\end{theorem}

Using our cylindrical density result proved in Section~\ref{sec:HW}, we establish the equivalence of Sobolev norms for $\mathscr{W}^{m,p}({O},P)$ when ${O}$ is convex.

\begin{theorem}\label{thm:norm}
Let $p\in[1,\infty)$ and let ${O}\subset\ell^2$ be nonempty, open and convex.
Then for every $m\ge2$, there exists $C=C({O},p,m)>0$ such that
\[
C^{-1}\|f\|_{\mathscr{W}^{m,p}({O},P)}\le \|f\|_{L^p({O},P)}
+\left\|\sqrt{\sum_{|\alpha|=m}\left| \mathbf{a}^{\alpha} \partial^{\alpha}f \right|^2}\right\|_{L^p({O},P)}\le C\|f\|_{\mathscr{W}^{m,p}({O},P)}
\]
for all $f\in \mathscr{W}^{m,p}({O},P)$.
\end{theorem}

On the contrary, we show in Proposition~\ref{1f3g} that a similar result does not hold for $W^{m,p}(O,P)$. As far as we know, Theorem~\ref{thm:norm} is the first higher-order graph-norm equivalence for Sobolev spaces on a proper open subset of an infinite-dimensional space.

Moreover, on the whole Gaussian space, our method applies directly to smooth cylindrical random variables and therefore solves the $p=1$ case left open in previous works. 

\subsection{Embeddings and logarithmic Sobolev inequalities}

Sobolev embedding theorems describe how the control of derivatives improves the integrability of functions and provides compactness for bounded families. In finite dimensions, the classical embedding $W^{1,p}(\mathbb{R}^n)\hookrightarrow L^{np/(n-p)}(\mathbb{R}^n)$ for $1\le p<n$, and its variants together with compact counterparts on bounded regular domains (e.g., \cite[Chapters~4 and 6]{AF}) are fundamental tools in the analysis of partial differential equations. However, such embeddings depend on the dimension: for fixed $p$, the exponent $np/(n-p)$ tends to $p$ as $n\to\infty$. Moreover, even in the finite-dimensional Gaussian setting, the Sobolev embedding $W^{1,p}\hookrightarrow L^q$ can fail to hold for any $1\le p<q<\infty$ (see \cite[Theorem~3.10]{LZZ23}). Thus, embeddings of this type are no longer the appropriate objects of study in infinite dimensions. 

In his celebrated paper \cite{Gro75}, Gross proved the following logarithmic Sobolev inequality
\begin{equation}\label{eq:Gross}
\int_{\dbR^n}|f|^2\log|f|\dd\gamma_n\le \int_{\dbR^n}|\nabla f|_2^2\dd\gamma_n+\|f\|_{L^2(\dbR^n,\gamma_n)}^2\log\|f\|_{L^2(\dbR^n,\gamma_n)}
\end{equation}
for every $f\in W^{1,2}(\dbR^n,\gamma_n)$, where $|\cdot|_2$ denotes the usual $\ell^2$-norm.
In contrast to the classical Sobolev inequality, the constant in \eqref{eq:Gross} is independent of $n$, making embeddings into Orlicz spaces a natural substitute for classical Sobolev embeddings in infinite dimensions.

Considerable effort has been devoted to generalizing or improving \eqref{eq:Gross}. Higher-order logarithmic Sobolev inequalities and Orlicz embeddings in finite dimensions was established by Adams in \cite{Ada79}, where the constants depend on the dimension. Ledoux \cite{Led88} obtained dimension-free first-order Gaussian Sobolev embedding into the Orlicz space $L^p\log^{p/2}L$ for the full range $1\le p<\infty$. Cianchi and Pick~\cite{CP09} studied optimal Gaussian Sobolev embeddings in rearrangement-invariant spaces, including Orlicz and Lorentz--Zygmund spaces. See also \cite{Led96} and the references therein for some related works.

Compact embedding properties in infinite dimensions have also attracted much attention (e.g., \cite{AMR22,DMN92,DZ95,Pes93,LZZ23}). When $p=2$, Da Prato and Zabczyk~\cite{DZ95} characterized compact Sobolev embeddings of weighted Sobolev spaces on a Hilbert space in terms of the relation between the Gaussian measure and gradient weights. Extending this picture to $\ell^r$ and $p\in[1,\infty)$, Luo, Zhang and Zhao~\cite{LZZ23} characterized continuous and compact logarithmic Sobolev embeddings of the weighted Sobolev spaces $W_{\mathbf{b}}^{1,p}(\ell^r,\gamma_{\mathbf{a}})$ in terms of the ratio $h(\mathbf{b}/\mathbf{a})$ between Gaussian variances $\mathbf{a}$ and gradient weights $\mathbf{b}$ (see \eqref{eq:r_M_delta}).

Using Caffarelli's contraction theorem \cite[Theorem~11]{Caf00} and the argument in \cite{LZZ23}, we prove the following theorem, with notations as in Section~\ref{subsec:embed}.
\begin{theorem}\label{thm:log}
		If ${O}$ is a convex set, then for every $1\le p<\infty$, the following conclusions hold.
		\begin{enumerate}
			\item[\upshape(i)] There is a continuous embedding
			\begin{equation*}
				W_{\mathbf{b}}^{1,p}({O},P)\hookrightarrow  L^p\log^{p/2}L({O},P)
			\end{equation*}
			if and only if $h(\mathbf{b}/\mathbf{a})>0$.
			
			\item[\upshape(ii)] For every $0\le q<p/2$, there is a compact embedding
			\[
			W_{\mathbf{b}}^{1,p}({O},P)\Subset L^p\log^q L({O},P)
			\]
			if and only if $\lim_{k\to\infty}r_k=\infty.$
		\end{enumerate}
	\end{theorem}

To the best of our knowledge, Theorem~\ref{thm:log} is the first logarithmic Sobolev embedding result for Sobolev spaces on a proper open subset of an infinite-dimensional space. The assumption $q<p/2$ is necessary in Theorem~\ref{thm:log} (see Theorem~\ref{20261002thm1} and Proposition~\ref{prop:sharp}). Our proof of Theorem~\ref{thm:log} also answers the questions left unresolved in \cite[Remark~5.3]{LZZ23} concerning the $p=1$, $0\le q<1/2$ case and the $q=p/2$ case for compact embeddings.

Since $W_{\mathbf{a}}^{1,p}({O},P)=\mathscr{W}^{1,p}({O},P)$ and $r_k\equiv 1$ when $\mathbf{a}=\mathbf{b}$, Theorem~\ref{thm:log} implies that $\mathscr{W}^{1,p}({O},P)$ is not compactly embedded into $L^p\log^q L({O},P)$ for every $0\le q<p/2$. In fact, we will prove in Theorem~\ref{20260928thm1} that for every nonempty open set $O$, neither ${W}^{1,p}({O},P)$ nor $\mathscr{W}^{1,p}({O},P)$ can be compactly embedded into the Orlicz space $L^\Phi({O},P)$ for any Young function $\Phi$. This is a surprising phenomenon, revealing a fundamental distinction between the infinite and finite dimensions.

A related trace problem was studied by Feo and Posteraro~\cite{FP13}, who provided logarithmic Sobolev trace inequalities on sufficiently regular finite-dimensional domains. As an application of Theorem~\ref{thm:log} and our results proved in Section~\ref{sec:tr}, in Theorems~\ref{thm:logtrlfb} and \ref{thm:logtr}, we establish logarithmic Sobolev trace inequalities for a broad class of open subsets of $\ell^2$.

In addition to these results, we provide in Section~\ref{subsec:ex} some examples of domains that behave unexpectedly. The first result in this direction is as follows.
\begin{theorem}\label{thm:logmain}
There exists a bounded, path-connected open set
${O}$ with $C^\infty$ boundary
such that for any $p\in[1,\infty)$, neither $W^{1,p}({O},P)$ nor $\mathscr W^{1,p}({O},P)$ is continuously embedded into the Orlicz space $L^\Phi({O},P)$ for any Young function $\Phi$ satisfying
\begin{equation}\label{eq:strictPhi}
                \lim_{t\to\infty}\frac{\Phi(t)}{t^p}=+\infty.
\end{equation}
\end{theorem}

In particular, Theorem~\ref{thm:logmain} implies that the integrability of functions in $W^{1,p}({O},P)$ or $\mathscr W^{1,p}({O},P)$ cannot be improved by any logarithmic or even iterated logarithmic power. This is substantially different from the finite-dimensional case.

In the case of the open ball $B_r$ of radius $r>0$, we also show in Theorem~\ref{thm:logtrace2} that the integrability of the traces of functions in $W^{1,2}(B_r,P)$ or $\mathscr{W}^{1,2}(B_r,P)$ cannot be improved by any logarithmic or even iterated logarithmic power. However, for every $p\in(2,\infty)$, the trace space $\mathscr{T}\!r(\mathscr{W}^{1,p}(B_r,P))$ can be continuously embedded into $L^p\log^{(p-2)/2}L(B_r,\mu_{\partial B_r})$ by Theorem~\ref{thm:logtr}.

\section{Preliminaries}\label{20250704sec1}

We introduce some notations to be used throughout this paper. Denote by $\mathbb{N}$ and $\mathbb{N}_{0}$ respectively the set of all positive integers and that of all nonnegative integers.
For any topological space $(X,\mathcal{T})$ and any subset $A$ of $X$, denote by $A^{\circ}$, $\partial A$ and $\mathscr{B}(X)$ the interior of $A$, the boundary of $A$ and the Borel $\sigma$-algebra on $X$ generated by $\mathcal{T}$, respectively. Denote by $\mathbf{1}_A$ the indicator function of $A$, i.e., $\mathbf{1}_A(x)=1$ if $x\in A$, and $\mathbf{1}_A(x)=0$ otherwise. Recall that $X$ is called a Lindel\"of space if every open cover of $X$ has a countable subcover (see \cite[p. 50]{Kel}). From \cite[Theorem 15, p. 49]{Kel}, it follows that every separable metric space is a Lindel\"of space. For any $p\in[1,\infty)$ and any $\sigma$-finite Borel measure $\mu$ on $X$, denote by $L^p(X,\mu)$ the Banach space of all
(equivalence classes of) real-valued, Borel measurable functions $f$ on $X$ for which $\int_X|f|^p\mathrm{d}\mu<\infty$. In particular, when $\mu$ is the usual Lebesgue measure on $\Omega\subset\dbR^n$, we simply write $L^p(\Omega,\mu)$ as $L^p(\O)$.

\subsection{Infinite-dimensional setting}

For any nonempty set $I\subset \mathbb{N}$, write
\begin{eqnarray*}
	\ell^2(I)\triangleq \left\{\mathbf{x}=(x_i)_{i\in I}\in \mathbb{R}^{I}:\sum_{i\in I}|x_i|^2<\infty\right\},
\end{eqnarray*}
which is a Hilbert space equipped with the natural inner product
$$
\left\langle \mathbf{x},\mathbf{y}\right\rangle_{\ell^2(I)}
\triangleq \sum_{i\in I}x_i y_i,\qquad
\forall\mathbf{x}=(x_i)_{i\in I},\,
\mathbf{y}=(y_i)_{i\in I}\in\ell^2(I),
$$
and the induced norm
\begin{eqnarray*}
	\left\lVert \mathbf{x}\right\rVert _{\ell^2(I)}\triangleq \sqrt{\left\langle \mathbf{x},\mathbf{x}\right\rangle_{\ell^2(I)}}= \left(\sum_{i\in I}|x_i|^2\right)^{{1}/{2}}.
\end{eqnarray*}
We simply write $\ell^2$, $\langle\cdot,\cdot\rangle$ and $\lVert\cdot\rVert$ instead of $\ell^2(\mathbb{N})$, $\langle\cdot,\cdot\rangle_{\ell^2(\dbN)}$ and $\left\lVert \cdot\right\rVert _{\ell^2(\mathbb{N})}$, respectively. Let $\supp(\mathbf{x})\triangleq\{i\in\dbN:x_i\neq 0\}$ denote the support of $\mathbf{x}$.
For subsets $A,B$ of $\ell^2$, write $\operatorname{dist}(A,B)\triangleq\inf\{\|\mathbf{x}-\mathbf{y}\|:\mathbf{x}\in A,\,\mathbf{y}\in B\}$, with the convention $\operatorname{dist}(A,B)=\infty$ if $A$ or $B$ is empty.
Denote by $B( \mathbf{x} ,r)$ the open ball in $\ell^{2}$ centered at $\mathbf{x}(\in \ell^{2})$ with radius $r>0$, and we simply write $B( \mathbf{0} ,r)$ as $B_r$.

For any $i,j\in\mathbb{N}$, set $\delta_{ij} \triangleq 1$, if $i=j$, and $\delta_{ij} \triangleq 0$, if $i\neq j.$ For each $k\in\mathbb{N}$, define $\textbf{e}_k\triangleq(\delta_{k,i})_{i\in\mathbb{N}}$. Then $\{\textbf{e}_k\}_{k=1}^{\infty}$ is an orthonormal basis of $\ell^2$. Let sgn$(x)\triangleq \frac{x}{|x|}$ for each $x\in\mathbb{R}\setminus\{0\}$, and sgn$(0)\triangleq 0$.

Throughout this paper, unless otherwise stated, we will fix a nonempty open subset $O$ of $\ell^{2}$ and a sequence $\left\{ a_{i} \right\}_{i=1}^{\infty}$ of positive numbers satisfying
\begin{equation}\label{20260424for1}
\sum\limits_{i=1}^{\infty} a_{i}^2< \infty.
\end{equation}

The following notion from \cite[Definition 2.1, p. 6]{WYZ1} will play a fundamental role in the sequel.
\begin{definition}\label{def of bounded contained}
A set $S\subset \ell^2$ is said to be \textbf{uniformly included} in $O$, denoted by $S\stackrel{\circ}{\subset} O$, if there exist $r,R\in(0, \infty)$ such that $\cup_{\mathbf{x}\in S}B(\mathbf{x},r)\subset O$ and $S\subset B_R$.
\end{definition}

We need two Hilbert spaces $H$ and $\widetilde{H}$ defined as follows:
\begin{equation}\label{20260428for2}
 H \triangleq \left\{ \mathbf{x}=(x_i)_{i\in\mathbb{N}} \in \ell^{2} :\sum\limits_{i=1}^{\infty} \frac{x_{i}^{2}}{a_{i}^{2}} <\infty \right\}
\end{equation}
equipped with the inner product and norm
\[\left\langle \mathbf{x} ,\mathbf{y} \right\rangle_{H} \triangleq \sum_{i=1}^{\infty} \frac{x_{i}y_{i}}{a_{i}^{2}} ,\qquad \left\| \mathbf{x} \right\|_{H} \triangleq \sqrt{\left\langle \mathbf{x} ,\mathbf{x} \right\rangle_{H}}\]
for any $\mathbf{x}=(x_i)_{i\in\mathbb{N}},\,\mathbf{y}=(y_i)_{i\in\mathbb{N}}\in H$,
 and
\begin{equation}\label{20260428for1}
\widetilde{H} \triangleq \left\{ \mathbf{x}=(x_i)_{i\in\mathbb{N}} \in \ell^{2} :\sum\limits_{i=1}^{\infty} \frac{x_{i}^{2}}{a_{i}^{4}} <\infty \right\}
\end{equation}
equipped with the inner product and norm
\[\left\langle\mathbf{x} ,\mathbf{y} \right\rangle_{\widetilde{H}} \triangleq \sum_{i=1}^{\infty} \frac{x_{i}y_{i}}{a_{i}^{4}} ,\qquad \left\| \mathbf{x} \right\|_{\widetilde{H}} \triangleq \sqrt{\left\langle \mathbf{x} ,\mathbf{x} \right\rangle_{\widetilde{H}}}\]
for any $\mathbf{x}=(x_i)_{i\in\mathbb{N}},\,\mathbf{y}=(y_i)_{i\in\mathbb{N}}\in \widetilde{H}$.

Using the method of product measures (e.g., \cite[Chapter~1]{Da06}, \cite[Section~2.2]{YZ}), one obtains a probability measure $P$ on $\ell^2$. Following (8) in \cite[p. 523]{YZ}, for each $k\in\mathbb{N}$, we set $\mathcal{N}^k\triangleq\prod\limits_{i=1}^{k}\mathcal{N}_{a_i}$, where
$$
\mathcal{N}_a(B)\triangleq \frac{1}{\sqrt{2\pi a^2}}\int_Be^{-\frac{x^2}{2a^2}}\,\mathrm{d}x,\qquad\forall\, B\in\mathscr{B}(\mathbb{R}),\;a\in(0,\infty).
$$
Since $\ell^2$ can be identified with $\mathbb{R}^k\times \ell^{2}(\mathbb{N}\setminus\{1,2,\dots,k\})$, we have the decomposition  $P=\mathcal{N}^k\times P^{\widehat{1,2,\dots,k}}$.  Here $P^{\widehat{1,2,\dots,k}}$ denotes the product measure obtained by omitting the $1,2,\dots,k$-th components, i.e., it is the restriction of the product measure $\prod\limits_{j\in\mathbb{N}\setminus\{1,\dots,k\}}\mathcal{N}_{a_j}$ to the space
$$
\left(\ell^{2}(\mathbb{N}\setminus\{1,\dots,k\}),\mathscr{B}\big(\ell^{2}(\mathbb{N}\setminus\{1,\dots,k\})\big)\right).
$$
Note that $H$ is the Cameron--Martin space of $P$.

We denote by $\mathbb{N}_0^{(\mathbb{N})}$ the set
\[\left\{ \alpha =( \alpha_{i})_{i\in\mathbb{N}} :\alpha_{i} \in \mathbb{N}_{0},\,\forall\, i\in \mathbb{N}, \text{ and } \alpha_{i} =0,\;\forall\,i\ge k \text{ for some }k\in \mathbb{N}\right\}.\]
In the sequel, we always abbreviate $\alpha =( \alpha_{i})_{i\in\mathbb{N}}$ as $\alpha$. For any $\alpha \in \mathbb{N}_{0}^{\left( \mathbb{N} \right)}$ and $\mathbf{x}=\left( x_{i} \right)_{i\in \mathbb{N}} \in \ell^{2}$, we write
$$\mathbf{x}^{\alpha} \triangleq x_{1}^{\alpha_{1}}x_{2}^{\alpha_{2}}\cdots,\qquad \mathbf{a}^{\alpha} \triangleq a_{1}^{\alpha_{1}}a_{2}^{\alpha_{2}}\cdots ,\qquad \left| \alpha \right|\triangleq \sum_{i=1}^{\infty} \alpha_{i},\qquad\alpha !\triangleq\alpha_{1} !\alpha_{2} !\cdots,$$
where $0!=1$. For any $\alpha ,\beta \in \mathbb{N}_{0}^{\left( \mathbb{N} \right)}$, we say $\alpha \le \beta$ if $\alpha_{j} \le \beta_{j}$ for all $j\in \mathbb{N}$. In this case, we write $C_{\beta}^{\alpha}\triangleq\frac{\beta !}{\left( \beta -\alpha \right) !\alpha !}$.

For a real-valued function $f$ on $O$, any $\mathbf{x}=(x_i)_{i\in\mathbb{N}}\in O$ and $\alpha \in\mathbb{N}_{0}^{\left( \mathbb{N} \right)}$, we employ the following notations:
$$
\delta_{i} f(\mathbf{x})\triangleq \frac{\partial f(\mathbf{x})}{\partial x_{i}} -\frac{x_{i}}{a_{i}^{2}} f(\mathbf{x}),\qquad \partial^{\alpha}f(\mathbf{x})\triangleq \frac{\partial^{\left| \alpha \right|} f(\mathbf{x})}{\partial x_{1}^{\alpha_{1}}\partial x_{2}^{\alpha_{2}}\cdots},\qquad\delta^{\alpha}f(\mathbf{x}) \triangleq \delta_{1}^{\alpha_{1}} \delta_{2}^{\alpha_{2}} \cdots f(\mathbf{x}).
$$
Here $\frac{\partial f(\mathbf{x})}{\partial x_{i}} $ is the first order partial derivative of $f(\cdot)$ (at $\mathbf{x}$) defined as follows:
$$
\begin{array}{ll}
\displaystyle\frac{\partial f(\mathbf{x})}{\partial x_i}\triangleq\lim_{\mathbb{R}\ni \tau\to 0}\frac{f(\mathbf{x}+\tau\textbf{e}_i )-f(\mathbf{x})}{\tau},
\end{array}
$$
which is sometimes abbreviated as $\partial_{x_i}f(\mathbf{x})$.
The higher order partial derivative $\frac{\partial^{\left| \alpha \right|} f(\mathbf{x})}{\partial x_{1}^{\alpha_{1}}\partial x_{2}^{\alpha_{2}}\cdots}$ is defined by induction. Moreover, for any $k\in\mathbb{N},\,i_1,\dots,i_k\in\mathbb{N}$, the higher order partial derivative $\frac{\partial^kf(\mathbf{x})}{\partial x_{i_1}\partial x_{i_2}\dots \partial x_{i_k}}$ is also defined by induction. In particular, when $\alpha = \mathbf{0}$, we write
$\mathbf{a}^{\alpha}\triangleq 1$, $\delta^{\alpha} f\triangleq f$ and $\partial^{\alpha}f\triangleq f.$

In the sequel, as in \cite[p.~7 and p.~17]{WYZ1}, for each $k\in\mathbb{N}$, we denote by $C^{k} ( {O})$ ({\it resp.} $C^{k}_{b}({O})$) the set of all real-valued functions $f$ on ${O}$ for which $f$ itself and all of its  partial derivatives up to order $k$ are continuous ({\it resp.} bounded and continuous) on ${O}$ with respect to the $\ell^2$ topology. We denote by $C( {O})$ ({\it resp.} $C_{b}({O})$) the set of all real-valued continuous (\resp bounded continuous) functions on ${O}$ with respect to the $\ell^2$ topology. We say $f\in C^{k} (\overline{O})$ (\resp $f\in C (\overline{O})$) if $f\in C^{k} (U)$ (\resp $f\in C (U)$) for some open subset $U$ of $\ell^2$ with $U\supset O$. Write
\begin{equation}\label{20260508for2}
C^{\infty} ({O})\triangleq \bigcap_{j=1}^{\infty}C^{j} ({O}),\qquad C^{\infty}_b({O})\triangleq \bigcap_{j=1}^{\infty}C^{j}_b({O}).
\end{equation}
For each $k\in\mathbb{N}_0$, write
$$
C^{k}_{0}({O})  \triangleq \left\{ f\in C^{k}_{b}\left({O} \right)  :\; f^{-1}(\mathbb{R}\setminus\{0\})\stackrel{\circ}{\subset} {O}\right\},\qquad C^{\infty}_0({O})\triangleq \bigcap_{j=1}^{\infty}C^{j}_0( {O}).
$$

Denote by $C^{k}( \mathbb{R}^n)$ (\resp $C^{k}_b( \mathbb{R}^n)$) the set of all real-valued functions $f$ on $\mathbb{R}^n$ for which $f$ itself and all of its  partial derivatives up to order $k$ are continuous (\resp bounded and continuous) on $\mathbb{R}^n$. Write
$$
C^{k}_{c}( \mathbb{R}^n)  \triangleq \left\{ f\in C^{k}( \mathbb{R}^n)  :\;\overline{f^{-1}(\mathbb{R}\setminus\{0\})}\text{ is compact in } \mathbb{R}^n\right\},\qquad C^{\infty}_c(\mathbb{R}^n)\triangleq \bigcap_{j=1}^{\infty}C^{j}_c( \mathbb{R}^n).
$$
Let $C_{\mathrm{pol}}^\infty(\dbR^n)$ be the collection of all $C^\infty$ functions on $\dbR^n$ whose derivatives of every order have at most polynomial growth.
Since any function on $\dbR^n$ can be regarded as a cylindrical function on $\ell^2$ that depends only on the first $n$ variables, we write
\[
\mathscr {C}_c^{\infty}\triangleq \bigcup_{k=1}^{\infty}C_c^{\infty}(\mathbb{R}^k),\qquad
\mathscr C_{b}^\infty\triangleq\bigcup_{n=1}^\infty C_{b}^\infty(\dbR^n),\qquad
\mathscr C_{\mathrm{pol}}^\infty\triangleq\bigcup_{n=1}^\infty C_{\mathrm{pol}}^\infty(\dbR^n).
\]
Clearly $\mathscr {C}_c^{\infty}\subset\mathscr C_{b}^\infty\subset\mathscr C_{\mathrm{pol}}^\infty$. Note that $\mathscr {C}_c^{\infty}$ is not a linear space, but $\mathscr C_{b}^\infty$ and $\mathscr C_{\mathrm{pol}}^\infty$ are linear spaces. Moreover, $\mathscr C_c^\infty$ contains no nonzero cylindrical polynomial functions, and $\mathscr C_b^\infty$ contains no nonconstant cylindrical polynomials.

Put
$$
C^{1}_{F}({O})\triangleq \left\{f\in C^{1}({O}):\;\sup_{\mathbf{x}\in E} \left\{ \left| f(\mathbf{x})\right|  +\sum^{\infty }_{i=1}  \left| \frac{\partial f (\mathbf{x})}{\partial x_i}\right|^{2}  \right\} <\infty,\; \forall\, E\stackrel{\circ}{\subset}{O}\right\}.
$$
 For each $k\in\mathbb{N}\cup \{\infty\}$, we write
\begin{equation}\label{20260508for1}
\begin{array}{ll}
\displaystyle C^{k}_{F}({O})\triangleq C^{k}\left(  {O}\right)\cap C^{1}_{F}({O}),\\[3mm]
C^{k}_{b,F}({O})\triangleq C^{k}_{b}({O})\cap C^{1}_{F}({O}),\quad
C^{k}_{0,F}({O})\triangleq C^{k}_{0}({O})\cap C^{1}_{F}({O}),\\[3mm]
\displaystyle C^{k}_{F^{k}}({O})\triangleq \left\{f\in C^{k}({O}):\;\partial^{\alpha} f\in C^{1}_{F}({O})\text{ for all }\alpha \in \mathbb{N}^{\left( \mathbb{N} \right)  }_{0}\text{ with }|\alpha| <k\right\},\\[3mm]
\displaystyle C^{k}_{b,F^{k}}({O})\triangleq C^{k}_{b}({O})\cap C^{k}_{F^{k}}({O}),\quad
C^{k}_{0,F^{k}}\left(  {O}\right)\triangleq C^{k}_{0}({O})\cap C^{k}_{F^{k}}({O}).
\end{array}
\end{equation}
For any $\mathbf{x}\in {O}$, write
\begin{eqnarray*}
{O}_{\mathbf{x},n}\triangleq\{(s_1,\dots,s_n)\in\mathbb{R}^n:\mathbf{x}+s_1\textbf{e}_1+\dots+s_n\textbf{e}_n\in {O}\}.
\end{eqnarray*}
Then ${O}_{\mathbf{x},n}$ is a nonempty open subset of $\mathbb{R}^n$.
\begin{definition}\label{20260427def12}
We say that a real-valued function $f$ on ${O}$ is \textbf{$\mathcal{F}$-continuous} if, for every $\mathbf{x}\in {O}$ and $n\in\mathbb{N}$ the following function:
\begin{eqnarray*}
	(s_1,\dots,s_n)\mapsto f(\mathbf{x}+s_1\mathbf{e}_1+\dots+s_n\mathbf{e}_n),\quad\forall\,(s_1,\dots,s_n)\in {O}_{\mathbf{x},n},
\end{eqnarray*}
is continuous (on ${O}_{\mathbf{x},n}$).
\end{definition}
For each $k\in\mathbb{N}_0$, we denote by $C_{\mathcal{F}}^k({O})$ (\resp $C_{\mathcal{F},b}^k({O})$) the set of all real-valued functions $f$ on ${O}$ for which $f$ itself and all of its partial derivatives up to order $k$ are $\mathcal{F}$-continuous and Borel measurable (\resp bounded, $\mathcal{F}$-continuous and Borel measurable) on ${O}$. Write $C_{\mathcal{F}}^{\infty}({O})=\bigcap\limits_{k=1}^{\infty}C_{\mathcal{F}}^k({O})$. For $m\in\dbN$ and $p\in[1,\infty)$, define
\[\| f \|_{C_{\mathscr{F}}^{m,p}({O})} \triangleq \sup\limits_{\mathbf{x}\in{O}} \left\{\Big( \sum\limits_{\left| \alpha \right| \le m} \left| \mathbf{a}^{\alpha} \partial^{\alpha}f(\mathbf{x})\right|^{p} \Big)^{{1}/{p}}\right\}\in[0,\infty]\]
for any $f\in C_{\mathcal{F}}^{m}({O})$, and set
$$\begin{aligned}
	&C_{\mathscr{F}}^{m,p}({O}) \triangleq \left\{ f\in C^{m}_{\mathcal{F}}({O}) :\| f \|_{C_{\mathscr{F}}^{m,p}({O})} <\infty \right\} ,\qquad C_{\mathscr{F}}^{\infty}({O}) \triangleq \bigcap_{m\in \mathbb{N} ,\,p\in [1,\infty)} C_{\mathscr{F}}^{m,p}({O}),\\
&C_{\mathscr{F}}^{m,p}({O} ,loc) \triangleq \left\{f\in C_{\mathcal{F}}^m({O}):f|_{V}\in C_{\mathscr{F}}^{m,p}(V) \text{ for any open set } V\s{O} \right\} ,\\
 &C_{\mathscr{F}}^{\infty}({O} ,loc) \triangleq\bigcap_{m\in \mathbb{N}  ,\,p\in [1,\infty)} C_{\mathscr{F}}^{m,p}({O} ,loc),\\&C_{0,\mathscr{F}}^{m,p}({O}) \triangleq C_{0}^{m}({O}) \bigcap C_{\mathscr{F}}^{m,p}({O}),\qquad C_{0,\mathscr{F}}^{\infty}({O}) \triangleq \bigcap_{m\in \mathbb{N} ,\,p\in [1,\infty)} C_{0,\mathscr{F}}^{m,p}({O}),\end{aligned}$$

\smallskip For any $f\in L^{p}({O},P)$, let ${U}_f$ denote the union of all open balls $B(\mathbf{x},r)$ satisfying
$$
B(\mathbf{x},r)\subset {O}\quad\text{and}\quad\int_{ B(\mathbf{x},r)}|f|^p\,\mathrm{d}P=0,
$$
and define $\supp f\triangleq  {O}\setminus U_f ,$ which is called the \textbf{support} of $f$ (In fact, the support of any more general locally integrable function was already introduced in \cite[Definition 3.1, p. 20]{WYZ1}; here we restrict our attention to $L^p$-functions).

Suppose that $(X,\lVert\cdot\rVert_X)$ and $(Y,\lVert\cdot\rVert_Y)$ are two normed spaces, ${\cal O}$ is a nonempty open subset of $X$, and $g$ is a mapping from ${\cal O}$ into $Y$. Denote by ${\cal L}(X;Y)$ the class of all bounded linear operators from $X$ into $Y$ (which is a normed space with the usual operator norm), and by $C_X({\cal O};Y)$ the set of all continuous mappings from ${\cal O}$ to $Y$. We simply denote $C_X({\cal O};\mathbb{R})$ by $C_X({\cal O})$. For any $\mathbf{x}\in {\cal O}$, we recall that $g$ is called Fr\'echet differentiable (from ${\cal O}$ into $Y$) at $\mathbf{x}$ if there exists an element in ${\cal L}(X;Y)$, denoted by  $Dg(\mathbf{x})$ (and called the Fr\'echet derivative of $g$ at $\mathbf{x}$), such that
$$
\lim\limits_{\Delta \mathbf{x}\to \mathbf{0}}\frac{\|g(\mathbf{x}+\Delta \mathbf{x})-g(\mathbf{x})-Dg(\mathbf{x})(\Delta \mathbf{x})\|_Y}{\|\Delta \mathbf{x}\|_X}=0.
$$
Furthermore,  $g$ is called (continuously) Fr\'{e}chet differentiable from ${\cal O}$ into $Y$ if $g$ is Fr\'{e}chet differentiable from ${\cal O}$ into $Y$ at every point of ${\cal O}$ (and the mapping $\mathbf{x}\mapsto Dg(\mathbf{x})$ from ${\cal O}$ into ${\cal L}(X;Y)$ is continuous). Denote by $C_X^1({\cal O};Y)$ the class of all continuously Fr\'{e}chet differentiable mappings from ${\cal O}$ into $Y$. Inductively, people define $C_X^k({\cal O};Y)$ for each $k\in\mathbb{N}$ and $C_X^{\infty}({\cal O};Y)\triangleq \bigcap\limits_{k=1}^{\infty}C_X^k({\cal O};Y)$. For each $r\in \mathbb{N}\cup\{\infty\}$, we simply denote $C_X^r({\cal O};\mathbb{R})$ by $C_X^r({\cal O})$.

For any $x\in\mathbb{R}$, we denote by $\left\lfloor x \right\rfloor$ the greatest integer not exceeding $x$. An elementary calculation shows the following result.
\begin{lemma}\label{zuhejishu}
	For any $\alpha \in \mathbb{N}_{0}^{\left( \mathbb{N} \right)}$, it holds that
 $$\# \left\{ \beta \in \mathbb{N}_{0}^{\left( \mathbb{N} \right)} :\beta \le \alpha \right\} =\prod_{i=1}^{\infty} \left( 1+\alpha_{i} \right),\qquad\max_{\beta \le \alpha} C_{\alpha}^{\beta}=\prod_{i=1}^{\infty} C_{\alpha_{i}}^{\left\lfloor \frac{\alpha_{i}}{2} \right\rfloor},$$
	where $\#$ denotes the cardinality of the set.
\end{lemma}

By Jensen's inequality, it is easy to show the following elementary result.
\begin{lemma}\label{nver1}
	For any $k\in\mathbb{N}$, $p\in[1,\infty)$, and $x_1,\dots,x_k\in\mathbb{R}$, it holds that
	\[\left|\sum_{i=1}^{k}x_i\right|^p \le k^{p-1} \sum_{i=1}^{k}\left|x_i\right|^p .\]
\end{lemma}

We now introduce some further facts that will be used in the sequel.
\begin{lemma}\label{chengjiguji}
If $f,g\in C_{\mathscr{F}}^{m,p}({O})$ ({\it resp.} $\in C_{F^{\infty}}^{\infty}({O})$), then $fg\in C_{\mathscr{F}}^{m,p}({O})$ ({\it resp.} $\in C_{F^{\infty}}^{\infty}({O})$) and
	\begin{equation}\label{1f31g}
		\| fg \|_{C_{\mathscr{F}}^{m,p}({O})} \le \sup_{ k=0,\dots, m}  \left( k+1 \right)^{k-\frac{k}{p}} \left( C_{k}^{\left\lfloor \frac{k}{2} \right\rfloor} \right)^{k}   \cdot \| f \|_{C_{\mathscr{F}}^{m,p}({O})} \cdot \| g \|_{C_{\mathscr{F}}^{m,p}({O})}.
	\end{equation}

\end{lemma}
\begin{proof}
If $f,g\in C_{\mathscr{F}}^{m,p}({O})$, then it is obvious that $fg\in C_{\mathcal{F}}^{m}({O})$. Thus
$$
\begin{aligned}\left\| fg \right\|_{C_{\mathscr{F}}^{m,p}({O})}^{p}&=\sup_{\mathbf{x}\in{O}}\sum_{k=0}^{m} \sum_{\left| \alpha \right| =k} \left| \mathbf{a}^{\alpha} \partial^{\alpha}\left( fg \right)(\mathbf{x}) \right|^{p} =\sup_{\mathbf{x}\in{O}}\sum_{k=0}^{m} \sum_{\left| \alpha \right| =k} \left\vert \mathbf{a}^{\alpha} \sum_{\gamma \le \alpha} C_{\alpha}^{\gamma}\partial^{\gamma}f(\mathbf{x})\cdot \partial^{\alpha -\gamma}g(\mathbf{x}) \right\vert^{p}\\
&\le\sup_{\mathbf{x}\in{O}} \sum_{k=0}^{m} \sum_{\left| \alpha \right| =k}   \prod_{i=1}^{\infty} \left( \alpha_{i} +1 \right)^{p-1} \prod_{i=1}^{\infty} \left| C_{\alpha_{i}}^{\left\lfloor \frac{\alpha_{i}}{2} \right\rfloor} \right|^{p}  \sum_{\gamma \le \alpha} \left| \mathbf{a}^{\alpha} \partial^{\gamma}f(\mathbf{x})\partial^{\alpha -\gamma}g(\mathbf{x}) \right|^{p}\\
&\le\sup_{\mathbf{x}\in{O}}\sum_{k=0}^{m}   \left( k+1 \right)^{kp-k} \left\vert C_{k}^{\left\lfloor \frac{k}{2} \right\rfloor} \right\vert^{kp} \sum_{\left| \alpha \right| =k} \sum_{\gamma \le \alpha} \left| \mathbf{a}^{\alpha} \partial^{\gamma}f(\mathbf{x})\partial^{\alpha -\gamma}g(\mathbf{x}) \right|^{p}\\
&=  \sup_{\mathbf{x}\in{O}}\sum_{k=0}^{m} \left( k+1 \right)^{kp-k} \left\vert C_{k}^{\left\lfloor \frac{k}{2} \right\rfloor} \right\vert^{kp} \sum_{\left| \alpha \right| =k} \sum_{\gamma \le \alpha} \left| \mathbf{a}^{\gamma} \partial^{\gamma}f(\mathbf{x})\mathbf{a}^{\alpha -\gamma} \partial^{\alpha -\gamma}g(\mathbf{x}) \right|^{p}\\
&\le \sup_{ k=0,\dots, m}  \left( k+1 \right)^{kp-k} \left\vert C_{k}^{\left\lfloor \frac{k}{2} \right\rfloor} \right\vert^{kp} \cdot\sup_{\mathbf{x}\in{O}}   \sum_{|\gamma_1|\le m, |\gamma_2|\le m} \left| \mathbf{a}^{\gamma_1} \partial^{\gamma_1}f(\mathbf{x}) \right|^{p}\cdot\left| \mathbf{a}^{\gamma_2} \partial^{\gamma_2}g(\mathbf{x})\right|^{p}\\
&\le \sup_{ k=0,\dots, m}  \left( k+1 \right)^{kp-k} \left\vert C_{k}^{\left\lfloor \frac{k}{2} \right\rfloor} \right\vert^{kp}   \cdot \left\| f \right\|_{C_{\mathscr{F}}^{m,p}({O})}^{p} \cdot \left\| g \right\|_{C_{\mathscr{F}}^{m,p}\left( {O}  \right)}^{p} <\infty ,\end{aligned}$$
	where the first inequality follows from Lemmas \ref{zuhejishu} and \ref{nver1}, and the second one follows from the fact that $C_{n}^{\left\lfloor \frac{n}{2} \right\rfloor}$ is nondecreasing in $n$, $\alpha_{i} \le \left| \alpha \right| =k$ for all $i=1,2,\dots$, and that at most $k$ of the $\alpha_{i}$ are non-zero.

If $f,g\in C_{F^{\infty}}^{\infty}({O})$ (recall the definition of $C_{F^{\infty}}^{\infty}({O})$ in \eqref{20260508for1}), then it is clear that $fg\in C^{\infty}({O})$. For any $\alpha \in \mathbb{N}_{0}^{\left( \mathbb{N} \right)}$ and $S\s {O}$,  the following holds on $S$:
	$$\begin{aligned}
		&\sum_{i=1}^{\infty} \left| \partial^{\alpha}\left( f\cdot \frac{\partial }{\partial x_i}g \right) \right|^{2} =\sum_{i=1}^{\infty} \left\vert \sum_{\gamma \le \alpha} C_{\alpha}^{\gamma}\partial^{\alpha -\gamma}f\cdot \partial^{\gamma} \frac{\partial }{\partial x_i}g \right\vert^{2}\\
		\le{}& \sum_{i=1}^{\infty}  \sum_{\gamma \le \alpha} \left| C_{\alpha}^{\gamma}\partial^{\alpha -\gamma}f \right|^{2} \sum_{\gamma \le \alpha} \left| \frac{\partial }{\partial x_i}\partial^{\gamma}g \right|^{2}  \le  \sup_{S} \sum_{\gamma \le \alpha} \left| C_{\alpha}^{\gamma}\partial^{\alpha -\gamma}f \right|^{2} \cdot \sum_{\gamma \le \alpha} \sum_{i=1}^{\infty} \left|  \frac{\partial }{\partial x_i}\partial^{\gamma}g \right|^{2}.
	\end{aligned}$$
Combining this with the fact that $\partial^{\gamma}g\in C_{F}^{\infty}({O})$ for any $\gamma\le \alpha$, we obtain
	$$\sup_{S} \sum_{i=1}^{\infty} \left| \partial^{\alpha}\left( f\cdot  \frac{\partial }{\partial x_i}g \right) \right|^{2} <\infty.$$
	Similarly,
	$$\sup_{S} \sum_{i=1}^{\infty} \left| \partial^{\alpha}\left( g\cdot  \frac{\partial }{\partial x_i}f \right) \right|^{2} <\infty,$$
	and consequently, from the inequality
	$$\begin{aligned}
		&\sup_{S}  \sum_{i=1}^{\infty} \left|  \frac{\partial }{\partial x_i}\partial^{\alpha}\left( fg \right) \right|^{2}   =\sup_{S}   \sum_{i=1}^{\infty} \left| \partial^{\alpha}\left(  \frac{\partial }{\partial x_i}f\cdot g \right) +\partial^{\alpha}\left( f\cdot  \frac{\partial }{\partial x_i}g \right) \right|^{2}  \\
		\le{}& 2\sup_{S}  \sum_{i=1}^{\infty} \left| \partial^{\alpha}\left(  \frac{\partial }{\partial x_i}f\cdot g \right) \right|^{2}   +2\sup_{S} \sum_{i=1}^{\infty} \left| \partial^{\alpha}\left( f\cdot  \frac{\partial }{\partial x_i}g \right) \right|^{2}   <\infty,
	\end{aligned}$$
	we have $\partial^{\alpha}\left( fg \right) \in C_{F}^{\infty}({O})$. Thus, $fg\in C_{F^{\infty}}^{\infty}({O})$.
\end{proof}
\begin{lemma}\label{fuheguji}
{\rm 1)} For any $n\in\mathbb{N}$ and $\eta_{i} \in C_{\mathscr{F}}^{m,p}\left( {O} ,loc \right)$ ({\it resp.} $\in C_{F^{\infty}}^{\infty}({O})$), $i=1,2,\dots ,n$, let
$$\Phi(\mathbf{x})\triangleq \left( \eta_{1} \left( \mathbf{x} \right) ,\eta_{2} \left( \mathbf{x} \right) ,\cdots ,\eta_{n} \left( \mathbf{x} \right) \right),\quad\forall\, \mathbf{x} \in{O}.$$
If $\Phi({O})$ is contained in an open set $V\subset\mathbb{R}^{n}$ and $\overline{\Phi(E)}$ is a compact subset of $V$ for every $E\s{O}$, then $\varphi \circ\Phi \in C_{\mathscr{F}}^{m,p}\left( {O} ,loc \right)$ ({\it resp.} $\in C_{F^{\infty}}^{\infty}({O})$) for any $\varphi \in C^{m}\left( V \right)(\text{\it resp.} \in C^{\infty}\left( V \right));$

	{\rm 2)} If $\psi \in C^{\infty}\left( \mathbb{R} \right)$ and $\eta \in C_{\mathscr{F}}^{m,p}({O})$ ({\it resp.} $\in  C_{\ell^{2}}^{\infty}\left( {O}\right) $), then $\psi \circ \eta  \in C_{\mathscr{F}}^{m,p}\left( {O}\right)$ ({\it resp.} $ \in  C_{\ell^{2}}^{\infty}\left( {O}\right) $).
\end{lemma}

\begin{proof}
For any $\alpha \in \mathbb{N}_{0}^{\left( \mathbb{N} \right)}\setminus \left\{ \textbf{0} \right\}$ with
$0<|\alpha|\le m$, analogously to the multivariate Fa\`a di Bruno formula (e.g., \cite[Theorem 4.2, p. 14]{LP07}), we have
$$\partial^{\alpha}\left(\varphi \circ\Phi \right) =\sum\limits_{\substack{\beta =\left( \beta_{j} \right)_{j=1}^n \in \mathbb{N}_{0}^n,\\ 0<|\beta| \le \left\vert \alpha \right\vert}}  \frac{(\partial^{\beta}\varphi) \circ \Phi}{\beta !} \sum\limits_{\substack{  \gamma  =(\gamma_{j,k})_{(j,k)\in I_{\beta}},\\ \sum \gamma =\alpha}} \frac{\alpha !\prod\limits_{(j,k)\in I_{\beta}}   \partial^{\gamma_{j,k}}\eta_{j}}{\gamma!}  ,$$
where $I_{\beta}\triangleq \{(j,k):j=1,\ldots,n,\,\beta_j>0,\,k=1,\ldots,\beta_j\}$,$\gamma_{j,k}\in \mathbb{N}_{0}^{\left( \mathbb{N} \right)} \setminus \left\{ \textbf{0} \right\}$ for any $(j,k)\in I_{\beta}$, $ \sum \gamma \triangleq\sum_{(j,k)\in I_{\beta}}\gamma_{j,k}$, $|\beta|\triangleq\sum_{j=1}^{n} \beta_{j}$, $\beta!\triangleq\prod_{j=1}^{n} \beta_{j}!$ and $\gamma!\triangleq \prod\limits_{(j,k)\in I_{\beta}}   \gamma_{j,k} ! $.

 Suppose $\eta_{i} \in C_{\mathscr{F}}^{m,p}\left( {O} ,loc \right)$, $i=1,2,\dots ,n$ and $\varphi \in C^{m}\left( V \right)$.
Clearly $\varphi \circ\Phi\in C_{\mathcal{F}}^{m}\left( {O} \right)$. For any nonempty open set $U\s {O}$, set
$$M\triangleq \sup_{U} \sum\limits_{\substack{\beta =\left( \beta_{j} \right)_{j=1}^n \in \mathbb{N}_{0}^n,\\  |\beta| \le m}}\left| \left(\partial^{\beta}\varphi \right)\circ\Phi \right| <\infty .$$
Also, put
\begin{equation}\label{13f1ff1}
F\left( m \right) \triangleq \sup_{\substack{0<\left| \alpha \right| \le m}}   \sum\limits_{\substack{\beta =\left( \beta_{j} \right)_{j=1}^n \in \mathbb{N}_{0}^n,\\ 0<|\beta| \le \left\vert \alpha \right\vert}} \;\sum\limits_{\substack{  \gamma  =(\gamma_{j,k})_{(j,k)\in I_{\beta}},\\ \sum \gamma =\alpha}}  1    <\infty.
\end{equation}

The following holds on $U$:
$$\begin{aligned}
	&\hphantom{{}={}}\!\sum_{0<\left| \alpha \right| \le m} \left| \mathbf{a}^{\alpha} \partial^{\alpha}\left(\varphi \circ\Phi \right)  \right|^{p}\\
&=\sum_{0<\left| \alpha \right| \le m} \left| \mathbf{a}^{\alpha} \sum\limits_{\substack{\beta =\left( \beta_{j} \right)_{j=1}^n \in \mathbb{N}_{0}^n,\\ 0<|\beta| \le \left\vert \alpha \right\vert}}  \frac{(\partial^{\beta}\varphi) \circ \Phi}{\beta !} \sum\limits_{\substack{  \gamma  =(\gamma_{j,k})_{(j,k)\in I_{\beta}},\\ \sum \gamma =\alpha}} \frac{\alpha !\prod\limits_{(j,k)\in I_{\beta}}   \partial^{\gamma_{j,k}}\eta_{j}}{\gamma!}    \right|^{p}\\
&\le M^{p}\sum_{0<\left| \alpha \right| \le m} \left( \mathbf{a}^{\alpha} \sum\limits_{\substack{\beta =\left( \beta_{j} \right)_{j=1}^n \in \mathbb{N}_{0}^n,\\ 0<|\beta| \le \left\vert \alpha \right\vert}}  \;\sum\limits_{\substack{  \gamma  =(\gamma_{j,k})_{(j,k)\in I_{\beta}},\\ \sum \gamma =\alpha}}\alpha !\prod\limits_{(j,k)\in I_{\beta}}   |\partial^{\gamma_{j,k}}\eta_{j}|   \right)^{p}\\
&\le (m!)^{pm}M^{p}\sum_{0<\left| \alpha \right| \le m} \left( \sum\limits_{\substack{\beta =\left( \beta_{j} \right)_{j=1}^n \in \mathbb{N}_{0}^n,\\ 0<|\beta| \le \left\vert \alpha \right\vert}}  \; \sum\limits_{\substack{  \gamma  =(\gamma_{j,k})_{(j,k)\in I_{\beta}},\\ \sum \gamma =\alpha}}\;\prod\limits_{(j,k)\in I_{\beta}}   |\mathbf{a}^{\gamma_{j,k}} \partial^{\gamma_{j,k}}\eta_{j}|    \right)^{p} .
\end{aligned}$$

By Lemma \ref{nver1}, we obtain
 $$\begin{aligned}
	&\hphantom{{}={}}\!\sum_{0<\left| \alpha \right| \le m} \left| \mathbf{a}^{\alpha} \partial^{\alpha}\left( \varphi \circ \Phi \right) \right|^{p}\\ &\le (m!)^{pm}M^{p}\left| F\left( m \right) \right|^{p-1} \sum_{0<\left| \alpha \right| \le m} \;\sum\limits_{\substack{\beta =\left( \beta_{j} \right)_{j=1}^n \in \mathbb{N}_{0}^n,\\ 0<|\beta| \le \left\vert \alpha \right\vert}} \; \sum\limits_{\substack{  \gamma  =(\gamma_{j,k})_{(j,k)\in I_{\beta}},\\ \sum \gamma =\alpha}}\;\prod\limits_{(j,k)\in I_{\beta}}   \left| \mathbf{a}^{\gamma_{j,k}} \partial^{\gamma_{j,k}}\eta_{j} \right|^{p} \\
&\le (m!)^{pm}M^{p}\left| F\left( m \right) \right|^{p-1}  \sum\limits_{\substack{\beta =\left( \beta_{j} \right)_{j=1}^n \in \mathbb{N}_{0}^n,\\ 0<|\beta| \le m}}\; \sum\limits_{\substack{ \left| \gamma_{j,k} \right| \le m,\\ (j,k)\in I_{\beta}}} \;\prod_{(j,k)\in I_{\beta}} \left| \mathbf{a}^{\gamma_{j,k}} \partial^{\gamma_{j,k}}\eta_{j} \right|^{p}\\
&=(m!)^{pm}M^{p}\left| F\left( m \right) \right|^{p-1}  \sum\limits_{\substack{\beta =\left( \beta_{j} \right)_{j=1}^n \in \mathbb{N}_{0}^n,\\ 0<|\beta| \le m}}\; \prod_{\substack{j=1,\ldots,n,\\
 \beta_j>0}}  \left( \sum_{\left| \alpha \right| \le m} \left| \mathbf{a}^{\alpha} \partial^{\alpha}\eta_{j} \right|^{p} \right)^{\beta_{j}} .\end{aligned}$$
Since $\eta_{1} ,\eta_{2} ,\ldots ,\eta_{n} \in C_{\mathscr{F}}^{m,p}\left( {O},loc \right)$, we see that
$\varphi \circ \Phi  \in C_{\mathscr{F}}^{m,p}\left( {O} ,loc \right).$ By the same proof as above and \cite[(8.2.1), p. 151]{Die69} one obtains the corresponding conclusion 2).

 Suppose $\eta_{i} \in C_{F^{\infty}}^{\infty}\left( {O} \right)$, $i=1,2,\dots ,n$ and $\varphi \in C^{\infty}\left( V \right)$.
 For any $\alpha \in \mathbb{N}_{0}^{\left( \mathbb{N} \right)}$, we have
$$\begin{aligned}
	&\hphantom{{}={}}\!\sum_{i=1}^{\infty} \left| \frac{\partial }{\partial x_i}\partial^{\alpha}\left( \varphi \circ\Phi  \right) \right|^{2} =\sum_{i=1}^{\infty} \left\vert \partial^{\alpha}\left( \sum_{j=1}^{n} \frac{\partial }{\partial x_j}\varphi \left(\Phi \right)\cdot \frac{\partial }{\partial x_i}\eta_{j} \right) \right\vert^{2}\\
&  =\sum_{i=1}^{\infty} \left\vert \sum_{j=1}^{n} \partial^{\alpha}\left( \frac{\partial }{\partial x_j}\varphi \left( \Phi \right) \cdot \frac{\partial }{\partial x_i}\eta_{j} \right) \right\vert^{2} =\sum_{i=1}^{\infty} \left\vert \sum_{j=1}^{n} \sum_{\beta \le \alpha} C_{\alpha}^{\beta}\partial^{\beta}\left( \frac{\partial}{\partial x_j}\varphi \left( \Phi \right) \right)\cdot \partial^{\alpha -\beta}\frac{\partial }{\partial x_i}\eta_{j} \right\vert^{2}\\
&  \le (\alpha !)^2\sum_{i=1}^{\infty} \left( \sum_{j=1}^{n} \sum_{\beta \le \alpha} \left\vert \partial^{\beta}\left( \frac{\partial }{\partial x_j}\varphi \left(\Phi \right) \right) \cdot\frac{\partial }{\partial x_i}\left( \partial^{\alpha -\beta}\eta_{j} \right) \right\vert \right)^{2}\\
&  \le (\alpha !)^2\sum_{i=1}^{\infty}  \sum_{j=1}^{n} \sum_{\beta \le \alpha} \left\vert \partial^{\beta}\left( \frac{\partial }{\partial x_j}\varphi \left( \Phi \right) \right) \right\vert^{2}  \cdot  \sum_{j=1}^{n} \sum_{\beta \le \alpha} \left| \frac{\partial }{\partial x_i}\left( \partial^{\alpha -\beta}\eta_{j} \right) \right|^{2} \\
&  =(\alpha !)^2\cdot  \sum_{j=1}^{n} \sum_{\beta \le \alpha} \left\vert \partial^{\beta}\left( \frac{\partial }{\partial x_j}\varphi \left( \Phi \right) \right) \right\vert^{2}  \cdot  \sum_{j=1}^{n} \sum_{\beta \le \alpha} \sum_{i=1}^{\infty} \left| \frac{\partial }{\partial x_i}\left( \partial^{\alpha -\beta}\eta_{j} \right) \right|^{2}  .\end{aligned}$$
Hence from $\eta_{i} \in C_{F^{\infty}}^{\infty}\left( {O} \right)$, $i=1,2,\dots ,n$, we obtain $\varphi \circ\Phi  \in C_{F^{\infty}}^{\infty}\left( {O} \right).$ We have thus completed the proof of Lemma \ref{fuheguji}.
\end{proof}

\begin{lemma}\label{1f1f13f1f3}
It holds that
	$$\begin{aligned}
		&\lVert\cdot\rVert  \in C_{\mathscr{F}}^{\infty}\left( \ell^{2} \setminus \left\{ \mathbf{0} \right\} ,loc \right) \bigcap C_{F^{\infty}}^{\infty}\left( \ell^{2} \setminus \left\{ \mathbf{0} \right\} \right) \bigcap C_{\ell^{2}}^{\infty}\left( \ell^{2} \setminus \left\{ \mathbf{0} \right\} \right) ,\\
&\lVert\cdot\rVert^{2} \in C_{\mathscr{F}}^{\infty}\left( \ell^{2} ,loc \right) \bigcap C_{F^{\infty}}^{\infty}\left( \ell^{2} \right) \bigcap C_{\ell^{2}}^{\infty}\left( \ell^{2}  \right) .\end{aligned}$$
\end{lemma}
\begin{proof}
Firstly, it is well known that $\lVert \cdot \rVert \in C_{\ell^{2}}^{\infty}\left( \ell^{2}\setminus\{\mathbf{0}\}\right)$ and $\lVert\cdot\rVert^{2} \in  C_{\ell^{2}}^{\infty}\left( \ell^{2} \right)$ (or by some simple computation). A direct computation shows that for any $\mathbf{x}=(x_i)_{i\in\mathbb{N}}\in\ell^2$,
	$$\sum_{\left| \alpha \right| =k} \mathbf{a}^{\alpha} \left| \partial^{\alpha}\left( \| \mathbf{x} \|^{2} \right) \right| =\begin{cases}2\sum\limits_{i=1}^{\infty} a_{i}\left| x_{i} \right| ,&k=1,\\ 2\sum\limits_{i=1}^{\infty} a_{i}^{2},&k=2,\\ 0,&k>2.\end{cases}$$
Thus it is obvious that $\lVert\cdot\rVert^{2} \in C_{\mathscr{F}}^{\infty}\left( \ell^{2} ,loc \right) \bigcap C_{F^{\infty}}^{\infty}\left( \ell^{2} \right)  . $
Since $\sqrt{\cdot}$ is a $C^{\infty}$ function on $\left(0,\infty \right)$, by Lemma \ref{fuheguji}, we obtain
$$\lVert\cdot\rVert =\sqrt{\lVert\cdot\rVert^{2}}\in C_{\mathscr{F}}^{\infty}\left( \ell^{2} \setminus \left\{ \mathbf{0} \right\} ,loc \right) \bigcap C_{F^{\infty}}^{\infty}\left( \ell^{2} \setminus \left\{ \mathbf{0} \right\} \right)  .$$
This completes the proof of Lemma \ref{1f1f13f1f3}.
\end{proof}

We also need the following infinite-dimensional partition of unity.
\begin{theorem}\label{danweifenjie}
	Let $U$ be a nonempty open subset of $\ell^{2}$ and $\left\{ U_{\alpha} \right\}_{\alpha \in \Lambda}$ be a family of nonempty open subsets of $\ell^2$ satisfying $U=\bigcup\limits_{\alpha \in \Lambda} U_{\alpha}$, where $\Lambda$ is an index set. Then, one can find a sequence
$\{f_{i}\}_{i=1}^{\infty}\subset C_{F^{\infty}}^{\infty}\left( \ell^{2} \right) \bigcap C_{0,\mathscr{F}}^{\infty}\left( \ell^{2} \right)\bigcap C_{\ell^{2}}^{\infty}\left( \ell^{2} \right)
$
such that:
		\begin{itemize}
		\item[$\mathrm{1)}$]There exists a map $a:\mathbb{N} \rightarrow \Lambda$ such that
$
f_{i}^{-1}(\mathbb{R}\setminus\{0\})\s U_{a\left( i \right)},\quad\forall\,i\in\mathbb{N}$.

		\item[$\mathrm{2)}$]$0\le f_{i}\left( \mathbf{x} \right) \le 1,$ for all $ \,\mathbf{x} \in \ell^{2} $ and $ i\in\mathbb{N}$.
		\item[$\mathrm{3)}$]$
			\sum\limits_{i=1}^{\infty} f_{i}\left( \mathbf{x} \right) =1,$ for all $ \mathbf{x} \in U$.
		\item[$\mathrm{4)}$] $\left\{ f_{i}^{-1}(\mathbb{R}\setminus\{0\})\bigcap U\right\}_{i=1}^{\infty}$ is a locally finite family in $U$.
	\end{itemize}
\end{theorem}
\begin{proof}
For each $\mathbf{x}\in U$, choose $\alpha(\mathbf{x})\in\Lambda$ and numbers $0<r_{\mathbf{x}}<R_{\mathbf{x}}$ such that $\overline{B(\mathbf{x},R_{\mathbf{x}})}\s U_{\alpha(\mathbf{x})}$. A smooth function of $\|\mathbf{y}-\mathbf{x}\|^2$ gives a function $\Phi_{\mathbf{x}}$ satisfying
$$
0\le\Phi_{\mathbf{x}}\le1,\qquad
\Phi_{\mathbf{x}}=1\ \hbox{on }B(\mathbf{x},r_{\mathbf{x}}),\qquad
\supp\Phi_{\mathbf{x}}\subset\overline{B(\mathbf{x},R_{\mathbf{x}})}.
$$
By Lemmas \ref{fuheguji} and \ref{1f1f13f1f3}, the function $\Phi_{\mathbf{x}}$ belongs to
$C_{F^{\infty}}^{\infty}(\ell^2)\cap C_{0,\mathscr{F}}^{\infty}(\ell^2)\cap C_{\ell^2}^{\infty}(\ell^2)$. Indeed, its positive-order derivatives vanish outside a bounded annulus, and the derivatives of the squared norm have the bounds in Lemma \ref{1f1f13f1f3}.

By the Lindel\"of property, choose $\{\mathbf{x}_i\}_{i=1}^{\infty}\subset U$ so that $U=\bigcup_{i=1}^{\infty}B(\mathbf{x}_i,r_{\mathbf{x}_i})$. If the selected subcover is finite, the functions below are followed by zero functions. Set $a(i)\triangleq\alpha(\mathbf{x}_i)$ and
$$
f_1\triangleq\Phi_{\mathbf{x}_1},\qquad
f_i\triangleq\Phi_{\mathbf{x}_i}\prod_{j=1}^{i-1}(1-\Phi_{\mathbf{x}_j}),\quad i\ge2.
$$
Finite products and Lemma \ref{chengjiguji} give the asserted function classes. Also, $0\le f_i\le1$ and $\supp f_i\subset\supp\Phi_{\mathbf{x}_i}\s U_{a(i)}$.

For any $\mathbf{x}\in U$, choose $j$ such that $\mathbf{x}\in B(\mathbf{x}_j,r_{\mathbf{x}_j})$. The function $\Phi_{\mathbf{x}_j}$ equals one throughout this ball, so $f_i$ vanishes throughout the ball for every $i>j$. Thus the family of the supports of $f_i$ is locally finite in $U$. Finally,
$$
\sum_{i=1}^n f_i=1-\prod_{i=1}^n(1-\Phi_{\mathbf{x}_i}),
$$
and the product on the right vanishes at $\mathbf{x}$ for every $n\ge j$. This proves conclusions 3) and 4), and completes the proof.
\end{proof}
Let us now examine the relationship between the Wiener measure $p_t$ with a variance parameter $t\in(0,\infty)$ appearing in \cite{Gro67} and the measure $P$ used in this paper.

Recall \eqref{20260428for2} for the definition of $H$. There is a natural injection $\iota:H\to \ell^2$ defined by $\iota\mathbf{x}=\mathbf{x}$ for all $\mathbf{x}\in H$, which identifies $H$ as a subset of $\ell^2$. Define a Hilbert--Schmidt operator $\mathcal{A}:H\to H$ by $\mathcal{A} \mathbf{x}\triangleq (a_ix_i)_{i\in\mathbb{N}}$ for any $\mathbf{x}=(x_i)_{i\in\mathbb{N}}\in H$. Then,
$\|\mathcal{A}\mathbf{x}\|_H=\|\iota\mathbf{x}\|,$ for any $\mathbf{x}\in H.$
By \cite[Exercise 17, p. 59]{Kuo}, $\|\mathcal{A}\cdot\|_H$ is a measurable norm on $H$ and it is easy to see that the completion of $H$ with respect to $\|\mathcal{A}\cdot\|_H$ is $\ell^2$. Following \cite[p. 127]{Gro67}, the triple $(H,\ell^2,\iota)$ is called an abstract Wiener space. For each $t\in(0, \infty)$, the Wiener measure $p_t$ is determined by
$$
p_t(\{\mathbf{x}\in \ell^2: \varphi(\mathbf{x})<s\})=\frac{1}{\sqrt{2\pi t\|\iota^*\varphi\|_{H^*}^2}}\int_{-\infty}^{s}e^{-\frac{y^2}{2t\|\iota^*\varphi\|_{H^*}^2}}\,\mathrm{d}y,\qquad\forall\,s\in\mathbb{R},\;\varphi\in (\ell^2)^*\setminus\{0\},
$$
where $H^*$(the real dual of $H$) can be identified with $H$ via the Riesz representation theorem and $(\ell^2)^*$ is the real dual of $\ell^2$. Observe that for any $\varphi\in (\ell^2)^*$, we have $\iota^*\varphi\in H^*$, which implies that there exists $\mathbf{y}=(y_i)_{i\in\mathbb{N}}\in H$ such that $\iota^*\varphi(\mathbf{x})=\sum\limits_{i=1}^{\infty}\frac{x_iy_i}{a_i^2}$ for all $\mathbf{x}=(x_i)_{i\in\mathbb{N}}\in H$. Since $\iota^*\varphi(\mathbf{x})= \varphi(\iota\mathbf{x})=\sum\limits_{i=1}^{\infty}x_i\cdot\frac{y_i}{a_i^2}$ for all $\mathbf{x}=(x_i)_{i\in\mathbb{N}}\in H$ and $\iota H$ is dense in $\ell^2$, we have
$$
\sum\limits_{i=1}^{\infty} \frac{y_i^2}{a_i^4}<\infty,\qquad\varphi(\mathbf{z})=\sum\limits_{i=1}^{\infty}z_i\cdot\frac{y_i}{a_i^2},\quad\forall\, \mathbf{z}=(z_i)_{i\in\mathbb{N}}\in \ell^2,
$$
and
$$
\int_{\ell^2}e^{\sqrt{-1}s\varphi(\mathbf{x})}\,\mathrm{d}P(\mathbf{x})=e^{-\frac{s^2a^2}{2}},\qquad \forall\,s\in\mathbb{R},
$$
where $a^2=\int_{\ell^2}\varphi^2\,\mathrm{d}P=\sum\limits_{i=1}^{\infty} \frac{y_i^2}{a_i^2}=\|\iota^*\varphi\|_{H^*}^2.$ From the properties of the characteristic functional, we have
$$
\int_{\ell^2}e^{\sqrt{-1}s\varphi(\mathbf{x})}\,\mathrm{d}P(\mathbf{x})=\int_{\ell^2}e^{\sqrt{-1}s\varphi(\mathbf{x})}\,\mathrm{d}p_1(\mathbf{x}),\qquad\forall\,s\in\mathbb{R},\,\varphi\in (\ell^2)^*,
$$
and by the uniqueness of the Fourier transform (e.g., \cite[Theorem 39.6, p. 530]{Dri}), it follows that:
\begin{lemma}\label{20250903lem1}
$P=p_1$.
\end{lemma}
Let $\widetilde{\textbf{e}_i}\triangleq a_i\textbf{e}_i$ for all $i\in\mathbb{N}$, where $\{\textbf{e}_i\}_{i=1}^{\infty}$ is defined before Definition \ref{def of bounded contained}. Then $\{\widetilde{\textbf{e}_i}\}_{i=1}^{\infty}\subset (\ell^2)^*$ forms an orthonormal basis of $H$. Recall that in \cite[p. 133]{Gro67}, for any $t\in(0,\infty)$ and  real-valued bounded Borel measurable function $f$ on $\ell^2$, the function $p_tf$ is defined by
$$
(p_tf)(\mathbf{x})\triangleq \int_{\ell^2}f(\mathbf{x}+\mathbf{y})\,\mathrm{d}p_t(\mathbf{y}),\qquad \forall\,\mathbf{x}\in \ell^2.
$$
For each $\mathbf{x}\in \ell^2$ and $i\in\mathbb{N}$, according to \cite[p. 152]{Gro67}, we have
$$
(p_tf)(\mathbf{x}+s\cdot\widetilde{\textbf{e}_i})-(p_tf)(\mathbf{x})=  Dp_tf (\mathbf{x})(s\cdot\widetilde{\textbf{e}_i})+o(|s|),\qquad \text{as}\,s\to 0,
$$
where $Dp_tf(\mathbf{x})$ denotes the $H$-Fr\'{e}chet derivative of $p_tf$ at $\mathbf{x}$. This implies
\[\begin{aligned}
 Dp_tf(\mathbf{x})(\widetilde{\textbf{e}_i})&=\lim_{s\to 0}\frac{(p_tf)(\mathbf{x}+s\cdot\widetilde{\textbf{e}_i})-(p_tf)(\mathbf{x})}{s}\\
&=a_i\cdot\lim_{s\to 0}\frac{(p_tf)(\mathbf{x}+s\cdot a_i \cdot\textbf{e}_i)-(p_tf)(\mathbf{x})}{sa_i}
=a_i\cdot\frac{\partial(p_tf)(\mathbf{x})}{\partial x_i}.
\end{aligned}\]
By induction, one can prove that
\begin{eqnarray}\label{20250903for1}
 D^kp_tf(\mathbf{x})(\widetilde{\textbf{e}_{i_1}},\dots,\widetilde{\textbf{e}_{i_k}})=a_{i_1}a_{i_2}\dots a_{i_k}\cdot \frac{\partial^k(p_tf)(\mathbf{x})}{\partial x_{i_1}\partial x_{i_2}\dots \partial x_{i_k}},\qquad
\forall\,k\in\mathbb{N},\,i_1,\dots,i_k\in\mathbb{N},
\end{eqnarray}
where $D^kp_tf(\mathbf{x})$ denotes the $k$-th $H$-Fr\'{e}chet derivative of $p_tf$ at $\mathbf{x}$ (see \cite[p. 155]{Gro67} for a definition). From \cite[(1), p. 70]{Lee}, for a real-valued bounded, Borel measurable function $f$ on $\ell^2$, we obtain
\begin{eqnarray}\label{20250903for2}
 \sum_{i_1,i_2,\dots,i_k=1}^{\infty}a_{i_1}^2a_{i_2}^2\dots a_{i_k}^2\left|\frac{\partial^k(p_1f)(\mathbf{x})}{\partial x_{i_1}\partial x_{i_2}\dots \partial x_{i_k}}\right|^2\le k!\cdot \int_{\ell^2} f(\mathbf{x}+\mathbf{y})^2\,\mathrm{d}P(\mathbf{y}),\quad \forall\,k\in\mathbb{N}.
\end{eqnarray}

The following result indicates that the constant ``2" appearing in the exponent of the left-hand side of \eqref{20250903for2} is sharp.
\begin{theorem}\label{gross moguang}
	For $r\in(0,\infty)$, the following property holds \textbf{if and only if} $r\in[2,\infty)$: for every $M\in(0,\infty)$, there exists a constant $C(r,M)>0$ such that, for every real-valued bounded Borel measurable function $f$ on $\ell^2$ satisfying $\sup_{\ell^2}|f|\le M$, we have
\begin{equation*}
	\sum_{i=1}^{\infty}a_i^r\left|\frac{\partial(p_1f)(\mathbf{x})}{\partial x_i}\right|^r\le C(r,M),\qquad
\forall\,\mathbf{x}\in\ell^2.
\end{equation*}
	\end{theorem}
\begin{proof}
By \eqref{20250903for2}, it is straightforward that the stated property holds for $r\in[2, \infty)$.
A direct computation yields
$$
\frac{\partial (p_1f)(\mathbf{x})}{\partial x_i} =\int_{\ell^{2}} \frac{y_{i}}{a_{i}^{2}} f\left( \mathbf{x}+\mathbf{y} \right)\, \mathrm{d}P(\mathbf{y}),\qquad \forall\,\mathbf{x},\,\mathbf{y}=(y_i)_{i\in\mathbb{N}}\in \ell^{2},\,i\in\mathbb{N}.
$$
For $ r\in (0,2) $, define
$$
f_{n}(\mathbf{x}) \triangleq\operatorname{sgn} \left( \sum\limits_{i=1}^{n} \frac{x_{i}}{a_{i}} \right),\quad \forall\,n\in\mathbb{N},\,\mathbf{x}=(x_i)_{i\in\mathbb{N}}\in\ell^2.
$$
Then,
 $$
 \begin{aligned}
			&\hphantom{{}={}}\!\sum\limits_{i=1}^{\infty} a_{i}^{r}\left| \frac{\partial (p_1f_n)(\mathbf{x})}{\partial x_i} \right|^{r} =\sum\limits_{i=1}^{\infty} \left| \int_{\ell^{2}} \frac{y_{i}}{a_{i}} \operatorname{sgn} \left( \sum\limits_{j=1}^{n} \frac{x_{j}+y_{j}}{a_{j}} \right)\,\mathrm{d}P(\mathbf{y}) \right|^{r}\\
 &  =\sum\limits_{i=1}^{n} \left| \int_{\mathbb{R}^{n}} \frac{y_{i}}{a_{i}} \operatorname{sgn} \left( \sum\limits_{j=1}^{n} \frac{x_{j}+y_{j}}{a_{j}} \right) \frac{1}{\left( \sqrt{2\pi} \right)^{n} \prod\limits_{j=1}^{n} a_{j}} e^{-\sum\limits_{j=1}^{n} \frac{y_{j}^{2}}{2a_{j}^{2}}}\,\mathrm{d} y_{1}\mathrm{d} y_{2}\dots \mathrm{d} y_{n} \right|^{r}\\
  &  =\sum\limits_{i=1}^{n} \left| \int_{\mathbb{R}^{n}} y_{i}\operatorname{sgn} \left( \sum\limits_{j=1}^{n} \frac{x_{j}}{a_{j}} +\sum\limits_{j=1}^{n} y_{j} \right) \frac{1}{\left( \sqrt{2\pi} \right)^{n}} e^{-\sum\limits_{j=1}^{n} \frac{y_{j}^{2}}{2}}\,\mathrm{d} y_{1}\mathrm{d} y_{2}\dots \mathrm{d} y_{n} \right|^{r}\\
   &  =\sum\limits_{i=1}^{n} \left| \frac{1}{n} \int_{\mathbb{R}^{n}} \left( \sum\limits_{j=1}^{n} y_{j} \right) \operatorname{sgn} \left( \sum\limits_{j=1}^{n} \frac{x_{j}}{a_{j}} +\sum\limits_{j=1}^{n} y_{j} \right) \frac{1}{\left( \sqrt{2\pi} \right)^{n}} e^{-\sum\limits_{j=1}^{n} \frac{y_{j}^{2}}{2}}\,\mathrm{d} y_{1}\mathrm{d} y_{2}\dots \mathrm{d} y_{n} \right|^{r},\end{aligned}
$$
where the last equality follows from symmetry of the integrand. Take an $n\times n$ orthogonal matrix $T$ whose first row consists entirely of $\frac{1}{\sqrt{n}}$. Applying the change of variables
$$
\begin{pmatrix}z_{1}\\ z_{2}\\ \vdots\\ z_{n}\end{pmatrix} =T\begin{pmatrix}y_{1}\\ y_{2}\\ \vdots\\ y_{n}\end{pmatrix},
$$
we obtain
$$
 \begin{aligned}
			&\hphantom{{}={}}\!\sum\limits_{i=1}^{n} \left| \frac{1}{n} \int_{\mathbb{R}^{n}}   \sum\limits_{j=1}^{n} y_{j}  \operatorname{sgn} \left( \sum\limits_{j=1}^{n} \frac{x_{j}}{a_{j}} +\sum\limits_{j=1}^{n} y_{j} \right) \frac{1}{\left( \sqrt{2\pi} \right)^{n}} e^{-\sum\limits_{j=1}^{n} \frac{y_{j}^{2}}{2}}\,\mathrm{d} y_{1}\mathrm{d} y_{2}\dots \mathrm{d} y_{n} \right|^{r}\\
    & =\sum\limits_{i=1}^{n} \left| \frac{1}{n} \int_{\mathbb{R}^{n}}   \sqrt{n} z_{1}   \operatorname{sgn} \left( \sum\limits_{j=1}^{n} \frac{x_{j}}{a_{j}} +\sqrt{n} z_{1} \right) \frac{1}{\left( \sqrt{2\pi} \right)^{n}} e^{-\sum\limits_{j=1}^{n} \frac{z_{j}^{2}}{2}}\,\mathrm{d} z_{1}\mathrm{d} z_{2}\dots \mathrm{d} z_{n} \right|^{r}\\
& =\sum\limits_{i=1}^{n} \left| \frac{1}{\sqrt{n}} \int_{\mathbb{R}} z_{1}\operatorname{sgn} \left( \sum\limits_{j=1}^{n} \frac{x_{j}}{a_{j}} +\sqrt{n} z_{1} \right) \frac{1}{\sqrt{2\pi}} e^{-\frac{z_{1}^{2}}{2}}\,\mathrm{d} z_{1} \right|^{r}\\
&  =\frac{1}{\left( 2\pi \right)^{\frac{r}{2}}n^{\frac{r}{2} -1}} \left| \int_{\mathbb{R}} z_{1}\operatorname{sgn} \left( \sum\limits_{j=1}^{n} \frac{x_{j}}{a_{j}} +\sqrt{n} z_{1} \right) e^{-\frac{z_{1}^{2}}{2}}\mathrm{d} z_{1} \right|^{r} =\frac{2^{r}}{\left( 2\pi \right)^{\frac{r}{2}}n^{\frac{r}{2} -1}}  e^{-\frac{r}{2n} \left( \sum\limits_{j=1}^{n} \frac{x_{j}}{a_{j}} \right)^{2}}.\end{aligned}
$$		
Thus we arrive at
\begin{equation}\label{13}
			\sum\limits_{i=1}^{\infty} a_{i}^{r}\left| \frac{\partial\left( p_{1}f_{n} \right)}{\partial x_i} \left( \frac{\mathbf{x}}{\sqrt{n}} \right) \right|^{r} =\frac{2^{r}}{\left( 2\pi \right)^{\frac{r}{2}}} \frac{1}{n^{\frac{r}{2} -1}} e^{-\frac{r}{2} \left( \frac{1}{n} \sum\limits_{j=1}^{n} \frac{x_{j}}{a_{j}} \right)^{2}},\quad\forall\, \mathbf{x}=(x_i)_{i\in\mathbb{N}}\in \ell^{2},\,n\in \mathbb{N} .
\end{equation}
If the stated property were true,
combining it with \eqref{13} would imply
\begin{eqnarray}\label{20251219for2}
\lim_{n\rightarrow \infty}   \frac{1}{n} \sum\limits_{j=1}^{n} \frac{x_{j}}{a_{j}}   =\infty ,\qquad\forall\, \mathbf{x}=(x_i)_{i\in\mathbb{N}}\in \ell^{2}.
\end{eqnarray}
Let $g_n(\mathbf{x})\triangleq \frac{x_n}{a_n}$, for any $n\in\mathbb{N}$ and $\mathbf{x}=(x_i)_{i\in\mathbb{N}}\in\ell^2$. Observe that $\{g_n\}_{n=1}^{\infty}$ is a sequence of independent and identically distributed random variables with respect to the measure $P$. By the Law of Large Numbers (see \cite[Theorem 2.4.1, p. 78]{Durrett}), we have
$$
\lim_{n\rightarrow \infty}  \frac{1}{n} \sum\limits_{j=1}^{n} \frac{x_{j}}{a_{j}}  =0,\qquad P\text{-a.e.,}
$$
which contradicts \eqref{20251219for2}. Hence, the stated property cannot hold for $r\in(0,2)$, completing the proof of Theorem \ref{gross moguang}.
\end{proof}
By an explicit quadratic construction, the compactly truncated sequence can be obtained under assumption \eqref{20260424for1}.
\begin{proposition}\label{230215Th1}
There exists a sequence $\{X_n\}_{n=1}^{\infty}\subset C_{\mathcal{F}}^{\infty}(\ell^{2})$ with the following properties:
\begin{itemize}
\item[$\mathrm{1)}$] $0\le X_n\le1$, and $X_n^{-1}(\mathbb{R}\setminus\{0\})$ is contained in a compact subset $K_n$ of $\ell^2$, for every $n\in\mathbb{N}$;
\item[$\mathrm{2)}$] For every $k\in\mathbb{N}$,
\begin{eqnarray}\label{20250724for2}
\sup_{\mathbf{x}\in\ell^2,\,n\in\mathbb{N}}\sum_{i_1,\ldots,i_k=1}^{\infty}a_{i_1}\cdots a_{i_k}\left|\frac{\partial^kX_n(\mathbf{x})}{\partial x_{i_1}\cdots\partial x_{i_k}}\right|^2<\infty.
\end{eqnarray}
Moreover, for every $p\in[1,\infty)$ there is a positive constant $C(k,p)$, independent of $n$, such that
\begin{eqnarray}\label{20260917cutoffestimate}
\sup_{\mathbf{x}\in\ell^2}\sum_{i_1,\ldots,i_k=1}^{\infty}a_{i_1}^p\cdots a_{i_k}^p\left|\frac{\partial^kX_n(\mathbf{x})}{\partial x_{i_1}\cdots\partial x_{i_k}}\right|^p\le C(k,p)n^{-kp};
\end{eqnarray}
\item[$\mathrm{3)}$] For $P$-almost every $\mathbf{x}\in\ell^2$, $X_n(\mathbf{x})=1$ and all positive-order partial derivatives of $X_n$ at $\mathbf{x}$ vanish for every sufficiently large $n$.
\end{itemize}
\end{proposition}
\begin{proof}
Choose a nondecreasing sequence $\{d_i\}_{i=1}^{\infty}\subset[1,\infty)$ such that
\begin{equation}\label{20260917di}
\lim_{i\to\infty}d_i=\infty,\qquad\sum_{i=1}^{\infty}a_i^2d_i^2<\infty.
\end{equation}
For example, choose strictly increasing integers $N_j$, $j\in\mathbb{N}$, such that $\sum_{i=N_j}^{\infty}a_i^2\le 2^{-j}(j+1)^{-2}$, set $d_i=1$ for $i<N_1$, and set $d_i=j+1$ for $N_j\le i<N_{j+1}$. These choices satisfy \eqref{20260917di}.

Take $\chi\in C^{\infty}(\mathbb{R};[0,1])$ such that $\chi(t)=1$ for $t\le1$ and $\chi(t)=0$ for $t\ge4$. Define
$$
Q(\mathbf{x})\triangleq\sum_{i=1}^{\infty}d_ix_i^2\in[0,\infty],\qquad
X_n(\mathbf{x})\triangleq
\begin{cases}
\chi\left(\frac{Q(\mathbf{x})}{n^2}\right),&Q(\mathbf{x})<\infty,\\
0,&Q(\mathbf{x})=\infty,
\end{cases}
\quad \mathbf{x}=(x_i)_{i\in\mathbb{N}}\in\ell^2.
$$
Put $Q_N(\mathbf{x})\triangleq\sum_{i=1}^N d_i x_i^2$.
For every $\mathbf{x}\in\ell^2$, $k\in\mathbb{N}$ and
$\mathbf{t}=(t_1,\dots,t_k)\in\mathbb{R}^k$,
\[
X_n\left(\mathbf{x}+\sum_{i=1}^k t_i\mathbf{e}_i\right)=
\begin{cases}
\displaystyle\chi\left(\frac{Q(\mathbf{x})+\sum_{i=1}^k d_i(2x_it_i+t_i^2)}{n^2}\right),&Q(\mathbf{x})<\infty,\\
0,&Q(\mathbf{x})=\infty.
\end{cases}
\]
Hence all partial derivatives of $X_n$ exist and are $\mathcal{F}$-continuous. Moreover, for every $\alpha\in\mathbb{N}_0^{(\mathbb{N})}$,
\[
\partial^\alpha X_n(\mathbf{x})
=\lim_{N\to\infty}\partial^\alpha\!\left[\chi\left(\frac{Q_N}{n^2}\right)\right](\mathbf{x}),\qquad \forall\mathbf{x}\in\ell^2.
\]
Thus all $\partial^\alpha X_n$ are Borel measurable and $X_n\in C_{\mathcal{F}}^\infty(\ell^2)$.

For every $n\in\mathbb{N}$, set
$$K_n\triangleq\{\mathbf{x}\in\ell^2:Q(\mathbf{x})\le4n^2\}.$$
Then $K_n$ is closed, $K_n\subset\overline{B_{2n}}$, and
$$
\sup_{\mathbf{x}\in K_n}\sum_{i>N}x_i^2\le\frac{4n^2}{\inf_{i>N}d_i}\longrightarrow0\qquad(N\longrightarrow\infty).
$$
By \cite[Theorem~18.27, p.~210--211]{CRV}, $K_n$ is totally bounded, and therefore compact. Since $X_n^{-1}(\mathbb{R}\setminus\{0\})\subset K_n$, conclusion 1) follows.

We next prove conclusion 2). On $\{Q<\infty\}$,
$$
\frac{\partial (Q/n^2)}{\partial x_i}=\frac{2d_ix_i}{n^2},\qquad
\frac{\partial^2 (Q/n^2)}{\partial x_i\partial x_j}=\frac{2d_i\delta_{ij}}{n^2},
\qquad \partial^{\alpha}(Q/n^2)=0\;\text{ if $|\alpha|\ge 3$}.
$$
Let $\mathcal{P}_{1,2}(k)$ denote the set of partitions of $\{1,\ldots,k\}$ into sets of size one or two. For each partition $\pi\in\mathcal{P}_{1,2}(k)$, let $s$ and $t$ denote the numbers of its one-element and two-element sets, respectively. Thus $s+2t=k$.
The higher-order chain rule gives
\[
\frac{\partial^k X_n(\mathbf{x})}{\partial{x_{i_1}}\cdots\partial{x_{i_k}}}=\sum_{\pi\in\mathcal{P}_{1,2}(k)}
T_{\pi,n}(\mathbf{x};i_1,\ldots,i_k),
\]
where
$$
T_{\pi,n}(\mathbf{x};i_1,\ldots,i_k)
\triangleq\frac{2^{s+t}}{n^{2(s+t)}}\chi^{(s+t)}\left(\frac{Q(\mathbf{x})}{n^2}\right)
\prod_{\{r\}\in\pi}d_{i_r}x_{i_r}
\prod_{\{r,q\}\in\pi}d_{i_r}\delta_{i_ri_q}.
$$
Moreover, $T_{\pi,n}(\mathbf{x};i_1,\ldots,i_k)\ne0$ implies $Q(\mathbf{x})<4n^2$.
By \eqref{20260917di},
$$
\sup_{i\in\mathbb{N}}a_id_i<\infty,\qquad
\sum_{i=1}^{\infty}a_i^2d_i<\infty.
$$
Consequently,
$$
\begin{aligned}
\sum_{i_1,\ldots,i_k=1}^{\infty}a_{i_1}\cdots a_{i_k}|T_{\pi,n}(\mathbf{x};i_1,\ldots,i_k)|^2&\le C(k)n^{-4(s+t)}
\left(\sum_{i=1}^{\infty}a_id_i^2x_i^2\right)^s
\left(\sum_{i=1}^{\infty}a_i^2d_i^2\right)^t\\
&\le C(k)n^{-4(s+t)+2s}=C(k)n^{-2k},
\end{aligned}
$$
where $\sum_i a_id_i^2x_i^2\le(\sup_i a_id_i)Q(\mathbf{x})$. Also, by the Cauchy--Schwarz inequality,
$$
\begin{aligned}
\sum_{i_1,\ldots,i_k=1}^{\infty}a_{i_1}\cdots a_{i_k}|T_{\pi,n}(\mathbf{x};i_1,\ldots,i_k)|&\le C(k)n^{-2(s+t)}
\left(\sum_{i=1}^{\infty}a_id_i|x_i|\right)^s
\left(\sum_{i=1}^{\infty}a_i^2d_i\right)^t\\
&\le C(k)n^{-2(s+t)}Q(\mathbf{x})^{s/2}
\left(\sum_{i=1}^{\infty}a_i^2d_i\right)^{s/2+t}
\le C(k)n^{-k}.
\end{aligned}
$$
There are only finitely many partitions $\pi$ for a fixed $k$. Summing these estimates proves \eqref{20250724for2}, with the stronger upper bound $C(k)n^{-2k}$, and the estimate
$$
\sup_{\mathbf{x}\in\ell^2}\sum_{i_1,\ldots,i_k=1}^{\infty}a_{i_1}\cdots a_{i_k}
\left|\frac{\partial^kX_n(\mathbf{x})}{\partial x_{i_1}\cdots\partial x_{i_k}}\right|\le C(k)n^{-k}.
$$
For every nonnegative summable sequence $\{b_j\}$ and $p\ge1$, $\sum_jb_j^p\le(\sum_jb_j)^p$. Applying this inequality to the last estimate gives \eqref{20260917cutoffestimate}. On $\{Q=\infty\}$ all the derivatives vanish, so the estimates hold on the whole space.

Finally, Tonelli's theorem and \eqref{20260917di} give
$$
\int_{\ell^2}Q(\mathbf{x})\,\mathrm{d}P(\mathbf{x})
=\sum_{i=1}^{\infty}d_ia_i^2<\infty.
$$
Thus $Q(\mathbf{x})<\infty$ for $P$-almost every $\mathbf{x}$. For such a point, $Q(\mathbf{x})/n^2<1$ for every sufficiently large $n$. Since $\chi$ is constant on $(-\infty,1]$, conclusion 3) follows. This completes the proof.
\end{proof}

The rest of this subsection provides the localization results used in both approximation theorems.

\begin{lemma}\label{20260917compactcutoff}
Let $U\subset\ell^2$ be open, $S\subset U$ be a nonempty compact set, and $0<r<\operatorname{dist}(S,\ell^2\setminus U)$. There exists
$$
\eta\in C_{\ell^2}^{\infty}(\ell^2)\bigcap C_{F^{\infty}}^{\infty}(\ell^2)\bigcap C_{0,\mathscr{F}}^{\infty}(\ell^2)
$$
such that $0\le\eta\le1$, $\eta=1$ on an open neighborhood of $S$, and
$$
 \eta^{-1}(\mathbb{R}\setminus\{0\})\subset\{\mathbf{x}\in\ell^2:\operatorname{dist}(\mathbf{x},S)<r\},\qquad \eta^{-1}(\mathbb{R}\setminus\{0\})\s U.
$$
\end{lemma}
\begin{proof}
Choose $\mathbf{x}_1,\ldots,\mathbf{x}_M\in S$ such that $S\subset\bigcup_{j=1}^M B(\mathbf{x}_j,r/8)$. Take $\chi\in C^{\infty}(\mathbb{R};[0,1])$ with $\chi(t)=1$ for $t\le r^2/64$ and $\chi(t)=0$ for $t\ge r^2/16$. Set
$$
\eta(\mathbf{x})\triangleq 1-\prod_{j=1}^M\left(1-\chi(\|\mathbf{x}-\mathbf{x}_j\|^2)\right),\qquad\mathbf{x}\in\ell^2.
$$
By Lemmas \ref{chengjiguji}, \ref{fuheguji} and \ref{1f1f13f1f3},
\[
\eta\in C_{\ell^2}^{\infty}(\ell^2)\cap C_{F^{\infty}}^{\infty}(\ell^2)\cap C_{\mathscr{F}}^{\infty}(\ell^2,loc).
\]
Choose $R>\max_{1\le j\le M}\|\mathbf{x}_j\|+r/4$.
Since $\partial^\alpha\eta=0$ on $\ell^2\setminus B_R$,
\[
\sup_{\ell^2}|\partial^\alpha\eta|
=\sup_{B_R}|\partial^\alpha\eta|<\infty,\qquad
\|\eta\|_{C_{\mathscr{F}}^{m,p}(\ell^2)}
=\|\eta\|_{C_{\mathscr{F}}^{m,p}(B_R)}<\infty
\]
for all $\alpha\in\mathbb{N}_0^{(\mathbb{N})}$,
$m\in\mathbb{N}$ and $p\in[1,\infty)$.
Thus $\eta\in C_{0,\mathscr{F}}^\infty(\ell^2)$.

The function $\eta$ equals $1$ on $\bigcup_{j=1}^M B(\mathbf{x}_j,r/8)$ and vanishes outside $\bigcup_{j=1}^M\overline{B(\mathbf{x}_j,r/4)}$. Hence its support is bounded and is contained in $\{\operatorname{dist}(\cdot,S)\le r/4\}$. Its distance from $\ell^2\setminus U$ is at least $\operatorname{dist}(S,\ell^2\setminus U)-r/4>0$. These observations prove the remaining assertions.
\end{proof}

\begin{lemma}\label{20260917finiteonuniformsets}
Let $K\subset\ell^2$ be compact, and let $\{A_i\}_{i=1}^{\infty}$ be a locally finite family of subsets of ${O}$. For each $i\in\mathbb{N}$, let $S_i\subset K\cap A_i$ be compact, and let $r_i>0$, with $r_i\to0$. Suppose that
$$
E_i\subset\{\mathbf{x}\in\ell^2:\operatorname{dist}(\mathbf{x},S_i)\le r_i\},\qquad i\in\mathbb{N},
$$
where $E_i=\varnothing$ when $S_i=\varnothing$. Then every set $E\s{O}$ meets only finitely many of the sets $E_i$.
\end{lemma}
\begin{proof}
Suppose that $E\s{O}$ meets infinitely many $E_i$. There are distinct indices $i_j\to\infty$ and points $\mathbf{y}_j\in E\cap E_{i_j}$, $\mathbf{x}_j\in S_{i_j}$ such that $\|\mathbf{x}_j-\mathbf{y}_j\|\le r_{i_j}+j^{-1}$. After passage to a subsequence, compactness gives $\mathbf{x}_j\to\mathbf{x}\in K$, and therefore $\mathbf{y}_j\to\mathbf{x}$. Since $E\s{O}$, its closure in $\ell^2$ is contained in ${O}$; hence $\mathbf{x}\in{O}$. A neighborhood of $\mathbf{x}$ meets only finitely many $A_i$. On the other hand, $\mathbf{x}_j\in A_{i_j}$ and $\mathbf{x}_j\to\mathbf{x}$, which is a contradiction.
\end{proof}

\subsection{Sobolev spaces}\label{sec4}
By \cite[Lemma~3.1, p.~17]{WYZ1}, $C_{0,F^\infty}^{\infty}({O})$ is dense in $L^p({O},P)$.
We will use $C_{0,F^{\infty}}^{\infty}({O})$ as the test function space to define Sobolev spaces on ${O}$ and establish the properties needed in the approximation proofs.

\begin{definition}\label{def:weak}
Let $f,g\in L^p({O},P)$ and $\alpha\in\mathbb{N}_0^{(\mathbb{N})}$. We write $\partial^{\alpha}f=g$ if
$$
\int_{{O}}f\delta^{\alpha}\phi\,\mathrm{d}P
=(-1)^{|\alpha|}\int_{{O}}g\phi\,\mathrm{d}P,
\qquad\forall\,\phi\in C_{0,F^{\infty}}^{\infty}({O}).
$$
\end{definition}


As in \cite[(2.3), p. 9]{WYZ1}, for each $f\in L^1(\ell^2, P)$ and $n\in\mathbb{N}$, we define a function $f_n$ on $\mathbb{R}^n$ as follows:
	\begin{equation}\label{20250920def1}
		f_n(\mathbf{x}_n)\triangleq \int_{\ell^2(\mathbb{N}\setminus\{1,\dots,n\})} f(\mathbf{x}_n,\mathbf{x}^n)\,\mathrm{d}P^{\widehat{1,\dots,n}}(\mathbf{x}^n),
	\end{equation}
	where $\mathbf{x}^n=(x_{i})_{i\in \mathbb{N}\setminus\{1,\dots,n\}}\in \ell^2(\mathbb{N}\setminus\{1,\dots,n\})$ and $\mathbf{x}_n=(x_1,\dots,x_n)\in \mathbb{R}^n$.
By a similar argument to its counterpart in \cite[Proposition 2.2, pp. 9--10]{WYZ1} (for the special case of $p=2$), one can prove the following result:
\begin{proposition}\label{Reduce dimension}
	For each $f\in L^{p} ( \ell^{2} ,P )$, the function $f_n$ (defined by \eqref{20250920def1}) can be regarded as a cylindrical function on $\ell^2$ (depending only on the first $n$ variables) and satisfies the following properties:
	\begin{itemize}
		\item[$\mathrm{1)}$]\label{z1}
		$\| f_{n}\|_{L^{p} ( \ell^{2},P )  }  \le \|f\|_{L^{p} ( \ell^{2},P )  }  $;
		\item[$\mathrm{2)}$] \label{z2}
		$\lim\limits_{n\rightarrow \infty } \| f_{n}-f\|_{L^{p} ( \ell^{2},P)  }  =0$;
		\item[$\mathrm{3)}$] \label{z3}
		$\lim\limits_{n\rightarrow \infty } \int_{\ell^2}\big|| f_{n}|^{p}  -| f|^{p}\big|\,\mathrm{d}P =0.$	
	\end{itemize}
\end{proposition}

\begin{lemma}\label{20260917testseparation}
If $h\in L^1({O},P)$ satisfies $\int_{{O}}h\phi\,\mathrm{d}P=0$ for all $\phi\in C_{0,F^{\infty}}^{\infty}({O})$, then $h=0$ $P$-almost everywhere on ${O}$. The same assertion holds on any open subset of $\ell^2(\mathbb{N}\setminus\{1,\ldots,n\})$, with the corresponding product Gaussian measure.
\end{lemma}
\begin{proof}
Apply Theorem \ref{danweifenjie} to a cover of ${O}$ by balls uniformly included in ${O}$, and denote the resulting partition by $\{\rho_i\}_{i=1}^{\infty}$. Extend each $h\rho_i$ by zero to $\ell^2$. For every $i,n\in\mathbb{N}$ and $\theta\in C_c^{\infty}(\mathbb{R}^n)$, the function $\rho_i(\mathbf{x})\theta(x_1,\ldots,x_n) \in C_{0,F^{\infty}}^{\infty}({O})$.  Hence$$\int_{\mathbb{R}^{n}} \left( h\rho_{i} \right)_{n} \theta \,\mathrm{d} \mathcal{N}^{n} =\int_{\ell^{2}} h\rho_{i} \theta \,\mathrm{d} P=\int_{{O}} h\rho_{i} \theta \mathrm{d} P=0,\qquad i=1,2,\cdots.$$
We have $$\left( h\rho_{i} \right)_{n} =0\quad \text{a.e.},\qquad i=1,2,\cdots.$$By Proposition \ref{Reduce dimension},
\[
\int_{\ell^{2}} \left| h\rho_{i} \right| \,\mathrm{d}P
=
\lim_{n\rightarrow\infty}
\int_{\ell^{2}}
\left| \left(h\rho_{i}\right)_{n} \right| \,\mathrm{d}P
=0,
\qquad i=1,2,\cdots.
\]
Hence
\[
h=\sum_{i=1}^{\infty}h\rho_{i}=0
\qquad \text{a.e.},
\]
which completes the proof.
\end{proof}

\begin{lemma}\label{20260917slices}
Let $f,g\in L^p({O},P)$, let $\alpha\in\mathbb{N}_0^{(\mathbb{N})}$, and choose $n\in\mathbb{N}$ such that $\alpha_j=0$ for $j>n$. Write $\mathbf{x}=(\mathbf{x}_n,\mathbf{x}^n)$ and define the open slice
$$
{O}_{\mathbf{x}^n}\triangleq\{\mathbf{x}_n\in\mathbb{R}^n:(\mathbf{x}_n,\mathbf{x}^n)\in{O}\}.
$$
Then $\partial^{\alpha}f=g$ if and only if, for $P^{\widehat{1,\ldots,n}}$-almost every $\mathbf{x}^n$, the ordinary distributional derivative of $f(\cdot,\mathbf{x}^n)$ of order $\alpha$ on ${O}_{\mathbf{x}^n}$ equals $g(\cdot,\mathbf{x}^n)$.
\end{lemma}
\begin{proof}
	Write
	$$
	P=\mathcal{N}^n\times P^{\widehat{1,\ldots,n}}.
	$$
	By Fubini's theorem, for $P^{\widehat{1,\ldots,n}}$-almost every $\mathbf{x}^n$, the functions
	$f(\cdot,\mathbf{x}^n)$ and $g(\cdot,\mathbf{x}^n)$
	belong to $L^p({O}_{\mathbf{x}^n},\mathcal{N}^n)$. Since the Gaussian measure $\mathcal{N}^n$ has a strictly positive smooth density with respect to Lebesgue measure, these functions are locally integrable with respect to Lebesgue measure on ${O}_{\mathbf{x}^n}$.
	
	Suppose first that
	$\partial^\alpha f=g$
	on ${O}$. Since $\ell^2$ is separable and ${O}$ is open, there exists a countable family of product neighborhoods
	$$
	V_j\times U_j\s{O},\qquad j\in\mathbb{N}
	$$
	covering ${O}$, where $V_j\subset\mathbb{R}^n$ and
	$U_j\subset\ell^2(\mathbb{N}\setminus\{1,\ldots,n\})$
	are bounded open sets.

Fix $j\in\mathbb{N}$ and $\theta\in C_c^\infty(V_j)$. For every
$\xi\in C_{0,F^\infty}^\infty(U_j),$
the zero extension to ${O}$ of the function
$$
(\mathbf{x}_n,\mathbf{x}^n)\longmapsto
\theta(\mathbf{x}_n)\xi(\mathbf{x}^n)
$$
belongs to $C_{0,F^\infty}^\infty({O})$.
Since $\alpha_k=0$ for $k>n$, Definition~\ref{def:weak} and Fubini's theorem give
	$$
	\begin{aligned}
		0&=\int_{U_j}\xi(\mathbf{x}^n)
		\Bigg[
		\int_{V_j}
		f(\mathbf{x}_n,\mathbf{x}^n)
		\delta^\alpha\theta(\mathbf{x}_n)
		\,\mathrm{d}\mathcal{N}^n(\mathbf{x}_n)\\
		&\hspace{23mm}
		-(-1)^{|\alpha|}
		\int_{V_j}
		g(\mathbf{x}_n,\mathbf{x}^n)
		\theta(\mathbf{x}_n)
		\,\mathrm{d}\mathcal{N}^n(\mathbf{x}_n)
		\Bigg]
		\,\mathrm{d}P^{\widehat{1,\ldots,n}}(\mathbf{x}^n).
	\end{aligned}
	$$
	For fixed $j$ and $\theta$, the function
	$$
	\begin{aligned}
		H_{j,\theta}(\mathbf{x}^n)\triangleq{}
		\int_{V_j}
		f(\mathbf{x}_n,\mathbf{x}^n)
		\delta^\alpha\theta(\mathbf{x}_n)
		\,\mathrm{d}\mathcal{N}^n(\mathbf{x}_n)
		-
		(-1)^{|\alpha|}
		\int_{V_j}
		g(\mathbf{x}_n,\mathbf{x}^n)
		\theta(\mathbf{x}_n)
		\,\mathrm{d}\mathcal{N}^n(\mathbf{x}_n)
	\end{aligned}
	$$
	belongs to $L^1(U_j,P^{\widehat{1,\ldots,n}})$. Hence Lemma \ref{20260917testseparation}, applied to $\ell^2(\mathbb{N}\setminus\{1,\dots,n\})$, yields
	$$
	H_{j,\theta}(\mathbf{x}^n)=0\qquad\text{for $P^{\widehat{1,\ldots,n}}$-almost every $\mathbf{x}^n\in U_j$}.
	$$
	
	Recall that $C_c^\infty(V_j)$ is separable in its usual topology. Choose a countable dense subset $\mathscr{D}_j\subset C_c^\infty(V_j)$. Since $\mathscr{D}_j$ is countable, there exists a set $E_j\subset U_j$ with
$P^{\widehat{1,\ldots,n}}(U_j\setminus E_j)=0$ such that, for every $\mathbf{x}^n\in E_j$, $f(\cdot,\mathbf{x}^n)$ and $g(\cdot,\mathbf{x}^n)$ are locally integrable, and
	$$
	H_{j,\theta}(\mathbf{x}^n)=0,
	\qquad
	\mathbf{x}^n\in E_j,\qquad \theta\in\mathscr{D}_j.
	$$
	For every $\mathbf{x}^n$ for which $f(\cdot,\mathbf{x}^n)$ and $g(\cdot,\mathbf{x}^n)$ are locally integrable, the map
	$$
	\theta\longmapsto H_{j,\theta}(\mathbf{x}^n)
	$$
	is a continuous linear functional on $C_c^\infty(V_j)$. Therefore, by density,
	$$
	\int_{V_j}
	f(\mathbf{x}_n,\mathbf{x}^n)
	\delta^\alpha\theta(\mathbf{x}_n)
	\,\mathrm{d}\mathcal{N}^n(\mathbf{x}_n)
	=
	(-1)^{|\alpha|}
	\int_{V_j}
	g(\mathbf{x}_n,\mathbf{x}^n)
	\theta(\mathbf{x}_n)
	\,\mathrm{d}\mathcal{N}^n(\mathbf{x}_n)
	$$
	for every $\theta\in C_c^\infty(V_j)$ and for $P^{\widehat{1,\ldots,n}}$-almost every $\mathbf{x}^n\in U_j$.
	
	Let
	$$
	w_n(\mathbf{x}_n)
	\triangleq\left(\prod_{k=1}^n\frac{1}{\sqrt{2\pi}a_k}\right)\exp\left(
	-\sum_{k=1}^n\frac{x_k^2}{2a_k^2}\right),
	$$
	so that
	$$
	\mathrm{d}\mathcal{N}^n(\mathbf{x}_n)
	=w_n(\mathbf{x}_n)\,\mathrm{d}\mathbf{x}_n.
	$$
	By the definition of $\partial^\alpha$ and $\delta^\alpha$, we have
	$$
	w_n\,\delta^\alpha\theta=\partial^\alpha(w_n\theta),
	$$
	which implies
	$$
	\int_{V_j}
	f(\mathbf{x}_n,\mathbf{x}^n)\partial^\alpha\!\bigl(w_n\theta\bigr)(\mathbf{x}_n)
	\,\mathrm{d}\mathbf{x}_n
	=
	(-1)^{|\alpha|}
	\int_{V_j}
	g(\mathbf{x}_n,\mathbf{x}^n)w_n(\mathbf{x}_n)\theta(\mathbf{x}_n)
	\,\mathrm{d}\mathbf{x}_n.
	$$
	Since $w_n$ is strictly positive and belongs to $C^\infty(\mathbb{R}^n)$, multiplication by $w_n$ is an automorphism of $C_c^\infty(V_j)$. Hence, for $P^{\widehat{1,\ldots,n}}$-almost every $\mathbf{x}^n\in U_j$,
	\begin{equation}\label{eq:exceptional}
	\int_{V_j}
	f(\mathbf{x}_n,\mathbf{x}^n)
	\partial^\alpha\psi(\mathbf{x}_n)
	\,\mathrm{d}\mathbf{x}_n
	=
	(-1)^{|\alpha|}
	\int_{V_j}
	g(\mathbf{x}_n,\mathbf{x}^n)
	\psi(\mathbf{x}_n)
	\,\mathrm{d}\mathbf{x}_n
	\end{equation}
	for every $\psi\in C_c^\infty(V_j)$. Thus the ordinary distributional derivative of
	$f(\cdot,\mathbf{x}^n)$ of order $\alpha$ equals
	$g(\cdot,\mathbf{x}^n)$ on $V_j$.
	
	Since the family $\{V_j\times U_j\}_{j\in\mathbb{N}}$ is countable, the set
	\[Z\triangleq\bigcup_{j=1}^{\infty}\left(U_j\setminus E_j\right)\]
	has $P^{\widehat{1,\ldots,n}}$-measure zero. If $\mathbf{x}^n\notin Z$, then the sets $V_j$ for those $j$ with $\mathbf{x}^n\in U_j$ cover ${O}_{\mathbf{x}^n}$, and \eqref{eq:exceptional} holds for every such $j$ and all $\psi\in C_c^\infty(V_j)$.  By the locality of distributional derivatives, it follows that
	$\partial^\alpha\bigl(f(\cdot,\mathbf{x}^n)\bigr)=g(\cdot,\mathbf{x}^n)$
	in the sense of distributions on ${O}_{\mathbf{x}^n}$.
	
	Conversely, suppose that, for $P^{\widehat{1,\ldots,n}}$-almost every $\mathbf{x}^n$,
	$\partial^\alpha\bigl(f(\cdot,\mathbf{x}^n)\bigr)=g(\cdot,\mathbf{x}^n)$
	in the sense of ordinary distributions on ${O}_{\mathbf{x}^n}$. Let
	$\phi\in C_{0,F^\infty}^\infty({O}).$
	For every $\mathbf{x}^n$ for which $\partial^\alpha\bigl(f(\cdot,\mathbf{x}^n)\bigr)=g(\cdot,\mathbf{x}^n)$ holds, the function
	$$
	\mathbf{x}_n\longmapsto\phi(\mathbf{x}_n,\mathbf{x}^n)
	$$
	belongs to $C_c^\infty({O}_{\mathbf{x}^n})$. Indeed, the support of $\phi$ is bounded and has positive distance from $\ell^2\setminus{O}$; hence its intersection with each finite-dimensional slice is a compact subset of ${O}_{\mathbf{x}^n}$.
	Therefore
	$$
	\psi_{\mathbf{x}^n}(\mathbf{x}_n)
	\triangleq
	w_n(\mathbf{x}_n)\phi(\mathbf{x}_n,\mathbf{x}^n)
	$$
	belongs to $C_c^\infty({O}_{\mathbf{x}^n})$, and
	$$
	\int_{{O}_{\mathbf{x}^n}}
	f(\mathbf{x}_n,\mathbf{x}^n)
	\partial^\alpha\psi_{\mathbf{x}^n}(\mathbf{x}_n)
	\,\mathrm{d}\mathbf{x}_n
	=
	(-1)^{|\alpha|}
	\int_{{O}_{\mathbf{x}^n}}
	g(\mathbf{x}_n,\mathbf{x}^n)
	\psi_{\mathbf{x}^n}(\mathbf{x}_n)
	\,\mathrm{d}\mathbf{x}_n.
	$$
	Using
	$$
	\partial^\alpha\!\left(
	w_n(\mathbf{x}_n)\phi(\mathbf{x}_n,\mathbf{x}^n)
	\right)
	=
	w_n(\mathbf{x}_n)
	\delta^\alpha\phi(\mathbf{x}_n,\mathbf{x}^n),
	$$
	we obtain
	\begin{equation}\label{eq:2.17}
	\int_{{O}_{\mathbf{x}^n}}
	f(\mathbf{x}_n,\mathbf{x}^n)
	\delta^\alpha\phi(\mathbf{x}_n,\mathbf{x}^n)\,\mathrm{d}\mathcal{N}^n(\mathbf{x}_n)
	=(-1)^{|\alpha|}
	\int_{{O}_{\mathbf{x}^n}}
	g(\mathbf{x}_n,\mathbf{x}^n)
	\phi(\mathbf{x}_n,\mathbf{x}^n)\,\mathrm{d}\mathcal{N}^n(\mathbf{x}_n)
	\end{equation}
	for $P^{\widehat{1,\ldots,n}}$-almost every $\mathbf{x}^n$.
	The integrands on both sides of \eqref{eq:2.17} are integrable on ${O}$ with respect to $P$. Indeed, $\phi$ and $\delta^\alpha\phi$ are bounded, while $f,g\in L^p({O},P)$ and $P({O})\le1$. Hence Fubini's theorem yields
	$$
	\int_{O} f\,\delta^\alpha\phi\,\mathrm{d}P
	=
	(-1)^{|\alpha|}
	\int_{O} g\,\phi\,\mathrm{d}P.
	$$
	This is precisely the identity
	$\partial^\alpha f=g$ on ${O}$ since $\phi\in C_{0,F^\infty}^\infty({O})$ is arbitrary.
\end{proof}

The proof also shows that weak derivatives can be restricted to open subsets and are local: if a proposed derivative belongs to $L^p({O},P)$ and satisfies the defining identity in a neighborhood of every point, then it satisfies the identity on ${O}$. To see the latter assertion, take a countable subcover of these neighborhoods, apply Lemma \ref{20260917slices} on each, and use locality on the finite-dimensional slices.

\begin{lemma}\label{1f3ff32}
If $f,\partial^{\alpha}f\in L^p({O},P)$, then $\supp \partial^{\alpha}f\subset\supp f$, with both supports taken relative to ${O}$.
\end{lemma}
\begin{proof}
	Every $\phi \in C_{0,F^{\infty}}^{\infty}\left( {O} \setminus \supp f\right)$ has a zero extension belonging to $C_{0,F^\infty}^\infty({O})$. Hence the defining identity yields
	$$
	\int_{{O}\setminus\supp f}\partial^{\alpha}f\,\phi\,\mathrm{d}P=0.
	$$
	By Lemma \ref{20260917testseparation},
	$\partial^{\alpha}f=0$
	almost everywhere on ${O}\setminus\supp f$. This proves Lemma \ref{1f3ff32}.
\end{proof}

\begin{lemma}\label{nini1f3ff3}
	Let $f,\partial^{\alpha}f\in L^p({O},P)$. For every $\phi\in C_{\mathcal{F},b}^{|\alpha|}({O})$ such that $\supp\phi\s{O}$, one has $$ \int_{{O}}f\delta^{\alpha}\phi\,\mathrm{d}P =(-1)^{|\alpha|}\int_{{O}}\partial^{\alpha}f\phi\,\mathrm{d}P. $$ If $\supp f\s{O}$, the assumption $\supp\phi\s{O}$ can be omitted.
\end{lemma}
\begin{proof}
	Choose $n\in\mathbb{N}$ such that $\alpha_j=0$ for every $j>n$, and write
	$\mathbf{x}=(\mathbf{x}_n,\mathbf{x}^n).$
	Let
	$w_n(\mathbf{x}_n)$ be as in Lemma~\ref{20260917slices} so that
	$$
	\mathrm{d}\mathcal{N}^n(\mathbf{x}_n)
	=
	w_n(\mathbf{x}_n)\,\mathrm{d}\mathbf{x}_n.
	$$
	By Lemma \ref{20260917slices}, for
	$P^{\widehat{1,\ldots,n}}$-almost every $\mathbf{x}^n$,
	$\partial^\alpha\left(f(\cdot,\mathbf{x}^n)\right)=\partial^\alpha f(\cdot,\mathbf{x}^n)$
	in the  distributional sense on ${O}_{\mathbf{x}^n}$. Thus, for every
	$\psi\in C_c^\infty({O}_{\mathbf{x}^n})$,
	\begin{eqnarray}\label{20260923sliceordinary}
		\int_{{O}_{\mathbf{x}^n}}
		f(\mathbf{x}_n,\mathbf{x}^n)
		\partial^\alpha\psi(\mathbf{x}_n)
		\,\mathrm{d}\mathbf{x}_n
		=
		(-1)^{|\alpha|}
		\int_{{O}_{\mathbf{x}^n}}
		\partial^\alpha f(\mathbf{x}_n,\mathbf{x}^n)
		\psi(\mathbf{x}_n)
		\,\mathrm{d}\mathbf{x}_n.
	\end{eqnarray}
	
	We first assume that $\supp\phi\s{O}$. Put $L\triangleq\supp\phi.$
	Then $L$ is bounded and closed in $\ell^2$. For each $\mathbf{x}^n$, set
	$$
	L_{\mathbf{x}^n}
	\triangleq
	\{\mathbf{x}_n\in\mathbb{R}^n:
	(\mathbf{x}_n,\mathbf{x}^n)\in L\}.
	$$
	Whenever $L_{\mathbf{x}^n}\neq\varnothing$, it is a compact subset of
	${O}_{\mathbf{x}^n}$. Indeed, $L\s{O}$ implies $\operatorname{dist}(L,\ell^2\setminus{O})>0,$
	and therefore
	$$
	\operatorname{dist}
	\bigl(L_{\mathbf{x}^n},\mathbb{R}^n\setminus{O}_{\mathbf{x}^n}\bigr)
	\ge \operatorname{dist}(L,\ell^2\setminus{O}).
	$$
	
Fix a $\mathbf{x}^n$ such that
\eqref{20260923sliceordinary} holds and that $f(\cdot,\mathbf{x}^n)$, $\partial^\alpha f(\cdot,\mathbf{x}^n)$ are locally integrable. By the definition of
	$C_{\mathcal F,b}^{|\alpha|}({O})$, the function
	$\phi_{\mathbf{x}^n}(\mathbf{x}_n)\triangleq\phi(\mathbf{x}_n,\mathbf{x}^n)$
	belongs to $C^{|\alpha|}({O}_{\mathbf{x}^n})$. Since $\phi(\mathbf{x})=0$ for $P$-almost every $\mathbf{x}\in{O}\setminus L$,
by Fubini's theorem and the $\mathcal{F}$-continuity of $\phi$, we may assume
	$\supp\phi_{\mathbf{x}^n}\subset L_{\mathbf{x}^n}$.
	Hence
	\[\psi_{\mathbf{x}^n}\triangleq w_n\phi_{\mathbf{x}^n}\in C_c^{|\alpha|}({O}_{\mathbf{x}^n}).\]

	We claim that \eqref{20260923sliceordinary} remains valid with
	$\psi=\psi_{\mathbf{x}^n}$. To see this, choose a compact set
	$K_0\subset{O}_{\mathbf{x}^n}$ whose interior contains
	$\supp\psi_{\mathbf{x}^n}$. By the standard finite-dimensional
	mollification argument, there exists a sequence
	$\{\psi_k\}\subset C_c^\infty({O}_{\mathbf{x}^n})$ with $\supp\psi_k\subset K_0,$
	such that
	$$
	\partial^\beta\psi_k
	\longrightarrow
	\partial^\beta\psi_{\mathbf{x}^n}
	\quad\text{uniformly on }\mathbb{R}^n
	$$
	for every $\beta\in\mathbb{N}_0^{(\mathbb{N})}$ with $|\beta|\le|\alpha|$.
	Since
	$f(\cdot,\mathbf{x}^n),
	\partial^\alpha f(\cdot,\mathbf{x}^n)
	\in L^1(K_0),$
	letting $k\to\infty$ in \eqref{20260923sliceordinary} gives
	$$
	\int_{{O}_{\mathbf{x}^n}}
		f(\mathbf{x}_n,\mathbf{x}^n) \partial^\alpha(w_n\phi_{\mathbf{x}^n})(\mathbf{x}_n)
		\,\mathrm{d}\mathbf{x}_n
		=
		(-1)^{|\alpha|}\int_{{O}_{\mathbf{x}^n}}
		\partial^\alpha f(\mathbf{x}_n,\mathbf{x}^n)w_n(\mathbf{x}_n)\phi(\mathbf{x}_n,\mathbf{x}^n)\,\mathrm{d}\mathbf{x}_n.
	$$
	By the definition of $\partial^\alpha$ and $\delta^\alpha$,
	$$
	\partial^\alpha
	\bigl(
	w_n\phi_{\mathbf{x}^n}
	\bigr)
	=
	w_n\delta^\alpha\phi_{\mathbf{x}^n}.
	$$
	Consequently,
	\begin{eqnarray}\label{20260923slicegaussian}
		\begin{aligned}
			\int_{{O}_{\mathbf{x}^n}}
			f(\mathbf{x}_n,\mathbf{x}^n)
			\delta^\alpha\phi(\mathbf{x}_n,\mathbf{x}^n)
			\,\mathrm{d}\mathcal{N}^n(\mathbf{x}_n)
			=
			(-1)^{|\alpha|}
			\int_{{O}_{\mathbf{x}^n}}
			\partial^\alpha f(\mathbf{x}_n,\mathbf{x}^n)
			\phi(\mathbf{x}_n,\mathbf{x}^n)
			\,\mathrm{d}\mathcal{N}^n(\mathbf{x}_n).
		\end{aligned}
	\end{eqnarray}
	
	Since $L$ is bounded and all partial derivatives of $\phi$ of order at most
	$|\alpha|$ are bounded, $\delta^\alpha\phi$ is bounded on $L$. Moreover,
	$\phi$ and all its partial derivatives vanish $P$-almost everywhere on
${O}\setminus L$. Hence there exists
	$C>0$ such that
	$$
	|f\delta^\alpha\phi|+|\partial^\alpha f\,\phi|
	\le
	C\bigl(|f|+|\partial^\alpha f|\bigr)\mathbf{1}_L,\qquad P\text{-a.e. on }O.
	$$
	Since $P$ is finite and
	$f,\partial^\alpha f\in L^p({O},P)$, the right-hand side belongs to
	$L^1({O},P)$. Fubini's theorem applied to
	\eqref{20260923slicegaussian} therefore yields
	$$
	\int_{O} f\delta^\alpha\phi\,\mathrm{d}P
	=
	(-1)^{|\alpha|}
	\int_{O} \partial^\alpha f\,\phi\,\mathrm{d}P.
	$$
	
	We now assume that $\supp f\s{O}$ and no longer impose any support
	condition on $\phi$. Set
	$K\triangleq\supp f.$
	By Lemma \ref{1f3ff32},
	$$
	\supp \partial^\alpha f\subset K.
	$$
	Since $K\s{O}$, we have
	$$
	\operatorname{dist}(K,\ell^2\setminus{O})>0.
	$$
	For $\mathbf{x}^n\in\ell^2(\dbN\setminus\{1,\dots,n\})$, define
	$$
	K_{\mathbf{x}^n}
	\triangleq
	\{\mathbf{x}_n\in\mathbb{R}^n:
	(\mathbf{x}_n,\mathbf{x}^n)\in K\}.
	$$
	For $P^{\widehat{1,\ldots,n}}$-almost every $\mathbf{x}^n$,
	\begin{eqnarray}\label{20260923slicesupport}
		\begin{aligned}
			f(\cdot,\mathbf{x}^n)&=0
			&&\text{almost everywhere on }
			{O}_{\mathbf{x}^n}\setminus K_{\mathbf{x}^n},\\
			\partial^\alpha f(\cdot,\mathbf{x}^n)&=0
			&&\text{almost everywhere on }
			{O}_{\mathbf{x}^n}\setminus K_{\mathbf{x}^n}.
		\end{aligned}
	\end{eqnarray}
	This follows directly from Fubini's theorem applied to the almost-everywhere
	vanishing of $f$ and $\partial^\alpha f$ outside $K$.
	
	Next, fix a $\mathbf{x}^n$ such that \eqref{20260923sliceordinary} and \eqref{20260923slicesupport} hold, and that $f(\cdot,\mathbf{x}^n)$, $\partial^\alpha f(\cdot,\mathbf{x}^n)$ are locally integrable.
	If $K_{\mathbf{x}^n}=\varnothing$, both sides of \eqref{20260923slicegaussian}
	are zero by \eqref{20260923slicesupport}. Suppose therefore that
	$K_{\mathbf{x}^n}\neq\varnothing$. Then $K_{\mathbf{x}^n}$ is compact and
	$$
	\operatorname{dist}
	\bigl(
	K_{\mathbf{x}^n},
	\mathbb{R}^n\setminus{O}_{\mathbf{x}^n}
	\bigr)
	\ge \operatorname{dist}(K,\ell^2\setminus{O}).
	$$
	Choose $\eta_{\mathbf{x}^n}\in C_c^\infty({O}_{\mathbf{x}^n})$
	such that $\eta_{\mathbf{x}^n}=1$
	on an open neighborhood of $K_{\mathbf{x}^n}$. Then
	$$
	\eta_{\mathbf{x}^n}\phi_{\mathbf{x}^n}
	\in
	C_c^{|\alpha|}({O}_{\mathbf{x}^n}).
	$$
	By the finite-dimensional approximation argument already proved above,
	the ordinary distributional identity
	\eqref{20260923sliceordinary} can be applied to
	$w_n\eta_{\mathbf{x}^n}\phi_{\mathbf{x}^n}.$
	Thus
	$$
	\begin{aligned}
		&\int_{{O}_{\mathbf{x}^n}}
		f(\mathbf{x}_n,\mathbf{x}^n)
		\partial^\alpha
		\bigl(
		w_n\eta_{\mathbf{x}^n}\phi_{\mathbf{x}^n}
		\bigr)(\mathbf{x}_n)
		\,\mathrm{d}\mathbf{x}_n
		\\
		={}&
		(-1)^{|\alpha|}
		\int_{{O}_{\mathbf{x}^n}}
		\partial^\alpha f(\mathbf{x}_n,\mathbf{x}^n)
		w_n(\mathbf{x}_n)
		\eta_{\mathbf{x}^n}(\mathbf{x}_n)
		\phi_{\mathbf{x}^n}(\mathbf{x}_n)
		\,\mathrm{d}\mathbf{x}_n.
	\end{aligned}
	$$
	Since $\eta_{\mathbf{x}^n}=1$ on a neighborhood of
	$K_{\mathbf{x}^n}$, all derivatives of
	$\eta_{\mathbf{x}^n}$ of positive order vanish on
	$K_{\mathbf{x}^n}$. Hence
	$$
	\partial^\alpha
	\bigl(
	w_n\eta_{\mathbf{x}^n}\phi_{\mathbf{x}^n}
	\bigr)
	=
	\partial^\alpha
	\bigl(
	w_n\phi_{\mathbf{x}^n}
	\bigr)\qquad\text{on $K_{\mathbf{x}^n}$.}
	$$
	Together with
	\eqref{20260923slicesupport}, this gives
	$$
	\begin{aligned}
		&\int_{{O}_{\mathbf{x}^n}}
		f(\mathbf{x}_n,\mathbf{x}^n)
		\partial^\alpha
		\bigl(
		w_n\phi_{\mathbf{x}^n}
		\bigr)(\mathbf{x}_n)
		\,\mathrm{d}\mathbf{x}_n
		\\
		={}&
		(-1)^{|\alpha|}
		\int_{{O}_{\mathbf{x}^n}}
		\partial^\alpha f(\mathbf{x}_n,\mathbf{x}^n)
		w_n(\mathbf{x}_n)
		\phi_{\mathbf{x}^n}(\mathbf{x}_n)
		\,\mathrm{d}\mathbf{x}_n.
	\end{aligned}
	$$
	Using again
	$$
	\partial^\alpha(w_n\phi_{\mathbf{x}^n})
	=
	w_n\delta^\alpha\phi_{\mathbf{x}^n},
	$$
	we obtain \eqref{20260923slicegaussian} without any support assumption on
	$\phi$.
	
	Finally, $K$ is bounded. Since $\phi$ and all its partial derivatives of
	order at most $|\alpha|$ are bounded, $\phi$ and
	$\delta^\alpha\phi$ are bounded on $K$. By
	$\supp f,\supp \partial^\alpha f\subset K$, there exists $C>0$ such that
	$$
	|f\delta^\alpha\phi|+|\partial^\alpha f\,\phi|
	\le C\bigl(|f|+|\partial^\alpha f|\bigr)\mathbf{1}_K,\qquad P\text{-a.e. on }O,
	$$
	of which the right-hand side belongs to $L^1({O},P)$. Fubini's theorem therefore
	gives
	$$
	\int_{O} f\delta^\alpha\phi\,\mathrm{d}P
	=
	(-1)^{|\alpha|}
	\int_{O} \partial^\alpha f\,\phi\,\mathrm{d}P.
	$$
	This proves Lemma \ref{nini1f3ff3}.
\end{proof}

Now we define the Sobolev spaces considered in this paper.
\begin{definition}\label{def:Sob1}
For $m\in\mathbb{N}$ and $p\in[1,\infty)$, the Sobolev space $W^{m,p}( {O},P) $ is defined by
	\[W^{m,p}( {O},P)\triangleq\left\{\begin{aligned}
 f\in L^{p}(O,P):{}&
  \partial^{\alpha}f\in L^{p}(O,P),\
  \forall \alpha\in \mathbb{N}_{0}^{(\mathbb{N})}
  \text{ with } |\alpha|\le m,\\
& \text{and }
  \sum_{|\alpha|\le m}
  \int_{O} \left|\mathbf{a}^{\alpha}\partial^{\alpha}f\right|^{p}
  \mathrm{d}P<\infty
\end{aligned}\right\},\]
on which the norm is given by
	$$\| f \|_{W^{m,p}( {O},P)} \triangleq \left(\sum_{ \left| \alpha \right| \le m} \int_{{O}} \left| \mathbf{a}^{\alpha} \partial^{\alpha}f \right|^{p} \mathrm{d} P\right)^{\frac{1}{p}},\qquad\forall f\in W^{m,p}( {O},P).$$
Moreover, write $W_{0,u}^{m,p}( {O},P) \triangleq \left\{ f\in W^{m,p}( {O},P) :\supp f\s{O} \right\}.$
\end{definition}

\begin{remark}
	It is easy to see that $W^{m,p}( {O},P)$ is a Banach space and $\mathscr{C}^{\infty}_{c} \subset W^{m,p}( {O},P)$ (we suppress the restriction notation for simplicity). The above letter ``u" in $W_{0,u}^{m,p}( {O},P)$ stands for ``uniformly included", which is defined in Definition \ref{def of bounded contained}.
\end{remark}
\begin{remark}\label{rmk:anothernorm}
For $f\in L^{p}( {O},P)$ such that $\partial^{\alpha}f\in L^{p}( {O},P)$ for all $\alpha \in \mathbb{N}_{0}^{\left( \mathbb{N} \right)}$ with $ \left| \alpha \right| \le m$, we also have another definition of Sobolev norm as follows:
\begin{equation*}
	\| f \|_{\widetilde{W}^{m,p}( {O},P)} \triangleq  \left(\int_{{O}} \left| f \right|^p \mathrm{d} P+\sum_{k=1}^{m} \sum_{i_{1},i_{2},\cdots ,i_{k}=1}^{\infty} a_{i_{1}}^{p}a_{i_{2}}^{p}\cdots a_{i_{k}}^{p}\int_{{O}} \left| \frac{\partial^{k} f}{\partial x_{i_{1}}\partial x_{i_{2}}\cdots \partial x_{i_{k}}} \right|^{p} \mathrm{d} P\right)^{\frac{1}{p}}.
\end{equation*}
	Obviously,
	$$\|f\|_{W^{m,p}( {O},P)} \le \|f\|_{\widetilde{W}^{m,p}( {O},P)} \le \sqrt[p]{m!} \|f\|_{W^{m,p}( {O},P)},$$
	so they are equivalent norms.
\end{remark}
\begin{remark}\label{20250920rem1}
	The sum appearing in the definition of $\| f \|_{\widetilde{W}^{m,p}( {O},P)}$ coincides with \cite[(5.2.1), p. 211]{Bog98} when $p=2$.
\end{remark}

Motivated by \cite[(5.2.1), p. 211]{Bog98}, we also consider the following Sobolev space.
\begin{definition}\label{def:Sob2}
For $m\in\mathbb{N}$ and $p\in[1,\infty)$, the Sobolev space $\mathscr{W}^{m,p}( {O},P) $ is defined by
	\[\mathscr{W}^{m,p}( {O},P)\triangleq\left\{\begin{aligned}
 f\in L^{p}(O,P):{}&
  \partial^{\alpha}f\in L^{p}(O,P),\
  \forall \alpha\in \mathbb{N}_{0}^{(\mathbb{N})}
  \text{ with } |\alpha|\le m,\\
& \text{and }
  \int_{O}\left(\sum_{0<|\alpha|\le m}
   \left|\mathbf{a}^{\alpha}\partial^{\alpha}f\right|^{2}\right)^{\frac{p}{2}}
  \mathrm{d}P<\infty
\end{aligned}\right\},\]
on which the norm is given by
	$$\| f \|_{\mathscr{W}^{m,p}( {O},P)} \triangleq \left(\int_{{O}} \left| f \right|^{p} \mathrm{d} P\right)^{\frac{1}{p}}+\left(\int_{O}\left(\sum_{0<|\alpha|\le m}
   \left|\mathbf{a}^{\alpha}\partial^{\alpha}f\right|^{2}\right)^{\frac{p}{2}}
  \mathrm{d}P\right)^{\frac{1}{p}},\qquad\forall f\in \mathscr{W}^{m,p}( {O},P).$$
Moreover, write $\mathscr{W}_{0,u}^{m,p}( {O},P) \triangleq \left\{ f\in \mathscr{W}^{m,p}( {O},P) :\supp f\s{O} \right\}.$
\end{definition}

Next we prove some properties to be used later.
\begin{proposition}\label{chengjijieduan}
	If $f\in W^{m,p}( {O},P)$ and $g\in C_{\mathscr{F}}^{m,p}( {O})$, then $fg\in W^{m,p}( {O},P)$, and there exists a positive constant $C( m,p )$ such that
	\begin{equation*}
		\left\| fg \right\|_{W^{m,p}( {O},P)} \le C( m,p ) \left\| g \right\|_{C_{\mathscr{F}}^{m,p}( {O})} \cdot \left\| f \right\|_{W^{m,p}( {O},P)}.
	\end{equation*}
\end{proposition}
\begin{proof}
	First, we claim that
	\begin{equation}\label{13f1f3f}
		\partial^{\alpha}\left( fg \right) =\sum_{\gamma \le \alpha} C_{\alpha}^{\gamma}\partial^{\gamma}f\cdot \partial^{\alpha -\gamma}g,\qquad\forall \,\alpha \in \mathbb{N}_{0}^{\left( \mathbb{N} \right)}\text{ with }|\alpha|\le m.
	\end{equation}
Indeed, for $\alpha \in \mathbb{N}_{0}^{\left( \mathbb{N} \right)}$ with $\left| \alpha \right| \le m$, it is clear that
	$\sum_{\gamma \le \alpha} C_{\alpha}^{\gamma}\partial^{\gamma}f\cdot \partial^{\alpha -\gamma}g\in L^{p}( {O},P).$
Hence, for any $\phi \in C_{0,F^{\infty}}^{\infty}( {O})$, we have
	$$\begin{aligned}
		&\hphantom{{}={}}\!\left( -1 \right)^{\left| \alpha \right|} \int_{{O}} \sum_{\gamma \le \alpha} C_{\alpha}^{\gamma}\partial^{\gamma}f\cdot \partial^{\alpha -\gamma}g\cdot \phi \mathrm{d} P
 =\left( -1 \right)^{\left| \alpha \right|} \sum_{\gamma \le \alpha} C_{\alpha}^{\gamma}\int_{{O}} \partial^{\gamma}f\cdot \partial^{\alpha -\gamma}g\cdot \phi \mathrm{d} P\\
&  =\left( -1 \right)^{\left| \alpha \right|} \sum_{\gamma \le \alpha} \left( -1 \right)^{\left| \gamma \right|} C_{\alpha}^{\gamma}\int_{{O}} f\delta^{\gamma} \left( \partial^{\alpha -\gamma}g\cdot \phi \right) \mathrm{d} P\\
&  =\left( -1 \right)^{\left| \alpha \right|} \sum_{\gamma \le \alpha} \sum_{\gamma^{\prime} \le \gamma} \left( -1 \right)^{\left| \gamma \right|} C_{\alpha}^{\gamma}C_{\gamma}^{\gamma^{\prime}}\int_{{O}} f\cdot \partial^{\gamma^{\prime}}\partial^{\alpha -\gamma}g\cdot \delta^{\gamma -\gamma^{\prime}} \phi   \mathrm{d} P\\
&  =\left( -1 \right)^{\left| \alpha \right|} \int_{{O}} f\sum_{\beta \le \alpha} \sum_{\gamma^{\prime} \le \alpha -\beta} \left( -1 \right)^{\left| \gamma^{\prime} \right| +\left| \beta \right|} C_{\alpha}^{\gamma^{\prime} +\beta}C_{\gamma^{\prime} +\beta}^{\gamma^{\prime}}\cdot \partial^{\alpha -\beta}g\cdot \delta^{\beta} \phi   \mathrm{d} P\\
&  =\left( -1 \right)^{\left| \alpha \right|} \int_{{O}} f\sum_{\beta \le \alpha} \sum_{\gamma^{\prime} \le \alpha -\beta} \left( -1 \right)^{\left| \gamma^{\prime} \right| +\left| \beta \right|} C_{\alpha}^{\beta}C_{\alpha -\beta}^{\gamma^{\prime}}\cdot \partial^{\alpha -\beta}g\cdot \delta^{\beta} \phi   \mathrm{d} P\\
&  =\left( -1 \right)^{\left| \alpha \right|} \int_{{O}} f\sum_{\beta \le \alpha} \left( -1 \right)^{\left| \beta \right|} C_{\alpha}^{\beta}\cdot \partial^{\alpha -\beta}g\cdot \delta^{\beta} \phi   \sum_{\gamma^{\prime} \le \alpha -\beta} \left( -1 \right)^{\left| \gamma^{\prime} \right|} C_{\alpha -\beta}^{\gamma^{\prime}}\mathrm{d} P\\
&  =\left( -1 \right)^{\left| \alpha \right|} \int_{{O}} f\left( -1 \right)^{\left| \alpha \right|} C_{\alpha}^{\alpha}\cdot \partial^{\alpha -\alpha}g\cdot \delta^{\alpha} \phi   \mathrm{d} P=\int_{{O}} fg\cdot\delta^{\alpha} \phi \mathrm{d} P,\end{aligned}$$
	where the second equality follows from Lemma \ref{nini1f3ff3}: the function $\partial^{\alpha-\gamma}g\cdot\phi$ is $\mathcal F$-continuous with bounded partial derivatives through order $|\gamma|$ and is uniformly supported in ${O}$. This shows that \eqref{13f1f3f} holds.
	
Next, we have
\begin{eqnarray}\label{20260420for1}
	\begin{aligned}\left\| fg \right\|_{W^{m,p}\left( {O} ,P\right)}^{p}&=\sum_{k=0}^{m} \sum_{\left| \alpha \right| =k} \int_{{O}} \left| \mathbf{a}^{\alpha} \partial^{\alpha}\left( fg \right) \right|^{p} \mathrm{d} P
=\sum_{k=0}^{m} \sum_{\left| \alpha \right| =k} \int_{{O}} \left\vert \mathbf{a}^{\alpha} \sum_{\gamma \le \alpha} C_{\alpha}^{\gamma}\partial^{\gamma}f\cdot \partial^{\alpha -\gamma}g \right\vert^{p} \mathrm{d} P\\
&\le \sum_{k=0}^{m} \sum_{\left| \alpha \right| =k}  \prod_{i=1}^{\infty} \left( \alpha_{i} +1 \right)^{p-1} \prod_{i=1}^{\infty} \left| C_{\alpha_{i}}^{\left\lfloor \frac{\alpha_{i}}{2} \right\rfloor} \right|^{p}   \sum_{\gamma \le \alpha} \int_{{O}} \left| \mathbf{a}^{\alpha} \partial^{\gamma}f\partial^{\alpha -\gamma}g \right|^{p} \mathrm{d} P\\
&\le \sum_{k=0}^{m}  \left( k+1 \right)^{kp-k} \left\vert C_{k}^{\left\lfloor \frac{k}{2} \right\rfloor} \right\vert^{kp}   \sum_{\left| \alpha \right| =k} \sum_{\gamma \le \alpha} \int_{{O}} \left| \mathbf{a}^{\alpha} \partial^{\gamma}f\partial^{\alpha -\gamma}g \right|^{p} \mathrm{d} P\\
&=\sum_{k=0}^{m}  \left( k+1 \right)^{kp-k} \left\vert C_{k}^{\left\lfloor \frac{k}{2} \right\rfloor} \right\vert^{kp}  \sum_{\left| \alpha \right| =k} \sum_{\gamma \le \alpha} \int_{{O}} \left| \mathbf{a}^{\gamma} \partial^{\gamma}f\mathbf{a}^{\alpha -\gamma} \partial^{\alpha -\gamma}g \right|^{p} \mathrm{d} P\\
&\le\sup_{ k=0,\ldots, m}  \left( k+1 \right)^{kp-k} \left\vert C_{k}^{\left\lfloor \frac{k}{2} \right\rfloor} \right\vert^{kp}   \sum_{\left| \alpha \right| \le m} \sum_{\gamma \le \alpha} \int_{{O}} \left| \mathbf{a}^{\gamma} \partial^{\gamma}f\mathbf{a}^{\alpha -\gamma} \partial^{\alpha -\gamma}g \right|^{p} \mathrm{d} P\\
&\le\sup_{k=0,\ldots, m}  \left( k+1 \right)^{kp-k} \left\vert C_{k}^{\left\lfloor \frac{k}{2} \right\rfloor} \right\vert^{kp}   \sum_{\left| \alpha \right| \le m} \sum_{|\beta|\le m} \int_{{O}} \left| \mathbf{a}^{\alpha} \partial^{\alpha}f\mathbf{a}^{\beta} \partial^{\beta}g \right|^{p} \mathrm{d} P\\
&\le \sup_{k=0,\ldots, m}   \left( k+1 \right)^{kp-k} \left\vert C_{k}^{\left\lfloor \frac{k}{2} \right\rfloor} \right\vert^{kp}  \cdot \left\| g \right\|_{C_{\mathscr{F}}^{m,p}( {O})}^{p} \cdot \left\| f \right\|_{W^{m,p}( {O},P)}^{p} <\infty ,\end{aligned}
\end{eqnarray}
where the first inequality follows from Lemmas \ref{zuhejishu} and \ref{nver1}, and the second inequality follows from the fact that $C_{n}^{\left\lfloor \frac{n}{2} \right\rfloor}$ is nondecreasing in $n$, $\alpha_{i} \le \left| \alpha \right| =k$ for all $i=1,2,\ldots$, and that at most $k$ of the $\alpha_{i}$ are non-zero. Thus we take
	$$C( m,p ) \triangleq \sup_{k=0,\ldots, m}   \left( k+1 \right)^{k-\frac{k}{p}} \left\vert C_{k}^{\left\lfloor \frac{k}{2} \right\rfloor} \right\vert^{k} \in(0,\infty),$$
	as the desired constant. This completes the proof of Proposition \ref{chengjijieduan}.
\end{proof}

The $\mathscr{W}^{m,p}$ version of Proposition~\ref{chengjijieduan} is as follows.

\begin{proposition} \label{prop:chengjijieduan2}
Let $f\in \mathscr{W}^{m,p}({O}, P), g\in C^{m, 2}_{\mathscr{F}}({O})$. There exists constant $C(m, p)>0$ such that
$$ \|fg\|_{\mathscr{W}^{m, p}({O}, P)} \le C(m, p) \|g\|_{C^{m, 2}_{\mathscr{F}}({O})} \|f\|_{\mathscr{W}^{m, p}({O}, P)}. $$
\end{proposition}
\begin{proof}
It follows from \eqref{13f1f3f} that
\[\begin{aligned}
  &\|fg\|_{\mathscr{W}^{m, p}({O}, P)}^p \\
  \le{}&  (m+1)^{p-1} \sum_{k=0}^m  \int_{{O}} \left( \sum_{|\alpha|=k} |\mathbf{a}^\alpha \partial^\alpha (fg)|^2  \right)^{\frac{p}{2}}  \dd P  \\
   \le{}&  (m+1)^{p-1} \sum_{k=0}^m  \int_{{O}} \left( \sum_{|\alpha|=k} \left|\sum_{\gamma \le \alpha} C_{\alpha}^{\gamma} \mathbf{a}^\gamma \partial^{\gamma}f\cdot \mathbf{a}^{\alpha-\gamma} \partial^{\alpha -\gamma}g\right|^2  \right)^{\frac{p}{2}}  \dd P \\
   \le{}& (m+1)^{p-1} \sum_{k=0}^m  \int_{{O}} \left( \sum_{|\alpha|=k} \prod_{i=1}^{\infty} \left( \alpha_{i} +1 \right) \prod_{i=1}^{\infty} \left| C_{\alpha_{i}}^{\left\lfloor \frac{\alpha_{i}}{2} \right\rfloor} \right|^{2} \sum_{\gamma \le \alpha} \left|  \mathbf{a}^\gamma \partial^{\gamma}f\cdot \mathbf{a}^{\alpha-\gamma} \partial^{\alpha -\gamma}g\right|^2  \right)^{\frac{p}{2}}  \dd P \\
   \le{}& (m+1)^{p-1} \sum_{k=0}^m (k+1)^{\frac{kp}{2}} \left| C_{k}^{\left\lfloor \frac{k}{2} \right\rfloor} \right|^{kp} \int_{{O}} \left( \sum_{|\alpha|=k}  \sum_{\gamma \le \alpha} \left|  \mathbf{a}^\gamma \partial^{\gamma}f\cdot \mathbf{a}^{\alpha-\gamma} \partial^{\alpha -\gamma}g\right|^2  \right)^{\frac{p}{2}}  \dd P \\
   \le{}& (m+1)^{p-1} \sup_{0\le k\le m} (k+1)^{\frac{kp}{2}} \left| C_{k}^{\left\lfloor \frac{k}{2} \right\rfloor} \right|^{kp}  \sum_{k=0}^m \int_{{O}} \left( \sum_{l=0}^k \sum_{|\gamma|=l} \sum_{|\beta|=k-l}  \left|  \mathbf{a}^\gamma \partial^{\gamma}f\right|^2 |\mathbf{a}^\beta \partial^{\beta}g|^2  \right)^{\frac{p}{2}}  \dd P \\
   \le{}& (m+1)^{p-1} \sup_{0\le k\le m} (k+1)^{\frac{kp}{2}} \left| C_{k}^{\left\lfloor \frac{k}{2} \right\rfloor} \right|^{kp}  \|g\|_{C^{m, 2}_{\mathscr{F}}({O})}^p \sum_{k=0}^m \int_{{O}} \left( \sum_{l=0}^k \sum_{|\gamma|=l}   \left|  \mathbf{a}^\gamma \partial^{\gamma}f\right|^2  \right)^{\frac{p}{2}}  \dd P.
\end{aligned}\]
Since
\[\begin{aligned}
\sum_{k=0}^m\int_O
\left(\sum_{l=0}^k \sum_{|\gamma|=l}
|\mathbf a^\gamma \partial^\gamma f|^2\right)^{p/2}\,\mathrm dP&\le(m+1)\int_O
\left(\sum_{|\gamma|\le m}
|\mathbf a^\gamma \partial^\gamma f|^2\right)^{p/2}\,\mathrm dP\\
&\le(m+1)\|f\|_{\mathscr W^{m,p}(O,P)}^p,
\end{aligned}\]
where the last inequality follows from Minkowski's inequality,
we obtain
$$ \|fg\|_{\mathscr{W}^{m, p}({O}, P)}^p \le (m+1)^{p} \sup_{0\le k\le m} (k+1)^{\frac{kp}{2}} \left| C_{k}^{\left\lfloor \frac{k}{2} \right\rfloor} \right|^{kp}  \|g\|_{C^{m, 2}_{\mathscr{F}}({O})}^p \|f\|_{\mathscr{W}^{m, p}({O}, P)}^p.  $$
Letting
$$C(m, p)= \left( (m+1)^{p} \sup_{0\le k\le m} (k+1)^{\frac{kp}{2}} \left| C_{k}^{\left\lfloor \frac{k}{2} \right\rfloor} \right|^{kp}  \right)^{\frac{1}{p}}$$
proves the desired result.
\end{proof}

\begin{proposition}\label{sguanghualuozaisobolev}
If $f\in C_{\mathcal{F}}^{m}( {O})$ satisfies $$\sum_{\left| \alpha \right| \le m} \int_{{O}} \left| \mathbf{a}^{\alpha} \partial^{\alpha}f \right|^{p} \,\mathrm{d} P<\infty,$$ then $f\in W^{m,p}( {O},P)$. Similarly, if $f\in C_{\mathcal F}^{m}(O)$ satisfies
\[
\int_O
\left(\sum_{|\alpha|\le m}|\mathbf a^\alpha \partial^\alpha f|^2\right)^{p/2}\,\mathrm dP<\infty,
\]
then $f\in\mathscr W^{m,p}(O,P)$.
\end{proposition}

\begin{proof}
	We only prove the $W^{m,p}$ case since the $\mathscr{W}^{m,p}$ case is similar.
Fix $\alpha\in\mathbb{N}_0^{(\mathbb{N})}$ with $|\alpha|\le m$, and choose $n\in\mathbb{N}$ such that $\alpha_j=0$ for all $j>n$. By assumption and Fubini's theorem, $f$ and its partial derivatives of order at most $m$ have locally integrable restrictions to almost every slice ${O}_{\mathbf{x}^n}$. Since $f\in C_{\mathcal{F}}^m(O)$, we have $f(\cdot,\mathbf{x}^n)\in C^m(O_{\mathbf{x}^n})$
on every nonempty $O_{\mathbf{x}^n}$.
The classical and distributional derivatives of $f(\cdot,\mathbf{x}^n)$ agree. Lemma~\ref{20260917slices} therefore proves
\begin{eqnarray}\label{1f1ffg}
\int_{{O}}f\delta^{\alpha}\phi\,\mathrm{d}P
=(-1)^{|\alpha|}\int_{{O}}\partial^{\alpha}f\phi\,\mathrm{d}P,
\qquad\phi\in C_{0,F^{\infty}}^{\infty}({O}).
\end{eqnarray}
The assumed weighted integrability then gives $f\in W^{m,p}({O},P)$.
\end{proof}

After slight modifications to the proofs, all subsequent results in this subsection hold with $W^{m,p}$ replaced by $\mathscr{W}^{m,p}$. For convenience, we only state the $W^{m,p}$ versions.

\begin{lemma}\label{20260917zeroextension}
If $f\in W_{0,u}^{m,p}({O},P)$, its zero extension $\widetilde f$ to $\ell^2$ belongs to $W^{m,p}(\ell^2,P)$, and
$$
\partial^{\alpha}\widetilde f=\widetilde{\partial^{\alpha}f},\quad |\alpha|\le m,
\qquad\|\widetilde f\|_{W^{m,p}(\ell^2,P)}=\|f\|_{W^{m,p}({O},P)}.
$$
On the right, $\widetilde{\partial^{\alpha}f}$ denotes the zero extension of the derivative from ${O}$.
\end{lemma}
\begin{proof}
If $\phi\in C_{0,F^{\infty}}^{\infty}(\ell^2)$, its restriction to ${O}$ and all its partial derivatives are bounded. The second assertion of Lemma \ref{nini1f3ff3} gives
$$
\int_{\ell^2}\widetilde f\delta^{\alpha}\phi\,\mathrm{d}P
=\int_{{O}}f\delta^{\alpha}\phi\,\mathrm{d}P
=(-1)^{|\alpha|}\int_{\ell^2}\widetilde{\partial^{\alpha}f}\phi\,\mathrm{d}P.
$$
This proves $\partial^{\alpha}\widetilde f=\widetilde{\partial^{\alpha}f}$, and the norm equality follows immediately.
\end{proof}

We shall need the following results.
\begin{proposition}\label{jiangweibijin}
	If $f\in W^{m,p}_{0,u}( \ell^{2} ,P)$, then $\{f_{n}\}_{n=1}^{\infty}\subset W^{m,p}( \ell^{2} ,P)$ and
	$$\lim_{n\rightarrow \infty} \| f_{n}-f \|_{W^{m,p}( \ell^{2} ,P)} =0,$$
	where $\{f_{n}\}_{n=1}^{\infty}$ are the functions given in \eqref{20250920def1}.
\end{proposition}
\begin{proof}
Choose $R>0$ such that $\supp f\subset B_R$, and define $(\partial^{\alpha}f)_n$ by the same formula as \eqref{20250920def1}. If $\phi\in C_{0,F^{\infty}}^{\infty}(\ell^2)$, let $\phi_n$ be defined as in Proposition \ref{Reduce dimension}. For derivatives involving only the first $n$ variables, differentiation under that integral is justified by the bounded partial derivatives of $\phi$. The function $\phi_n$, regarded as a cylindrical function, and each of its partial derivatives are bounded. Thus, if $\alpha_j=0$ for all $j>n$, Fubini's theorem and Lemma \ref{nini1f3ff3} give
$$
\int_{\ell^2}f_n\delta^{\alpha}\phi\,\mathrm{d}P
=\int_{\ell^2}f\delta^{\alpha}\phi_n\,\mathrm{d}P
=(-1)^{|\alpha|}\int_{\ell^2}\partial^{\alpha}f\phi_n\,\mathrm{d}P=(-1)^{|\alpha|}\int_{\ell^2}(\partial^{\alpha}f)_n\phi\,\mathrm{d}P.
$$
If $\alpha_j>0$ for some $j>n$, we integrate first with respect to $x_j$; the function $f_n$ is independent of $x_j$, while
\[
\int_{\mathbb R}\delta^\alpha\phi(\mathbf x)\,\mathrm d\mathcal N_{a_j}(x_j)
=\frac{1}{\sqrt{2\pi}a_j}\int_{\mathbb R}\frac{\partial}{\partial x_j}
\left(e^{-\frac{x_j^2}{2a_j^2}}\delta^{\alpha-\mathbf e_j}\phi(\mathbf x)\right)\,\mathrm dx_j
=0.
\]
Hence
\begin{eqnarray}\label{zhouwenhui}
\partial^{\alpha}f_n=\begin{cases}
(\partial^{\alpha}f)_n,&\alpha_j=0\ \hbox{for all }j>n,\\
0,&\hbox{otherwise},
\end{cases}\qquad |\alpha|\le m.
\end{eqnarray}
By the contraction property in Proposition \ref{Reduce dimension},
$$
\|f_n\|_{W^{m,p}(\ell^2,P)}^p
\le\sum_{|\alpha|\le m}\int_{\ell^2}|\mathbf{a}^{\alpha}\partial^{\alpha}f|^p\,\mathrm{d}P.
$$
For each fixed $\alpha$, Proposition \ref{Reduce dimension} and \eqref{zhouwenhui} give $\partial^{\alpha}f_n\to \partial^{\alpha}f$ in $L^p(\ell^2,P)$. Also, for every $n$,
$$
\int_{\ell^2}|\mathbf{a}^{\alpha}\partial^{\alpha}(f_n-f)|^p\,\mathrm{d}P
\le2^p\int_{\ell^2}|\mathbf{a}^{\alpha}\partial^{\alpha}f|^p\,\mathrm{d}P.
$$
Hence, by the dominated convergence theorem with respect to the counting measure on $\{\alpha\in\mathbb N_0^{(\mathbb N)}:|\alpha|\le m\}$,
\[
\lim_{n\to\infty}\|f_n-f\|_{W^{m,p}(\ell^2,P)}^p
=\sum_{|\alpha|\le m}\lim_{n\to\infty}\int_{\ell^2}|\mathbf a^\alpha \partial^\alpha(f_n-f)|^p\,\mathrm dP
=0,
\]
which completes the proof.
\end{proof}

\begin{lemma}\label{20260917sobolevcutoff}
Let $\{X_n\}_{n=1}^{\infty}$ and $\{K_n\}_{n=1}^{\infty}$ be as in Proposition \ref{230215Th1}. For every $f\in W^{m,p}({O},P)$,
$$
X_nf\in W^{m,p}({O},P),\qquad
\lim_{n\to\infty}\|X_nf-f\|_{W^{m,p}({O},P)}=0.
$$
The support in $\ell^2$ of the zero extension of $X_nf$ is contained in $K_n\cap\overline{{O}}$. This assertion does not claim that the zero extension is a Sobolev function on $\ell^2$.
\end{lemma}
\begin{proof}
The estimates in Proposition \ref{230215Th1} imply $X_n\in C_{\mathscr{F}}^{m,p}(\ell^2)$, and Proposition \ref{chengjijieduan} gives $X_nf\in W^{m,p}({O},P)$. Put $Y_n\triangleq X_n-1$. The Leibniz formula and the same estimate as in \eqref{20260420for1} give
\begin{eqnarray}\label{20260427for1}
\|Y_nf\|_{W^{m,p}({O},P)}^p
\le C(m,p)^p\int_{{O}}
\left(\sum_{|\alpha|\le m}|\mathbf{a}^{\alpha}\partial^{\alpha}f|^p\right)
\left(\sum_{|\beta|\le m}|\mathbf{a}^{\beta}\partial^{\beta}Y_n|^p\right)\,\mathrm{d}P.
\end{eqnarray}
The second sum is bounded independently of $n$ and $\mathbf{x}$ and tends to zero for $P$-almost every $\mathbf{x}$, by Proposition \ref{230215Th1}. The first sum is integrable. Dominated convergence proves the asserted limit. The final assertion follows because $X_n$ vanishes outside the closed set $K_n$.
\end{proof}

\begin{proposition}\label{quankongjiandebijin}
The space $W_{0,u}^{m,p}(\ell^2,P)$ is dense in $W^{m,p}(\ell^2,P)$.
\end{proposition}
\begin{proof}
Apply Lemma \ref{20260917sobolevcutoff} with ${O}=\ell^2$. Since $K_n$ is compact, it is bounded, and therefore $X_nf\in W_{0,u}^{m,p}(\ell^2,P)$ for every $n$.
\end{proof}

\subsection{Surface measure}\label{subsec:surfmes}
For a nonempty set $I\subset\dbN$, we denote by $P_I$ the orthogonal normal projection (in $\ell^2$) whose range is the closed linear space spanned by $\{\mathbf{e}_i\}_{i\in I}$, i.e., $P_{I}\mathbf{x}\triangleq (x_i)_{i\in I}$ for any $\mathbf{x}=(x_j)_{j\in \mathbb{N}}\in\ell^2$.

We first recall the definition of $C^1$-differentiable surfaces of $\ell^2$ with codimension $1$; see \cite[Definition 3]{Goo} for a more general notion of $H$-$C^1$ surfaces.
\begin{definition}\label{20241121def1}
Suppose that $S$ is a nonempty subset of $\ell^2$. If for each $\mathbf{x}\in S$, there exists $i_0\in\mathbb{N}$, an open neighborhood $U$ of $\mathbf{x}$ in $\ell^2$, an open subset $U_{i_0}$ of $\ell^2(\mathbb{N}\setminus\{i_0\})$ and a function $f\in C_{\ell^2(\mathbb{N}\setminus\{i_0\})}^1(U_{i_0})$ such that
\begin{enumerate}[label=(\arabic*)]
\item
$P_{\mathbb{N}\setminus\{i_0\}}  |_{S\cap U}$ is a homeomorphism from $S\cap U$ onto $U_{ i_0 }$;
\item
$
S\cap U=\left\{\mathbf{x}_{\mathbb{N}\setminus\{i_0\}}+f(\mathbf{x}_{\mathbb{N}\setminus\{i_0\}})\mathbf{e}_{i_0}:\mathbf{x}_{ \mathbb{N}\setminus\{i_0\} } \in U_{ i_0 }\right\},
$
\end{enumerate}
then we say that $S$ is a \textbf{$C^1$-differentiable surface of $\ell^2$ with codimension $1$} and $(S\cap U,i_0, f)$ is a \textbf{coordinate triple} of $S$. Furthermore, we say that two coordinate triples $(S\cap U,i_1, f_1)$ and $(S\cap V,i_2, f_2)$ are \textbf{$C^1$-compatible} if either $S\cap U\cap V=\varnothing$, or $S\cap U\cap V\not=\varnothing$ and
\begin{eqnarray*}
&&P_{\mathbb{N}\setminus\{i_2\}} \left(P_{\mathbb N\setminus\{i_1\}}\big|_{S\cap U}\right)^{-1}\in C^1_{\ell^2(\mathbb{N}\setminus\{i_1\})}(P_{\mathbb{N}\setminus\{i_1\}}(S\cap U\cap V);\ell^2(\mathbb{N}\setminus\{i_2\})),\\
&&P_{\mathbb{N}\setminus\{i_1\}} \left(P_{\mathbb N\setminus\{i_2\}}\big|_{S\cap V}\right)^{-1}\in C^1_{\ell^2(\mathbb{N}\setminus\{i_2\})}(P_{\mathbb{N}\setminus\{i_2\}}(S\cap U\cap V);\ell^2(\mathbb{N}\setminus\{i_1\})).
\end{eqnarray*}
\end{definition}

Then a local surface measure on each compatible coordinate triple is defined as follows.

\begin{definition}\label{20241012def2}
Suppose that $S$ is a $C^1$-differentiable surface of $\ell^2$ with codimension $1$, and $(S\cap U,i, f)$ is a compatible coordinate triple of $S$. Then $f\in C_{\ell^2(\mathbb{N}\setminus\{i\})}^1(U_{i})$ and
$$
S\cap U=\left\{\mathbf{x}_{\mathbb{N}\setminus\{i\}}+f(\mathbf{x}_{\mathbb{N}\setminus\{i\}})\mathbf{e}_i:\mathbf{x}_{\mathbb{N}\setminus\{i\}}\in U_{i}\right\},
$$
where $U_{i}$ is an open subset of $\ell^2(\mathbb{N}\setminus\{i\})$. For every $E\in \mathscr{B}(S\cap U)$, we denote the following quantity by $\mu(E,S\cap U,i,f)$:
\begin{equation}\label{eq:surfmes}
\int_{P_{\mathbb{N}\setminus\{i\}}E}\sqrt{1+\sum_{j\in \mathbb{N}\setminus\{i\}}\left|\partial_{x_j} f(\mathbf{x}_{\mathbb{N}\setminus\{i\}})\right|^2}\cdot \frac{1}{\sqrt{2\pi a_i^2}}e^{-\frac{f(\mathbf{x}_{\mathbb{N}\setminus\{i\}})^2}{2a_i^2}}\,\mathrm{d}P^{\widehat{i}}(\mathbf{x}_{\mathbb{N}\setminus\{i\}}).
\end{equation}
Here $P^{\widehat{i}}$ denotes the product measure obtained by omitting the $i$-th component, i.e., it is the restriction of the product measure $\prod\limits_{j\in\mathbb{N}\setminus\{i\}}\mathcal{N}_{a_j}$ to the space
$$
\left(\ell^{2}(\mathbb{N}\setminus\{i\}),\mathscr{B}\big(\ell^{2}(\mathbb{N}\setminus\{i\})\big)\right).
$$
\end{definition}

\begin{lemma}[{cf. \cite[Lemma 4.5, p. 36]{WYZ}}]\label{20241013lem1}
We adopt the same notations and assumptions as in Definition \ref{20241012def2}. Suppose that $(S\cap V,i_1, f')$ is another compatible coordinate triple of $S$ such that $(S\cap U)\cap (S\cap V)\neq \varnothing$.
Then it holds that
\[\mu(E,S\cap V,i_1,f')=\mu(E,S\cap U,i,f),\qquad\forall\,E\in \mathscr{B}(S\cap U\cap V).\]
\end{lemma}
Note that Lemma \ref{20241013lem1} implies that the quantity $\mu(E,S\cap U,i,f)$ defined in Definition \ref{20241012def2} is independent of $i$ and $f$, and hence we will use the notation $\mu(E,S\cap U)$ to denote it. Lemma~\ref{20241013lem1} also implies that if $(S\cap U)\cap (S\cap V)\neq \varnothing$, then
$$
\mu(E,S\cap U)=\mu(E,S\cap V),\qquad \forall\, E\in \mathscr{B}(S\cap U\cap V).
$$
Obviously, $\mu(\cdot,S\cap U)$ is a Borel measure and we call it a \textbf{local version of surface measure on $S$}.

\begin{proposition}\label{20241013prop1}
Suppose that $S$ is a $C^1$-differentiable surface of $\ell^2$ with codimension $1$. Then there exists a Borel measure $\mu_S$ on $S$ such that
$$
\mu_S(E)=\mu(E,S\cap U),\qquad\forall\, E\in \mathscr{B}(S\cap U),
$$
for any compatible coordinate triple $(S\cap U,i,f)$ in $S$.
\end{proposition}
\begin{proof}
Since $\ell^2$ has a countable base, by \cite[Theorem~15, p.~49]{Kel}, we can find a sequence of sets $\{U_{k}\}_{k=1}^{
\infty}$ such that $S\subset\bigcup_{k=1}^{\infty}U_{k}$ with compatible coordinate triples $(S\cap U_{k},i_k,f_k)$ of $S$. Write $W_k=S\cap U_{k}$ for each $k\in\mathbb{N}$, $Y_1=W_1$ and $Y_n=W_n\setminus  \left(\bigcup _{k=1}^{n-1}W_k\right)$ for $n>1$. Clearly, $\{Y_n\}_{n=1}^{\infty}$  is a sequence of pairwise disjoint Borel subsets of $S$ and $Y_n\subset W_n$ for each $n\in\mathbb{N}$. Now we define a measure $\mu_S$ on $(S,\mathscr{B}(S))$ by
\begin{eqnarray*}
\mu_S(E)\triangleq\sum_{n=1}^{\infty}\mu(E\cap Y_n, W_n),\qquad \forall\,E\in\mathscr{B}(S).
\end{eqnarray*}
Suppose that $\mu(\cdot, S\cap U)$ is a local version of surface measure of $S$. Then for any Borel subset $E$ of $S\cap U$, by Lemma \ref{20241013lem1}, we have
\begin{eqnarray*}
\mu_S(E)=\sum_{n=1}^{\infty}\mu(E\cap Y_n, W_n)
=\sum_{n=1}^{\infty}\mu(E\cap Y_n, S\cap U)=\mu(E,S\cap U).
\end{eqnarray*}
It is easy to see that $\mu_S$ is a measure. This completes the proof of Proposition \ref{20241013prop1}.
\end{proof}
We will call the Borel measure $\mu_{S}$ in Proposition \ref{20241013prop1} the \textbf{surface measure} on $S$. This construction is the same as in \cite{Goo}, and hence by \cite[p. 417]{Goo}, if $S$ is a $C^1$-differentiable surface of $\mathbb{R}^n$ with codimension $1$, then
\begin{equation}\label{eq:surfmfin}
\mu(\cdot,S\cap U)= \prod_{i=1}^{n}\frac{1}{\sqrt{2\pi a_i^2}} e^{-\frac{x_i^2}{2a_i^2}}\,\mathrm{d}S,
\end{equation}
where $\mathrm{d}S$ is the (Lebesgue) surface measure on $S$.

The next Gauss--Green type theorem will be used in Sections~\ref{sec:ext} and \ref{sec:tr}. See also the divergence theorem \cite[Theorem~2]{Goo} of Goodman, where the differentiability assumption on the integrand is stronger.

\begin{theorem}[{\cite[Corollary~6.2, p.~76]{WYZ}}]\label{thm:GG}
	Let $O$ be a nonempty open subset of $\ell^2$ such that $\partial O$ is a $C^1$-differentiable surface of $\ell^2$ with codimension $1$. Suppose that $\overline{O}$ is a region with domain whose boundary points are exactly the points in $\partial O$ (see \cite[Definition~5.4, p.~43]{WYZ}). Let $V$ be an open subset of $\ell^2$ satisfying $\overline{O}\subset V$, and let $f\in C_{\mathcal{F}}^1(V)$ be a Borel measurable function on $V$. Suppose that $\left\{\partial_{x_i}f|_O,(x_if)|_O,f|_{O}\right\}\subset L^1(O,P)$ and that $f|_{\partial O}\in L^1(\partial O,\mu_{\partial O})$. Then
	\begin{equation*}
		\int_O \partial_{x_i}f\dd P- \int_O\frac{x_i f}{a_i^2}\dd P= \int_{\partial O}f\,\nu_{ i }^{\partial O}\,\mathrm d\mu_{\partial O},
	\end{equation*}
where for each $\mathbf{x}\in\partial O$, if we adopt the notation after \cite[(5.54), p.~43]{WYZ},
$$
\nu_{ i }^{\partial O}(\mathbf{x})\triangleq \frac{\partial_{x_i}\rho( \mathbf{x})}{\sqrt{\sum_{j=1}^{\infty}|\partial_{x_j}\rho( \mathbf{x})|^2}}.
$$
\end{theorem}

One important class of open sets to which Theorem~\ref{thm:GG} applies is as follows.
\begin{example}\label{ex:globalfunction}
	Let $F\in C^1_{\ell^2}(\ell^2)$ such that $DF(\mathbf{x})\neq\mathbf{0}$ whenever $F(\mathbf{x})=0$. Then by the implicit function theorem (e.g., \cite[Chapter~I, Theorem~5.9, p.~19]{Lan99}), $O=\{\mathbf{x}\in\ell^2:F(\mathbf{x})<0\}$ satisfies the assumptions of Theorem~\ref{thm:GG} with
	$$\nu_{ i }^{\partial O}(\mathbf{x})=\frac{\partial_{x_i}F( \mathbf{x})}{\sqrt{\sum_{j=1}^{\infty}|\partial_{x_j}F( \mathbf{x})|^2}}.$$
\end{example}

We conclude this subsection with the following property, which will be used in Section~\ref{subsec:trrefined}.

\begin{proposition}\label{prop:radon}
   Any $C^1$-differentiable surface $S$ of $\ell^2$ with codimension $1$ satisfies the following properties:
   \begin{enumerate}[label=(\roman*)]
      \item $\mu_S$ is locally finite;
      \item there exists a sequence of compact sets $\{K_j\}_{j\in\dbN}$ with $K_j\subset S$ for each $j$ such that
      \[
      \mu_S\left(S\setminus \bigcup_{j=1}^\infty K_j\right)=0.
      \]
   \end{enumerate}
\end{proposition}
\begin{proof}
   The property (i) follows from the continuity of the integrand in \eqref{eq:surfmes}.

   To prove (ii), note that by (i) and the Lindel\"of property of $S$, there is a countable open cover $\{U_j\}_{j\in\dbN}$ of $S$ together with coordinate triples $(S\cap U_j,i_j,f_j)$ of $S$ such that $\mu_S(S\cap U_j)<\infty$ for every $j\in\dbN$.
   By the condition (1) in Definition~\ref{20241121def1} and \cite[Theorem~3.2, p.~29]{Par67}, for any $n\in\dbN$, there exists a compact set $K^{n}_j\subset S\cap U_j$ such that
   \[\mu_S((S\cap U_j)\setminus K^{n}_j)<2^{-jn}.\]
   Then since
   \[S\setminus \bigcup_{j,n=1}^\infty K^{n}_j\subset \bigcup_{j=1}^\infty((S\cap U_j)\setminus K^{m}_j),\qquad \forall m\in\dbN,\]
   we have
   \[
   \mu_S\left(S\setminus \bigcup_{j,n=1}^\infty K^{n}_j\right)\le \sum_{j=1}^{\infty}\mu_S((S\cap U_j)\setminus K^{m}_j)\le\sum_{j=1}^{\infty}2^{-jm}=\frac{1}{2^m-1},\qquad \forall m\in\dbN.
   \]
   Hence letting $m\to\infty$ yields
   \[\mu_S\left(S\setminus \bigcup_{j,n=1}^\infty K^{n}_j\right)=0,\]
   and reordering the sequence $\{K^n_j\}_{j,n=1}^\infty=\{K_j\}_{j=1}^\infty$ completes the proof.
\end{proof}

\section{Smooth approximation}\label{sec:HW}
After slight modifications to the proofs, all approximation and non-density conclusions in this section remain
valid with $W^{m,p}$ replaced by $\mathscr W^{m,p}$ for every
$m\in\mathbb N$ and $p\in[1,\infty)$. For convenience, we only state the $W^{m,p}$ versions.

\subsection{Approximation by smooth functions}\label{sec5}
The main purpose of this subsection is to prove Theorem \ref{zhuyaodingyi}. We begin with the following result.
\begin{lemma}\label{0bijin}	
The following three conclusions hold true:
	
$\mathrm{1)}$ The space $W_{0,u}^{m,p}( {O} ,P)$ is contained in the closure of $\mathscr{C}^{\infty}_{c}$ in $W^{m,p}( {O} ,P)$;

$\mathrm{2)}$ The set $C_{\ell^{2}}^{\infty}( {O}) \bigcap C_{0,\mathscr{F}}^{\infty}( {O})$ is dense in $W_{0,u}^{m,p}( {O} ,P)$;

$\mathrm{3)}$ If $f\in W^{m,p}( {O} ,P)$, then for any open set $\widetilde{{O}}\s{O}$ and $\varepsilon > 0$, there exists $g\in\mathscr{C}^{\infty}_{c}$ such that $\left\| f-g \right\|_{W^{m,p}\left( \widetilde{{O}} ,P \right)} <\varepsilon$.
	\end{lemma}
\begin{proof}
\textbf{Step 1.} Let $u\in W_{0,u}^{m,p}({O},P)$. By Lemma \ref{20260917zeroextension}, its zero extension $\widetilde u$ belongs to $W_{0,u}^{m,p}(\ell^2,P)$. Choose $R>0$ such that $\supp\widetilde u\subset B_R$. Let $u_n$ be the corresponding function of $\widetilde u$ defined by \eqref{20250920def1}. Proposition \ref{jiangweibijin} gives
\begin{eqnarray}\label{1f1f}
\lim_{n\to\infty}\|u_n-\widetilde u\|_{W^{m,p}(\ell^2,P)}=0.
\end{eqnarray}
Viewed as a function on $\mathbb{R}^n$, $u_n$ and all its weak derivatives of order at most $m$ vanish outside the closed ball of radius $R$. By \eqref{zhouwenhui} and Lemma \ref{20260917slices}, these are also its ordinary finite-dimensional weak derivatives. The Gaussian density on that ball is bounded above and below by positive constants. Thus $u_n\in W^{m,p}(\mathbb{R}^n)$ in the usual finite-dimensional sense.

Choose $\chi_n\in C_c^{\infty}(\mathbb{R}^n;[0,\infty))$ with $\supp\chi_n\subset\overline{B_1}$ and $\int_{\mathbb{R}^n}\chi_n\,\mathrm{d}\mathbf{x}_n=1$. For $0<\delta<1$, set
$$
\chi_{n,\delta}(\mathbf{x}_n)\triangleq\delta^{-n}\chi_n(\mathbf{x}_n/\delta),\qquad
u_{n,\delta}(\mathbf{x}_n)\triangleq\int_{\mathbb{R}^n}u_n(\mathbf{y}_n)\chi_{n,\delta}(\mathbf{x}_n-\mathbf{y}_n)\,\mathrm{d}\mathbf{y}_n.
$$
Then $u_{n,\delta}\in C_c^{\infty}(\mathbb{R}^n)$, its support is contained in the closed ball of radius $R+1$, and
$$\partial^{\alpha}u_{n,\delta}=(\partial^{\alpha}u_n)*\chi_{n,\delta},\qquad |\alpha|\le m.$$
Finite-dimensional mollification converges in $W^{m,p}(\mathbb{R}^n)$, including when $p=1$. Since the Gaussian density is bounded on the common ball and only finitely many derivatives in the first $n$ variables occur, it follows that
\begin{eqnarray}\label{13f1ff}
\lim_{\delta\to0^+}\|u_{n,\delta}-u_n\|_{W^{m,p}(\ell^2,P)}=0,\qquad n\in\mathbb{N}.
\end{eqnarray}
Here the functions are regarded as cylindrical functions on $\ell^2$. Choosing first $n$ and then $\delta$ in \eqref{1f1f} and \eqref{13f1ff}, and restricting to ${O}$, proves conclusion 1).

\medskip
\textbf{Step 2.} Fix $u\in W_{0,u}^{m,p}({O},P)$ and $\varepsilon>0$. By Lemma \ref{20260917sobolevcutoff}, choose $N$ such that $\|X_Nu-u\|_{W^{m,p}({O},P)}<\varepsilon/2$. The set
$$S\triangleq K_N\cap\supp\widetilde u$$
is compact and is contained in ${O}$, because $\supp u\s{O}$. If $S=\varnothing$, the zero function gives the required approximation. Otherwise, Lemma \ref{20260917compactcutoff} gives a function $\eta$ equal to one on a neighborhood of $S$, with $\supp\eta\s{O}$, in the three global smooth classes stated there. In particular, $\eta X_Nu=X_Nu$ almost everywhere on ${O}$. By Step 1, choose $v\in\mathscr{C}_c^{\infty}$ such that
$$
\|v-X_Nu\|_{W^{m,p}({O},P)}
<\frac{\varepsilon}{2C(m,p)\max\{1,\|\eta\|_{C_{\mathscr{F}}^{m,p}(\ell^2)}\}}.
$$
The function $\eta v$ belongs to $C_{\ell^2}^{\infty}({O})\cap C_{0,\mathscr{F}}^{\infty}({O})$. Proposition \ref{chengjijieduan} gives
$$
\|\eta v-u\|_{W^{m,p}({O},P)}
\le C(m,p)\|\eta\|_{C_{\mathscr{F}}^{m,p}(\ell^2)}\|v-X_Nu\|_{W^{m,p}({O},P)}
+\|X_Nu-u\|_{W^{m,p}({O},P)}<\varepsilon.
$$
This proves conclusion 2).

\medskip
\textbf{Step 3.} Let $f\in W^{m,p}({O},P)$, let $\widetilde{O}\s{O}$ be open, and let $\varepsilon>0$. Choose $N$ such that $\|X_Nf-f\|_{W^{m,p}({O},P)}<\varepsilon/2$. The set $S\triangleq K_N\cap\overline{\widetilde{O}}$, where the closure is in $\ell^2$, is compact and is contained in ${O}$. If $S\neq\varnothing$, choose $\eta$ as in Lemma \ref{20260917compactcutoff}, with $\eta=1$ near $S$ and $\supp\eta\s{O}$; if $S=\varnothing$, put $\eta=0$. Then $u\triangleq\eta X_Nf\in W_{0,u}^{m,p}({O},P)$ and $u=X_Nf$ almost everywhere on $\widetilde{O}$. Their weak derivatives agree there by locality. By Step 1, choose $g\in\mathscr{C}_c^{\infty}$ such that $\|g-u\|_{W^{m,p}({O},P)}<\varepsilon/2$. Therefore
$$
\|f-g\|_{W^{m,p}(\widetilde{O},P)}
\le\|f-X_Nf\|_{W^{m,p}({O},P)}+\|u-g\|_{W^{m,p}({O},P)}<\varepsilon.
$$
This proves conclusion 3) and completes the proof.
\end{proof}

Combining Proposition \ref{quankongjiandebijin} and Lemma \ref{0bijin}, we obtain the following result immediately.
\begin{corollary}\label{1f13f3}
Both the sets $\operatorname{span}\mathscr{C}^{\infty}_{c}$ and $C_{\ell^{2}}^{\infty}( \ell^{2}) \bigcap C_{0,\mathscr{F}}^{\infty}( \ell^{2} )$ are dense in $W^{m,p}( \ell^{2} ,P)$.	
\end{corollary}
\begin{remark}\label{20251220rem1}
	Recall that V.I.~Bogachev defined four versions of Sobolev spaces on the whole space in \cite[pp. 211-217]{Bog98}. Precisely, if $\gamma$ is a centered Radon Gaussian measure on a locally convex space $X$, and $E$ is a Hilbert space (or Banach space), then the four versions of $E$-valued Sobolev spaces on $X$, i.e., $W^{p,m}(\gamma,E),D^{p,m}(\gamma,E),G^{p,m}(\gamma,E),H^{p,m}(\gamma,E)$ are defined respectively by completion of smooth cylindrical functions, two notions of partial derivatives and the Ornstein-Uhlenbeck semigroup. In \cite[Proposition 5.4.6, p. 220]{Bog98} and \cite[Theorem 5.7.2, p. 235]{Bog98}, Bogachev proved that for a separable Hilbert space $E$,
	\[
	W^{p,m}(\gamma,E)=D^{p,m}(\gamma,E)=H^{p,m}(\gamma,E)\subset G^{p,m}(\gamma,E),\qquad \forall\,m\in\mathbb{N},\,p\in(1,\infty),
	\]
in which the technique of conditional expectation (see \cite[Proposition 5.4.5, p. 220]{Bog98}) plays a crucial role; this is essentially the same as that in \eqref{20250920def1}. He also showed that for $E=\mathbb{R}$,
	\[
	W^{p,1}(\gamma,E)=D^{p,1}(\gamma,E)=H^{p,1}(\gamma,E)= G^{p,1}(\gamma,E),\qquad \forall\, p\in(1,\infty),
	\]
	(see \cite[Corollary 5.4.7, p. 221]{Bog98}).
	
	In the special case that $X=\ell^2,\,\gamma=P$ and $E=\mathbb{R}$, the partial derivatives used in the present paper are more general than those employed by Bogachev.  Combining this with Remark \ref{20250920rem1}, we obtain
	\begin{eqnarray}\label{20250920for1}
		W^{2,m}(P,\mathbb{R})=D^{2,m}(P,\mathbb{R})=H^{2,m}(P,\mathbb{R})\subset G^{2,m}(P,\mathbb{R})\subset W^{m,2}(\ell^2,P),\quad \forall\,m\in\mathbb{N},
	\end{eqnarray}
	where the last notation $W^{m,2}(\ell^2,P)$ is the one introduced in this paper. Moreover, Corollary \ref{1f13f3} implies that all five spaces appearing in \eqref{20250920for1} actually coincide. We stress that the key technique in establishing Corollary \ref{1f13f3} is essentially the same as that of \cite[Proposition 5.4.5, p. 220]{Bog98} (though presented with different names and notation).
	
	Furthermore, one can see that for any $p\in[1,2]$, $m\in\mathbb{N}$ and $f\in W^{p,m}(P,\mathbb{R})\cap W^{m,p}(\ell^2,P)$,
	\[\begin{aligned}
		\left\lVert f\right\rVert_{W^{p,m}(P,\mathbb{R})}&=\sum_{k=0}^{m}\left(\int_{\ell^2}\left(\sum_{i_1,\dots,i_k=1}^{\infty}a_{i_1}^2a_{i_2}^2\cdots a_{i_k}^2\cdot \left|\frac{\partial^{k} f}{\partial x_{i_{1}}\partial x_{i_{2}}\cdots \partial x_{i_{k}}}\right|^2 \right)^{\frac{p}{2}}\,\mathrm{d}P\right)^{\frac{1}{p}}\\
		&\le\sum_{k=0}^{m}\left(\int_{\ell^2} \sum_{i_1,\dots,i_k=1}^{\infty}a_{i_1}^pa_{i_2}^p\cdots a_{i_k}^p\cdot \left|\frac{\partial^{k} f}{\partial x_{i_{1}}\partial x_{i_{2}}\cdots \partial x_{i_{k}}}\right|^p   \,\mathrm{d}P\right)^{\frac{1}{p}}\\
		&\le (m+1)\left\lVert f\right\rVert_{\widetilde{W}^{m,p}(\ell^{2} ,P)},
	\end{aligned}\]
	while for any $p\in(2,\infty)$, $m\in\mathbb{N}$ and $f\in W^{p,m}(P,\mathbb{R})\cap W^{m,p}(\ell^2,P)$,
	\begin{align*}
		\left\lVert f\right\rVert_{W^{p,m}(P,\mathbb{R})}&=\sum_{k=0}^{m}\left(\int_{\ell^2}\left(\sum_{i_1,\dots,i_k=1}^{\infty}a_{i_1}^2a_{i_2}^2\cdots a_{i_k}^2\cdot \left|\frac{\partial^{k} f}{\partial x_{i_{1}}\partial x_{i_{2}}\cdots \partial x_{i_{k}}}\right|^2 \right)^{\frac{p}{2}}\,\mathrm{d}P\right)^{\frac{1}{p}}\\
		&\ge\sum_{k=0}^{m}\left(\int_{\ell^2} \sum_{i_1,\dots,i_k=1}^{\infty}a_{i_1}^pa_{i_2}^p\cdots a_{i_k}^p\cdot \left|\frac{\partial^{k} f}{\partial x_{i_{1}}\partial x_{i_{2}}\cdots \partial x_{i_{k}}}\right|^p \,\mathrm{d}P\right)^{\frac{1}{p}}\\
		&\ge\left(\int_{\ell^2} \sum_{k=0}^{m}\sum_{i_1,\dots,i_k=1}^{\infty}a_{i_1}^pa_{i_2}^p\cdots a_{i_k}^p\cdot \left|\frac{\partial^{k} f}{\partial x_{i_{1}}\partial x_{i_{2}}\cdots \partial x_{i_{k}}}\right|^p  \,\mathrm{d}P\right)^{\frac{1}{p}}\\
		&= \left\lVert f\right\rVert_{\widetilde{W}^{m,p}(\ell^{2} ,P)}.
	\end{align*}
	Consequently, for any $m\in\mathbb{N}$ and $p\in[1,2)$,
	\begin{eqnarray}\label{20250920for2}
		W^{p,m}(P,\mathbb{R})=D^{p,m}(P,\mathbb{R})=H^{p,m}(P,\mathbb{R})\subset G^{p,m}(P,\mathbb{R})\subset W^{m,p}(\ell^{2} ,P),
	\end{eqnarray}
	and Corollary \ref{1f13f3} again implies that the five spaces appearing in \eqref{20250920for2} coincide.
\end{remark}

We are now in a position to prove Theorem \ref{zhuyaodingyi}.

\begin{proof}[Proof of Theorem \ref{zhuyaodingyi}]
Let $f\in W^{m,p}({O},P)$ and $\varepsilon>0$. By Lemma \ref{20260917sobolevcutoff}, choose $N\in\mathbb{N}$ such that, with $u\triangleq X_Nf$,
$$\|u-f\|_{W^{m,p}({O},P)}<\varepsilon/2.$$
Choose a countable cover of ${O}$ by balls uniformly included in ${O}$. Apply Theorem \ref{danweifenjie} and denote the resulting partition of unity by $\{\rho_i\}_{i=1}^{\infty}$. Each $\rho_i$ is defined on $\ell^2$, its support is uniformly included in ${O}$, and the family of these supports is locally finite in ${O}$.

For every $i\in\mathbb{N}$, set
$$S_i\triangleq K_N\cap\supp\rho_i,\qquad u_i\triangleq\rho_i u\quad\hbox{on }{O}.$$
The support of $\rho_i$ is taken in $\ell^2$. Thus $S_i$ is a compact subset of ${O}$, and the zero extension $\widetilde u_i$ belongs to $W_{0,u}^{m,p}(\ell^2,P)$ by Proposition \ref{chengjijieduan} and Lemma \ref{20260917zeroextension}. Its support is contained in $S_i$.

If $S_i=\varnothing$, put $\eta_i=g_i=0$ and $r_i=2^{-i}$. For each remaining index choose
$$0<r_i<\min\left\{2^{-i},\frac{1}{4}\operatorname{dist}(S_i,\ell^2\setminus{O})\right\}.$$
Lemma \ref{20260917compactcutoff} gives $\eta_i\in C_{\ell^2}^{\infty}(\ell^2)\cap C_{F^{\infty}}^{\infty}(\ell^2)\cap C_{0,\mathscr{F}}^{\infty}(\ell^2)$ such that $0\le\eta_i\le1$, $\eta_i=1$ near $S_i$, and
$$\supp\eta_i\subset\{\mathbf{x}:\operatorname{dist}(\mathbf{x},S_i)<r_i\}.$$
By conclusion 1) of Lemma \ref{0bijin}, applied on $\ell^2$, choose $g_i\in\mathscr{C}_c^{\infty}$ such that
$$
\|g_i-\widetilde u_i\|_{W^{m,p}(\ell^2,P)}
<\frac{\varepsilon}{2^{i+1}C(m,p)\max\{1,\|\eta_i\|_{C_{\mathscr{F}}^{m,p}(\ell^2)}\}}.
$$
Since $\eta_i\widetilde u_i=\widetilde u_i$, Proposition \ref{chengjijieduan} yields
\begin{eqnarray}\label{20260917localerrors}
\|\eta_i g_i-\widetilde u_i\|_{W^{m,p}(\ell^2,P)}<\frac{\varepsilon}{2^{i+1}}
\end{eqnarray}
for every nonempty $S_i$; for an empty $S_i$ the left side is zero.

Define
$$f_{\varepsilon}(\mathbf{x})\triangleq\sum_{i=1}^{\infty}\eta_i(\mathbf{x})g_i(\mathbf{x}),\qquad\mathbf{x}\in{O}.$$
Lemma \ref{20260917finiteonuniformsets}, with $K=K_N$ and $A_i=\supp\rho_i$, shows that on every set $E\s{O}$ only finitely many summands can be nonzero. In particular, the sum is locally finite in the $\ell^2$ topology. Every summand has bounded weighted partial derivative sums of every order, and belongs to $C_{\ell^2}^{\infty}(\ell^2)\cap C_{F^{\infty}}^{\infty}(\ell^2)$. Hence
$$
f_{\varepsilon}\in C_{\ell^2}^{\infty}({O})\bigcap C_{F^{\infty}}^{\infty}({O})\bigcap C_{\mathscr{F}}^{\infty}({O},loc).
$$

By \eqref{20260917localerrors}, the series
$$\sum_{i=1}^{\infty}(\eta_i g_i-u_i)$$
converges absolutely in the Banach space $W^{m,p}({O},P)$ to a function $v$, with $\|v\|_{W^{m,p}({O},P)}\le\varepsilon/2$. Both $\{\supp\rho_i\}$ and $\{\supp\eta_i\cap{O}\}$ are locally finite in ${O}$. On each sufficiently small ball, the series is therefore a finite sum, and $\sum_i u_i=u$ almost everywhere because $\sum_i\rho_i=1$. It follows that $v=f_{\varepsilon}-u$ almost everywhere on ${O}$. Consequently, $f_{\varepsilon}\in W^{m,p}({O},P)$ and
$$
\|f_{\varepsilon}-f\|_{W^{m,p}({O},P)}
\le\|v\|_{W^{m,p}({O},P)}+\|u-f\|_{W^{m,p}({O},P)}<\varepsilon.
$$
This proves Theorem \ref{zhuyaodingyi}.
\end{proof}
	

\subsection{The segment conditions}\label{20260509sec1}
In this subsection, we will define open subsets of $\ell^2$ that satisfy the following properties:
$(\mathrm{1})$ they have continuous boundary;
$(\mathrm{2})$ they have Lipschitz continuous boundary;
$(\mathrm{3})$ they have regular boundary;
$(\mathrm{4})$ they satisfy the segment condition;
$(\mathrm{5})$ they satisfy the strong segment condition;
$(\mathrm{6})$ they satisfy the finitely convex condition.
Among these properties, the strong segment condition plays an important role in the final section. Moreover, we will explore the relationships among them. Roughly speaking, the following are true:
\begin{itemize}
\item $(\mathrm{1})$ is equivalent to $(\mathrm{4})$;
\item $(\mathrm{5})$ is strictly stronger than $(\mathrm{4})$;
\item  Each of $(\mathrm{2})$, $(\mathrm{3})$ and $(\mathrm{6})$ implies $(\mathrm{5})$.
\end{itemize}

Let us first recall the following classical definition in \cite[3.21, p. 68]{AF}.
\begin{definition}
We say that the nonempty open subset $\Omega$ of $\mathbb{R}^n$ satisfies the \textbf{segment condition} if every $\mathbf{x}\in\partial \Omega$ has a neighbourhood $U_{\mathbf{x}}$ of $\mathbf{x}$ and a nonzero vector $\mathbf{y}_{\mathbf{x}}\in\mathbb{R}^n$ such that if $\mathbf{z}\in \overline{\Omega}\cap U_{\mathbf{x}}$, then $\mathbf{z}+t\mathbf{y}_{\mathbf{x}}\in \Omega$ for all $t\in(0,1)$.
\end{definition}

Recall that a mapping $T:\mathbb{R}^n\to \mathbb{R}^n$ is called a rigid motion  if $T\mathbf{x}=\mathbf{c}+R\mathbf{x}$ for all $\mathbf{x}\in\mathbb{R}^n$, where $R$ is a rotation and $\mathbf{c}\in \mathbb{R}^n$. We also recall the following definition (e.g., \cite[Definition 9.57, p. 273]{Leo17}).
\begin{definition}\label{20260501def1}
	We say that $\Omega$ has a \textbf{continuous ({\it resp.} Lipschitz continuous) boundary} if for every $\mathbf{x}_0 \in \partial \Omega$, there exists a rigid motion $T$ from $\mathbb{R}^n$ into itself with $T\mathbf{x}_0=\mathbf{0}$, an open neighborhood $V_{\mathbf{x}_0}$ of $\mathbf{x}_0$, and a continuous ({\it resp.} Lipschitz continuous) function $f:\mathbb{R}^{n-1} \to \mathbb{R}$ such that
	$$\begin{aligned}
		&T \left( V_{\mathbf{x}_0}\cap \Omega \right) =\left\{ \left( y_{1},\mathbf{y}^{1} \right) \in TV_{\mathbf{x}_0}:y_{1}<f\left( \mathbf{y}^{1} \right) \right\} ,\\
 &T \left(V_{\mathbf{x}_0}\cap \partial \Omega \right) =\left\{ \left( y_{1},\mathbf{y}^{1} \right) \in TV_{\mathbf{x}_0}:y_{1}=f\left( \mathbf{y}^{1} \right) \right\} .
	\end{aligned}$$
\end{definition}
By \cite[Theorem 11.34, pp. 328--330]{Leo17}, the following result holds.
\begin{theorem}\label{20260501thm1}
The open set $\Omega$ satisfies the segment condition if and only if $\Omega$ has a continuous boundary.
\end{theorem}

Now we return to the arbitrarily given nonempty open subset $O$ of $\ell^2$ and recall that the space $\widetilde{H}$ is defined in \eqref{20260428for1}.
\begin{definition}\label{20260427def1}
	We say that $O$ satisfies the \textbf{strong segment condition} ({\it resp.} \textbf{segment condition}) if for any $\mathbf{x}\in\partial{O}$ there exists a neighbourhood $U_{\mathbf{x}}\subset\ell^{2}$ of $\mathbf{x}$ and a nonzero vector $\mathbf{y}_{\mathbf{x}}$ in $\widetilde{H}$ ({\it resp.} in $\ell^2$) such that
	$$\left\{\mathbf{z}+t\mathbf{y}_{\mathbf{x}}:t\in\left(0,1\right)\right\}\subset{O},\quad\forall\,\mathbf{z}\in\overline{O}\cap U_{\mathbf{x}}.$$
\end{definition}

\begin{example}
Take a vector $\mathbf{x}_0 \in \ell^{2} \setminus \widetilde{H}$ such that $\left\| \mathbf{x}_0 \right\| =1$, and let
$$f(\mathbf{x}) \triangleq \left\langle\mathbf{x} ,\mathbf{x}_0 \right\rangle ,\qquad\qquad Q\mathbf{x} \triangleq\mathbf{x}-f(\mathbf{x}) \mathbf{x}_0,\qquad\forall\,\mathbf{x} \in \ell^{2}.$$
We consider the following nonempty open subset of $\ell^2$:
$${O}_0 \triangleq\left\{ \mathbf{x} \in \ell^{2} :f(\mathbf{x}) >\sqrt{\left\| Q(\mathbf{x}) \right\|} \right\}.$$
Obviously,
$$\overline{{O}_0} =\left\{ \mathbf{x} \in \ell^{2} :f(\mathbf{x}) \ge \sqrt{\left\| Q(\mathbf{x}) \right\| } \right\},\quad \text{and}\quad\partial {O}_0 =\left\{ \mathbf{x} \in \ell^{2} :f(\mathbf{x}) =\sqrt{\left\| Q(\mathbf{x}) \right\| } \right\}.
$$
Then for any $\mathbf{z} \in \overline{{O}_0}$ and $t\in \left( 0,1 \right)$, it holds that
$$f\left( \mathbf{z} +t\mathbf{x}_0 \right) =f\left( \mathbf{z} \right) +t>f\left( \mathbf{z} \right) \ge \sqrt{\left\| Q \mathbf{z}  \right\| } =\sqrt{\left\| Q\left( \mathbf{z} +t\mathbf{x}_0 \right) \right\| },$$
which implies that $\mathbf{z} +t\mathbf{x}_0 \in {O}_0$. Hence ${O}_0$ satisfies the segment condition.

However, note that the zero vector $\mathbf{0} \in \partial {O}_0$. If there exists $\mathbf{y} \in \widetilde{H} \setminus \left\{ \mathbf{0} \right\}$ such that $\mathbf{0}+t\mathbf{y} \in {O}_0$ for any $t\in \left( 0,1 \right)$, then
$f\left( t\mathbf{y} \right) >\sqrt{\left\| Q\left( t\mathbf{y} \right) \right\| },$
which implies that
$$
\sqrt{t} f(\mathbf{y}) >\sqrt{\left\| Q\mathbf{y}  \right\|}.
$$
Letting $t\rightarrow 0^{+}$ in the above, we obtain that
$\sqrt{\left\| Q\mathbf{y} \right\| } =0$, which leads to $Q\mathbf{y}=0,$ and hence $\mathbf{y} \in \mathrm{span} \left\{ \mathbf{x}_0 \right\}\subset \widetilde{H},$
which contradicts the assumption that $\mathbf{x}_0 \in \ell^{2} \setminus \widetilde{H}$. Thus ${O}_0$ does not satisfy the strong segment condition.
\end{example}

Further, write
$$\ell^2_F\triangleq \{\mathbf{x}=(x_i)_{i\in\mathbb{N}}\in\ell^2:\#\{i\in\mathbb{N}:x_i\neq 0\}\in\mathbb{N}_0\}.
$$
We now introduce the following notion.
\begin{definition}
	We say that ${O}$ satisfies the \textbf{finitely convex condition} if for any $\mathbf{x}=(x_i)_{i\in\mathbb{N}},\mathbf{y}=(y_i)_{i\in\mathbb{N}}\in{O}$ satisfying
$\mathbf{x}-\mathbf{y}\in\ell^2_F$, it holds that
$$\left\{ t \mathbf{x} +\left( 1-t \right) \mathbf{y} :t \in \left[ 0,1 \right] \right\} \subset {O}.$$
\end{definition}
\begin{remark}
Obviously, the concept of finitely convex is more general than the notion of $H$-convex sets defined in \cite[p. 668]{Hino11}.
\end{remark}
\begin{proposition}\label{youxiantubixianduantiaojian}
If ${O}$ is finitely convex, then ${O}$ satisfies the strong segment condition.
\end{proposition}
\begin{proof}
	Fix any $\mathbf{x}\in\partial{O}$ and $\mathbf{y}\in{O}$. Since ${O}$ is open, there exists $r>0$ such that $B(\mathbf{y},r)\subset{O}$. Since $\ell^2_F$ is dense in $\ell^2$, there exists $\mathbf{y}_{\mathbf{x}}\in\ell^2_F$, such that
	$$\left\|\mathbf{y}_{\mathbf{x}}-(\mathbf{y}-\mathbf{x})\right\|=\left\|\mathbf{x}+\mathbf{y}_{\mathbf{x}}-\mathbf{y}\right\|<r.$$
Then $\mathbf{x}+\mathbf{y}_{\mathbf{x}}\in{O}$, and hence there exists $\delta>0$ such that
	$$B(\mathbf{x}+\mathbf{y}_{\mathbf{x}},\delta)\subset{O}.$$
	Set $U_{\mathbf{x}}\triangleq B(\mathbf{x},\delta)$. For any $\mathbf{z}\in\overline{{O}}\cup U_{\mathbf{x}}$, we have
	$$\left\|\mathbf{z}+\mathbf{y}_{\mathbf{x}}-(\mathbf{x}+\mathbf{y}_{\mathbf{x}})\right\| =\left\|\mathbf{z}-\mathbf{x}\right\| <\delta,$$
	which implies $\mathbf{z}+\mathbf{y}_{\mathbf{x}}\in{O}$. If $\mathbf{z}\in{O}$, then by the finite convexity of ${O}$,
	$$\mathbf{z}+t\mathbf{y}_{\mathbf{x}}\in{O},\quad\forall\, t\in(0,1).$$
	If $\mathbf{z}\in\partial{O}$, then we choose $\varepsilon>0$ such that
	$$B(\mathbf{z}+\mathbf{y}_{\mathbf{x}},\varepsilon)\subset{O}.$$
	For any $t\in(0,1)$, take $\mathbf{z}_{\varepsilon}\in{O}$ satisfying
	$$\left\|\mathbf{z}_{\varepsilon}-\mathbf{z}\right\|<\frac{t\varepsilon}{1-t}\quad \text{and}\quad \mathbf{z}_{\varepsilon}-\mathbf{z}\in \ell^2_F.$$
Let
	$\mathbf{a}_{\varepsilon}\triangleq(1-t)\mathbf{z}_{\varepsilon}+t(\mathbf{z}+\mathbf{y}_{\mathbf{x}})\in{O}.$
	Observe that
	 $$\left\|\frac{(1-t)\mathbf{z}+t(\mathbf{z}+\mathbf{y}_{\mathbf{x}})-\mathbf{a}_{\varepsilon}}{t}\right\|=\frac{1-t}{t}\left\|\mathbf{z}-\mathbf{z}_{\varepsilon}\right\|<\varepsilon,$$
	and
	 $$\frac{(1-t)\mathbf{z}+t(\mathbf{z}+\mathbf{y}_{\mathbf{x}})-\mathbf{a}_{\varepsilon}}{t}+\mathbf{z}+\mathbf{y}_{\mathbf{x}}-\mathbf{z}_{\varepsilon}=\frac{1}{t}(\mathbf{z}-\mathbf{z}_{\varepsilon})+\mathbf{y}_{\mathbf{x}}\in\ell^2_F.$$
Consequently,
 $$\mathbf{z}+t\mathbf{y}_{\mathbf{x}}=(1-t)\mathbf{z}+t(\mathbf{z}+\mathbf{y}_{\mathbf{x}})=(1-t)\mathbf{z}_{\varepsilon}+t\left(\mathbf{z}+\mathbf{y}_{\mathbf{x}}+\frac{(1-t)\mathbf{z}+t(\mathbf{z}+\mathbf{y}_{\mathbf{x}})-\mathbf{a}_{\varepsilon}}{t}\right)\in{O}.$$
	Thus we have shown that
	$$\mathbf{z}+t\mathbf{y}_{\mathbf{x}}\in{O},\quad\forall\,t\in(0,1),$$
	which completes the proof of Proposition \ref{youxiantubixianduantiaojian}.
\end{proof}
For any $i\in\dbN$, as in Section~\ref{subsec:surfmes}, $P_{\mathbb{N}\setminus\{i\}}$ denotes the natural projection from $\ell^{2}$ onto $\ell^{2} \left( \mathbb{N} \setminus \left\{ i \right\} \right)$ defined by $P_{\mathbb{N}\setminus\{i\}}\mathbf{x}\triangleq (x_j)_{j\in \mathbb{N}\setminus\{i\}}$ for any $\mathbf{x}=(x_j)_{j\in \mathbb{N}}\in\ell^2$.
\begin{definition}\label{20260501def2}
	We say that ${O}$ has a \textbf{regular boundary} if for every $\mathbf{x} \in \partial {O}$, there exists a neighbourhood $U_{\mathbf{x}}$ of $\mathbf{x}$, an index $i_{}\in \mathbb{N}$, and a real-valued function $f$ on $P_{\mathbb{N}\setminus\{i\}}U_{\mathbf{x}}$ such that
 \begin{equation}\label{20260528e1}
 \begin{array}{ll}
 \displaystyle
 U_{\mathbf{x}}\cap {O} =\left\{ \mathbf{y}=(y_j)_{j\in\mathbb{N}} \in U_{\mathbf{x}}:y_{i_{}}<f\left(P_{\mathbb{N}\setminus\{i\}} \mathbf{y} \right) \right\},\\[3mm]
\displaystyle	U_{\mathbf{x}}\cap \partial {O} =\left\{ \mathbf{y}=(y_j)_{j\in\mathbb{N}} \in U_{\mathbf{x}}:y_{i_{}}=f\left(P_{\mathbb{N}\setminus\{i\}} \mathbf{y}\right) \right\}.
	\end{array}
\end{equation}
\end{definition}
Inspired by Definitions \ref{20260501def1} and \ref{20260501def2}, we introduce the following notion.
\begin{definition}\label{lianxubianjiedingyi}
	We say that ${O}$ has a \textbf{continuous ({\it resp.} Lipschitz continuous) boundary} if for every $\mathbf{x}_0 \in \partial {O}$, there exist a surjective isometry mapping $T$ from $\ell^{2}$ onto itself ( i.e., $T\ell^2=\ell^2$ and $\|T\mathbf{y}-T\mathbf{y}'\|=\|\mathbf{y}-\mathbf{y}'\|,\,\forall\,\mathbf{y},\mathbf{y}'\in \ell^2$) with $T\mathbf{x}_0 =\mathbf{0}$, an open neighborhoods $U_{\mathbf{x}_0}$ of $\mathbf{x}_0$, and a continuous ({\it resp.} Lipschitz continuous) function $f$ on $P_{\mathbb{N}\setminus\{1\}}T U_{\mathbf{x}_0}$ such that
	\begin{equation}\label{20260528e2}
 \begin{array}{ll}
 \displaystyle
		T\left( U_{\mathbf{x}_0}\cap {O} \right) =\left\{ \mathbf{y}=(y_j)_{j\in\mathbb{N}} \in T U_{\mathbf{x}_0} :y_{1}<f\left( P_{\mathbb{N}\setminus\{1\}}\mathbf{y}  \right) \right\},\\[3mm]
\displaystyle T\left( U_{\mathbf{x}_0}\cap \partial {O} \right) =\left\{  \mathbf{y}=(y_j)_{j\in\mathbb{N}}  \in T U_{\mathbf{x}_0}:y_{1}=f\left( P_{\mathbb{N}\setminus\{1\}}\mathbf{y} \right) \right\} .
	\end{array}
\end{equation}
\end{definition}

\begin{remark}\label{20260501rem1}
The Mazur--Ulam theorem (see e.g. \cite{MazurUlam32}) states that for a surjective isometry $T$ on $\ell^{2}$, there exists a unitary operator $U$ on $\ell^2$ and a vector $\mathbf{x}_0\in\ell^2$ such that $T\mathbf{x}=U(\mathbf{x}-\mathbf{x}_0)$ for every $\mathbf{x}\in \ell^2$.
\end{remark}
\begin{remark}
	Our definition here is slightly different from its finite-dimensional counterpart in Definition \ref{20260501def1}. More precisely, we use a function defined on an open neighborhood rather than on the whole space. Nevertheless, these two definitions are equivalent in both finite and infinite dimensions. Indeed, for any function $f$ defined on an open neighborhood $V_{\mathbf{x}}$ of $\mathbf{x}$ such that $\overline{T  V_{\mathbf{x}} } \subset T U_{\mathbf{x}}$, the Tietze Extension Theorem yields a continuous function $\widetilde{f}$ on $\ell^{2}$ satisfying $\widetilde{f} |_{\overline{T  V_{\mathbf{x}}  }}=f$. Furthermore, if $f$ is a Lipschitz continuous function on $T V_{\mathbf{x}}$, then by \cite[Theorem 3.1, p. 102]{EG92}, $f$ can be extended to a globally Lipschitz continuous function on $\ell^{2}$.
\end{remark}
\begin{example}\label{ex:Hilbert polydisc}
	Consider the Hilbert polydisc $\mathbb{D}_2^{\infty}$ as in \eqref{eq:Hilbert polydisc}.
For any $\mathbf{x}=\left( x_{i} \right)_{i\in \mathbb{N}}  \in \partial \mathbb{D}_2^{\infty}$, there exists $m\in\mathbb{N}$ and distinct positive integers $i_{1},i_{2},\ldots ,i_{m}$ such that
$$I\triangleq \left\{ n\in \mathbb{N} :\left| x_{n} \right| =1 \right\} =\left\{ i_{1},i_{2},\ldots ,i_{m} \right\}.$$
For each $\mathbf{z}=\sum_{i=1}^{\infty}z_{i}\textbf{e}_i \in  \ell^2$, let
$$
P_{\mathbf{x}}\mathbf{z}\triangleq\sum_{k=1}^{m} z_{i_k}\textbf{e}_k+\sum_{ i_k\in\{i_{1},i_{2},\ldots ,i_{m} \}\setminus \{1,2,\ldots ,m \}} z_{k}\textbf{e}_{i_k} +\sum_{i\in \mathbb{N}\setminus(\{i_{1},i_{2},\ldots ,i_{m} \}\cup \{1,2,\ldots ,m \})} z_{i}\textbf{e}_i,
$$
and
$$S_{\mathbf{x}} \mathbf{z}\triangleq \sum_{k=1}^{m} \operatorname{sgn} \left( x_{i_{k}} \right) z_{k}\textbf{e}_k+\sum_{k=m+1}^{\infty} z_{k}\textbf{e}_{k}.$$
Then $S_{\mathbf{x}} P_{\mathbf{x}}$ is a unitary operator on $\ell^2$. Let
$$
\widetilde{\mathbf{x}} \triangleq  S_{\mathbf{x}}  P_{\mathbf{x}} \mathbf{x}   =\left( \widetilde{x_{i}} \right)_{i\in \mathbb{N}}.$$
Then we have
$\widetilde{x_{i}} =1,\,\forall\,i=1,\ldots, m,$ and $\left| \widetilde{x_{i}} \right| <1,\,\forall\,i>m.$
Choose a natural number $N>m$ such that
$\left| \widetilde{x_{i}} \right| \le \frac{1}{2} ,\,\forall\,i>N.$
Take
$$\delta \in \left( 0,\min \left\{ \frac{1}{4} ,\min_{m<i\le N} \left( 1-\left| \widetilde{x_{i}} \right| \right) \right\} \right)$$
and let
$$U_{\mathbf{x}}\triangleq B( \mathbf{x} ,\delta ) ,\qquad\widetilde{U_{\mathbf{x}}} \triangleq  S_{\mathbf{x}} P_{\mathbf{x}}   U_{\mathbf{x}} =B( \widetilde{\mathbf{x}} ,\delta).$$
For any $\mathbf{z}=\sum_{i=1}^{\infty}z_{i}\textbf{e}_i \in \widetilde{U_{\mathbf{x}}} $, if $i>m$, then we have
$$\left| z_{i} \right| \le \left| \widetilde{x_{i}} \right| +\left| z_{i}-\widetilde{x_{i}} \right| \le \left| \widetilde{x_{i}} \right| +\left\| \mathbf{z} -\widetilde{\mathbf{x}} \right\|  <\left| \widetilde{x_{i}} \right| +\delta <1;$$
and if $i=1,\ldots,m$, we have
$$\left| z_{i}-1 \right| =\left| z_{i}-\widetilde{x_{i}} \right| \le \left\| \mathbf{z} -\widetilde{\mathbf{x}} \right\|  <\delta <\frac{1}{4}.$$
Thus for any $\mathbf{z}=\sum_{i=1}^{\infty}z_{i}\textbf{e}_i  \in \widetilde{U_{\mathbf{x}}}$, it holds that
$\mathbf{z} \in \mathbb{D}_2^{\infty}$ if and only if $z_{j}<1,\,\forall\,j=1,\ldots, m.$ By a similar argument, we see that for any $\mathbf{z}=\sum_{i=1}^{\infty}z_{i}\textbf{e}_i  \in  S_{\mathbf{x}}  P_{\mathbf{x}}(U_{\mathbf{x}}-\mathbf{x})$,
\begin{equation}\label{20260505for1}
\mathbf{z} \in S_{\mathbf{x}}  P_{\mathbf{x}}(U_{\mathbf{x}}\cap \mathbb{D}_2^{\infty} -\mathbf{x})
\ \text{ if and only if }\ z_{j}<0,\quad\forall\,j=1,\ldots, m.
\end{equation}
Take a real $m\times m$ orthogonal matrix $\widetilde{R_{\mathbf{x}}}=(R_{jk})_{m\times m}$ whose first row consists entirely of $\frac{1}{\sqrt{m}}$. Then $\widetilde{R_{\mathbf{x}}}^{-1}=\widetilde{R_{\mathbf{x}}}^{\top}$, the transpose of $\widetilde{R_{\mathbf{x}}}$. For any $j=1,2,\ldots ,m$ and $\mathbf{y}_m=(y_1,\cdots,y_m)^{\top} \in \mathbb{R}^{m}$, we denote the $j$-th component of $\widetilde{R_{\mathbf{x}}}^{-1} \mathbf{y}_m $ and $\widetilde{R_{\mathbf{x}}} \mathbf{y}_m $ by $\left( \widetilde{R_{\mathbf{x}}}^{-1} \mathbf{y}_m \right)_{j}$ and $\left( \widetilde{R_{\mathbf{x}}}  \mathbf{y}_m \right)_{j}$, respectively. Then it holds that
$$\left( \widetilde{R_{\mathbf{x}}}^{-1} \mathbf{y}_m \right)_{j} =\frac{1}{\sqrt{m}} y_{1}+\sum_{k=2}^{m} R_{jk}y_{k}.$$
We define a unitary operator on $\ell^{2}$ by
$$R_{\mathbf{x}} \mathbf{y} \triangleq\sum_{i= 1}^{m}\left( \widetilde{R_{\mathbf{x}}}  \mathbf{y}_m \right)_{i}\textbf{e}_i+\sum_{i=m+1}^{\infty}y_{i}\textbf{e}_i,\qquad
\forall\,\mathbf{y}=\sum_{i=1}^{\infty}y_{i}\textbf{e}_i\in \ell^{2},$$
where $\mathbf{y}_m \triangleq (y_1,\cdots,y_m)^{\top}\in \mathbb{R}^{m}$. Let
$$T \mathbf{y}  \triangleq \left( R_{\mathbf{x}}  S_{\mathbf{x}} P_{\mathbf{x}} \right) (\mathbf{y}-\mathbf{x})   ,\quad \forall\,\mathbf{y} \in \ell^{2},$$
and
$$f\left( \mathbf{y}^{1} \right) \triangleq \min_{1\le j\le m} \left( -\sqrt{m} \sum_{k=2}^{m} R_{jk}y_{k} \right) ,\quad\forall\, \mathbf{y}^{1}=(y_i)_{i\in\mathbb{N}\setminus\{1\}} \in \ell^{2} \left( \mathbb{N} \setminus \left\{ 1 \right\} \right).$$
Note that $T$ is a surjective isometry on $\ell^2$ and $f$ is the minimum of $m$ Lipschitz continuous functions, and it is well known that such a function is still Lipschitz continuous.

 Note that for any $\mathbf{z}=\sum_{i=1}^{\infty}z_{i}\textbf{e}_i   \in T  U_{\mathbf{x}}$, by \eqref{20260505for1}, we have
$$\mathbf{z} \in T \left( \mathbb{D}_2^{\infty} \bigcap U_{\mathbf{x}} \right)\ \text{ if and only if }\ \left\langle R_{\mathbf{x}}^{-1}\mathbf{z},\textbf{e}_{j} \right\rangle  <0,\quad\forall\,j=1,2,\ldots ,m,$$
$$\left\langle R_{\mathbf{x}}^{-1}\mathbf{z}, \textbf{e}_{j}\right\rangle  =\frac{1}{\sqrt{m}} z_{1}+\sum_{k=2}^{m} R_{jk}z_{k},\quad\,\forall\,j=1,2,\ldots ,m,$$
and hence
$$\mathbf{z} \in T \left( \mathbb{D}_2^{\infty} \bigcap U_{\mathbf{x}} \right)\ \text{ if and only if }\ z_{1}<f\left(P_{\mathbb{N}\setminus\{1\}} \mathbf{z}  \right).$$
Similarly, one can prove that
$$\mathbf{z} \in T \left( (\partial\mathbb{D}_2^{\infty} )\bigcap U_{\mathbf{x}} \right)\ \text{ if and only if }\ z_{1}=f\left(P_{\mathbb{N}\setminus\{1\}} \mathbf{z}  \right),$$
and
$$\mathbf{z} \in T \left(  U_{\mathbf{x}}\setminus \overline{\mathbb{D}_2^{\infty}}\right)\ \text{ if and only if }\ z_{1}>f\left(P_{\mathbb{N}\setminus\{1\}} \mathbf{z}  \right).$$
Thus we have shown that $\mathbb{D}_2^{\infty}$ has a Lipschitz continuous boundary. Combining with Proposition \ref{1f1f1f} below, this implies that $\mathbb{D}_2^{\infty}$ also satisfies the strong segment condition.
\end{example}
\begin{proposition}\label{1111}
	If ${O}$ has a regular boundary, then ${O}$ satisfies the strong segment condition.
\end{proposition}
\begin{proof}
For any $\mathbf{x}\in \partial {O}$, there exists a neighbourhood $U_{\mathbf{x}}$ of $\mathbf{x}$, an index $i\in \mathbb{N}$, and a real-valued function $f$ on $P_{\mathbb{N}\setminus\{i\}}U_{\mathbf{x}}$ such that
\eqref{20260528e1} holds.
	Choose a neighbourhood $V_{\mathbf{x}}\subset U_{\mathbf{x}}$ of $\mathbf{x}$ such that there exists $\delta>0$ sufficiently small with
	$$\mathbf{z} -\delta t\textbf{e}_{i}\in U_{\mathbf{x}},\quad \forall\,(\mathbf{z},t) \in V_{\mathbf{x}}\times \left( 0,1 \right).$$
	Consider
	$\mathbf{y}_{\mathbf{x}} \triangleq -\delta \textbf{e}_{i}.$
	For any $\mathbf{z}=(z_i)_{i\in\mathbb{N}} \in V_{\mathbf{x}}\bigcap \overline{{O}}$ and $t\in \left( 0,1 \right)$ we have
	$z_i -\delta t <f\left( P_{\mathbb{N}\setminus\{i\}}\mathbf{z}   \right),$
	which implies $\mathbf{z} +t\mathbf{y}_{\mathbf{x}}\in {O}$. Hence ${O}$ satisfies the strong segment condition, which proves Proposition \ref{1111}.
\end{proof}
Now, we give an infinite-dimensional analogue of Theorem \ref{20260501thm1}.
\begin{theorem}\label{jihekehua}
	The open set ${O}$ satisfies the segment condition if and only if ${O}$ has a continuous boundary.
\end{theorem}
\begin{proof}
	\textbf{The ``if" part:} For every $\mathbf{x}_0 \in \partial {O}$, there exist a surjective isometry mapping $T$ from $\ell^{2}$ onto itself with $T\mathbf{x}_0 =\textbf{0}$, an open neighborhoods $U_{\mathbf{x}_0}$ of $\mathbf{x}_0$, and a continuous
function $f$ on $P_{\mathbb{N}\setminus\{1\}}TU_{\mathbf{x}_0}$ such that
	\eqref{20260528e2} holds.
By the proof Proposition \ref{1111}, there exists an open set $V\subset TU_{\mathbf{x}_0}$ and $\mathbf{z}_{\mathbf{x}_{0}} \in \ell^{2}\setminus\{\textbf{0}\}$ such that
$$
\textbf{0}\in V,\qquad\left\{ \mathbf{z}+t\mathbf{z}_{\mathbf{x}_{0}}  :t\in \left( 0,1 \right) ,\,\mathbf{z} \in T\overline{{O}} \bigcap V  \right\} \subset T {O}  .$$
Then there exists a nonzero vector $\mathbf{y}_{\mathbf{x}_0}$ such that $T(\mathbf{x}_0+\mathbf{y}_{\mathbf{x}_0})=\mathbf{z}_{\mathbf{x}_{0}}$. Clearly, $V_{\mathbf{x}_0}\triangleq T^{-1}V$ is an open neighborhood of $\mathbf{x}_0$. Hence
$$
T\mathbf{x}+tT(\mathbf{x}_0+\mathbf{y}_{\mathbf{x}_0})\in T {O} ,\qquad\forall\,t\in(0,1),\,\mathbf{x}\in \overline{{O}} \bigcap V_{\mathbf{x}_0}.
$$
Combining the fact that $T\mathbf{x}_0 =\textbf{0}$ with Remark \ref{20260501rem1}, we have $T(\mathbf{x}+t \mathbf{y}_{\mathbf{x}_0})=T\mathbf{x}+tT(\mathbf{x}_0+\mathbf{y}_{\mathbf{x}_0})$, and hence
$$\left\{ \mathbf{x} +t\mathbf{y}_{\mathbf{x}_{0}} :t\in \left( 0,1 \right) ,\,\mathbf{x} \in \overline{{O}} \bigcap V_{\mathbf{x}_{0}} \right\} \subset {O} .$$
Thus we have proved the ``if" part.	

\medskip\textbf{The ``only if" part:} For any $\mathbf{x}_{0}\in \partial {O}$, take a bounded connected neighborhood $U_{\mathbf{x}_{0}}\subset\ell^{2}$ of $\mathbf{x}_{0}$ and $\mathbf{y}_{\mathbf{x}_{0}}\in\ell^{2}\setminus \left\{ \textbf{0} \right\}$ such that
	\begin{equation}\label{111}
		\left\{ t\in \left( 0,1 \right) :\mathbf{z} +t\mathbf{y}_{\mathbf{x}_{0}} \right\} \subset {O} ,\qquad \forall\,\mathbf{z} \in \overline{O} \cap U_{\mathbf{x}_{0}}.
	\end{equation}	
Assume that $\mathbf{x}_{0} =\textbf{0} $, $\mathbf{y}_{\mathbf{x}_{0}} =-\textbf{e}_{1}$, and $U_{\mathbf{x}_{0}}=\left( a,b \right) \times W$ is an open ball with diameter less than $1$, where $a,b\in\mathbb{R}$ with $a<0<b$, and $W$ is an open set in $\ell^2(\mathbb{N}\setminus\{1\})$. Set
$$f\left( \mathbf{y}^{1} \right) \triangleq \begin{cases}\sup \left\{ y_{1}\in \mathbb{R} :\left( y_{1},\mathbf{y}^{1} \right) \in U_{\mathbf{x}_{0}}\bigcap \overline{O} \right\} ,&\mathbf{y}^{1} \in P_{\mathbb{N}\setminus\{1\}} \left( U_{\mathbf{x}_{0}}\bigcap \overline{O} \right),\\ a,&\mathbf{y}^{1} \in P_{\mathbb{N}\setminus\{1\}}   U_{\mathbf{x}_{0}}   \setminus P_{\mathbb{N}\setminus\{1\}} \left( U_{\mathbf{x}_{0}}\bigcap \overline{O} \right).\end{cases}$$	
	Clearly,
\begin{equation}\label{20260506for1}
\left( f\left( \mathbf{y}^{1} \right) ,\mathbf{y}^{1} \right) \in U_{\mathbf{x}_{0}}\bigcap \partial O ,\qquad\forall\,\mathbf{y}^{1} \in P_{\mathbb{N}\setminus\{1\}} \left( U_{\mathbf{x}_{0}}\bigcap \overline{O} \right).
\end{equation}
For any $\left( y_{1},\mathbf{y}^{1} \right) \in U_{\mathbf{x}_{0}}\bigcap O$, since $U_{\mathbf{x}_{0}}\bigcap O$ is open, we have $y_{1}<f\left( \mathbf{y}^{1} \right)$. For any $\left( y_{1},\mathbf{y}^{1} \right) \in U_{\mathbf{x}_{0}}\bigcap \partial O$, if $y_{1}<f\left( \mathbf{y}^{1} \right)$, then
	$$0<\left\| \left( f\left( \mathbf{y}^{1} \right) ,\mathbf{y}^{1} \right) -\left( y_{1},\mathbf{y}^{1} \right) \right\| = f\left( \mathbf{y}^{1} \right) -y_{1}  <1,$$
and
$$\left(f\left( \mathbf{y}^{1} \right),\mathbf{y}^{1} \right)+(f\left( \mathbf{y}^{1} \right) -y_{1} )\mathbf{y}_{\mathbf{x}_{0}}=\left( y_{1} ,\mathbf{y}^{1} \right) \in U_{\mathbf{x}_{0}}\bigcap \partial O ,$$
	which contradicts \eqref{111}. Hence $y_{1}=f\left( \mathbf{y}^{1} \right)$. Similarly, one can prove that for any $\left( y_{1},\mathbf{y}^{1} \right) \in U_{\mathbf{x}_{0}}\setminus\overline{O}$, $y_{1}>f\left( \mathbf{y}^{1} \right)$.		

 Therefore, we have proved that
	$$\begin{aligned}
		&U_{\mathbf{x}_{0}}\bigcap O =\left\{ \left( y_{1},\mathbf{y}^{1} \right) \in U_{\mathbf{x}_{0}}  :y_{1}<f\left( \mathbf{y}^{1} \right) \right\},\\ &U_{\mathbf{x}_{0}}\bigcap \partial O =\left\{ \left( y_{1},\mathbf{y}^{1} \right) \in U_{\mathbf{x}_{0}}  :y_{1}=f\left( \mathbf{y}^{1} \right) \right\} .
	\end{aligned}$$

We now prove the continuity of $f$ on $P_{\mathbb{N}\setminus\{1\}} U_{\mathbf{x}_{0}}$ by showing that
$f$ is both lower and upper semi-continuous  at every point of this set.

Fix any $\mathbf{y}^{1}_{0} \in P_{\mathbb{N}\setminus\{1\}}U_{\mathbf{x}_{0}}$ and $\varepsilon >0$. We divide the proof of lower semi-continuity of $f$ on $P_{\mathbb{N}\setminus\{1\}} U_{\mathbf{x}_{0}}$ into two cases.

\medskip\noindent\textbf{Case 1.} $f\left( \mathbf{y}^{1}_{0} \right) = a$. By the definition of $f$, we have
$$f\left( \mathbf{y}^{1} \right) \ge a,\qquad\forall\,\mathbf{y}^{1} \in P_{\mathbb{N}\setminus\{1\}} U_{\mathbf{x}_{0}} .$$

\medskip\noindent\textbf{Case 2.} $f\left( \mathbf{y}^{1}_{0} \right) > a$. Then $\mathbf{y}_{0}^{1} \in P_{\mathbb{N}\setminus\{1\}}\left( U_{\mathbf{x}_{0}}\bigcap \overline{O} \right)$. Choose $\delta_0 \in \left( 0,\min \left\{ 1,f\left( \mathbf{y}_{0}^{1} \right) - a,\varepsilon \right\} \right)$. By the definition of supremum, there exists $y_{1}^{\ast}>f\left( \mathbf{y}_{0}^{1} \right) -\frac{\delta_0}{2}$ such that
$$\left( y_{1}^{\ast},\mathbf{y}_{0}^{1} \right) \in U_{\mathbf{x}_{0}}\bigcap \overline{O}.$$
Since $0<\delta_0<1$, by \eqref{111} and \eqref{20260506for1},
we obtain
$$\left(  f\left( \mathbf{y}_{0}^{1} \right) ,\mathbf{y}_{0}^{1} \right) +\delta_0\mathbf{y}_{\mathbf{x}_{0}}=\left( f\left( \mathbf{y}_{0}^{1} \right) -\delta_0 ,\mathbf{y}_{0}^{1} \right) \in O \bigcap U_{\mathbf{x}_{0}}.$$
Since $O  \bigcap U_{\mathbf{x}_{0}}$ is open, there exists $\delta\in(0,\varepsilon)$ such that for any $\mathbf{z}^1\in\ell^2(\mathbb{N}\setminus\{1\})$ with $$\left\| \mathbf{z}^{1} -\mathbf{y}_{0}^{1} \right\|_{\ell^{2} \left( \mathbb{N} \setminus \left\{ 1 \right\} \right)} <\delta,$$ it holds that
$$\left( f\left( \mathbf{y}_{0}^{1} \right) -\delta ,\mathbf{z}^{1} \right) \in O \bigcap U_{\mathbf{x}_{0}},$$
and hence
$$f\left( \mathbf{z}^{1} \right) \ge f\left( \mathbf{y}^{1}_{0} \right) -\delta > f\left( \mathbf{y}^{1}_{0} \right) -\varepsilon.$$
Thus we have showed that $f$ is lower semi-continuous at $\mathbf{y}^{1}_{0}$.

We divide the proof of upper semi-continuity of $f$ on $P_{\mathbb{N}\setminus\{1\}} U_{\mathbf{x}_{0}}$ into three cases.

\medskip\noindent\textbf{Case 1.} $f\left( \mathbf{y}_{0}^{1} \right) = b$. By the definition of $f$, we have
$$f\left( \mathbf{y}^{1} \right) \le b,\qquad\forall\,\mathbf{y}^{1} \in P_{\mathbb{N}\setminus\{1\}} U_{\mathbf{x}_{0}} .$$

\medskip\noindent\textbf{Case 2.} $f\left( \mathbf{y}_{0}^{1} \right) = a$. Suppose that $f$ were not upper semi-continuous at $\mathbf{y}_{0}^{1}$. Then there would exist an $\varepsilon_{0} >0$ and a  sequence $\{\mathbf{y}_{n}^{1} \}_{n=1}^{\infty}\subset P_{\mathbb{N}\setminus\{1\}}U_{\mathbf{x}_{0}} $ such that
$$\mathbf{y}_{n}^{1} \rightarrow \mathbf{y}_{0}^{1},\quad \text{as}\, n\to\infty,\quad \text{and}\quad f\left( \mathbf{y}_{n}^{1} \right) \ge a+2\varepsilon_{0},\quad\forall\,n\in\mathbb{N}.$$
By the definition of supremum, for each $n\in\mathbb{N}$, there exists $y_{1,n}\ge a+\varepsilon_{0}$ such that
$$\left( y_{1,n},\mathbf{y}_{n}^{1} \right) \in U_{\mathbf{x}_{0}}\bigcap \overline{O}.$$
Since $\left\{ y_{1,n} \right\}_{n=1}^{\infty}$ is a sequence of bounded real numbers, it has a convergent subsequence $\{y_{1,n_{k}}\}_{k=1}^{\infty}$. Let $\lim_{k\to\infty}y_{1,n_{k}}= y_{1}^{\star}\in \left[ a+\varepsilon_{0} ,b \right]$. Then
$$\lim_{k\to\infty}\left( y_{1,n_{k}},\mathbf{y}_{n_{k}}^{1} \right)= \left( y_{1}^{\star},\mathbf{y}_{0}^{1} \right).$$
 If $\left( y_{1}^{\star},\mathbf{y}_{0}^{1} \right) \in U_{\mathbf{x}_{0}}\bigcap \overline{O}$, then $f\left(\mathbf{y}_{0}^{1} \right)>a$, which contradicts to $f\left( \mathbf{y}_{0}^{1} \right) = a$. Hence, $\left( y_{1}^{\star},\mathbf{y}_{0}^{1} \right) \notin U_{\mathbf{x}_{0}}\bigcap \overline{O},$ which implies that $y_{1}^{\star}=b$.  Choose $\delta^{\prime} \in \left( 0,\min \left\{ 1,\varepsilon_{0} ,b-a \right\} \right)$. For sufficiently large $k\in\mathbb{N}$, we have $b>y_{1,n_{k}}> b-\delta^{\prime}$, by \eqref{111} and
$$
0<y_{1,n_{k}}-\left( b-\delta^{\prime} \right) <\delta^{\prime} <1,$$
we have
$$\left( b-\delta^{\prime} ,\mathbf{y}_{n_{k}}^{1} \right) =\left( y_{1,n_{k}},\mathbf{y}_{n_{k}}^{1} \right)+\left(y_{1,n_{k}}-\left( b-\delta^{\prime} \right) \right)\mathbf{y}_{\mathbf{x}_{0}}
\in U_{\mathbf{x}_{0}}\bigcap O.$$
Letting $k\rightarrow \infty$, the above shows that $\left( b-\delta^{\prime} ,\mathbf{y}_{0}^{1} \right) \in \overline{O}.$ Since $\mathbf{y}^{1}_{0} \in P_{\mathbb{N}\setminus\{1\}}U_{\mathbf{x}_{0}},$ $U_{\mathbf{x}_{0}}=\left( a,b \right) \times W$ and $a< b-\delta^{\prime}<b$, we have $\left( b-\delta^{\prime} ,\mathbf{y}_{0}^{1} \right) \in U_{\mathbf{x}_{0}}$. Thus $\left( b-\delta^{\prime} ,\mathbf{y}_{0}^{1} \right) \in U_{\mathbf{x}_{0}}\cap  \overline{O}$,
and hence $f\left( \mathbf{y}_{0}^{1} \right) \neq a$, which is a contradiction. Therefore, $f$ is upper semi-continuous at $\mathbf{y}_{0}^{1}$ in this case.

\medskip\noindent\textbf{Case 3.} $a<f\left( \mathbf{y}_{0}^{1} \right) <b$. In this case we have
$$\mathbf{y}_{0}^{1} \in P_{\mathbb{N}\setminus\{1\}} \left(U_{\mathbf{x}_{0}}\bigcap \overline{O}\right),$$
and
$$\left( f\left( \mathbf{y}_{0}^{1} \right) ,\mathbf{y}_{0}^{1} \right) \in U_{\mathbf{x}_{0}}\bigcap \overline{O}.$$
Suppose that $f$ were not upper semi-continuous at $\mathbf{y}_{0}^{1}$. Then there would exist an
$$
\varepsilon_{0}^{\prime} \in \left( 0,\min\left\{1,\frac{b-f\left( \mathbf{y}_{0}^{1} \right)}{2}\right\} \right),
$$ and $\{\mathbf{y}_{n}^{1}\}_{n=1}^{\infty} \subset  P_{\mathbb{N}\setminus\{1\}}U_{\mathbf{x}_{0}} $ with $\lim_{n\to\infty}\mathbf{y}_{n}^{1}= \mathbf{y}^{1}_{0}$ such that
$$f\left( \mathbf{y}_{n}^{1} \right) \ge f\left( \mathbf{y}_{0}^{1} \right) +2\varepsilon_{0}^{\prime},\qquad\forall\,n\in\mathbb{N}.$$
By the definition of supremum, for each $n\in\mathbb{N}$, we may choose a $y_{1,n}\in \left( f\left( \mathbf{y}_{n}^{1} \right) -\frac{\varepsilon_{0}^{\prime} }{2} ,b \right)$ satisfying
$\left( y_{1,n},\mathbf{y}_{n}^{1} \right) \in U_{\mathbf{x}_{0}}\bigcap \overline{O}.$
Note that
$y_{1,n}> f\left( \mathbf{y}_{0}^{1} \right) +\frac{3}{2} \varepsilon_{0}^{\prime} .$
By
$$0<y_{1,n}-\left( f\left( \mathbf{y}_{0}^{1} \right) +\frac{1}{2} \varepsilon_{0}^{\prime} \right)<\varepsilon_{0}^{\prime}
 <1,$$
and noting \eqref{111}, we obtain that
$$\left( f\left( \mathbf{y}_{0}^{1} \right) +\frac{1}{2} \varepsilon_{0}^{\prime} ,\mathbf{y}_{n}^{1} \right)
=\left( y_{1,n},\mathbf{y}_{n}^{1} \right)+\left(y_{1,n}-\left( f\left( \mathbf{y}_{0}^{1} \right) +\frac{1}{2} \varepsilon_{0}^{\prime} \right)\right)\mathbf{y}_{\mathbf{x}_{0}}
 \in U_{\mathbf{x}_{0}}\bigcap O.$$
Letting $n\rightarrow \infty$ in the above and by a similar argument as in the Case 2, we have
$$\left( f\left( \mathbf{y}_{0}^{1} \right) +\frac{1}{2} \varepsilon_{0}^{\prime} ,\mathbf{y}_{0}^{1} \right) \in U_{\mathbf{x}_{0}}\bigcap \overline{O},$$
which contradicts the definition of $f\left( \mathbf{y}_{0}^{1} \right)$. Thus $f$ is upper semi-continuous at $\mathbf{y}_{0}^{1}$ in this case.

Now, we have showed that $f$ is both lower and upper semi-continuous at every point of $P_{\mathbb{N}\setminus\{1\}} U_{\mathbf{x}_{0}}$ and hence $f$ is continuous on $P_{\mathbb{N}\setminus\{1\}} U_{\mathbf{x}_{0}}$.

In the general case, we can take a unitary operator $U$ on $\ell^{2}$ such that $U\mathbf{y}_{\mathbf{x}_{0}}=-\|\mathbf{y}_{\mathbf{x}_{0}}\|\textbf{e}_{1}$. Let $T\mathbf{x}\triangleq \frac{1}{\|\mathbf{y}_{\mathbf{x}_{0}}\|}U(\mathbf{x}-\mathbf{x}_{0}),$ for any $\mathbf{x}\in\ell^2$. Choose a small enough open neighborhood $V_{\mathbf{x}_{0}}\subset U_{\mathbf{x}_{0}}$ of $\mathbf{x}_0$, so that $T V_{\mathbf{x}_{0}}=\left( a,b \right) \times W$ is an open ball containing $\textbf{0}$ with diameter less than $1$, where $a,b\in\mathbb{R}$ with $a<0<b$, and $W$ is an open set in $\ell^2(\mathbb{N}\setminus\{1\})$. For any
$$\mathbf{z} \in \overline{T  O  } \bigcap T  V_{\mathbf{x}_{0}}  =T\left( \overline{O} \bigcap V_{\mathbf{x}_{0}} \right),\,t\in \left( 0,1 \right),$$
we have
$$
 T^{-1}  \mathbf{z}  +t\mathbf{y}_{\mathbf{x}_{0}}\in O.
 $$
Observe that $\mathbf{z} -t\textbf{e}_{1}=T( T^{-1}  \mathbf{z} ) +\frac{1}{\|\mathbf{y}_{\mathbf{x}_{0}}\|}U(t\mathbf{y}_{\mathbf{x}_{0}})=T( T^{-1}  \mathbf{z}  +t\mathbf{y}_{\mathbf{x}_{0}})$, which implies that
$$\mathbf{z} -t\textbf{e}_{1} \in T O.$$
This reduces to the previous case, and thus there exists a continuous
function $f$ on $P_{\mathbb{N}\setminus\{1\}}  T  V_{\mathbf{x}_{0}} $ such that
	$$\begin{aligned}
		& T  O  \bigcap T  V_{\mathbf{x}_{0}} =\left\{ \mathbf{y}=(y_j)_{j\in\mathbb{N}} \in T  V_{\mathbf{x}_{0}} :y_{1}<f\left( P_{\mathbb{N}\setminus\{1\}}\mathbf{y}  \right) \right\},\\
 &T\left( \partial O \right)\bigcap T  V_{\mathbf{x}_{0}}  =\left\{  \mathbf{y}=(y_j)_{j\in\mathbb{N}}  \in T  V_{\mathbf{x}_{0}} :y_{1}=f\left( P_{\mathbb{N}\setminus\{1\}}\mathbf{y} \right) \right\} .
	\end{aligned}$$
Therefore, we have
	$$\begin{aligned}
		& (\|\mathbf{y}_{\mathbf{x}_{0}}\| \cdot T  O ) \bigcap (\|\mathbf{y}_{\mathbf{x}_{0}}\|\cdot T  V_{\mathbf{x}_{0}}) =\left\{ \mathbf{y}=(y_j)_{j\in\mathbb{N}} \in \|\mathbf{y}_{\mathbf{x}_{0}}\|\cdot T  V_{\mathbf{x}_{0}} :y_{1}<\|\mathbf{y}_{\mathbf{x}_{0}}\| f\left( P_{\mathbb{N}\setminus\{1\}}\mathbf{y}  \right) \right\},\\
 &(\|\mathbf{y}_{\mathbf{x}_{0}}\| \cdot T  (\partial O )) \bigcap (\|\mathbf{y}_{\mathbf{x}_{0}}\|\cdot T  V_{\mathbf{x}_{0}}) =\left\{ \mathbf{y}=(y_j)_{j\in\mathbb{N}} \in \|\mathbf{y}_{\mathbf{x}_{0}}\|\cdot T  V_{\mathbf{x}_{0}} :y_{1}=\|\mathbf{y}_{\mathbf{x}_{0}}\| f\left( P_{\mathbb{N}\setminus\{1\}}\mathbf{y}  \right) \right\}.
	\end{aligned}$$
Note that $\|\mathbf{y}_{\mathbf{x}_{0}}\| \cdot T$ is a surjective isometry mapping from $\ell^{2}$ onto itself. This completes the proof of Theorem \ref{jihekehua}.
\end{proof}

\begin{proposition}\label{1f1f1f}
	If $ O$ has a Lipschitz continuous boundary, then $ O$ satisfies the strong segment condition.
\end{proposition}
\begin{proof}
	For every $\mathbf{x}_{0} \in \partial  O$, there exist a surjective isometry mapping $T$ from $\ell^{2}$ onto itself, an open neighborhood $U_{\mathbf{x}_{0}}\subset \ell^{2}$ of $\mathbf{x}_{0}$, and a Lipschitz continuous function $f$ on $P_{\mathbb{N}\setminus\{1\}}T  U_{\mathbf{x}_{0}}$ such that $T\left( \mathbf{x}_{0} \right) =\textbf{0}$ and
\begin{equation}\label{20260506for2}
\begin{aligned}
		&T\left( U_{\mathbf{x}_{0}}\bigcap  O \right) =\left\{ \left( y_{1},\mathbf{y}^{1} \right) \in T  U_{\mathbf{x}_{0}}  :y_{1}<f\left( \mathbf{y}^{1} \right) \right\} ,\\
&T\left( U_{\mathbf{x}_{0}}\bigcap \partial  O \right) =\left\{ \left( y_{1},\mathbf{y}^{1} \right) \in T U_{\mathbf{x}_{0}} :y_{1}=f\left( \mathbf{y}^{1} \right) \right\} .
	\end{aligned}
\end{equation}
Then there exists a positive constant $C $ such that
	$$\left| f\left( \mathbf{y}^{1} \right) -f\left( \mathbf{z}^{1} \right) \right| \le C\left\| \mathbf{y}^{1} -\mathbf{z}^{1} \right\|_{\ell^{2} \left( \mathbb{N} \setminus \left\{ 1 \right\} \right)} ,\qquad\forall\,\mathbf{y}^{1} ,  \,\mathbf{z}^{1} \in P_{\mathbb{N}\setminus\{1\}}\left( T\left( \overline{ O} \bigcap U_{\mathbf{x}_{0}} \right) \right).$$
	Since $\widetilde{H}$ is dense in $\ell^{2}$, and $T\left( \cdot +\mathbf{x}_{0} \right)$ is an isometry from $\ell^{2}$ onto itself, we see that $T\left( \widetilde{H} +\mathbf{x}_{0} \right)$ is also dense in $\ell^{2}$.
	
Choose a sufficient small open neighborhood $V_{\mathbf{x}_{0}}(\subset U_{\mathbf{x}_{0}})$ of $\mathbf{x}_{0}$ and $r \in(0,\infty)$, such that
$$
B( \mathbf{z} ,r) \subset TU_{\mathbf{x}_{0}},\qquad \forall\, \mathbf{z} \in TV_{\mathbf{x}_{0}}.
$$
Consider the nonempty open set
	$$O\triangleq \left\{ \mathbf{x} =\left( x_{i} \right)_{i\in \mathbb{N}}\in\ell^2 :x_{1}<-C\left\| P_{\mathbb{N}\setminus\{1\}} \mathbf{x}  \right\|_{\ell^{2} \left( \mathbb{N} \setminus \left\{ 1 \right\} \right)}, \|\mathbf{x}\|<r\right\},$$
and we take
	$$\mathbf{w} =\left( w_{i} \right)_{i\in \mathbb{N}} \in O\bigcap T\left( \widetilde{H} +\mathbf{x}_{0} \right) \neq \varnothing.$$
Now for any $\mathbf{z} =\left( z_{i} \right)_{i\in \mathbb{N}} \in T\left(V_{\mathbf{x}_{0}}\bigcap \overline{ O} \right) ,$ and $t\in \left( 0,1 \right)$, we have
$\|t\mathbf{w}\|<r$, $\mathbf{z} +t\mathbf{w}\in T U_{\mathbf{x}_{0}}$,
$$\begin{aligned}
\langle \mathbf{z} +t\mathbf{w},\textbf{e}_1 \rangle_{\ell^2} &\le f\left(P_{\mathbb{N}\setminus\{1\}} \mathbf{z}  \right) +t w_{1} <f\left( P_{\mathbb{N}\setminus\{1\}}\mathbf{z}  \right) -Ct\left\| P_{\mathbb{N}\setminus\{1\}}\mathbf{w}  \right\|_{\ell^{2} \left( \mathbb{N} \setminus \left\{ 1 \right\} \right)} \\
&\le f\left( P_{\mathbb{N}\setminus\{1\}}(\mathbf{z}  +t\mathbf{w}) \right).
\end{aligned}$$
By \eqref{20260506for2}, we have $\mathbf{z} +t\mathbf{w} \in T\left( U_{\mathbf{x}_{0}}\bigcap  O \right)\subset T  O$. Set $\mathbf{y}_{\mathbf{x}_{0}} \triangleq T^{-1}\mathbf{w}  -\mathbf{x}_{0} \in \widetilde{H}$. Similar to the proof of the ``if" part of Theorem \ref{jihekehua}, we see that $\mathbf{z} +t\mathbf{y}_{\mathbf{x}_{0}} \in  O$ for ant $\mathbf{z}\in V_{\mathbf{x}_{0}}\bigcap \overline{ O}$ and $ t\in \left( 0,1 \right)$. Hence, $ O$ satisfies the strong segment condition. This completes the proof of Proposition \ref{1f1f1f}.
\end{proof}

\subsection{Approximation by smooth cylindrical functions}\label{20260427sec1}
Recall that a classical result in \cite[Theorem 3.2, p. 68]{AF} states that if $\Omega$ is a domain in $\mathbb{R}^n$ satisfying the segment condition, then $C_c^{\infty}(\mathbb{R}^n)$ is dense in $W^{m,p}(\Omega)$ for every $m\in\mathbb{N}$ and $p\in[1,\infty)$. We will establish its infinite-dimensional counterpart in this subsection by proving Theorem \ref{zhuhanshuguanghuabijin}.

Firstly, an example is given below showing that such an approximation result does not hold for every connected open subset of $\ell^2$.
\begin{example}\label{examp7.1}
Let
		$${O} \triangleq B_1 \setminus \left\{ \mathbf{x}=(x_i)_{i\in\mathbb{N}} \in \ell^{2} :x_{1}\in \left[ 0,1 \right),\, x_{2}=0 \right\} ,$$
		$$\begin{aligned}
			&\omega_{+} \triangleq B_1 \bigcap \left\{  \mathbf{x}=(x_i)_{i\in\mathbb{N}} \in \ell^{2} :x_{1}>\frac{3}{4} ,\,x_{2}>0 \right\} ,\\
			&\omega_{-} \triangleq B_1 \bigcap \left\{  \mathbf{x}=(x_i)_{i\in\mathbb{N}} \in \ell^{2} :x_{1}>\frac{3}{4} ,\,x_{2}<0 \right\} .
		\end{aligned}$$
Take $\varphi \in C^{\infty}\left( \mathbb{R} \right)$ with
		$$\varphi (t) =0,\quad \forall\,t\le \frac{1}{2} ,\qquad\ \varphi (t) =1,\quad \forall\,t\ge \frac{3}{4},$$
		and let
		$$u( \mathbf{x}) \triangleq \varphi ( x_{1}) \cdot\operatorname{sgn} (x_{2}),\qquad \forall\, \mathbf{x}=(x_i)_{i\in\mathbb{N}}\in\ell^2.$$
		Obviously $u\in C_{\mathscr{F}}^{\infty}( {O}) \bigcap C_{\ell^{2}}^{\infty}\left( {O} ,\mathbb{R} \right)$, thus $u\in W^{m,p}( {O} ,P)$ for any $m\in\mathbb{N}$ and $p\in[1,\infty)$. If there existed a sequence $\{u_{n}\}_{n=1}^{\infty}\subset C^{m}( {O}) \bigcap C( \overline{{O}} )\bigcap W^{m,p}( {O} ,P)$ such that
		$$\lim_{n\rightarrow \infty} \| u_{n}-u \|_{W^{m,p}( {O} ,P)} =0,$$
then we would have
		$$\begin{aligned}
			&\lim_{n\rightarrow \infty} \int_{\omega_{+}} \left| u_{n}-1 \right|^{p} \mathrm{d} P=0,\quad \lim_{n\rightarrow \infty} \int_{\omega_{-}} \left| u_{n}+1 \right|^{p} \mathrm{d} P=0,\\
			&\lim_{n\rightarrow \infty} \int_{\omega_{+}} \left| \frac{\partial u_{n}}{\partial x_{2}} \right|^{p} \mathrm{d} P=0,\quad \lim_{n\rightarrow \infty} \int_{\omega_{-}} \left| \frac{\partial u_{n}}{\partial x_{2}}  \right|^{p} \mathrm{d} P=0.
		\end{aligned}$$
		By the conclusion (b) of \cite[Theorem 6.3.1, pp. 171--172]{Res} and \cite[(i), p. 181]{Res}, we obtain
		$$\begin{aligned}
			&\lim_{k\rightarrow \infty} u_{n_{k}}\left( \mathbf{x} \right) =1,\ \text{a.e. on } \omega_{+},\qquad \lim_{k\rightarrow \infty} u_{n_{k}}\left( \mathbf{x} \right) =-1,\ \text{a.e. on } \omega_{-};\\
			&\lim_{k\rightarrow \infty} \frac{\partial u_{n_k}}{\partial x_{2}} \left( \mathbf{x} \right) =0,\ \text{a.e. on } \omega_{+}\bigcup \omega_{-}.
		\end{aligned}$$
		Fubini's theorem yields the existence of $x_{1}^{0}\in \left( \frac{3}{4} ,1 \right) $ and $\mathbf{x}^{2}_{0} \in \ell^{2} \left( \mathbb{N} \setminus \left\{ 1,2 \right\} \right)$ such that
		$$\delta^{2} \triangleq \left| x_{1}^{0} \right|^{2} +\left\lVert \mathbf{x}_{0}^{2} \right\rVert_{\ell^{2} \left( \mathbb{N} \setminus \left\{ 1,2 \right\} \right)}^{2} <1$$
		and
		$$\begin{aligned}
			&\lim_{k\rightarrow \infty} u_{n_{k}}\left( x_{1}^{0},x_{2},\mathbf{x}_{0}^{2} \right) =1,\ \text{a.e. on } \left( 0,\sqrt{1-\delta^{2}} \right); \\
			&\lim_{k\rightarrow \infty} u_{n_{k}}\left( \mathbf{x} \right) =-1,\ \text{a.e. on } \left( -\sqrt{1-\delta^{2}} ,0 \right);\\
			&\lim_{k\rightarrow \infty} \frac{\partial u_{n_k}}{\partial x_{2}} \left( x_{1}^{0},x_{2},\mathbf{x}_{0}^{2} \right) =0,\ \text{a.e. on } \left( -\sqrt{1-\delta^{2}} ,\sqrt{1-\delta^{2}} \right).
		\end{aligned}$$
		For almost every $x_{2}\in \left( 0,\sqrt{1-\delta^{2}} \right) $ and $y_{2}\in \left( -\sqrt{1-\delta^{2}} ,0 \right)$, the Newton-Leibniz formula gives
		$$u_{n_{k}}\left( x_{1}^{0},x_{2},\mathbf{x}_{0}^{2} \right) -u_{n_{k}}\left( x_{1}^{0},y_{2},\mathbf{x}_{0}^{2} \right) =\int_{y_{2}}^{x_{2}} \frac{\partial u_{n_k}}{\partial x_{2}} \left( x_{1}^{0},t,\mathbf{x}_{0}^{2} \right) \mathrm{d} t,$$
		and the Dominated Convergence Theorem gives
		$$2=\lim_{k\rightarrow \infty} \int_{y_{2}}^{x_{2}}  \frac{\partial  u_{n_{k}}\left( x_{1}^{0},t,\mathbf{x}_{0}^{2} \right)}{\partial x_{2}} \mathrm{d} t=0,$$
		which is a contradiction! In particular, $\operatorname{span}\mathscr{C}^{\infty}_{c}$ is not dense in $W^{m,p}( {O} ,P)$.
	\end{example}

We shall use compact truncation and translations in $\widetilde H$. The compact support assertions in the next lemma refer to the support in $\ell^2$ of the zero extension, not to compact containment in ${O}$.
\begin{lemma}\label{jinhuaxuliebijin}
For every $f\in W^{m,p}({O},P)$ and $\varepsilon>0$, there is $g\in W^{m,p}({O},P)$ such that the support of its zero extension is compact in $\ell^2$ and $\|g-f\|_{W^{m,p}({O},P)}<\varepsilon$.
\end{lemma}
\begin{proof}
Take $g=X_Nf$ with $N$ sufficiently large and apply Lemma \ref{20260917sobolevcutoff}. Its zero extension has support contained in $K_N\cap\overline{O}$. No regularity across $\partial{O}$ is asserted.
\end{proof}

\begin{lemma}\label{jeduan}
Let $f\in L^p(\ell^2,P)$ and $\supp f\subset B_R$ for some $R>0$. For every $\mathbf{y}\in\widetilde H$,
\begin{eqnarray}\label{1f332g2gg1}
\lim_{t\to0}\|f(\cdot+t\mathbf{y})-f\|_{L^p(\ell^2,P)}=0.
\end{eqnarray}
Moreover, for $|t|\le1$,
\begin{eqnarray}\label{20260425for1}
\|f(\cdot+t\mathbf{y})\|_{L^p(\ell^2,P)}^p
\le e^{R\|\mathbf{y}\|_{\widetilde H}}\|f\|_{L^p(\ell^2,P)}^p.
\end{eqnarray}
\end{lemma}
\begin{proof}
For $\mathbf{y}=(y_i)_{i\in\mathbb{N}}\in\widetilde H$, the series $\sum_i x_i y_i/a_i^2$ converges absolutely for every $\mathbf{x}\in\ell^2$ and defines a continuous linear functional. The Cauchy--Schwarz inequality gives
$$
\left|\sum_{i=1}^{\infty}\frac{x_i y_i}{a_i^2}\right|\le\|\mathbf{x}\|\,\|\mathbf{y}\|_{\widetilde H}.
$$
Also $\mathbf{y}\in H$, since $\sup_i a_i<\infty$. Put
\begin{equation}\label{eq:Rt}
R_t(\mathbf x)\triangleq
\exp\left(
t\sum_{i=1}^{\infty}\frac{x_i y_i}{a_i^2}
-\frac{t^2}{2}\|\mathbf y\|_H^2
\right).
\end{equation}
We claim that
\begin{equation}\label{20260917translationdensity}
\int_{\ell^2}F(\mathbf x+t\mathbf y)\,\mathrm dP(\mathbf x)
=\int_{\ell^2}F(\mathbf x)R_t(\mathbf x)\,\mathrm dP(\mathbf x)
\end{equation}
for every nonnegative Borel function $F$. Indeed, it follows from the Gaussian moment formula that for every $n\in\mathbb N$,
\[
\int_{\ell^2}
\exp\left(
t\sum_{i>n}\frac{x_i y_i}{a_i^2}
-\frac{t^2}{2}\sum_{i>n}\frac{y_i^2}{a_i^2}
\right)\,\mathrm dP=1.
\]
Hence, if $F(\mathbf x)=F_n(x_1,\ldots,x_n)$, then
\[
\begin{aligned}
\int_{\ell^2}F R_t\,\mathrm dP &=\int_{\mathbb R^n}F_n(x_1,\ldots,x_n)
\prod_{i=1}^n\frac{e^{-(x_i-ty_i)^2/(2a_i^2)}}{\sqrt{2\pi}a_i}\,\mathrm dx_1\cdots\mathrm dx_n\\
&=\int_{\ell^2}F(\mathbf x+t\mathbf y)\,\mathrm dP(\mathbf x).
\end{aligned}
\]
By the monotone class and monotone convergence theorems, this proves
\eqref{20260917translationdensity} for every nonnegative Borel
function $F$.
In particular, translation by $t\mathbf{y}$ preserves the class of $P$-null sets.
Applying \eqref{20260917translationdensity} to $F=|f|^p$ gives \eqref{20260425for1}.

Fix a radial function $\theta\in C_{\ell^2}^{\infty}(\ell^2)$ with $0\le\theta\le1$, $\theta=1$ on $B_R$, and $\supp\theta\subset\overline B_{R+1}$. Smooth cylindrical functions are dense in $L^p(\ell^2,P)$: this follows from Proposition \ref{Reduce dimension} and finite-dimensional approximation in $L^p(\mathbb{R}^n,\mathcal N^n)$. Therefore, given $\varepsilon>0$, choose $v\in\mathscr C_c^{\infty}$ with $\|v-f\|_{L^p(\ell^2,P)}<\varepsilon$, and set $g\triangleq\theta v$. Then $\|g-f\|_{L^p}<\varepsilon$, and $f-g$ is supported in $\overline B_{R+1}$. The estimate \eqref{20260425for1}, applied with radius $R+2$, gives
$$
\|(f-g)(\cdot+t\mathbf{y})\|_{L^p}
\le e^{(R+2)\|\mathbf{y}\|_{\widetilde H}/p}\|f-g\|_{L^p},\qquad |t|\le1.
$$
Since $g$ is bounded and continuous, the dominated convergence theorem gives $\|g(\cdot+t\mathbf{y})-g\|_{L^p}\to0$. The triangle inequality now yields
$$
\limsup_{t\to0}\|f(\cdot+t\mathbf{y})-f\|_{L^p}
\le\left(e^{(R+2)\|\mathbf{y}\|_{\widetilde H}/p}+1\right)\varepsilon.
$$
Letting $\varepsilon\to0$ proves \eqref{1f332g2gg1}.
\end{proof}

\begin{lemma}\label{20260917translatedderivatives}
Let $G\subset\ell^2$ be open and $f\in W^{m,p}(G,P)$. Suppose that $f$ and all its weak derivatives of order at most $m$ vanish almost everywhere outside $G\cap B_R$, for some $R>0$. For $\mathbf{y}\in\widetilde H$ and $|t|\le1$, define $f_t(\mathbf{x})\triangleq f(\mathbf{x}+t\mathbf{y})$ on $G-t\mathbf{y}$. Then $f_t\in W^{m,p}(G-t\mathbf{y},P)$ and
$$\partial^{\alpha}f_t(\mathbf{x})=(\partial^{\alpha}f)(\mathbf{x}+t\mathbf{y}),\qquad |\alpha|\le m,$$
almost everywhere on $G-t\mathbf{y}$.
\end{lemma}
\begin{proof}
Let $R_t$ be as in \eqref{eq:Rt}.
Formula \eqref{20260917translationdensity}, applied to the zero extensions of the derivatives, gives
$$
\sum_{|\alpha|\le m}\int_{G-t\mathbf{y}}|\mathbf{a}^{\alpha}(\partial^{\alpha}f)(\mathbf{x}+t\mathbf{y})|^p\,\mathrm{d}P(\mathbf{x})
\le e^{R\|\mathbf{y}\|_{\widetilde H}}\|f\|_{W^{m,p}(G,P)}^p.
$$
For $\phi\in C_{0,F^{\infty}}^{\infty}(G-t\mathbf{y})$, set $\psi(\mathbf{x})\triangleq R_t(\mathbf{x})\phi(\mathbf{x}-t\mathbf{y})$ on $G$. We have $\psi \in C_{0,F^{\infty}}^{\infty}\left( G \right)$. In fact, its support is uniformly included in $G$ and is bounded; the gradients of every fixed partial derivative of $R_t$ are bounded on bounded sets, since $(y_i/a_i^2)_{i\in\mathbb{N}}\in\ell^2$. A direct computation gives
$$
\delta^{\alpha}\psi(\mathbf{x})=R_t(\mathbf{x})(\delta^{\alpha}\phi)(\mathbf{x}-t\mathbf{y}).
$$
Using this equality, \eqref{20260917translationdensity}, and the weak derivative identity on $G$, we obtain
$$
\begin{aligned}
\int_{G-t\mathbf{y}}f_t\delta^{\alpha}\phi\,\mathrm{d}P
&=\int_G f\delta^{\alpha}\psi\,\mathrm{d}P
=(-1)^{|\alpha|}\int_G \partial^{\alpha}f\psi\,\mathrm{d}P\\
&=(-1)^{|\alpha|}\int_{G-t\mathbf{y}}(\partial^{\alpha}f)(\mathbf{x}+t\mathbf{y})\phi(\mathbf{x})\,\mathrm{d}P(\mathbf{x}).
\end{aligned}
$$
Together with the preceding integrability estimate, this proves the lemma.
\end{proof}

Now we prove the approximation by smooth cylindrical functions. The proof uses compact truncation directly and does not require Theorem \ref{zhuyaodingyi}.
\begin{proof}[Proof of Theorem \ref{zhuhanshuguanghuabijin}]
Let $f\in W^{m,p}({O},P)$ and $\varepsilon>0$. By Lemma \ref{20260917sobolevcutoff}, choose $N\in\mathbb{N}$ such that, with $u\triangleq X_Nf$,
\begin{eqnarray}\label{20260917secondtruncation}
\|u-f\|_{W^{m,p}({O},P)}<\varepsilon/3.
\end{eqnarray}
Set $K\triangleq K_N\cap\overline{O}$. This is compact, and the support of the zero extension of $u$ is contained in $K$. If $u=0$ almost everywhere, the choice $g=0$ proves the assertion. We henceforth suppose that $u$ is not zero.

\medskip
\textbf{Step 1.} Choose finitely many points $\mathbf{x}_j\in K$ and bounded open neighborhoods $V_j,U_j$ of $\mathbf{x}_j$, $1\le j\le M$, such that
	$$
	K\subset\bigcup_{j=1}^M V_j,\qquad \overline{V_j}\s U_j.
	$$
	If $\mathbf{x}_j\in{O}$, choose $U_j\s{O}$ and put $\mathbf{y}_j=\textbf{0}$. If $\mathbf{x}_j\in\partial{O}$, use the strong segment condition to choose $U_j$ and a nonzero vector $\mathbf{y}_j\in\widetilde H$ such that
\begin{eqnarray}\label{20260426for3}
\mathbf{z}+t\mathbf{y}_j\in{O},\qquad\mathbf{z}\in U_j\cap\overline{O},\qquad 0<t<1.
\end{eqnarray}
After multiplying each such vector by a sufficiently small positive number, we may also assume
\begin{eqnarray}\label{20260426for4}
\|\mathbf{y}_j\|<\min\{1,\operatorname{dist}(\overline{V_j},\ell^2\setminus U_j)\}.
\end{eqnarray}
The strong segment condition is preserved by this multiplication. Such neighborhoods can first be chosen for each point of $K$, after which compactness yields a finite subcover.

Apply Theorem \ref{danweifenjie} to the cover $\{V_j\}_{j=1}^M$ of $\bigcup_{j=1}^M V_j$. Denote the partition by $\{\rho_i\}$ and its assignment to the sets $V_j$ by $a(i)$. Its supports form a locally finite family on a neighborhood of $K$. Only finitely many of these supports meet $K$. Retain these functions and relabel them $\rho_1,\ldots,\rho_{n_0}$. Then
$$
\supp\rho_i\s V_{a(i)},\qquad
\sum_{i=1}^{n_0}\rho_i=1\quad\hbox{on }K.
$$
Define on $\ell^2$
$$
f_i(\mathbf{x})\triangleq
\begin{cases}\rho_i(\mathbf{x})u(\mathbf{x}),&\mathbf{x}\in{O},\\0,&\mathbf{x}\notin{O},
\end{cases}
\qquad S_i\triangleq K\cap\supp\rho_i,\qquad B_i\triangleq S_i\cap\partial{O}.
$$
Both $S_i$ and $B_i$ are compact, and
\begin{eqnarray}\label{20260426for6}
u=\sum_{i=1}^{n_0}f_i\quad\hbox{almost everywhere on }{O}.
\end{eqnarray}
For $|\alpha|\le m$, let $\widetilde{\partial^{\alpha}f_i}$ denote the zero extension to $\ell^2$ of $\partial^{\alpha}(\rho_i u)$ from ${O}$. By Proposition \ref{chengjijieduan} and Lemma \ref{1f3ff32},
$$
\sum_{|\alpha|\le m}\int_{\ell^2}|\mathbf{a}^{\alpha}\widetilde{\partial^{\alpha}f_i}|^p\,\mathrm{d}P<\infty,
\qquad\supp\widetilde{\partial^{\alpha}f_i}\subset S_i.
$$
In particular, these functions have supports in a common ball $B_{R_i}$ for some $R_i>0$.

Let $G_i\triangleq\ell^2\setminus B_i$. The function $f_i|_{G_i}$ belongs to $W^{m,p}(G_i,P)$ and has weak derivatives $\widetilde{\partial^{\alpha}f_i}|_{G_i}$. To verify this, on ${O}$ use the product rule; on $\ell^2\setminus\overline{O}$ use the zero function. Every point of $\partial{O}\setminus B_i$ has a neighborhood disjoint from $S_i$, on which the function and the proposed derivatives vanish almost everywhere. These neighborhoods cover $G_i$, and locality established after Lemma \ref{20260917slices} proves the assertion. Thus no Sobolev regularity of the zero extension across $B_i$ is assumed.

\medskip
\textbf{Step 2.} For $0<t<1$, put
$$
G_{i,t}\triangleq\ell^2\setminus(B_i-t\mathbf{y}_{a(i)}),\qquad
f_{i,t}(\mathbf{x})\triangleq f_i(\mathbf{x}+t\mathbf{y}_{a(i)}).
$$
If $\mathbf{x}_{a(i)}\in{O}$, then $B_i=\varnothing$ and $\mathbf{y}_{a(i)}=\textbf{0}$. If $\mathbf{x}_{a(i)}\in\partial{O}$, suppose that $\mathbf{x}\in(B_i-t\mathbf{y}_{a(i)})\cap\overline{O}$. Then $\mathbf{z}\triangleq\mathbf{x}+t\mathbf{y}_{a(i)}\in B_i\subset\overline{V_{a(i)}}$. By \eqref{20260426for4}, $\mathbf{x}=\mathbf{z}-t\mathbf{y}_{a(i)}\in U_{a(i)}$. Formula \eqref{20260426for3} implies $\mathbf{z}\in{O}$, contrary to $\mathbf{z}\in\partial{O}$. Hence, for every $i$,
\begin{eqnarray}\label{20260426for5}
(B_i-t\mathbf{y}_{a(i)})\cap\overline{O}=\varnothing,\qquad\overline{O}\subset G_{i,t}.
\end{eqnarray}
By Lemma \ref{20260917translatedderivatives}, $f_{i,t}\in W^{m,p}(G_{i,t},P)$ and
$$
\partial^{\alpha}f_{i,t}(\mathbf{x})=\widetilde{\partial^{\alpha}f_i}(\mathbf{x}+t\mathbf{y}_{a(i)}),\qquad |\alpha|\le m,
$$
almost everywhere on $G_{i,t}$. For each fixed $\alpha$, Lemma \ref{jeduan} gives
$$
\lim_{t\to0^+}\int_{\ell^2}\left|\mathbf{a}^{\alpha}\left[\widetilde{\partial^{\alpha}f_i}(\mathbf{x}+t\mathbf{y}_{a(i)})-\widetilde{\partial^{\alpha}f_i}(\mathbf{x})\right]\right|^p\,\mathrm{d}P(\mathbf{x})=0.
$$
For $0<t<1$, the same integral is bounded by
$$
2^{p-1}\left(e^{R_i\|\mathbf{y}_{a(i)}\|_{\widetilde H}}+1\right)
\int_{\ell^2}|\mathbf{a}^{\alpha}\widetilde{\partial^{\alpha}f_i}|^p\,\mathrm{d}P.
$$
The bounds are summable over $|\alpha|\le m$. By the dominated convergence theorem with respect to the counting measure on $\{\alpha\in\mathbb N_0^{(\mathbb N)}:|\alpha|\le m\}$, we obtain
\begin{equation}\label{1f3f33g3g}
\begin{aligned}
\|f_{i,t}-f_i\|_{W^{m,p}(O,P)}^p
&\le \sum_{|\alpha|\le m}\int_{\ell^2}\left|\mathbf a^\alpha\left[\widetilde{D^\alpha f_i}(\mathbf x+t\mathbf y_{a(i)})-\widetilde{D^\alpha f_i}(\mathbf x)\right]\right|^p\,\mathrm dP(\mathbf x)\\
&\longrightarrow0\qquad(t\to0^+),\qquad 1\le i\le n_0.
\end{aligned}
\end{equation}
For each $i$ with $\mathbf{x}_{a(i)}\in\partial{O}$, choose $t_i\in(0,1)$ so that
\begin{eqnarray}\label{1q}
\|f_{i,t_i}-f_i\|_{W^{m,p}({O},P)}<\frac{\varepsilon}{3n_0}.
\end{eqnarray}
For each $i$ with $\mathbf{x}_{a(i)}\in{O}$, take $t_i=0$, $G_{i,0}=\ell^2$ and $f_{i,0}=f_i$; the left side of \eqref{1q} is then zero.

\medskip
\textbf{Step 3.} Define
$$T_i\triangleq(S_i-t_i\mathbf{y}_{a(i)})\cap\overline{O}.$$
This is a compact subset of $G_{i,t_i}$, by \eqref{20260426for5}. If $T_i\neq\varnothing$, use Lemma \ref{20260917compactcutoff} to choose
$$
\eta_i\in C_{\ell^2}^{\infty}(\ell^2)\cap C_{F^{\infty}}^{\infty}(\ell^2)\cap C_{0,\mathscr{F}}^{\infty}(\ell^2)
$$
with $\eta_i=1$ near $T_i$ and $\supp\eta_i\s G_{i,t_i}$. If $T_i=\varnothing$, put $\eta_i=0$. Proposition \ref{chengjijieduan} gives $\eta_i f_{i,t_i}\in W_{0,u}^{m,p}(G_{i,t_i},P)$. Denote its zero extension to $\ell^2$ by $F_i$. Lemma \ref{20260917zeroextension} gives $F_i\in W^{m,p}(\ell^2,P)$.

The support of $f_{i,t_i}$, as an $L^p$ function on $\ell^2$, is contained in $S_i-t_i\mathbf{y}_{a(i)}$. On ${O}$, this support is contained in $T_i$, where $\eta_i=1$. Therefore
\begin{eqnarray}\label{2q}
F_i=f_{i,t_i}\quad\hbox{almost everywhere on }{O},
\qquad\|F_i-f_{i,t_i}\|_{W^{m,p}({O},P)}=0.
\end{eqnarray}
The second equality follows from locality of weak derivatives. This localization is made before zero extension; no assertion that the entire support of $f_{i,t_i}$ is contained in a neighborhood of $T_i$ is needed.

By Corollary \ref{1f13f3}, choose $g_i\in\operatorname{span}\mathscr C_c^{\infty}$ such that
\begin{eqnarray}\label{3q}
\|g_i-F_i\|_{W^{m,p}(\ell^2,P)}<\frac{\varepsilon}{3n_0}.
\end{eqnarray}
Set $g\triangleq\sum_{i=1}^{n_0}g_i\in\operatorname{span}\mathscr C_c^{\infty}$. Combining \eqref{20260917secondtruncation}, \eqref{20260426for6}, \eqref{1q}, \eqref{2q} and \eqref{3q}, we obtain
$$
\begin{aligned}
\|g-f\|_{W^{m,p}({O},P)}
&\le\|u-f\|_{W^{m,p}({O},P)}
+\sum_{i=1}^{n_0}\left(\|g_i-F_i\|_{W^{m,p}(\ell^2,P)}+\|f_{i,t_i}-f_i\|_{W^{m,p}({O},P)}\right)\\
&<\frac{\varepsilon}{3}+n_0\frac{\varepsilon}{3n_0}+n_0\frac{\varepsilon}{3n_0}=\varepsilon.
\end{aligned}
$$
This proves Theorem \ref{zhuhanshuguanghuabijin}.
\end{proof}

At the end of this section, we provide an example of an open subset ${O}$ of $\ell^2$ that satisfies the segment condition but not the strong segment condition, and for which $\operatorname{span}\mathscr{C}^{\infty}_{c}$ is not dense in $W^{m,p}({O},P)$ for any $m\in\mathbb N$ and $p\in[1,\infty)$.
Before proceeding, we need a one-dimensional estimate.
\begin{lemma}\label{lem:oned}
Let $1\leq p<\infty$, $A\geq0$, and $v\in C^1([-1,1])$. Then
\begin{equation*}\int_{-1}^1\bigl(|v(t)-A\operatorname{sgn}(t)|^p+|v'(t)|^p\bigr)\,\mathrm d t\geq2^{2-p}A^p.\end{equation*}
\end{lemma}
\begin{proof}
For $0<t<1$, the fundamental theorem of calculus and the triangle inequality yield
\[|v(0)-A|\leq|v(t)-A|+\int_0^t|v'(r)|\,\mathrm d r.\]
Integrating this over $t\in(0,1)$ gives
\[|v(0)-A|\leq\int_0^1|v(t)-A|\,\mathrm d t+\int_0^1|v'(t)|\,\mathrm d t.\]
Similarly,
\[|v(0)+A|\leq\int_{-1}^0|v(t)+A|\,\mathrm d t+\int_{-1}^0|v'(t)|\,\mathrm d t.\]
Hence,
\[2A\leq|v(0)-A|+|v(0)+A|\leq\int_{-1}^1|v(t)-A\operatorname{sgn}(t)|\,\mathrm d t+\int_{-1}^1|v'(t)|\,\mathrm d t.\]
Combining this with H\"{o}lder's inequality and Lemma~\ref{nver1}, we obtain
\[(2A)^p\leq4^{p-1}\int_{-1}^1\bigl(|v(t)-A\operatorname{sgn}(t)|^p+|v'(t)|^p\bigr)\,\mathrm d t,\]
from which the result follows.
\end{proof}

\begin{theorem}\label{thm:seg}
There exist a nonempty connected open set ${O}\subset\ell^2$, a unit vector $\mathbf v\in\ell^2\setminus H$, and a real-valued Borel function $u$ with the following properties:
\begin{itemize}
\item[1)] ${O}$ has a continuous boundary, $P(\partial{O})=0$, and
\begin{equation}\label{eq:segment}
	\mathbf z+t\mathbf v\in{O},\qquad \mathbf z\in \overline{{O}},\quad t>0.
\end{equation}
In particular, ${O}$ satisfies the segment condition.
\item[2)] $\mathbf 0\in\partial{O}$, and for any $\mathbf h\in H\setminus\{\mathbf 0\}$, there exists no open neighborhood $U$ of $\mathbf 0$ such that
$$
\mathbf z+t\mathbf h\in{O},\qquad \mathbf z\in\overline{{O}}\cap U ,\quad 0<t<1.
$$
 Thus ${O}$ does not satisfy the strong segment condition.
 \item[3)] For every $m\in \mathbb N$ and $1\leq p<\infty$, $u\in W^{m,p}({O},P)$, and
\begin{equation}\label{eq:lower}
\inf_{g\in\operatorname{span}\mathscr C_c^\infty}\|u-g \|_{W^{m,p}({O},P)}^p\geq 2^{1-p}\frac{e^{-1/2}}{\sqrt{2\pi}}>0.
\end{equation}

\end{itemize}
\end{theorem}
\begin{proof} Since $\sum_{i=1}^\infty a_i^2<\infty$, $a_i\to0$. Define recursively
\begin{equation}\label{eq:nj}
	n_0=1,\qquad n_j=\min\{n>n_{j-1}:a_n\leq2^{-j}\},\qquad j\in\mathbb N.
\end{equation}
Put
\[I=\mathbb N\setminus(\{1\}\cup\{n_j:j\in\mathbb N\}),\qquad Y=\ell^2(\mathbb N),\qquad Z=\ell^2(I).\]
If $I=\varnothing$, we take $Z=\{\mathbf 0\}$. We identify $\mathbf x\in\ell^2$ isometrically by its coordinates with
\begin{equation}\label{eq:decomp}
	\mathbf x=(s,\mathbf y,\mathbf z)\in \mathbb R\oplus Y\oplus Z,
\end{equation}
where $s=x_1$, $\mathbf y=(y_j)_{j\in\mathbb N}=(x_{n_j})_{j\in\mathbb N}$, and $\mathbf z=(z_i)_{i\in I}=(x_i)_{i\in I}$.
The independence of the coordinates gives the product decomposition
\begin{equation}\label{eq:product}
	P=\mathcal N_{a_1}\times\nu\times\eta,\qquad \nu=\prod_{j=1}^\infty\mathcal N_{a_{n_j}}\Bigg|_{Y},\qquad \eta=\prod_{i\in I}\mathcal N_{a_i}\Bigg|_{Z}.
\end{equation}
When $I=\varnothing$, $\eta$ is unit mass at $\mathbf 0$.

For $\mathbf y=(y_j)_{j\in\mathbb N}\in Y$, define
\begin{equation*}
Q(\mathbf y)\triangleq\sum_{j=1}^\infty2^jy_j^2\in[0,\infty],\quad G\triangleq\{\mathbf y\in Y:Q(\mathbf y)<\infty\},\quad K\triangleq\{\mathbf y\in Y:Q(\mathbf y)\leq4\}.
\end{equation*}
Tonelli's theorem and \eqref{eq:nj} give
\begin{equation*}
	\int_YQ\,\mathrm d\nu=\sum_{j=1}^\infty2^ja_{n_j}^2\leq\sum_{j=1}^\infty2^{-j}=1.
\end{equation*}
Consequently, $\nu(G)=1$, and Markov's inequality yields
\begin{equation}\label{eq:markov}
	\nu(\{\mathbf y\in Y:Q(\mathbf y)\leq2\})\geq\frac{1}{2}.
\end{equation}

We check that $K$ is compact in $Y$. Since $Q$ is the supremum of its continuous partial sums, $Q$ is lower semicontinuous and $K$ is closed. For every $\mathbf y=(y_j)_{j\in\mathbb N}\in K$,
\begin{equation}\label{eq:tail}
	\|\mathbf y\|_Y^2\leq2,\qquad \sum_{j>N}y_j^2\leq2^{-(N+1)}Q(\mathbf y)\leq2^{1-N}.
\end{equation}
The projection of $K$ onto its first $N$ coordinates is bounded and hence totally bounded in $\mathbb R^N$. Together with \eqref{eq:tail}, this shows that $K$ is totally bounded in $Y$. Since $K$ is closed in the complete space $Y$, and by \cite[Theorem~18.27, p.~210--211]{CRV}, $K$ is compact.

Set
\begin{equation*}
	\mathbf e\triangleq(2^{-j/2})_{j\in\mathbb N}\in Y,\qquad \mathcal{R}\mathbf y\triangleq\mathbf y-\langle \mathbf y,\mathbf e\rangle \mathbf e,\qquad \mathbf y\in Y.
\end{equation*}
Thus $\|\mathbf e\|_Y=1$, $\mathcal{R}$ is the orthogonal projection onto $\mathbf e^\perp$, and
\[Q(\mathbf e)=\sum_{j=1}^\infty1=\infty.\]
Since $Q(\mathbf y+\mathbf z)\leq2Q(\mathbf y)+2Q(\mathbf z)$ and $Q(c\mathbf y)=c^2Q(\mathbf y)$, for all $\mathbf y,\mathbf z\in Y$ and $c\in\mathbb R$, $G$ is a linear subspace. If $\mathbf y,\mathbf z\in K$ and $\mathcal{R}\mathbf y=\mathcal{R}\mathbf z$, then $\mathbf y-\mathbf z=\lambda \mathbf e$ for some $\lambda\in\mathbb R$. Since $\mathbf y-\mathbf z\in G$, and $\mathbf e\notin G$, we have $\lambda=0$ and $\mathbf y=\mathbf z$.

Consequently, $\mathcal{R}|_K$ is a continuous bijection from the compact space $K$ onto $\mathcal{R}(K)\subset  \mathbf e^\perp$, with continuous inverse. The function
$f_0$ defined by
\begin{equation*}
	f_0(\mathcal{R}\mathbf y)\triangleq\langle \mathbf y,\mathbf e\rangle,\qquad \mathbf y\in K,
\end{equation*}
is well defined and continuous on $\mathcal{R}(K)$. By \eqref{eq:tail}, $|f_0(\mathcal{R}\mathbf y)|\leq \|\mathbf y\|_Y\leq\sqrt2$ for every $\mathbf y\in K$. Since every metric space is normal and $\mathcal{R}(K)$ is closed in $e^\perp$,
the Tietze extension theorem \cite[Theorem 35.1]{Munkres}
gives a bounded continuous extension $f:e^\perp\to\mathbb R$ of $f_0$. In particular,
\begin{equation}\label{eq:onK}
	f(\mathcal{R}\mathbf y)=\langle \mathbf y,\mathbf e\rangle\quad (\mathbf y\in K),\qquad f(\mathbf 0)=0.
\end{equation}

With the decomposition \eqref{eq:decomp}, define
\begin{equation*}
{O}\triangleq\left\{(s,\mathbf y,\mathbf z)\in \mathbb R\oplus Y\oplus Z:\langle \mathbf y,\mathbf e\rangle>f(\mathcal{R}\mathbf y)-s^2\right\},\qquad \mathbf v\triangleq(0,\mathbf e,\mathbf 0).
\end{equation*}
By \eqref{eq:nj}, we have
$$
\|\mathbf v\|_{H}^2=\sum_{j=1}^{\infty}\frac{2^{-j }}{a_{n_j}^2}\geq \sum_{j=1}^{\infty}1=\infty,
$$
and hence $\mathbf v\in\ell^2\setminus H$.
Let $\Phi(s,\mathbf y,\mathbf z)\triangleq\langle \mathbf y,\mathbf e\rangle-f(\mathcal{R}\mathbf y)+s^2$ for every $(s,\mathbf y,\mathbf z)\in\ell^2$. Since $\Phi$ is continuous, ${O}$ is open. As $\mathcal{R}\mathbf e=\mathbf 0$,
\begin{equation}\label{eq:translate}
	\Phi(\mathbf x+t\mathbf v)=\Phi(\mathbf x)+t,\qquad \mathbf x\in\ell^2,\quad t\in\mathbb R.
\end{equation}
Since $\Phi(\mathbf 0)=0$, we have $t\mathbf v\in{O}$ for every $t>0$, so ${O}\ne\varnothing$.

Continuity gives $\overline{O}\subset\{\mathbf x\in\ell^2:\Phi(\mathbf x)\geq0\}$. If $\Phi(\mathbf x)=0$, \eqref{eq:translate} shows that $\mathbf x+t\mathbf v\in{O}$ and $\mathbf x-t\mathbf v\notin{O}$ for every $t>0$. Hence
\begin{equation*}
	\overline{O}=\{\mathbf x\in\ell^2:\Phi(\mathbf x)\geq0\},\qquad \partial{O}=\{\mathbf x\in\ell^2:\Phi(\mathbf x)=0\}.
\end{equation*}
This and \eqref{eq:translate} immediately give \eqref{eq:segment}.

Write $\mathbf y=\mathbf z_0+r\mathbf e$, where $\mathbf z_0\in \mathbf e^\perp$ and $r\in\mathbb R$. Then
\[\begin{aligned}
{O}&=\{(s,\mathbf z_0+r\mathbf e,\mathbf z):r>f(\mathbf z_0)-s^2\},\\
\partial{O}&=\{(s,\mathbf z_0+r\mathbf e,\mathbf z):r=f(\mathbf z_0)-s^2\}.
\end{aligned}\]
Thus the boundary is the range of the continuous mapping
$$
(s,\mathbf z_0,\mathbf z)\mapsto (s,\mathbf z_0,\mathbf z)+ (f(\mathbf z_0)-s^2)\mathbf v,\qquad (s,\mathbf z_0,\mathbf z)\in \mathbb R\oplus \mathbf e^\perp\oplus Z .
$$
The map
\[(s,\mathbf z_0,\mathbf z,t)\longmapsto\bigl(s,\mathbf z_0+(f(\mathbf z_0)-s^2+t)\mathbf e,\mathbf z\bigr)\]
is a homeomorphism from $(\mathbb R\oplus \mathbf e^\perp\oplus Z)\times(0,\infty)$ onto ${O}$; its inverse is given by $\mathbf z_0=\mathcal{R}\mathbf y$ and $t=\Phi(s,\mathbf y,\mathbf z)$. Its domain is path connected, so ${O}$ is path connected.

For any $(s,\mathbf y,\mathbf z)\in\partial {O}$, we have
\[s^2=f(\mathcal{R}\mathbf y)-\langle \mathbf y,\mathbf e\rangle.\]
For fixed $(\mathbf y,\mathbf z)$ the above equation on $s$ has at most two real solutions, which is an $\mathcal N_{a_1}$-null set. By \eqref{eq:product} and Fubini's theorem, $P(\partial{O})=0.$
Moreover, \eqref{eq:onK} gives
\begin{equation}\label{eq:section}
	{O}\cap(\mathbb R\oplus K\oplus Z)=(\mathbb R\setminus\{0\})\oplus K\oplus Z.
\end{equation}

We now show part 2) of Theorem~\ref{thm:seg}. If $\mathbf h=(h_1,\mathbf h_Y,\mathbf h_Z)=(h_i)_{i\in\mathbb N}\in H$, where $h_1\in\mathbb R,\mathbf h_Y\in Y,\mathbf h_Z\in Z$, then
\begin{equation*}
	Q(\mathbf h_Y)=\sum_{j=1}^\infty 2^ja_{n_j}^2\frac{h_{n_j}^2}{a_{n_j}^2}\leq\sum_{j=1}^\infty\frac{h_{n_j}^2}{a_{n_j}^2}<\infty.
\end{equation*}
Consequently, $\pm t\mathbf h_Y\in K$ for sufficiently small $t>0$.

If $h_1=0$, \eqref{eq:section} implies $t\mathbf h\notin{O}$ for all sufficiently small $t>0$; hence there exists no open neighborhood $U$ of $\mathbf 0$ such that
$$
\mathbf z+t\mathbf h\in{O},\qquad \mathbf z\in\overline{{O}}\cap U ,\quad 0<t<1.
$$

If $h_1\ne0$, given any neighborhood $U$ of $\mathbf 0$, choose $t\in(0,1)$ sufficiently small that $-t\mathbf h\in U$ and $-t\mathbf h_Y\in K$. By \eqref{eq:section}, $\mathbf z\triangleq-t\mathbf h\in{O}\cap U$, whereas $\mathbf z+t\mathbf h=\mathbf 0\notin{O}$. This proves part 2) of Theorem~\ref{thm:seg}.

Let
\[\theta(t)\triangleq\begin{cases}\exp(-1/t),&t>0,\\0,&t\leq0,\end{cases}\qquad
\chi(t)\triangleq\frac{\theta(4-t)}{\theta(4-t)+\theta(t-2)},\qquad t\in\mathbb R.\]
The denominator of $\chi$ is everywhere strictly positive. Hence $\chi\in C^\infty(\mathbb R;[0,1])$ and
\begin{equation}\label{eq:chi}
	\chi(t)=1\quad\text{for all }t\leq2 ,\qquad \chi(t)=0\quad\text{for all }t\geq4 .
\end{equation}
Define a Borel measurable function on $Y$,
\begin{equation}\label{eq:X}
	X(\mathbf y)\triangleq\begin{cases}\chi(Q(\mathbf y)),&\mathbf y\in G,\\0,&\mathbf y\in Y\setminus G.\end{cases}
\end{equation}
The function $X$, as well as all of its partial derivatives below, vanishes on $Y\setminus K$. By \eqref{eq:markov} and \eqref{eq:chi}, for every $1\leq p<\infty$,
\begin{equation}\label{eq:mass}
	\int_KX(\mathbf y)^p\,\mathrm d\nu(\mathbf y)\geq \nu(\{\mathbf y\in Y:Q(\mathbf y)\leq2\})\geq\frac{1}{2}.
\end{equation}


If $\mathbf y\notin G$, and $\mathbf z=(z_i)_{i\in\mathbb N}$ such that $z_i\neq 0$ only for finitely many $i\in\mathbb N$, then $\mathbf y+\mathbf z\notin G$ and $X(\mathbf y+\mathbf z)=0$. If $\mathbf y\in G$ and $\mathbf z=(z_i)_{i\in\mathbb N}$ such that $\{i:z_i\neq 0\}\subset\{1,\ldots,N\}$, then
\begin{equation*}
	Q(\mathbf y+\mathbf z)=Q(\mathbf y)+2\sum_{j=1}^N2^jy_jz_j+\sum_{j=1}^N2^jz_j^2.
\end{equation*}

For $k\in\mathbb N$, let $\mathcal{P}_{1,2}(k)$ be as in the proof of Proposition~\ref{230215Th1}, and for every $\pi\in\mathcal{P}_{1,2}(k)$, we denote the number of blocks in $\pi$ by $|\pi|$. The finite-dimensional chain rule gives, for every $\mathbf y\in G$ and $i_1,\ldots,i_k\in\mathbb N$,
\begin{equation}\label{eq:chain}
	\partial_{y_{i_1}}\cdots \partial_{y_{i_k}}X(\mathbf y)=\sum_{\pi\in\mathcal{P}_{1,2}(k)}\chi^{(|\pi|)}(Q(\mathbf y))\prod_{\{j\}\in\pi}\partial_{y_{i_j}}Q(\mathbf y)\prod_{\{l,n\}\in\pi,l\neq n}\partial_{y_{i_l}}\partial_{y_{i_n}}Q(\mathbf y),
\end{equation}
Here
\[\partial_{y_i}Q(\mathbf y)=2^{i+1}y_i,\qquad \partial_{y_i}\partial_{y_j}Q(\mathbf y)=2^{i+1}\delta_{ij}.\]

If a term in \eqref{eq:chain} is nonzero, then $Q(\mathbf y)\leq4$, so $|y_i|\leq2\cdot2^{-i/2}$. Together with \eqref{eq:nj}, this gives
\[\begin{aligned}
a_{n_i}|\partial_{y_i}Q(\mathbf y)|&\leq2|y_i|\leq4\cdot2^{-i/2},\\
a_{n_i}a_{n_j}|\partial_{y_i}\partial_{y_j}Q(\mathbf y)|&=2^{i+1}a_{n_i}a_{n_j}\delta_{ij}\leq 2^{1-j}\delta_{ij}=2\cdot2^{-i/2}2^{-j/2}\delta_{ij}.
\end{aligned}\]
For every $r,k\in\mathbb N$, $\chi^{(r)}$ is a bounded function, and $\mathcal{P}_{1,2}(k)$ is a finite set. Therefore, there exists some $C_k<\infty$ depending only on $k$ and $\chi$, such that
\begin{equation}\label{eq:weighted}
\left|\left(\prod_{r=1}^ka_{n_{i_r}}\right)\partial_{y_{i_1}}\cdots \partial_{y_{i_k}}X(\mathbf y)\right|\leq C_k\prod_{r=1}^k2^{-i_r/2},
\end{equation}
for all $\mathbf y\in Y$ and all indices $i_1,\ldots,i_k$. All partial derivatives of $X$ vanish on $Y\setminus G$; they also vanish on $G\setminus K$ by \eqref{eq:chi} and \eqref{eq:chain}. Thus these cases are covered by \eqref{eq:weighted}.

By \eqref{eq:weighted}, for every $p\geq1$,
\begin{equation}\label{eq:summable}
\sup_{\mathbf y\in Y}\sum_{i_1,\ldots,i_k\geq1}\left|\left(\prod_{r=1}^ka_{n_{i_r}}\right)\partial_{y_{i_1}}\cdots \partial_{y_{i_k}}X(\mathbf y)\right|^p
\leq C_k^p\left(\sum_{j=1}^\infty2^{-jp/2}\right)^k<\infty.
\end{equation}

For every $(s,\mathbf y,\mathbf z)\in{O}$, we define
\begin{equation*}
u(s,\mathbf y,\mathbf z)\triangleq\operatorname{sgn}(s)\, X(\mathbf y).
\end{equation*}
Then $u$ is a real-valued Borel measurable function. If $s\ne0$, it equals $X(\mathbf y)$ or $-X(\mathbf y)$ in a sufficiently small neighborhood of $(s,\mathbf y,\mathbf z)$. If $(0,\mathbf y,\mathbf z)\in{O}$, then \eqref{eq:section} implies $\mathbf y\notin K$. Since $K$ is closed, there exists a neighborhood $U$ of $(0,\mathbf y,\mathbf z)$ such that
$$
\mathbf y'\notin K,\qquad (s',\mathbf y',\mathbf z')\in U.
$$
By the argument after \eqref{eq:X}, we obtain $u=0$ on $U$. Hence $u$ is smooth on every finite-coordinate section of ${O}$.

Every partial derivative involving the first coordinate vanishes, as does every  partial derivative involving a coordinate in $I$. For a finitely supported multi-index $\beta=(\beta_j)_{j\in\mathbb N}$ involving only $Y$ coordinates,
\begin{equation}\label{eq:Du}
	\partial_{\mathbf y}^\beta u(s,\mathbf y,\mathbf z)=\operatorname{sgn}(s)\,\partial_{\mathbf y}^\beta X(\mathbf y),\qquad (s,\mathbf y,\mathbf z)\in {O}.
\end{equation}

By \eqref{eq:summable}, $|u|\leq1$ on ${O}$, and $P({O})\leq1$,
\begin{equation}\label{eq:boundSob}
\begin{aligned}
\sum_{|\alpha|\leq m}\int_{O}|\mathbf a^\alpha \partial^\alpha u|^p\,\mathrm d P
&\leq
\sum_{k=0}^{m}\int_Y \sup_{\mathbf y\in Y}\sum_{i_1,\ldots,i_k\geq1}\left|\left(\prod_{r=1}^ka_{n_{i_r}}\right)\partial_{y_{i_1}}\cdots \partial_{y_{i_k}}X(\mathbf y)\right|^p\,\mathrm d \nu(\mathbf y)\nonumber\\
&\leq
1+\sum_{k=1}^m C_k^p\left(\sum_{j=1}^\infty2^{-jp/2}\right)^k<\infty,
\end{aligned}
\end{equation}
where $\partial^\alpha u$ denotes the partial derivatives of $u$. We show next that these are indeed weak derivatives of $u$.

Fix $\alpha=(\alpha_i)_{i\in\mathbb N}\in\mathbb N_0^{(\mathbb N)}$ and a test function $\varphi\in C_{0,F^{\infty}}^\infty({O})$. There exist $R_0,d_0\in(0,\infty)$ such that
\[ \varphi^{-1}(\mathbb R\setminus\{0\})\subset\{\mathbf x\in {O}:\|\mathbf x\|\leq R_0\},\qquad \operatorname{dist}(\varphi^{-1}(\mathbb R\setminus\{0\}),\ell^2\setminus{O})\geq d_0.\]
Since $\varphi\in C_{0,F^{\infty}}^\infty({O})$, any partial derivative of $\varphi$ is bounded.

Choose $N$ such that $\alpha_i=0$ for $i>N$, and write $\mathbf x=(\mathbf z,\mathbf y)\in\mathbb R^N\oplus\ell^2(\mathbb N\setminus\{ 1, 2,\ldots, N\})$. For fixed $\mathbf y\in \ell^2(\mathbb N\setminus\{ 1, 2,\ldots, N\})$, let
\[{O}_{\mathbf y}\triangleq\{\mathbf z\in\mathbb R^N:(\mathbf z,\mathbf y)\in{O}\},\qquad \varphi_{\mathbf y}(\mathbf z)\triangleq\varphi(\mathbf z,\mathbf y).\]
Then ${O}_{\mathbf y}$ is an open subset of $\mathbb R^N$. The support of $\varphi_{\mathbf y}$ is bounded and lies at distance at least $d_0$ from $\mathbb R^N\setminus{O}_{\mathbf y}$; thus $\varphi_{\mathbf y}\in C_c^\infty({O}_{\mathbf y})$. The preceding argument shows that $u(\cdot,\mathbf y)\in C^{\infty}({O}_{\mathbf y})$.

Define
\[w_N(\mathbf z)\triangleq\left(\prod_{i=1}^N\frac1{\sqrt{2\pi}a_i}\right)\exp\left(-\sum_{i=1}^N\frac{z_i^2}{2a_i^2}\right),\qquad \mathbf z=(z_1,\ldots,z_N)\in\mathbb R^N.\]
Since $\delta_i\varphi_{\mathbf y}=w_N^{-1}\partial_{y_i}(w_N\varphi_{\mathbf y})$ for all $i=1,\ldots,N$, iteration gives
\[w_N\delta^\alpha\varphi_{\mathbf y}=\partial^\alpha(w_N\varphi_{\mathbf y}).\]
Finite-dimensional integration by parts therefore yields
\begin{equation}\label{eq:parts}
\begin{aligned}
\int_{{O}_{\mathbf y}}u(\mathbf z,\mathbf y)\delta^\alpha\varphi_{\mathbf y}(\mathbf z)w_N(\mathbf z)\,\mathrm d \mathbf z
&=\int_{{O}_{\mathbf y}}u(\mathbf z,\mathbf y)\partial^\alpha(w_N\varphi_{\mathbf y})(\mathbf z)\,\mathrm d \mathbf z\nonumber\\
&=(-1)^{|\alpha|}\int_{{O}_{\mathbf y}}\partial^\alpha u(\mathbf z,\mathbf y)\varphi_{\mathbf y}(\mathbf z)w_N(\mathbf z)\,\mathrm d \mathbf z.
\end{aligned}\end{equation}
There is no boundary term because $w_N\varphi_{\mathbf y}$ has compact support inside ${O}_{\mathbf y}$.

Combining \eqref{eq:weighted} and $\mathbf a^\alpha>0$, $\partial^\alpha u$ is bounded on ${O}$. Since $\varphi^{-1}(\mathbb R\setminus \{0\})\stackrel{\circ}{\subset}{O}$, $\delta^\alpha\varphi$ is also bounded: it is a finite sum of products of finitely many partial derivatives of $\varphi$ and polynomials in finitely many coordinates.   Thus both sides of \eqref{eq:parts} are absolutely integrable with respect to the variable $\mathbf y$ under the measure $P^{\widehat{1,\ldots,N}}$. Integrating both sides with respect to $P^{\widehat{1,\ldots,N}}$ and applying Fubini's theorem, we obtain

$$
\int_{O} u\,\delta^\alpha\varphi\,\,\mathrm d P
=
(-1)^{|\alpha|}
\int_{O} \partial^\alpha u\,\varphi\,\,\mathrm d P .
$$

All partial derivatives above are consequently weak derivatives. By \eqref{eq:boundSob},
\begin{equation*}
	u\in W^{m,p}({O},P)\qquad\text{for all }m\in\mathbb N,\ 1\leq p<\infty.
\end{equation*}
In particular,
\begin{equation}\label{eq:zero}
	\partial_{x_1} u =0\qquad\text{in the weak-derivative sense on }{O}.
\end{equation}

Fix $m\geq1$, $1\leq p<\infty$, and any $g\in\operatorname{span}\mathscr C_c^\infty$. For fixed $\mathbf y\in K$, $\mathbf z\in Z$, put
\[v_{\mathbf y,\mathbf z}(t)\triangleq g(a_1t,\mathbf y,\mathbf z),\qquad-1\leq t\leq1.\]
Because $g$ is a finite linear combination of whole-space smooth cylindrical functions, $v_{\mathbf y,\mathbf z}\in C^1([-1,1])$, and
\begin{equation}\label{eq:vprime}
	v_{\mathbf y,\mathbf z}'(t)=a_1 (\partial_{x_1} g) (a_1t,\mathbf y,\mathbf z).
\end{equation}
By \eqref{eq:section}, for $\mathbf y\in K$ and $t\ne0$, $(a_1t,\mathbf y,\mathbf z)\in{O}$, and
\[u(a_1t,\mathbf y,\mathbf z)=\operatorname{sgn}(t)X(\mathbf y).\]
Apply Lemma~\ref{lem:oned} with $A=X(\mathbf y)$ and $v=v_{\mathbf y,\mathbf z}$. Write
\[\rho(t)\triangleq\frac1{\sqrt{2\pi}}e^{-t^2/2},\qquad t\in\mathbb R.\]
For every $|t|\leq1$, we have $\rho(t)\geq\rho(1)$, and hence
\begin{equation}\label{eq:weightedoned}
\int_{-1}^1\bigl(|v_{\mathbf y,\mathbf z}(t)-X(\mathbf y)\operatorname{sgn}(t)|^p+|v_{\mathbf y,\mathbf z}'(t)|^p\bigr)\rho(t)\,\mathrm d t
\geq\rho(1)2^{2-p}X(\mathbf y)^p.
\end{equation}
By Definition \ref{def:Sob1}, \eqref{eq:zero}, \eqref{eq:vprime}, and the change of variables $s=a_1t$,
\begin{align}
\|u-g \|_{W^{m,p}({O},P)}^p
&\geq\int_{O}\left(|u-g|^p+a_1^p\left|\partial_{x_1} g \right|^p\right)\,\mathrm d P\notag\\
&\geq\int_Z\int_K\int_{-1}^1\bigl(|v_{\mathbf y,\mathbf z}(t)-X(\mathbf y)\operatorname{sgn}(t)|^p+|v_{\mathbf y,\mathbf z}'(t)|^p\bigr)\rho(t)\,\mathrm d t\,\mathrm d\nu(\mathbf y)\,\mathrm d\eta(\mathbf z)\notag\\
&\geq\rho(1)2^{2-p}\int_KX(\mathbf y)^p\,\mathrm d\nu(\mathbf y)\notag\\
&\geq\rho(1)2^{1-p}=2^{1-p}\frac{e^{-1/2}}{\sqrt{2\pi}},\label{eq:finalbound}
\end{align}
where the second inequality only restricts the nonnegative integral to the section in \eqref{eq:section}; the third uses \eqref{eq:weightedoned} and $\eta(Z)=1$, and the fourth uses \eqref{eq:mass}. Since $g$ was arbitrary, \eqref{eq:finalbound} proves \eqref{eq:lower}. This completes the proof of Theorem~\ref{thm:seg}.
\end{proof}

\section{Extension properties}\label{sec:ext}
In this section, we construct extension operators for some open sets in $\ell^2$.

\subsection{Domains with uniformly regular boundaries}\label{subsec:ext}
As is well known, on $\mathbb{R}^n$ equipped with Lebesgue measure, extension operators exist for domains satisfying the uniform $C^m$-regularity condition and having bounded boundary. See for example \cite[Theorem 5.22]{AF}. Motivated by this classical result, we address the analogous problem in the infinite dimensional setting.

Let $O\subset \ell^2$ be an open set. For $k\in\dbN\cup\{\infty\}$, we say $O$ has $C^k$ boundary if for every $\mathbf{x}\in \partial O$, there exists an open set $U$ containing $\mathbf{x}$ and $g\in C^k_{\ell^2}(U)$ such that

\begin{enumerate}[label=(\arabic*)]
   \item $U\cap O=\{\mathbf{x}\in U: g(\mathbf{x})>0  \}$ and $U\cap \partial O= \{  \mathbf{x}\in U: g(\mathbf{x})=0  \}$;
   \item $(\partial_{x_i} g(\mathbf{x}))_{i\ge 1}\neq\mathbf{0}$ for each $\mathbf{x}\in U\cap\partial O$.
\end{enumerate}

By the Lindel\"of property of $\ell^2$, when $O$ has $C^1$ boundary, we can find a countable open cover $\{U_j\}_{j\ge 1}$ of $\partial O$ and 
$g_j\in C^1_{\ell^2}(U_j)$ satisfying (1)--(2) above for each $j$. 

\begin{definition}\label{def:smbdry}
Let $O\subset \ell^2$ be an open set having $C^1$ boundary. We say $O$ satisfies the  \emph{uniform $C^1$-regularity condition} if the following conditions hold:
\begin{enumerate}[label=(\roman*)]
   \item For each $j\ge 1$, there exists an open set $W_j$ such that $U_j\subset W_j$. There exist $R>0, M\in\dbN$ such that $\operatorname{dist}(U_j,\partial W_j)\ge R$ and every collection of $M+1$ of the sets $W_j$ has empty intersection.
   \item For each $j\ge 1$, there exists $i_j^*\in\dbN$ such that the mapping
$$\Phi_j:U_j\to\ell^2,\qquad \Phi_j(\mathbf{x})\triangleq  g_j(\mathbf{x})\textbf{e}_{i_j^*} + \sum_{i\neq i_j^*}x_i \textbf{e}_i ,\qquad\forall\mathbf{x}=(x_i)_{i\in\dbN}\in U_j$$
is a $C^1$ diffeomorphism from $U_j$ onto  $V_j\triangleq \Phi_j(U_j)$, and $V_j$ is symmetric with respect to the hyperplane $\Pi_j=\{\mathbf{y}=(y_i)_{i\in\dbN}\in \ell^2: y_{i_j^*}=0\}$. The inverse of $\Phi_j$ is thus
$$ \Psi_j(\mathbf{y})\triangleq  h_j(\mathbf{y})\textbf{e}_{i_j^*} + \sum_{i\neq i_j^*}y_i \textbf{e}_i, $$
where $h_j(\mathbf{y})=x_{i_j^*}$ and $\Phi_j(\mathbf{x})=\mathbf{y}$. Moreover, for any sets $E_j\subset U_j$, $E_j'\subset V_j$ with $\operatorname{dist}(E_j,\partial U_j)>0$ and $\operatorname{dist}(E_j',\partial V_j)>0$, we have 
$\operatorname{dist}(\Phi_j(E_j),\partial V_j)>0$ and $\operatorname{dist}(\Psi_j(E'_j),\partial U_j)>0$.
\item There exist $\delta, N>0$ such that $O^\delta\triangleq \{\mathbf{x}\in O : \operatorname{dist}(\mathbf{x},\partial O)<\delta\}\subset \bigcup_{j\ge 1} U_j$ and $i_j^*\le N$  for all $j\ge 1$. In addition, there exists $A>0$ such that for all $j\ge 1$, and any $\mathbf{x}\in U_j\cap O$, $\mathbf{y}\in V_j^+\triangleq\Phi_j(U_j\cap O)$, we have
$$ \left| \frac{\mathrm{d}P\circ \Phi_j}{\mathrm{d} P}(\mathbf{x}) \right|\le A, \qquad  \left| \frac{\mathrm{d}P\circ \Psi_j}{\mathrm{d} P}(\mathbf{y}) \right|\le A, $$
and
$$ \sup_{j\ge 1} \|g_j \|_{C_{\mathscr{F}}^{1,1}(O\cap U_j)}\le A, \qquad \sup_{j\ge 1} \|h_j\|_{C_{\mathscr{F}}^{1,1}(V_j)}\le A.  $$
\end{enumerate}
\end{definition}

We give some examples of $O$ satisfying the uniform $C^1$-regularity condition.

\begin{example}
	Fix $i_0\in\dbN$ and choose a globally Lipschitz function $\psi\in C_{\ell^2(\mathbb N\setminus\{i_0\})}^{1}
\bigl(\ell^2(\mathbb N\setminus\{i_0\})\bigr)$ such that
   \[\sup_{\ell^2(\dbN\setminus\{i_0\})}|\psi|<\infty,\qquad\sup_{\ell^2(\dbN\setminus\{i_0\})}\sum_{k\neq i_0}a_k|\partial_{x_k}\psi|<\infty.\]
   Concrete examples of $\psi$ include constant functions, $\sum_{k\neq i_0}2^{-k}\sin x_k$, and $\exp\left(-\sum_{k\neq i_0}x_k^2\right)$.
   Then for any $\varepsilon_0\in\{1,-1\}$,
   \[O=\left\{\mathbf{x}=(x_j)_{j\in\dbN}\in\ell^2:\varepsilon_0 x_{i_0}>\psi((x_k)_{k\neq i_0})\right\}\]
   satisfies the uniform $C^1$-regularity condition. Indeed, set
   \[g(\mathbf{x})=\varepsilon_0 x_{i_0}-\psi((x_k)_{k\neq i_0})\]
   and
   \[U=\{\mathbf{x}\in\ell^2:|g(\mathbf{x})|<1\},\qquad W=\{\mathbf{x}\in\ell^2:|g(\mathbf{x})|<2\}.\]
   Since $\psi$ is globally Lipschitz on $\ell^2(\dbN\setminus\{i_0\})$, $g$ is globally Lipschitz on $\ell^2$; i.e., there exists $L>0$ such that
   \[|g(\mathbf{x})-g(\mathbf{y})|\le L\|\mathbf{x}-\mathbf{y}\|,\qquad\forall\mathbf{x},\mathbf{y}\in\ell^2.\]
   Then $\{U\}$ is an open cover of $\partial O$, $U\subset W$, and $\operatorname{dist}(U,\partial W)\ge L^{-1}$.
   Moreover,
   \[|g(\mathbf{x})|\le L\cdot \operatorname{dist}(\mathbf{x},\partial O),\qquad\forall\mathbf{x}\in O,\]
   which implies $O^\delta\subset U$ when $\delta<L^{-1}$. Note that
   \[\left| \frac{\mathrm{d}P\circ \Phi_j}{\mathrm{d} P}(\mathbf{x}) \right|=\exp\left(\frac{x_{i_0}^2-g(\mathbf{x})^2}{2a_i^2}\right)\]
   is bounded on $U\cap O$ by the properties of $\psi$. The remaining conditions in Definition~\ref{def:smbdry} are easy to verify.
\end{example}

To construct extension operators, we need the following two lemmas.
\begin{lemma}\label{uniformcondition}
Let $\{ O_j \}_{j\ge 1}$ be a collection of measurable sets in $\ell^2$. Let $f_j$ be an integrable function defined on $O_j$. If there exists $M\in\dbN$ such that every collection of $M+1$ of the sets $O_j$ has empty intersection, then for any measurable set $O\subset \bigcup_{j\ge 1} O_j$, we have
$$  \| \sum_j f_j \|_{L^p(O, P)}^p \le  M^{p-1} \sum_j \| f_j \|_{L^p(O_j, P)}^p.  $$
\end{lemma}
\begin{proof}
For any $\mathbf{x}\in O$, there are at most $M$ of $O_j$ containing $\mathbf{x}$, say $\mathbf{x}\in O_{j_1}\cap\dots\cap O_{j_m}$ ($m\le M$). Then it follows from Lemma~\ref{nver1} that
$$ |\sum_j f_j(\mathbf{x})|^p=|f_{j_1}(\mathbf{x})+\dots + f_{j_m}(\mathbf{x})|^p\le M^{p-1}( |f_{j_1}(\mathbf{x})|^p+\dots + |f_{j_m}(\mathbf{x})|^p ) =M^{p-1}\sum_j |f_{j}(\mathbf{x})|^p. $$
The result follows.
\end{proof}

\begin{lemma}\label{Sobolevnormundertransform}
Let $\Phi_j, \Psi_j, U_j, V_j$ be as in Definition~\ref{def:smbdry}. There exists $C>0$ independent of $j$ such that
$$ \| f\circ\Phi_j \|_{\mathscr{W}^{1,p}(U_j\cap O, P)}\le C \|f\|_{\mathscr{W}^{1,p}(V_j^+, P)}, \qquad \forall f\in \mathscr{W}^{1,p}(V_j^+, P),$$
and
$$ \| F\circ\Psi_j \|_{\mathscr{W}^{1,p}(V_j^+, P)}\le C \|\widetilde{f}\|_{\mathscr{W}^{1,p}(U_j\cap O, P)}, \qquad \forall F\in \mathscr{W}^{1,p}(U_j\cap O, P). $$
The same holds if one replaces $\mathscr{W}^{1,p}$ by $W^{1,p}$.
\end{lemma}
\begin{proof}
	We only prove the $\mathscr{W}^{1,p}$ case since the $W^{1,p}$ case is similar.
It suffices to prove the first inequality. By the condition (iii) in Definition~\ref{def:smbdry}, we have
$$  \int_{U_j\cap O} |f(\Phi_j(\mathbf{x}))|^p \mathrm{d} P(\mathbf{x})= \int_{V_j^+}  |f(\mathbf{y})|^p \mathrm{d} P(\Psi_j(\mathbf{y}))\le A \int_{V_j^+}  |f(\mathbf{y})|^p \mathrm{d} P(\mathbf{y}). $$
On the other hand, for any $\mathbf{x}\in U_j$, a direct calculation yields
$$ |\partial_{x_k} f(\Phi_j(\mathbf{x}))|^2\le \begin{cases} 2|(\partial_{x_{i_j^*}}f)(\Phi_j(\mathbf{x}))|^2|(\partial_{x_k}g_j)(\mathbf{x})|^2 + 2|(\partial_{x_k}f)(\Phi_j(\mathbf{x}))|^2, & k\neq i_j^*, \\ |(\partial_{x_{i_j^*}}f)(\Phi_j(\mathbf{x}))|^2|(\partial_{x_{i_j^*}}g_j)(\mathbf{x})|^2, & k=i_j^*.\end{cases}  $$
It follows that
\[\begin{aligned}
 &\hphantom{{}={}}\!\left( \sum_{k\ge1 } a_k^2 |\partial_{x_k} f(\Phi_j(\mathbf{x}))|^2  \right)^{\frac{p}{2}} \\
   &\le 2^{\frac{p}{2}} \left(  |(\partial_{x_{i_j^*}} f)(\Phi_j(\mathbf{x}))|^2 \sum_{k\ge1 } a_k^2 |\partial_{x_k} g_j(\mathbf{x})|^2 +\sum_{k\ge1} a_k^2 |(\partial_{x_k} f)(\Phi_j(\mathbf{x}))|^2 \right)^{\frac{p}{2}}.
\end{aligned}\]
Then we obtain
\[\begin{aligned}
    &\hphantom{{}={}}\! \int_{U_j\cap O} \left( \sum_{k\ge1 } a_k^2 |\partial_{x_k} f(\Phi_j(\mathbf{x}))|^2  \right)^{\frac{p}{2}} \mathrm{d} P(\mathbf{x}) \\
   &\le  2^{\frac{p}{2}}\left( \frac{1}{a_{i_j^*}^{2}}\|g_j\|_{C_{\mathscr{F}}^{1,2}(U_j\cap O)}^2  +1  \right)^{\frac{p}{2}} 2^{\frac{p}{2}} \int_{U_j\cap O} \left( \sum_{k\ge1 } a_k^2 |(\partial_{x_k} f)(\Phi_j(\mathbf{x}))|^2  \right)^{\frac{p}{2}} \mathrm{d} P(\mathbf{x})\\
   &\le 2^{p}A\left(\frac{1}{a_{i_j^*}^{2}}\|g_j\|_{C_{\mathscr{F}}^{1,2}(U_j\cap O)}^2  +1  \right)^{\frac{p}{2}} \int_{V_j^+} \left( \sum_{k\ge1 } a_k^2 |(\partial_{x_k} f)(\mathbf{y})|^2  \right)^{\frac{p}{2}} \mathrm{d} P(\mathbf{y}).
\end{aligned}\]
Again by the condition (iii) in Definition~\ref{def:smbdry}, the result follows.
\end{proof}

Next, we construct some smooth truncation functions with controlled partial derivatives.

\begin{lemma}\label{1f31ff3}
	Let $U$ and $V$ be nonempty open subsets of $\ell^2$ such that $U\subset V$ and $\operatorname{dist}( U,\partial V ) >0$. Then, for every $C>1$, there exists $\eta \in C_{\ell^2}^{\infty}( \ell^{2} )$ satisfying
	$$\eta |_{U}=1,\qquad 0\le \eta \le 1,\qquad\| \nabla \eta \|_{\infty}\triangleq \sup_{\mathbf{x}\in\ell^2}\|(\partial_{x_i}\eta(\mathbf{x}))_{i\in\dbN}\| \le \frac{C}{\operatorname{dist}( U,\partial V)},$$
	and
	\[\supp \eta \subset V,\qquad \operatorname{dist}( \supp \eta ,\partial V ) >0.\]
\end{lemma}
\begin{proof}
	For any $C>1$, choose $r\in \left( C^{-1}\operatorname{dist}( U,\partial V ) ,\operatorname{dist}( U,\partial V ) \right)$ and define
	$$f(\mathbf{x}) \triangleq \max \left\{ 0,1-\frac{\operatorname{dist}( \mathbf{x} ,U )}{r} \right\} .$$
	It is readily verified that $f$ is $\frac{1}{r}$-Lipschitz, i.e.,
	\[|f(\mathbf{x})-f(\mathbf{y})|\le \frac{1}{r}\|\mathbf{x}-\mathbf{y}\|,\qquad\forall\mathbf{x},\mathbf{y}\in\ell^2.\]
	Choose $\epsilon \in \left( 0,\frac{1}{8} \right)$ and  $\sigma >0$ such that
	\[\frac{\left( 1+\sigma \right)^{2}}{1-4\epsilon} <\frac{Cr}{\delta}.\]
	By \cite[Theorem~1, p.~2]{Aza07}, there exists $g\in C_{\ell^2}^{\infty}( \ell^{2})$ such that
	$$\left| g-f \right| <\epsilon ,\qquad\| \nabla g \|_{\infty} \le \frac{1+\sigma}{r}.$$
	Choose $\theta \in C^{\infty}( \mathbb{R})$ such that
	$$0\le\theta\le 1,\qquad\theta |_{\left( -\infty ,2\epsilon \right]}=0,\qquad\theta |_{\left[ 1-2\epsilon ,\infty \right)}=1,\qquad\left| \theta^{\prime} \right| \le \frac{1+\sigma}{1-4\epsilon}.$$
	Consider $$\eta (\mathbf{x}) \triangleq \theta \circ g(\mathbf{x}).$$ Clearly, $0\le \eta \le 1$.
	
	On $U$, we have $f=1$, and hence $$\left| g \right| \ge 1-\left| g-1 \right| \ge 1-\epsilon \ge 1-2\epsilon,$$ which implies that $\eta|_U =1$.
	If $\operatorname{dist}( \mathbf{x} ,U ) >r$, then $f(\mathbf{x})=0$, so $|g(\mathbf{x})|<\epsilon$ and $\eta(\mathbf{x}) =0$. Therefore, we have $\operatorname{dist}( \supp \eta ,\partial V ) >0.$
	
	Finally, we also have $$\|\nabla \eta \|_{\infty} \le \sup \left| \theta^{\prime} \right| \cdot \| \nabla g \|_{\infty} \le \frac{1+\sigma}{1-4\epsilon} \cdot \frac{1+\sigma}{r} \le \frac{C}{\delta},$$ which proves Lemma \ref{1f31ff3}.
\end{proof}

Then we use Lemma~\ref{1f31ff3} to construct a smooth partition of unity with controlled partial derivatives.

\begin{theorem}\label{Partionofunity}
	Suppose that a nonempty set $A\subset \ell^{2}$ admits an open cover $\bigcup_{i=1}^{\infty} U_{i}$ satisfying the following conditions:
	\begin{enumerate}[label=(\arabic*)]
		\item there exist open sets $V_{i}\subset U_{i}$ ($i=1,2,\dots$) such that $$A\subset \bigcup_{i=1}^{\infty} V_{i},\qquad \inf_{i\ge 1} \operatorname{dist}( V_{i},\partial U_{i}) >0;$$
		\item there exists $N\in \dbN$ such that every $\mathbf{x} \in A$ belongs to at most $N$ of the sets $U_i$.
	\end{enumerate}
	Then there exist $\rho_{i} \in C_{\ell^2}^{\infty}( \ell^{2} )$ ($i=1,2,\dots$) such that
	\begin{enumerate}[label=(\roman*)]
		\item $0\le\rho_i\le 1$ for all $i$ and $\sum_{i=1}^{\infty} \rho_{i} (\mathbf{x}) =1$ for every $\mathbf{x} \in A$;
		\item $\supp \rho_{i} \subset U_{i}$ and $\operatorname{dist}( \supp \rho_{i} ,U_{i} ) >0$ for all $i$;
		\item $\{\supp \rho_{i}\}_{i\ge1}$ is locally finite with respect to the topology of $A$;
		\item $\sup_{i\ge 1} \| \nabla\rho_{i} \|_{\infty} <\infty$.
	\end{enumerate}
\end{theorem}
\begin{proof}
	Define the open sets
	$$W_{i}\triangleq \left\{ \mathbf{x} \in U_{i}:\operatorname{dist}( \mathbf{x} ,V_{i} ) <\frac{1}{4} \inf_{k\ge 1} \operatorname{dist}( V_{k},\partial U_{k}) \right\} ,\qquad i\in\dbN.$$ For $\mathbf{x} \in A$, suppose that $j\in \mathbb{N}$ satisfies
	$$W_{j}\bigcap B\left( \mathbf{x} ,\frac{1}{4} \inf_{i\ge 1} \operatorname{dist}( V_{i},\partial U_{i}) \right) \neq \varnothing .$$ Now, for $\mathbf{z} \in W_{j}\bigcap B\left( \mathbf{x} ,\frac{1}{4} \inf_{i\ge 1} \operatorname{dist}( V_{i},\partial U_{i}) \right)$, choose $\mathbf{y} \in V_{j}$ such that
	$$\| \mathbf{z} -\mathbf{y} \| <\frac{1}{4} \inf_{i\ge 1} \operatorname{dist}( V_{i},\partial U_{i}),$$
	and hence
	$$\| \mathbf{x} -\mathbf{y} \| \le\| \mathbf{z} -\mathbf{y} \| +\| \mathbf{x} -\mathbf{z} \| <\frac{1}{2} \inf_{i\ge 1} \operatorname{dist}( V_{i},\partial U_{i}) <\inf_{i\ge 1} \operatorname{dist}( V_{i},\partial U_{i}).$$ It follows that $\mathbf{x} \in U_{j}$. Therefore, $B\left( \mathbf{x} ,\frac{1}{4} \inf_{i\ge 1} \operatorname{dist}( V_{i},\partial U_{i}) \right)$ intersects at most $N$ of the sets $W_i$.
	
	By Lemma \ref{1f31ff3}, for each $i\in \dbN$, we may choose $\eta_{i} \in C_{\ell^2}^{\infty}( \ell^{2} )$ such that
	$$\eta_{i} |_{V_{i}}=1,\qquad 0\le \eta_{i} \le 1,\qquad\supp \eta_{i} \subset W_{i},\qquad \operatorname{dist}( \supp \eta_{i} ,\partial W_{i}) >0,$$
   and
   \[\| \nabla\eta_{i} \|_{\infty} \le \frac{2}{\operatorname{dist}( V_{i},\partial W_{i} )} \le \frac{8}{\inf_{i\ge 1} \operatorname{dist}( V_{i},\partial U_{i})} .\]
   Define
	$$\rho_{1} \triangleq \eta_{1} ,\qquad\rho_{i} \triangleq \eta_{i} \prod_{k=1}^{i-1} \left( 1-\eta_{k} \right) ,\qquad i=2,3,\dots.$$ Clearly,
	$$\rho_{i} \in C^{\infty}( \ell^{2} ),\qquad 0\le \rho_i\le 1,\qquad i=1,2,\dots$$
	and
	$$\supp \rho_{i} \subset \supp \eta_{i} \subset W_{i}\subset U_{i},\qquad \operatorname{dist}( \supp \rho_{i} ,\partial U_{i} ) >0,\qquad i=1,2,\dots ,$$
	which also shows that $\{\supp \rho_{i}\}_{i\ge1}$ is locally finite with respect to the topology of $A$.
	
	Since for every $\mathbf{x}\in A$ there exists some $j$ such that $\mathbf{x} \in V_{j}$, we have $\eta_j (\mathbf{x}) =1$, and therefore
	$$\lim_{n\rightarrow \infty} \sum_{i=1}^{n} \rho_{i} (\mathbf{x}) =\lim_{n\rightarrow \infty} \left( 1-\prod_{i=1}^{n} \left( 1-\eta_{i} (\mathbf{x}) \right) \right) =1.$$ For $\mathbf{x}\in A$, in a sufficiently small neighborhood of $\mathbf{x}$, using
	$$\left| \prod_{i=1}^{n} a_{i}-\prod_{i=1}^{n} b_{i} \right| \le \sum_{i=1}^{n} \left| a_{i}-b_{i} \right| ,\qquad\forall a_{i},b_{i}\in \left[ 0,1 \right],$$
	we obtain
	$$\left| \rho_{i} (\mathbf{x}) -\rho_{i} (\mathbf{y}) \right| \le \sum_{k=1}^{i} \left| \eta_{k} (\mathbf{x}) -\eta_{k} (\mathbf{y}) \right| \le \frac{8 N}{\inf_{j\ge 1} \operatorname{dist}( V_{j},\partial U_{j} )} \| \mathbf{x} -\mathbf{y} \| ,$$
	which clearly yields
	$$\sup_{i\ge 1} \| \nabla \rho_{i} \|_{\infty} \le \frac{8N}{\inf_{j\ge 1} \operatorname{dist}( V_{j},\partial U_{j} )} <\infty.$$ This completes the proof.
\end{proof}

Then we prove our first extension theorem.
\begin{proof}[Proof of Theorem~\ref{thm:extlfb}]
By the condition (iii) in Definition~\ref{def:smbdry}, $\{U_j\}_{j\ge 1}$ is a cover of $O^\delta$. Thus there exists an open subset $U_0$ of $O$ such that $\operatorname{dist}(U_0,\partial O)>0$ and $O\subset \bigcup_{j\ge 0} U_j$. Let $\{w_j\}_{j\ge 0}$ be the partition of unity subordinate to $\{U_j\}_{j\ge 0}$ given by Theorem~\ref{Partionofunity}.

Since $\mathscr{C}_b^\infty$ is dense in $\mathscr{W}^{1,p}(O,P)$ by the $\mathscr{W}^{m,p}$ version of Theorem~\ref{zhuhanshuguanghuabijin}, it suffices to consider $f\in \mathscr{C}^{\infty}_{b}$. Then for all $\mathbf{x}\in O$, we have
\[ f(\mathbf{x})= \sum_{j\ge0}f_j(\mathbf{x}),\qquad f_j(\mathbf{x})\triangleq w_j(\mathbf{x})\cdot f(\mathbf{x}). \]
By Proposition \ref{prop:chengjijieduan2}, we have $f_0\in \mathscr{W}^{1,p}(\ell^2, P)$ and
\[\|f_0\|_{\mathscr{W}^{1,p}(\ell^2, P)}\le C_p \left\| w_0 \right\|_{C_{\mathscr{F}}^{1,2}\left( U_0 \right)} \cdot \left\| f \right\|_{\mathscr{W}^{1,p}\left( U_0 ,P \right)}.\]

For $j\ge 1$, let $\psi_j(\mathbf{y})=f_j(\Psi_j(\mathbf{y})).$ Set
$$ \widetilde{\psi}_j(\mathbf{y})=\begin{cases}\psi_j(\mathbf{y}), & y_{i_j^*}>0,\\ \psi_j(-y_{i_j^*}, \mathbf{y}'), & y_{i_j^*}<0,\end{cases} $$
where $\mathbf{y}'=(y_i)_{i\neq i_j^*}$.
Set $\widetilde{f}_j(\mathbf{x})= \widetilde{\psi}_j(\Phi_j(\mathbf{x})).$ Then we have $\widetilde{f}_j(\mathbf{x})=f_j(\mathbf{x})$ for all $\mathbf{x}\in U_j\bigcap O$. Moreover, by the condition (iii) in Definition~\ref{def:smbdry} and Lemma \ref{Sobolevnormundertransform}, we obtain
$$ \|\widetilde{f}_j \|_{\mathscr{W}^{1,p}(\ell^2, P)}\le C_1 \|\widetilde\psi_j\|_{\mathscr{W}^{1,p}(V_j, P)}\le C_2 \|\psi_j\|_{\mathscr{W}^{1,p}(V_j^+, P)} \le C_3 \|f_j \|_{\mathscr{W}^{1,p}(U_j\cap O, P)},$$
where $C_3$ is independent of $j$. Then define the extension operator $E$ by
$$ Ef(\mathbf{x})= f_0(\mathbf{x}) + \sum_{j\ge 1} \widetilde{f}_j(\mathbf{x}).  $$
It is clear that $Ef(\mathbf{x})=f(\mathbf{x})$ for all $\mathbf{x}\in O$. Moreover, we have by (i) and Lemma \ref{uniformcondition},
\[\begin{aligned}
    \|Ef \|_{\mathscr{W}^{1,p}(\ell^2, P)}^p &\le 2^{p-1} \|f_0 \|_{\mathscr{W}^{1,p}(\ell^2, P)}^p + 2^{p-1}\|\sum_{j\ge 1} \widetilde{f}_j \|_{\mathscr{W}^{1,p}(\ell^2, P)}^p  \\
   &\le  2^{p-1}\|f_0 \|_{\mathscr{W}^{1,p}(\ell^2, P)}^p + 2^{p-1}M^{p-1} \sum_{j\ge 1} \| \widetilde{f}_j \|_{\mathscr{W}^{1,p}(\ell^2, P)}^p\\
   &\le 2^{p-1} \|f_0 \|_{\mathscr{W}^{1,p}(\ell^2, P)}^p + 2^{p-1}M^{p-1}C_3^p \sum_{j\ge 1} \| f_j \|_{\mathscr{W}^{1,p}(U_j\cap O, P)}^p\\
   &\le 2^{p-1}C_p^p A^p\|f \|_{\mathscr{W}^{1,p}(U_0\cap O, P)}^p + 2^{p-1}M^{p-1}C_3^p C_p^p A^p \sum_{j\ge 1} \| f \|_{\mathscr{W}^{1,p}(U_j\cap O, P)}^p\\
   &\le 2^{p-1}C_p^p A^p(1+M^{2p-2}C_3^p) \|f \|_{\mathscr{W}^{1,p}(O, P)}^p.
\end{aligned}\]
This completes the proof.
\end{proof}

\subsection{A more refined approach}

\begin{definition}\label{def:reflection}
	We call $O \subset \ell^{2}$ a \emph{Sobolev reflection domain} if $O$ has a $C^{1}$ boundary with finite surface measure $\mu_{\partial O}$ and there exists a mapping
	\[
	T:\ell^{2}\setminus O \rightarrow \overline{O},\qquad
	\mathbf{x}=(x_i)_{i\in\dbN}\mapsto (T_i(\mathbf{x}))_{i\in\dbN}
	\]
	and a constant $C>0$ such that
   \begin{align}
      &T\mathbf{x}=\mathbf{x},\qquad \mathbf{x}\in\partial O;\label{eqar:00}\\
      &T_i\in C^{1}\left(\ell^{2}\setminus\overline{O}\right)
		\bigcap C\left(\ell^{2}\setminus O\right),\qquad \forall i\in\mathbb{N};\label{eqar:11}
   \end{align}
	moreover, for any $n\in\mathbb{N}$, $\xi_1,\xi_2,\dots,\xi_n\in\mathbb{R}$, and $\mathbf{x}\in\ell^{2}\setminus\overline{ O}$, one has
	\begin{eqnarray}\label{eqar:22}
		\sum_{i=1}^{\infty}a_i^2
		\left| \sum_{j=1}^{n}\partial_{x_i}T_j(\mathbf{x})\xi_j\right|^2
		\le C\sum_{i=1}^{n}a_i^2\xi_i^2;
	\end{eqnarray}
	and
	\begin{eqnarray}\label{eqar:33}
	 \int_{\ell^{2} \setminus \overline{ O}} \left| f(T\mathbf{x} ) \right| \, \mathrm{d} P\le C\int_{O} |f(\mathbf{x} )|\, \mathrm{d} P,
		\qquad
		\forall f\in\mathscr{C}_{b}^{\infty}.
	\end{eqnarray}
\end{definition}

\begin{remark}\label{rmk:measure0}
	For $O$ as in Definition~\ref{def:reflection}, it follows immediately from Definition~\ref{20241121def1} that $P(\partial O)=0$.
\end{remark}

\begin{remark}
	Following an argument similar to that of \cite[Lemma 4.2, p. 92]{Ku19}, it is readily seen from \eqref{eqar:11} that $T$ is a Borel measurable mapping. Furthermore, for any $f\in\mathscr{C}_{b}^{\infty}$, one has
	\[
	f\circ T\in  C^{1}\left(\ell^{2}\setminus\overline{O}\right)\bigcap C\left(\ell^{2}\setminus O\right).
	\]
\end{remark}

\begin{remark}\label{33gengqiang}
	We point out that \eqref{eqar:33} can be extended to every
	$f\in L^{1}\left(O,P\right)$.
	Let $\mu$ denote the pushforward measure induced by $T$. Clearly, $\mu$ is a finite measure. Since $\mathscr{C}_{b}^{\infty}$ is dense in
	$L^{1}\left(\overline{O},P+\mu\right)$, for every bounded measurable function $f$, there exists a sequence
	$\left\{f_{n}\right\}_{n=1}^{\infty}\subset\mathscr{C}_{b}^{\infty}$ such that
	\[
	\lim_{n\to\infty}\int_{\overline{O}}|f_{n}-f|\,\mathrm{d}(\mu+P)=0.
	\]
	Therefore, by \eqref{eqar:33}, we obtain
	\[
	\begin{aligned}
		\int_{\ell^{2}\setminus\overline{O}}
		\left|f(T\mathbf{x})\right|\,\mathrm{d}P
		&=\int_{\overline{O}}\left|f(\mathbf{x})\right|\,\mathrm{d}\mu =\lim_{n\to\infty}\int_{\overline{O}}\left|f_{n}(\mathbf{x})\right|\,\mathrm{d}\mu \\
		&=\lim_{n\to\infty}\int_{\ell^{2}\setminus\overline{O}}\left|f_{n}(T\mathbf{x})\right|\,\mathrm{d}P \\
		&\leq C\lim_{n\to\infty}\int_{O}\left|f_{n}(\mathbf{x})\right|\,\mathrm{d}P \\
		&=C\int_{O}\left|f(\mathbf{x})\right|\,\mathrm{d}P.
	\end{aligned}
	\]
	
	For a general $f\in L^{1}\left(O,P\right)$, choose an increasing sequence of nonnegative bounded measurable functions
	$\{\tilde{f}_{n}\}_{n=1}^{\infty}$ such that
	\[
	\lim_{n\to\infty}\tilde{f}_{n}=|f|.
	\]
	Then, by the monotone convergence theorem (Levi's theorem),
	\[
	\begin{aligned}
		\int_{\ell^{2}\setminus\overline{O}}
		\left|f(T\mathbf{x})\right|\,\mathrm{d}P
		&=
		\lim_{n\to\infty}
		\int_{\ell^{2}\setminus\overline{O}}
		\left|\tilde{f}_{n}(T\mathbf{x})\right|\,\mathrm{d}P \\
		&\leq
		C\lim_{n\to\infty}
		\int_{O}
		|\tilde{f}_{n}(\mathbf{x})|\,\mathrm{d}P \\
		&\leq
		C\int_{O}
		|f(\mathbf{x})|\,\mathrm{d}P.
	\end{aligned}
	\]
	Hence, \eqref{eqar:33} extends to every $f\in L^{1}\left(O,P\right)$.
\end{remark}

\begin{proof}[Proof of Theorem~\ref{thm:ext}]
For $f\in \mathscr{C}_{b}^{\infty}$, consider
$$
\left( Ef \right) (\mathbf{x}) \triangleq
\begin{cases}
	f\left( T\mathbf{x} \right) , & \mathbf{x} \in \ell^{2} \setminus \overline{O}\\
	f(\mathbf{x}) , & \mathbf{x} \in \overline{O}.
\end{cases}
$$
We have
\[
Ef\in C^{1}\left(\ell^{2}\setminus\overline{O}\right) \cap C\left(\ell^{2}\setminus O\right).
\]
Moreover, using \eqref{eqar:33} for $f^p$, we have $Ef\in L^{p}\left(\ell^{2},P\right)$ and
\begin{eqnarray}
	\int_{\ell^{2}} |Ef|^{p}\,\mathrm{d}P
	\le (C+1)\int_{O}|f|^{p}\,\mathrm{d}P.
	\label{44}
\end{eqnarray}
For $\mathbf{x}\in\ell^{2}\setminus\overline{O}$, it follows from
\eqref{eqar:22} that
\[
\begin{aligned}
	\sum_{i=1}^{\infty}a_{i}^{2}
	\left|\partial_{x_i}\bigl(f(T\mathbf{x})\bigr)\right|^{2}
	&=
	\sum_{i=1}^{\infty}a_{i}^{2}
	\left|
	\sum_{j=1}^{\infty}
	\partial_{x_i}T_{j}(\mathbf{x})
	(\partial_{x_j}f)(T\mathbf{x})
	\right|^{2} \\
	&\le
	C\sum_{i=1}^{\infty}a_{i}^{2}
	\left|(\partial_{x_i}f)(T\mathbf{x})\right|^{2}.
\end{aligned}
\]
Consequently,
\begin{equation}\label{55}
	\begin{aligned}
		\int_{\ell^{2}\setminus\overline{O}}\left(\sum_{i=1}^{\infty}a_{i}^{2}\left|\partial_{x_i}(f(T\mathbf{x}))\right|^{2}\right)^{\frac{p}{2}}\,\mathrm{d}P &\le
		C^{\frac{p}{2}}\int_{\ell^{2}\setminus\overline{O}}\left(\sum_{i=1}^{\infty}a_{i}^{2}\left|(\partial_{x_i}f)(T\mathbf{x})\right|^{2}\right)^{\frac{p}{2}}\,\mathrm{d}P \\
		&\le C^{1+\frac{p}{2}}\int_{O}\left(\sum_{i=1}^{\infty}a_{i}^{2}|\partial_{x_i}f|^{2}\right)^{\frac{p}{2}}\,\mathrm{d}P,
	\end{aligned}
\end{equation}
where the second inequality follows from \eqref{eqar:33} and Remark~\ref{33gengqiang}.
Therefore, we have
\begin{equation}\label{88}
   \begin{aligned}
	&\hphantom{{}={}}\int_{\ell^{2}\setminus\overline{O}}\left(\sum_{i=1}^{\infty} a_{i}^{2}\left|\partial_{x_i}(f(T\mathbf{x}))\right|^{2}\right)^{\frac{p}{2}}\,\mathrm{d}P
	+\int_{O}\left(\sum_{i=1}^{\infty} a_{i}^{2}\left|\partial_{x_i}f\right|^{2}\right)^{\frac{p}{2}}\,\mathrm{d}P\\
   &\le
	\left(1+C^{1+\frac{p}{2}}\right)\int_{O}\left(\sum_{i=1}^{\infty}a_{i}^{2}\left|\partial_{x_i}f\right|^{2}\right)^{\frac{p}{2}}\,\mathrm{d}P.	
   \end{aligned}
\end{equation}

Denote by
\[
\mathbf{\nu}(\mathbf{x})=\left(\nu_{1}(\mathbf{x}),\nu_{2}(\mathbf{x}),\dots\right)
\]
the unit outward normal vector on $\partial O$. By Theorem~\ref{thm:GG} and \eqref{eqar:00}, for
$\phi\in C_{0,F^{\infty}}^{\infty}\left(\ell^{2}\right)$, we have
\[\begin{aligned}
	\int_{\ell^{2}}\partial_{x_j}(Ef)(\mathbf{x})\phi(\mathbf{x})\,\mathrm{d}P(\mathbf{x})&=-\int_{\ell^{2}}(Ef)(\mathbf{x})\,(\delta_{x_j}\phi)(\mathbf{x})\,\mathrm{d}P(\mathbf{x}) \\
   &=
	-\int_{\ell^{2}\setminus\overline{O}} f(T\mathbf{x})\,(\delta_{x_j}\phi)(\mathbf{x})\,\mathrm{d}P(\mathbf{x})
	-
	\int_{O}f(\mathbf{x})\,(\delta_{x_j}\phi)(\mathbf{x})\,\mathrm{d}P(\mathbf{x}) \\
	&=
	-\int_{\partial O}
	f(T\mathbf{x})\phi(\mathbf{x})\,\nu_{j}(\mathbf{x})
	\,\mathrm{d}\mu_{\partial O}(\mathbf{x})
	+
	\int_{\partial O}f(\mathbf{x})\phi(\mathbf{x})\,\nu_{j}(\mathbf{x})\,\mathrm{d}\mu_{\partial O}(\mathbf{x}) \\
	&\hphantom{{}={}}{+} \int_{\ell^{2}\setminus\overline{O}}\partial_{x_j}\bigl(f(T\mathbf{x})\bigr)\phi(\mathbf{x})\,\mathrm{d}P(\mathbf{x})
	+
   \int_{O}\partial_{x_j}f(\mathbf{x})\phi(\mathbf{x})\,\mathrm{d}P(\mathbf{x})\\
   &=\int_{\ell^{2}\setminus\overline{O}}\partial_{x_j}\bigl(f(T\mathbf{x})\bigr)\phi(\mathbf{x})\,\mathrm{d}P(\mathbf{x})
	+\int_{O}\partial_{x_j}f(\mathbf{x})\phi(\mathbf{x})\,\mathrm{d}P(\mathbf{x}).
\end{aligned}\]
Therefore, combining \eqref{44} and \eqref{88}, we have $Ef\in \mathscr{W}^{1,p}\left(\ell^{2},P\right)$ and
$$\begin{aligned}
&\hphantom{{}={}}\| Ef \|^p_{\mathscr{W}^{1,p}(\ell^{2},P)} \\
&= \int_{\ell^2}|Ef|^p\dd P+\int_{\ell^2}\left[\sum_{i=1}^{\infty}a_i^2\left(|\partial_{x_i}f(\mathbf{x})|^2\cdot \mathbf{1}_O+|\partial_{x_i}f(T\mathbf{x})|^2\cdot \mathbf{1}_{\ell^2\setminus\overline{O}}\right)\right]^{\frac{p}{2}}\dd P(\mathbf{x})\\
&\le \int_{\ell^2}|Ef|^p\dd P+2^{\frac{p}{2}}\int_{O}\left(\sum_{i=1}^{\infty}a_i^2|\partial_{x_i} f|^2\right)^{\frac{p}{2}}\dd P+2^{\frac{p}{2}}\int_{\ell^2\setminus\overline{O}}\left(\sum_{i=1}^{\infty}a_i^2|\partial_{x_i} f(T\mathbf{x})|^2\right)^{\frac{p}{2}}\dd P\\
&\le\max \left\{1+C ,\,2^{\frac{p}{2}}\left( 1+C^{1+\frac{p}{2}} \right)\right\} \| f \|^p_{\mathscr{W}^{1,p}(O ,P)}.
\end{aligned}$$

By the $\mathscr{W}^{m,p}$ version of Theorem~\ref{zhuhanshuguanghuabijin}, $\mathscr{C}_b^\infty$ is dense in $\mathscr{W}^{1,p}\left( O ,P \right)$. Thus we obtain
a bounded linear operator
$${E}:\mathscr{W}^{1,p}\left( O ,P \right) \rightarrow \mathscr{W}^{1,p}( \ell^{2} ,P)$$
such that
$$(Ef)(\mathbf{x})=f(\mathbf{x}),\quad \text{for $P$-almost every $\mathbf{x}\in O$},\qquad \forall f\in \mathscr{W}^{1,p}\left( O ,P \right).$$
This completes the proof of Theorem \ref{thm:ext}.
\end{proof}

\begin{proposition}\label{prop:ellipsoid}
Fix
$$\mathbf{x}^{0} \triangleq \left( x_{i}^{0} \right)_{i\in\dbN} \in \ell^{2} ,\qquad\left\{ c_{i} \right\}_{i=1}^{\infty} \subset \left( 0,\infty \right).$$
If $\sup\limits_{i\ge 1} c_{i}<\infty$, then
$$O \triangleq \left\{ (x_i)_{i\in\dbN}\in \ell^{2} :\sum_{i=1}^{\infty} c_{i}\left( x_{i}-x_{i}^{0} \right)^{2} <1 \right\}$$
is a Sobolev reflection domain.
\end{proposition}
\begin{proof}
	\textbf{Step 1: Construction of $T$.}
Choose an integer $N\ge 2$ such that
\[
\sum_{j=N}^{\infty}\left(4a_{j}^{2}c_{j}+2c_{j}\left|x_{j}^{0}\right|^{2}\right)<1.
\]
Then define
\[
r_{i}\triangleq
\begin{cases}
	1, & 1\le i<N,\\[2mm]
	\displaystyle\frac{\sum_{j=1}^{N-1}\left(4a_{j}^{2}c_{j}+2c_{j}\left|x_{j}^{0}\right|^{2}\right)}{1-\sum_{j=N}^{\infty}\left(4a_{j}^{2}c_{j}+2c_{j}\left|x_{j}^{0}\right|^{2}\right)}, & i\ge N.
\end{cases}
\]
It is readily verified that
\begin{equation}\label{eq:rN}
r_{i}=r_N=\sum_{j=1}^{\infty}r_{j}\left(4a_{j}^{2}c_{j}+2c_{j}\left|x_{j}^{0}\right|^{2}\right),
\qquad \forall\, i\ge N.
\end{equation}

Set
\[
\alpha_{i}\triangleq r_{i}a_{i}^{2}c_{i}, \qquad i\in\dbN.
\]
For $\mathbf{x}=(x_i)_{i\in\dbN}\in\ell^{2}\setminus O$, it is clear that the function
\[
\mathbb{R}\longrightarrow (0,\infty),
\qquad
t\longmapsto
\sum_{i=1}^{\infty}
c_i e^{-2\alpha_i t}
\left(x_i-x_i^{0}\right)^2
\]
is continuously differentiable and
\[\frac{\mathrm{d}}{\mathrm{d}t}
	\left[
	\sum_{i=1}^{\infty}
	c_i e^{-2\alpha_i t}
	\left(x_i-x_i^{0}\right)^2
	\right]
	=-2\sum_{i=1}^{\infty}
	\alpha_i c_i e^{-2\alpha_i t}
	\left(x_i-x_i^{0}\right)^2 <0.\]
In addition,
\[
\begin{aligned}
	\sum_{i=1}^{\infty} c_i e^{-2\alpha_i\cdot 0}\left(x_i-x_i^{0}\right)^2
	&=\sum_{i=1}^{\infty}c_i\left(x_i-x_i^{0}\right)^2,\\
	\lim_{t\to+\infty}\left[\sum_{i=1}^{\infty}c_i e^{-2\alpha_i t}\left(x_i-x_i^{0}\right)^2\right]
	&=0,\\
	\lim_{t\to-\infty}\left[\sum_{i=1}^{\infty}c_i e^{-2\alpha_i t}\left(x_i-x_i^{0}\right)^2\right]
	&=+\infty.
\end{aligned}
\]
Thus there exists a unique $t=t(\mathbf{x})\in\left[ 0,\infty \right),$
such that
\[
\sum_{i=1}^{\infty} c_{i}e^{-2\alpha_{i} t(\mathbf{x})}\left( x_{i}-x_{i}^{0} \right)^{2} =1.
\]

We define
$T:\ell^{2}\setminus O\to\overline{O}$
by
\[
T(\mathbf{x})\triangleq \bigl(T_i(\mathbf{x})\bigr)_{i=1}^{\infty},
\]
where
\[
T_i(\mathbf{x})\triangleq x_i^{0}+e^{-2\alpha_i t(\mathbf{x})}\left(x_i-x_i^{0}\right).
\]
By the implicit function theorem (e.g., \cite[Chapter~I, Theorem~5.9, p.~19]{Lan99}), it is readily seen that
\[
T_i\in C^{1}\left(\ell^{2}\setminus\overline{O}\right)\cap C\left(\ell^{2}\setminus O\right),
\qquad \forall i\in\dbN.
\]
Moreover, it is clear that
\[
T(\mathbf{x})=\mathbf{x},
\qquad
\forall\,\mathbf{x}\in\partial O.
\]

For $\mathbf{x}\in\ell^{2}\setminus\overline{O}$
and $i,j\in\dbN$, a direct computation yields
\begin{equation}\label{99}
	\partial_{x_i} T_j(\mathbf{x})=
	e^{-2\alpha_j t(\mathbf{x})}\delta_{ij}
	-
	\frac{
		2\alpha_j c_i
		e^{-2(\alpha_i+\alpha_j)t(\mathbf{x})}
		\left(x_i-x_i^{0}\right)
		\left(x_j-x_j^{0}\right)
	}{
		\displaystyle
		\sum_{k=1}^{\infty}
		\alpha_k c_k
		e^{-2\alpha_k t(\mathbf{x})}
		\left(x_k-x_k^{0}\right)^2
	}.
\end{equation}

\medskip\noindent\textbf{Step 2: Proof of \eqref{eqar:22}.}
 Now we claim that, for every $n\in\mathbb{N}$,
 $\xi_{1},\xi_{2},\dots,\xi_{n}\in\mathbb{R}$, and
 $\mathbf{x}\in\ell^{2}\setminus\overline{O}$,
 \begin{equation*}
 	\sum_{i=1}^{\infty} a_{i}^{2}\left|\sum_{j=1}^{n}\partial_{x_i}T_{j}(\mathbf{x})\xi_{j}\right|^{2}
 	\le
 	\max\left\{r_{N},\frac{1}{r_{N}}\right\}\sum_{i=1}^{n}a_{i}^{2}\xi_{i}^{2}.
 \end{equation*}
To prove this estimate, for $i\in\dbN$, set
 	\[
 	b_i\triangleq \sqrt{r_i}\,a_i,\qquad
 	d_i\triangleq e^{-2\alpha_i t(\mathbf{x})},
 	\qquad
 	z_i\triangleq x_i-x_i^0.
 	\]
 Then $\alpha_i=b_i^2c_i$, and hence
 	\[
 	\sum_{k=1}^{\infty}
 	\alpha_kc_k
 	e^{-2\alpha_kt(\mathbf{x})}
 	\left(x_k-x_k^0\right)^2
 	=
 	\sum_{k=1}^{\infty}
 	b_k^2c_k^2d_kz_k^2.
 	\]

 For convenience, extend $\{\xi_1,\dots,\xi_n\}$ to $\{\xi_i\}_{i\in\dbN}$ by setting $\xi_i=0$ for $i>n$, and define
 	\[
 	u=(u_i)_{i\in\dbN}\triangleq(b_i\xi_i)_{i\in\dbN},
 	\qquad
 	q=(q_i)_{i\in\dbN}\triangleq(b_ic_iz_i)_{i\in\dbN},
 	\]
 By \eqref{99}, we obtain
 	\begin{equation}\label{eq:preceding}
		\begin{aligned}
 		b_i\sum_{j=1}^{n}
 		\partial_{x_i}T_j(\mathbf{x})\xi_j
 		&=
 		b_id_i\xi_i
 		-
 		\frac{
 			2b_ic_id_iz_i
 		}{
 			\displaystyle
 			\sum_{k=1}^{\infty}b_k^2c_k^2d_kz_k^2
 		}
 		\sum_{j=1}^{\infty}
 		b_j^2c_jd_jz_j\xi_j\\
 		&=
 		d_iu_i
 		-
 		\frac{
 			2d_iq_i
 		}{
 			\displaystyle
 			\sum_{k=1}^{\infty}d_kq_k^2
 		}
 		\sum_{j=1}^{\infty}d_jq_ju_j .
 	\end{aligned}
	\end{equation}
 Let $\mathsf{D}_d$ denote the diagonal operator with diagonal entries $d_i$. Then \eqref{eq:preceding} can be written as
 	\[
 	\left(
 	b_i\sum_{j=1}^{n}
 	\partial_{x_i}T_j(\mathbf{x})\xi_j
 	\right)_{i=1}^{\infty}
 	=
 	\mathsf{D}_d(u)
 	-
 	2\mathsf{D}_d(q)\,
 	\frac{\langle q,\mathsf{D}_d(u)\rangle}
 	{\langle q,\mathsf{D}_d(q)\rangle}.
 	\]
 Since $d_i>0$, we may write
 	\[
 	\mathsf{D}_d(u)-2\mathsf{D}_d(q)\,\frac{\langle q,\mathsf{D}_d(u)\rangle}{\langle q,\mathsf{D}_d(q)\rangle}
 	=
 	\mathsf{D}_d^{1/2}
 	\left(
 	\mathrm{Id}-
 	2\frac{\mathsf{D}_d^{1/2}q\otimes \mathsf{D}_d^{1/2}q}{
 		\left\|\mathsf{D}_d^{1/2}q\right\|_{\ell^2}^{2}}
 	\right)
 	\mathsf{D}_d^{1/2}u,
 	\]
	where $\mathsf{D}_d^{1/2}q\otimes \mathsf{D}_d^{1/2}q:\ell^2\to\ell^2$ is defined by
	\[\left(\mathsf{D}_d^{1/2}q\otimes \mathsf{D}_d^{1/2}q\right)(\mathbf{y})=\left\langle\mathsf{D}_d^{1/2}q,\mathbf{y}\right\rangle\mathsf{D}_d^{1/2}q,\qquad\forall\mathbf{y}\in\ell^2.\]
 The operator
 	\[
 	\mathrm{Id}-
 	2\frac{
 		\mathsf{D}_d^{1/2}q\otimes \mathsf{D}_d^{1/2}q
 	}{
 		\left\|\mathsf{D}_d^{1/2}q\right\|_{\ell^2}^{2}
 	}
 	\]
 is an orthogonal reflection on $\ell^2$. Moreover, since
 $0<d_i\le1$, the operator $\mathsf{D}_d^{1/2}$ is a contraction. Therefore,
 \begin{equation}\label{eq:Testimate1}
 	\begin{aligned}
 		\sum_{i=1}^{\infty}
 		r_ia_i^2
 		\left|
 		\sum_{j=1}^{n}
 		\partial_{x_i}T_j(\mathbf{x})\xi_j
 		\right|^2&=
 		\left\|
 		\mathsf{D}_d^{1/2}
 		\left(
 		I-
 		2\frac{
 			\mathsf{D}_d^{1/2}q\otimes \mathsf{D}_d^{1/2}q
 		}{
 			\left\|\mathsf{D}_d^{1/2}q\right\|_{\ell^2}^{2}
 		}
 		\right)
 		\mathsf{D}_d^{1/2}u
 		\right\|_{\ell^2}^{2}\\
 		&\le
 		\left\|
 		\left(
 		I-
 		2\frac{
 			\mathsf{D}_d^{1/2}q\otimes \mathsf{D}_d^{1/2}q
 		}{
 			\left\|\mathsf{D}_d^{1/2}q\right\|_{\ell^2}^{2}
 		}
 		\right)
 		\mathsf{D}_d^{1/2}u
 		\right\|_{\ell^2}^{2}\\
 		&=
 		\left\|\mathsf{D}_d^{1/2}u\right\|_{\ell^2}^{2}
 		\le
 		\left\|u\right\|_{\ell^2}^{2}
 		=
 		\sum_{i=1}^{n}r_ia_i^2\xi_i^2.
 	\end{aligned}
 	 \end{equation}
 By the definition of $r_i$, we have $r_i=1$ for $1\le i<N$ and
 $r_i=r_N$ for $i\ge N$. Thus
 	\begin{equation}\label{eq:Testimate2}
 	\min\{1,r_N\}\sum_{i=1}^{\infty}a_i^2\left|\sum_{j=1}^{n}\partial_{x_i}T_j(\mathbf{x})\xi_j\right|^2\le\sum_{i=1}^{\infty}r_ia_i^2\left|\sum_{j=1}^{n}\partial_{x_i}T_j(\mathbf{x})\xi_j\right|^2,
	\end{equation}
	while
	\begin{equation}\label{eq:Testimate3}
 	\sum_{i=1}^{n}r_ia_i^2\xi_i^2
 	\le
 	\max\{1,r_N\}
 	\sum_{i=1}^{n}a_i^2\xi_i^2.
	\end{equation}
 Combining \eqref{eq:Testimate1}--\eqref{eq:Testimate3} gives
 \[\sum_{i=1}^{\infty}
 		a_i^2
 		\left|
 		\sum_{j=1}^{n}
 		\partial_{x_i}T_j(\mathbf{x})\xi_j
 		\right|^2
		\le
 		\frac{\max\{1,r_N\}}
 		{\min\{1,r_N\}}
 		\sum_{i=1}^{n}a_i^2\xi_i^2
		=
 		\max\left\{
 		r_N,\frac{1}{r_N}
 		\right\}
 		\sum_{i=1}^{n}a_i^2\xi_i^2,\]
		which proves our claim.

\medskip\noindent\textbf{Step 3: Proof of \eqref{eqar:33}.}
 Let
 $$C_{0}\triangleq \sup_{t\ge 0} \left( \frac{r_{N}}{2} t-\min \left\{ r_{N}t,\min_{1\le i<N} \frac{\sinh \left( 2r_{i}a_{i}^{2}c_{i}t \right)}{2a_{i}^{2}c_{i}} \right\} \right) <\infty.$$
 We now prove that, for $f\in C_{b}^{\infty}\left( \mathbb{R}^{m} \right),m\in\dbN$, the following inequality holds:
 \begin{equation}\label{1f1fg13g1}
 	\int_{\ell^{2} \setminus \overline{O}} \left| f(T\mathbf{x} ) \right| \, \mathrm{d} P\le e^{C_{0}}\int_{O} |f(\mathbf{x} )|\, \mathrm{d} P.
 \end{equation}
 For each $n\in\dbN$, define
 $$\begin{aligned}
   \mathbf{x}_{n}^{0}&\triangleq \left(x_{1}^{0},x_{2}^{0},\dots,x_{n}^{0}\right),\qquad O_n\triangleq \left\{\left( x_{i} \right)_{i=1}^n \in \mathbb{R}^{n} :\sum_{i=1}^{n} c_{i}\left( x_{i}-x_{i}^{0} \right)^{2} <1 \right\} .
\end{aligned}$$
 First, by the dominated convergence theorem and Remark~\ref{rmk:measure0},
 $$\begin{aligned}
 	\lim_{n\rightarrow \infty} \int_{O_n} \left| f \right| \dd\mathcal{N}^n &=\lim_{n\rightarrow \infty} \int_{\ell^{2}} \left| f \right| \cdot\mathbf{1}_{O_n \times \ell^{2} \left( \mathbb{N} \setminus \left\{ 1,2,\dots ,n \right\} \right)}\dd P\\
	&=\int_{\ell^{2}} \left| f \right| \cdot\lim_{n\rightarrow \infty} \mathbf{1}_{O_n \times \ell^{2} \left( \mathbb{N} \setminus \left\{ 1,2,\dots ,n \right\} \right)}\dd P\\
	&=\int_{\ell^{2}} \left| f \right|\cdot \mathbf{1}_{O}\dd P=\int_{O} \left| f \right|\dd P,\end{aligned}$$
 and
 $$\begin{aligned}
 	&\hphantom{{}={}}{\lim_{n\rightarrow \infty}} \int_{\mathbb{R}^{n} \setminus \overline{O_n}} \left| f\left( x_{1}^{0}+e^{-2\alpha_{1} t(\mathbf{x}_n')}\left( x_{1}-x_{1}^{0} \right) ,\dots ,x_{m}^{0}+e^{-2\alpha_{m} t(\mathbf{x}_n')}\left( x_{m}-x_{m}^{0} \right) \right) \right|\dd \mathcal{N}^n(\mathbf{x}_n)\\
   & =\lim_{n\rightarrow \infty} \left( \int_{\ell^{2}} \left| f\left( x_{1}^{0}+e^{-2\alpha_{1} t(\mathbf{x}_n')}\left( x_{1}-x_{1}^{0} \right) ,\dots ,x_{m}^{0}+e^{-2\alpha_{m} t(\mathbf{x}_n')}\left( x_{m}-x_{m}^{0} \right) \right) \right| \right.\\
   &\hphantom{={}}\left.\times \mathbf{1}_{(\mathbb{R}^{n} \setminus \overline{O_n} )\times \ell^{2} \left( \mathbb{N} \setminus \left\{ 1,2,\dots ,n \right\} \right)}\dd P \right)\\
   & =\int_{\ell^{2} \setminus \overline{O}} \left| f\left( x_{1}^{0}+e^{-2\alpha_{1} t(\mathbf{x})}\left( x_{1}-x_{1}^{0} \right) ,\dots ,x_{m}^{0}+e^{-2\alpha_{m} t(\mathbf{x})}\left( x_{m}-x_{m}^{0} \right) \right) \right|\dd P\\
   &=\int_{\ell^{2} \setminus \overline{O}} \left| f\left( T\mathbf{x} \right) \right|\dd P,\end{aligned}$$
 where
 $$\mathbf{x}_{n}\triangleq(x_1,\dots,x_n),\qquad\mathbf{x}_{n}' \triangleq (x'_i)_{i\in\dbN},\qquad
 x'_i=\begin{cases}
	x_i & (i\le n)\\
	x_{i}^0 & (i>n)
 \end{cases},\qquad n\in\dbN.$$
 Next, we prove that, for $n>\max \left\{ N,m \right\}$,
 \begin{equation}\label{13f1f3'}
 	\int_{\mathbb{R}^{n} \setminus \overline{O_n}} \left| f(T^{\left( n \right)}(\mathbf{x}_n )) \right| \, \mathrm{d} \mathcal{N}^{n} \le e^{C_{0}}\int_{O_n} |f(\mathbf{x}_n )|\, \mathrm{d} \mathcal{N}^{n},
 \end{equation}
 where
 $$T^{\left( n \right)}( \mathbf{x}_{n} ) \triangleq \left(x_{1}^{0}+e^{-2\alpha_{1} t(\mathbf{x}_n')}\left( x_{1}-x_{1}^{0} \right),\dots,x_{n}^{0}+e^{-2\alpha_{n} t(\mathbf{x}_n')}\left( x_{n}-x_{n}^{0} \right)\right) $$
 for $\mathbf{x}_{n} \in \mathbb{R}^{n} \setminus \overline{O_n}$.

Notice that $$T^{\left( n \right)}:\mathbb{R}^{n} \setminus \overline{O_n} \rightarrow O_n \setminus \left\{ \mathbf{x}_{n}^{0} \right\}$$ is a diffeomorphism since we can write down its inverse in a way similar to the construction of $T$ in Step 1. Using \eqref{99} and the notation in Step 2, we have
\[\left(\partial_{x_i}T_j^{(n)}(\mathbf{x}_n)\right)_{1\le i,j\le n}=\left[I-\begin{pmatrix}
	c_1d_1'z_1\\
	\vdots\\
	c_nd_n'z_n
\end{pmatrix}\cdot\frac{\left(b_1^2c_1z_1,\cdots,b_n^2c_nz_n\right)}{\sum_{k=1}^{n}b_k^2c_k^2d_k'z_k^2}\right]\cdot\operatorname{diag}(d_1',\dots, d_n'),\]
where $I$ is the $n\times n$ identity matrix and $\operatorname{diag}(d_1',\dots, d_n')$ is the diagonal matrix with diagonal entries $d_i'=e^{-2\alpha_it(\mathbf{x}_n')}$.
Thus by the matrix determinant lemma,
\begin{equation}\label{eq:jacobian}
\left| \det \left(\partial_{x_i}T_j^{(n)}(\mathbf{x}_n)\right)_{1\le i,j\le n} \right| =\prod_{i=1}^nd_i'=e^{-2t(\mathbf{x}_n') \sum_{i=1}^{n} \alpha_{i}}.
\end{equation}
Define
$$g_{n}\left( \mathbf{x}_{n} \right) \triangleq \prod_{i=1}^{n} \left( \frac{1}{\sqrt{2\pi} a_{i}} e^{-\frac{x_{i}^{2}}{2a_{i}^{2}}} \right) ,\qquad n\in\dbN.$$
Let $\mathbf{y}_{n} \triangleq T^{\left( n \right)}\left( \mathbf{x}_{n} \right) \in O_n \setminus \left\{ \mathbf{x}_{n}^{0} \right\}$ and
$$\xi_{i} =\xi_{i} \left( \mathbf{x}_{n} \right) \triangleq e^{-\alpha_{i} t(\mathbf{x}_n')}\left( x_{i}-x_{i}^{0} \right) ,\qquad i=1,2,\dots ,n.$$
We first prove that, for $\mathbf{x}_{n} \in \mathbb{R}^{n} \setminus \overline{O_n}$, the following inequality holds:
\begin{equation}\label{1f1ff'}
\begin{aligned}
	&\hphantom{{}={}}\!\log\frac{g_{n}\left( \mathbf{x}_{n} \right)}{g_{n}\left( \mathbf{y}_{n} \right) \cdot e^{-2t(\mathbf{x}_n') \sum_{i=1}^{n} \alpha_{i}}} \\
	&=2t(\mathbf{x}_n')\sum_{i=1}^{n} \alpha_{i} -\sum_{i=1}^{n} \frac{\sinh \left( 2\alpha_{i} t(\mathbf{x}_n') \right)}{a_{i}^{2}} \xi_{i}^{2}-2\sum_{i=1}^{n} \frac{x_{i}^{0}\xi_{i}}{a_{i}^{2}} \sinh \left( \alpha_{i} t(\mathbf{x}_n') \right)\\
	&\le C_{0}.
\end{aligned}
\end{equation}
The equality in \eqref{1f1ff'} follows from a direct calculation.
For $i=1,2,\dots ,n$, we have
$$\begin{aligned}
	&\hphantom{{}={}}\frac{\xi_{i}^{2} \sinh \left( 2\alpha_{i} t \right)}{2a_{i}^{2}} +\frac{\left| x_{i}^{0} \right|^{2}}{a_{i}^{2}} \tanh \left( \alpha_{i} t \right) +\frac{2x_{i}^{0}\xi_{i}}{a_{i}^{2}} \sinh \left( \alpha_{i} t \right)\\
	& =\frac{\sinh \left( \alpha_{i} t \right) \cosh \left( \alpha_{i} t \right)}{a_{i}^{2}} \left( \xi_{i} +\frac{x_{i}^{0}}{\cosh \left( \alpha_{i} t \right)} \right)^{2} \ge 0.
\end{aligned}$$
Hence,
$$\begin{aligned}
	-\frac{2x_{i}^{0}\xi_{i}}{a_{i}^{2}} \sinh \left( \alpha_{i} t \right)&\le \frac{\xi_{i}^{2} \sinh \left( 2\alpha_{i} t \right)}{2a_{i}^{2}} +\frac{\left| x_{i}^{0} \right|^{2}}{a_{i}^{2}} \tanh \left( \alpha_{i} t \right)\\
	&\le \frac{\xi_{i}^{2} \sinh \left( 2\alpha_{i} t \right)}{2a_{i}^{2}} +\frac{\left| x_{i}^{0} \right|^{2}}{a_{i}^{2}} \alpha_{i} t\\
	&=\frac{\xi_{i}^{2} \sinh \left( 2\alpha_{i} t \right)}{2a_{i}^{2}} +r_{i}c_{i}\left| x_{i}^{0} \right|^{2} t.
\end{aligned}$$
Combining this with \eqref{eq:rN} and $\sum_{i=1}^{n} c_{i}\xi_{i}^{2} =1$, we have
$$\begin{aligned}
	\log \frac{g_{n}\left( \mathbf{x}_{n} \right)}{g_{n}\left( \mathbf{y}_{n} \right) \cdot e^{-2t(\mathbf{x}_n') \sum_{i=1}^{n} \alpha_{i}}}&\le 2t(\mathbf{x}_n')\sum_{i=1}^{n} \alpha_{i} +t(\mathbf{x}_n')\sum_{i=1}^{n} r_{i}c_{i}\left| x_{i}^{0} \right|^{2} -\sum_{i=1}^{n} \frac{\xi_{i}^{2} \sinh \left( 2\alpha_{i} t(\mathbf{x}_n') \right)}{2a_{i}^{2}}\\
	&\le\frac{r_{N}}{2} t(\mathbf{x}_n')-\sum_{i=1}^{n} c_{i}\xi_{i}^{2} \frac{\sinh \left( 2r_{i}a_{i}^{2}c_{i}t(\mathbf{x}_n') \right)}{2a_{i}^{2}c_{i}}\\
	&\le \sup_{t\ge 0} \left( \frac{r_{N}}{2} t-\min \left\{ r_{N}t,\min_{1\le i<N} \frac{\sinh \left( 2r_{i}a_{i}^{2}c_{i}t \right)}{2a_{i}^{2}c_{i}} \right\} \right)\\
	&= C_{0},
\end{aligned}$$
which yields \eqref{1f1ff'}.

Now, by applying the change of variables formula and \eqref{eq:jacobian}--\eqref{1f1ff'}, we obtain
$$\begin{aligned}\int_{\mathbb{R}^{n} \setminus \overline{O_n}} \left| f(T^{\left( n \right)}(\mathbf{x}_{n}) ) \right| \, \mathrm{d} \mathcal{N}^{n}&=\int_{O_n} \left| f\left( \mathbf{y}_{n} \right) \right| \, \frac{g_{n}\left( \mathbf{x}_{n} \right)}{g_{n}\left( \mathbf{y}_{n} \right) \cdot e^{-2t(\mathbf{x}_n') \sum_{i=1}^{n} \alpha_{i}}} \mathrm{d} \mathcal{N}^{n} \left( \mathbf{y}_{n} \right)\\ &\le e^{C_{0}}\int_{O_n} \left| f\left( \mathbf{y}_{n} \right) \right| \, \mathrm{d} \mathcal{N}^{n} \left( \mathbf{y}_{n} \right) =e^{C_{0}}\int_{O_n} \left| f\left( \mathbf{x}_{n} \right) \right| \, \mathrm{d} \mathcal{N}^{n} \left( \mathbf{x}_{n} \right) ,\end{aligned}$$
which proves \eqref{13f1f3'}. Letting $n\rightarrow \infty$ on both sides of \eqref{13f1f3'}, we obtain \eqref{1f1fg13g1}. Therefore, we have proved that $O$ is a Sobolev reflection domain.
\end{proof}

\begin{theorem}\label{bankongjiankeyantuo}
	Suppose that $\mathbf{b} =\left( b_{i} \right)_{i\in\dbN} \in \ell^{2} \setminus \left\{ \mathbf{0} \right\}$, $r\in \mathbb{R}$, $p\ge 1$ and
	$$O =\left\{ (x_i)_{i\in\dbN} \in \ell^{2} :\sum_{i=1}^{\infty} b_{i}x_{i}<r \right\}.$$
	Then $O$ is a $\mathscr{W}^{1,p}$-extension domain.
\end{theorem}
\begin{proof}
	Set
	$$\mathbf{h} \triangleq (a_i^2b_i)_{i\in\dbN}\in \ell^{2}.$$
	For $f\in \mathscr{C}_{b}^{\infty}$, define
	$$(Ef)(\mathbf{x})\triangleq F(\mathbf{x})=
	\begin{cases}
		f\left( \mathbf{x} -\dfrac{2\left( \sum_{i=1}^{\infty} b_{i}x_{i}-r \right)}{\sum_{i=1}^{\infty} a_{i}^{2}b_{i}^{2}} \mathbf{h} \right)\cdot \exp\left(\displaystyle\frac{2r\left( \sum_{i=1}^{\infty} b_{i}x_{i}-r \right)}{p\sum_{i=1}^{\infty} a_{i}^{2}b_{i}^{2}}\right),&\mathbf{x} \notin O\\[3mm]
		f(\mathbf{x}) ,&\mathbf{x} \in O .
	\end{cases}$$
	It is easy to see that
	$$F\in C\left( \ell^{2} \right) \bigcap C_{\ell^{2}}^{\infty}\left( \ell^{2} \setminus \partial O \right),\quad F|_{\partial O}=f|_{\partial O}.$$
	By the Cameron--Martin formula and the extension theorem for Gaussian-measure-preserving orthogonal transformations (see \cite[Corollary 2.4.3, p. 61] {Bog98} and \cite [Theorem 6.2.2, p. 283]{Bog98}), we have
	\begin{equation}\label{eq:changevariable}
	\int_{\ell^{2} \setminus \overline{O}} \left| F(\mathbf{x}) \right|^{p} \mathrm{d} P(\mathbf{x}) {=} \int_{O} \left| f(\mathbf{y}) \right|^{p} \mathrm{d} P(\mathbf{y})
	\end{equation}
	after the change of variable
	\begin{equation*}
	\mathbf{y} =\mathbf{x} -\frac{2\left( \sum_{i=1}^{\infty} b_{i}x_{i}-r \right)}{\sum_{i=1}^{\infty} a_{i}^{2}b_{i}^{2}} \mathbf{h}.
	\end{equation*}

	For $i,j\in\dbN$, we first note that
	\[
		\partial_{x_i}\left(
		x_{j}
		-\frac{2\left(\sum_{k=1}^{\infty}b_{k}x_{k}-r\right)}
		{\sum_{k=1}^{\infty}a_{k}^{2}b_{k}^{2}}
		a_{j}^{2}b_{j}
		\right)
		=
		\delta_{ij}
		-\frac{2b_{i}a_{j}^{2}b_{j}}
		{\sum_{k=1}^{\infty}a_{k}^{2}b_{k}^{2}}.
		\]
	Hence, for $\mathbf{x}\notin\overline{O}$, by the chain rule, we have
		\[\begin{aligned}
			\partial_{x_i}F(\mathbf{x})&=
			\left[
			\sum_{j=1}^{\infty}
			\partial_{y_j}f(\mathbf{y})
			\left(\delta_{ij}-\frac{2b_{i}a_{j}^{2}b_{j}}{\sum_{k=1}^{\infty}a_{k}^{2}b_{k}^{2}}\right)+\frac{2rb_{i}}
			{p\sum_{k=1}^{\infty}a_{k}^{2}b_{k}^{2}}
			f(\mathbf{y})
			\right]\\
			&\hphantom{{}={}}\times\exp\left(\frac{2r\left(\sum_{k=1}^{\infty}b_{k}x_{k}-r\right)}
				{p\sum_{k=1}^{\infty}a_{k}^{2}b_{k}^{2}}\right)\\
			&=
			\Bigg[
			\partial_{y_i}f(\mathbf{y})-\frac{2b_{i}}
			{\sum_{k=1}^{\infty}a_{k}^{2}b_{k}^{2}}
			\sum_{j=1}^{\infty}a_{j}^{2}b_{j}
			\partial_{y_j}f(\mathbf{y})+\frac{2rb_{i}}
			{p\sum_{k=1}^{\infty}a_{k}^{2}b_{k}^{2}}
			f(\mathbf{y})
			\Bigg]\\
			&\hphantom{{}={}}\times\exp\left(\frac{2r\left(\sum_{k=1}^{\infty}b_{k}x_{k}-r\right)}
				{p\sum_{k=1}^{\infty}a_{k}^{2}b_{k}^{2}}\right).
		\end{aligned}
		\]
	Since
		\[
		\mathbf{x}=\mathbf{y}-\frac{2\left(\sum_{k=1}^{\infty}b_{k}y_{k}-r\right)}
		{\sum_{k=1}^{\infty}a_{k}^{2}b_{k}^{2}}\mathbf{h}
		\]
	and
		\[
		\sum_{k=1}^{\infty}b_{k}x_{k}-r
		=
		-\left(
		\sum_{k=1}^{\infty}b_{k}y_{k}-r
		\right),
		\]
	the above expression of $\partial_{x_i}F(\mathbf{x})$ gives
		\[
		\begin{aligned}
			&\hphantom{{}={}}\!\left(
			\sum_{i=1}^{\infty}
			a_{i}^{2}\left|\partial_{x_i}F(\mathbf{x})\right|^{2}
			\right)\cdot\exp\left(\frac{4r\left(\sum_{k=1}^{\infty}b_{k}y_{k}-r\right)}
				{p\sum_{k=1}^{\infty}a_{k}^{2}b_{k}^{2}}\right)\\
			&=\sum_{i=1}^{\infty}a_{i}^{2}
			\left|
			\partial_{y_i}f(\mathbf{y})
			-\frac{2b_{i}}
			{\sum_{k=1}^{\infty}a_{k}^{2}b_{k}^{2}}
			\sum_{j=1}^{\infty}a_{j}^{2}b_{j}
			\partial_{y_j}f(\mathbf{y})
			+\frac{2rb_{i}}
			{p\sum_{k=1}^{\infty}a_{k}^{2}b_{k}^{2}}
			f(\mathbf{y})
			\right|^{2}.
		\end{aligned}
		\]
	
	Set
		\[
		u_{i}
		\triangleq
		\partial_{y_i}f(\mathbf{y})
		-\frac{2rb_{i}}
		{p\sum_{k=1}^{\infty}a_{k}^{2}b_{k}^{2}}
		f(\mathbf{y}).
		\]
	We then have
		\[
		\begin{aligned}
			&\hphantom{{}={}}\partial_{y_i}f(\mathbf{y})
			-\frac{2b_{i}}
			{\sum_{k=1}^{\infty}a_{k}^{2}b_{k}^{2}}
			\sum_{j=1}^{\infty}a_{j}^{2}b_{j}
			\partial_{y_j}f(\mathbf{y})
			+\frac{2rb_{i}}
			{p\sum_{k=1}^{\infty}a_{k}^{2}b_{k}^{2}}
			f(\mathbf{y})\\
			&=
			u_{i}
			-\frac{2b_{i}}
			{\sum_{k=1}^{\infty}a_{k}^{2}b_{k}^{2}}
			\sum_{j=1}^{\infty}a_{j}^{2}b_{j}u_{j}.
		\end{aligned}
		\]
	Consequently,
		\[
		\begin{aligned}
			&\hphantom{{}={}}\!\sum_{i=1}^{\infty}a_{i}^{2}
			\left|
			u_{i}
			-\frac{2b_{i}}
			{\sum_{k=1}^{\infty}a_{k}^{2}b_{k}^{2}}
			\sum_{j=1}^{\infty}a_{j}^{2}b_{j}u_{j}
			\right|^{2}\\
			&=
			\sum_{i=1}^{\infty}a_{i}^{2}|u_{i}|^{2}
			-\frac{4}
			{\sum_{k=1}^{\infty}a_{k}^{2}b_{k}^{2}}
			\left|
			\sum_{j=1}^{\infty}a_{j}^{2}b_{j}u_{j}
			\right|^{2}+\frac{4\sum_{i=1}^{\infty}a_{i}^{2}b_{i}^{2}}
			{\left(\sum_{k=1}^{\infty}a_{k}^{2}b_{k}^{2}\right)^{2}}
			\left|
			\sum_{j=1}^{\infty}a_{j}^{2}b_{j}u_{j}
			\right|^{2}\\
			&=
			\sum_{i=1}^{\infty}a_{i}^{2}|u_{i}|^{2}.
		\end{aligned}
		\]
	It follows that
		\[
		\begin{aligned}
			&\hphantom{{}={}}\!\left(
			\sum_{j=1}^{\infty}
			a_{j}^{2}\left| \partial_{x_j} F(\mathbf{x}) \right|^{2}
			\right)\exp\left(\frac{4r\left( \sum_{i=1}^{\infty} b_{i}y_{i}-r \right)}
				{p\sum_{i=1}^{\infty} a_{i}^{2}b_{i}^{2}}\right)
			\\
			&=
			\sum_{i=1}^{\infty}
			a_{i}^{2}
			\left|
			\partial_{y_i} f(\mathbf{y})
			-\frac{2rb_{i}}
			{p\sum_{j=1}^{\infty} a_{j}^{2}b_{j}^{2}}
			f(\mathbf{y})
			\right|^{2}\\
			&\le
			2\sum_{i=1}^{\infty}
			a_{i}^{2}
			\left|
			\partial_{y_i} f(\mathbf{y})
			\right|^{2}
			+
			2\sum_{i=1}^{\infty}
			a_{i}^{2}
			\left|
			\frac{2rb_{i}}
			{p\sum_{j=1}^{\infty} a_{j}^{2}b_{j}^{2}}
			f(\mathbf{y})
			\right|^{2}\\
			&=
			2\sum_{i=1}^{\infty}
			a_{i}^{2}
			\left|
			\partial_{y_i} f(\mathbf{y})
			\right|^{2}
			+
			\frac{8r^{2}}
			{p^{2}\left(
				\sum_{i=1}^{\infty}a_{i}^{2}b_{i}^{2}
				\right)}
			\left|
			f\left(\mathbf{y}\right)
			\right|^{2}.
		\end{aligned}
		\]
	Similarly to \eqref{eq:changevariable}, we obtain
	$$\begin{aligned}
		&\hphantom{{}={}}\!\int_{\ell^{2} \setminus \overline{O}} \left( \sum_{i=1}^{\infty} \left| a_{i}\partial_{i} F(\mathbf{x}) \right|^{2} \right)^{\frac{p}{2}} \mathrm{d} P(\mathbf{x}) \\
		&=\int_{O} \left[ \left( \sum_{j=1}^{\infty} a_{j}^{2}\left| \partial_{x_j} F(\mathbf{x}) \right|^{2} \right)\middle|_{\mathbf{x}=\mathbf{y}-\frac{2\left(\sum_{k=1}^{\infty}b_{k}y_{k}-r\right)}
		{\sum_{k=1}^{\infty}a_{k}^{2}b_{k}^{2}}\mathbf{h}}\cdot\exp\left(\frac{4r\left( \sum_{i=1}^{\infty} b_{i}y_{i}-r \right)}{p\sum_{i=1}^{\infty} a_{i}^{2}b_{i}^{2}}\right) \right]^{\frac{p}{2}} \mathrm{d} P(\mathbf{y})\\
		&\le \int_{O} \left( 2\sum_{i=1}^{\infty} a_{i}^{2}\left| \partial_{x_i} f \right|^{2} +\frac{8r^{2}}{p^{2}\left( \sum_{i=1}^{\infty} a_{i}^{2}b_{i}^{2} \right)} \left| f \right|^{2} \right)^{\frac{p}{2}} \mathrm{d} P\\
		&\le C\left( \int_{O} \left( \sum_{i=1}^{\infty} a_{i}^{2}\left| \partial_{x_i} f \right|^{2} \right)^{\frac{p}{2}} \mathrm{d} P+\int_{O} \left\vert f \right\vert^{p} \mathrm{d} P \right) .\end{aligned}$$
	Therefore, there exists a constant $\widetilde{C}>0$, independent of $f$, such that
	$$\left\| Ef \right\|_{\mathscr{W}^{1,p}(\ell^{2},P)} \le \widetilde{C}\left\| f \right\|_{\mathscr{W}^{1,p}(O,P)},\quad \forall f\in\mathscr{C}_{b}^{\infty}.$$
	Now by the density of $\mathscr{C}_b^\infty$ in $\mathscr{W}^{1,p}({O},P)$, we obtain the desired operator $E$.
\end{proof}

\begin{example}
	Fix $m\in\mathbb{N}$, and put $H_{m}\triangleq\ell^{2}\left(\mathbb{N}\setminus\{m\}\right).$
	Let $\Phi\in C_{b}^{1}\left(H_{m},\mathbb{R}\right).$
	Then
	\[
	{O}_{\Phi}
	\triangleq
	\left\{(x_i)_{i\in\dbN}\in\ell^{2}:x_{m}>\Phi\left((x_j)_{j\neq m}\right)\right\}
	\]
	is a Sobolev reflection domain.
	In particular, if
	$\boldsymbol{\beta}=(\beta_{j})_{j\neq m}\in H_{m}$
	has infinitely many nonzero components and $M>0$, then
	\[
	\Phi\left((x_j)_{j\neq m}\right)=
	M\tanh\left(\sum_{j\neq m}\beta_{j}\sin x_{j}\right)
	\]
	gives an example.
\end{example}

To study the extension properties of $W^{1,p}(O,P)$, we have the following definition.
\begin{definition}\label{20260918def1}
	We call $O \subset \ell^{2}$ a \emph{Sobolev reflection domain of type II} if $\mathscr C_b^{\infty}$ is dense in $W^{1,1}(O, P)$, $P(\partial O)=0$, and there exist a Borel mapping
	\[
	S:\ell^{2} \rightarrow \ell^2,\qquad
	\mathbf{x}=(x_i)_{i\in\dbN}\mapsto (S_i(\mathbf{x}))_{i\in\dbN}
	\]
	and a constant $C>0$ such that
	\begin{enumerate}[label=(\roman*)]
		\item $\mathbf0\in O,\qquad
S\mathbf x=\mathbf x\quad(\mathbf x\in\overline O),\qquad
S(\ell^2\setminus\overline O)=O\setminus\{\mathbf0\};$
\item for every $i,j\in\mathbb N$ and $ \mathbf x^{\widehat j}=(x_i)_{i\in\mathbb N\setminus \{j\}}\in\ell^2(\mathbb N\setminus\{j\})$ , the mapping defined by
$$
\widehat{S}_i(x_j)\triangleq S_i(\mathbf x),\qquad x_j\in\mathbb R,
$$
where $\mathbf x=(x_k)_{k\in\mathbb N},$ is continuous and piecewise $C^1$;
\item for $P$-almost every
$\mathbf x\in\ell^2\setminus\overline O$, any $n\in\mathbb{N}$, and any $\xi_1,\xi_2,\dots,\xi_n\in\mathbb{R}$, one has
	\begin{equation*}
		\sum_{i=1}^{\infty}a_i
		\left| \sum_{j=1}^{n}\partial_{x_i}S_j(\mathbf{x})\xi_j\right|
		\le C\sum_{i=1}^{n}a_i|\xi_i|;
	\end{equation*}
	\item for every nonnegative Borel function $h$ on $O$, it holds that
	\begin{equation*}
	 \int_{\ell^{2} \setminus \overline{ O}}   h(S\mathbf{x} )   \, \mathrm{d} P(\mathbf x)\le C\int_{O} h\, \mathrm{d} P.
	\end{equation*}
	\end{enumerate}
	

\end{definition}
Then by the same proof as in the following Section~\ref{subsubsec:W11}, we have the following theorem.

\begin{theorem}\label{20260918thm1}
If $O \subset \ell^{2}$ is a Sobolev reflection domain of type II, then $O$ is a $W^{1,1}$-extension domain.
\end{theorem}

\subsection{The Hilbert polydisc}\label{subsec:Hilbert polydisc}
Now we focus on the open Hilbert polydisc $\mathbb{D}_2^{\infty}$ as in \eqref{eq:Hilbert polydisc}. We will prove the following results.

\begin{theorem}\label{thm:HS-main}
Let $m\ge1$ and $1\le p<\infty$.  There do not exist $C\in(0,\infty)$ and an operator
\[
 \mathscr E:\mathscr{W}^{m,p} (\mathbb{D}_2^{\infty},P)\longrightarrow \mathscr{W}^{m,p} (\ell^2, P)
\]
such that $\mathscr Ef=f$ almost everywhere on $\mathbb{D}_2^{\infty}$ and
$\|\mathscr Ef\|_{\mathscr{W}^{m,p} (\ell^2,P)}\le C\|f\|_{\mathscr{W}^{m,p} (\mathbb{D}_2^{\infty},P)}$ for every $f\in\mathscr{W}^{m,p} (\mathbb{D}_2^{\infty},P)$.
\end{theorem}

\begin{theorem}\label{thm:co-main}
Let $m\ge1$ and $1<p<\infty$.
There do not exist $C\in(0,\infty)$ and an operator
\[
 E:W^{m,p} (\mathbb{D}_2^{\infty}, P)\longrightarrow W^{m,p} (\ell^2, P)
\]
such that $Ef=f$ almost everywhere on $\mathbb{D}_2^{\infty}$ and
$\|Ef\|_{W^{m,p} (\ell^2, P)}\le C
\|f\|_{W^{m,p} (\mathbb{D}_2^{\infty}, P)}$ for every $f\in W^{m,p} (\mathbb{D}_2^{\infty}, P)$.
\end{theorem}

\subsubsection{Preparatory results}\label{subsubsec:pre}
Let
\[
 R_0=(-1,0)\times(0,1),\qquad \Gamma_0=\{0\}\times(0,1),
\]
and write
\[
\eta_p\triangleq\inf\left\{\int_{R_0} \left(\sum_{i=1}^{2}|\partial_{x_i} w |^2\right)^{\frac{p}{2}}\dd x_1\dd x_2:
 w\in W^{1,p}(R_0),\,\mathrm{Tr}_{0}w=e^{-x_2}\right\},
\]
where $\mathrm{Tr}_{0}w\in L^p(\Gamma_0)$ is the trace of $w$ on $\Gamma_0$.

\begin{lemma}\label{lem:eta}
For every $p\in[1,\infty)$, we have $0<\eta_p<\infty$.
\end{lemma}
\begin{proof}
The function $w(x_1,x_2)=e^{x_1-x_2}$ satisfies $\mathrm{Tr}_{0}w=e^{-x_2}$, so $\eta_p<\infty$.
If $\eta_p=0$, choose $\{w_n\}_{n=1}^{\infty}\subset W^{1,p}(R_0)$ with
\[\mathrm{Tr}_{0}w_n=e^{-x_2},\qquad \lim_{n\to\infty}\int_{R_0} \left(\sum_{i=1}^{2}|\partial_{x_i} w_n |^2\right)^{\frac{p}{2}}\dd x_1\dd x_2=0,\]
and let
\[c_n\triangleq\int_{R_0}w_n(x_1,x_2)\dd x_1\dd x_2.\]
The Poincar\'e inequality on the bounded connected rectangle $R_0$  gives $\|w_n-c_n\|_{L^p(R_0)}\to0$.  Then the continuity of the
trace operator $W^{1,p}(R_0)\to L^p(\Gamma_0)$ gives
$\|e^{-x_2}-c_n\|_{L^p(0,1)}\to0$, which is impossible because the constants form a closed
one-dimensional subspace of $L^p(0,1)$ and $e^{- x_2}$ is not constant.
\end{proof}

Let $d\in\mathbb N$ and $\Omega_d\triangleq(-\infty,0)\times(0,\infty)^d$. We write each point of $\Omega_d$ as $(t,\mathbf y)$, where
$t<0$ and $\mathbf y=(y_1,\ldots,y_d)\in(0,\infty)^d$. Put
\[
 g_d(\mathbf y)\triangleq e^{-\sum_{i=1}^d y_i},\qquad \mathbf f_d\triangleq \frac{1}{\sqrt d}(1,\dots,1),\qquad D_{\mathbf f_d}v\triangleq\frac{1}{\sqrt d}\sum_{i=1}^d\partial_{y_i}v\quad(v\in W^{1,p}_{\mathrm{loc}}(\Omega_d)),
\]
where the derivatives are understood in the weak sense.
The diagonal slicing argument gives the following estimate.

\begin{proposition}\label{prop:diagHS}
If $v\in W^{1,p}_{\mathrm{loc}}(\Omega_d)$ has trace $g_d$ on $\{0\}\times(0,\infty)^d$, then
\[
 \int_{\Omega_d}\big(|\partial_{t} v(t, \mathbf y)|^2+|D_{\mathbf{f}_d}v(t, \mathbf y)|^2\big)^{p/2}\dd t\dd \mathbf y
 \ge \eta_p \,d^{(p-1)/2}p^{1-d}.
\]
\end{proposition}
\begin{proof}
Write $\mathbf y=\mathbf z+s\mathbf{f}_d\in (0,\infty)^d$ with $\mathbf z=(z_1,\dots,z_d)$ and $\mathbf z\perp \mathbf{f}_d$, and let
$s_0(\mathbf z)\triangleq\max\limits_{i=1,\ldots,d}(-\sqrt d\,z_i)$. Then it holds that $z_i+sd^{-1/2}>0$ for every $i=1,\ldots,d$, which is equivalent to $s>s_0(\mathbf z)$. Since $\mathbf z\perp \mathbf{f}_d$, we have $\sum_{i=1}^d z_i=0$ and $g_d(\mathbf y)=e^{-\sqrt{d} s}$. 

Consider the rectangle
$$
R_{\mathbf z}\triangleq\left\{(t,s)\in \mathbb{R}^2:-\frac{1}{\sqrt{d}}<t<0,\,s_0(\mathbf z)<s<s_0(\mathbf z)+\frac{1}{\sqrt{d}}\right\},
$$
which will be transformed into $(-1,0)\times(0,1)$ under the transformation
$x_1=\sqrt{d} t$, $x_2=\sqrt{d}(s-s_0(\mathbf z))$. For $\mathbf z\perp \mathbf{f}_d$, define
$$
w_{\mathbf z}(x_1,x_2)\triangleq e^{\sqrt{d} s_0(\mathbf z)}v\left(\frac{x_1}{\sqrt{d}},\mathbf z+\left(\frac{x_2}{\sqrt{d}}+s_0(\mathbf z)\right)\mathbf{f}_d\right)
,\qquad (x_1,x_2)\in (-1,0)\times(0,1).
$$
We may assume the left-hand side of the asserted inequality is finite.
Then for almost every $\mathbf z \in\mathbf{f}_d^{\perp}$, we have $w_{\mathbf z}\in W^{1,p}(R_0)$ and
\[\mathrm{Tr}_{0}w_{\mathbf z}=e^{\sqrt{d} s_0(\mathbf z)}v\left(0,\mathbf z+\left(\frac{x_2}{\sqrt{d}}+s_0(\mathbf z)\right)\mathbf{f}_d\right)=e^{-x_2};\]
hence Lemma \ref{lem:eta} yields
$$
\int_{R_0}\left( \sum_{i=1}^{2}|\partial_{x_i} w_{\mathbf z} |^2\right)^{\frac{p}{2}}\dd x_1\dd x_2\geq \eta_p>0.
$$
Let $u_{\mathbf z}(t,s)\triangleq v(t,\mathbf z+s\mathbf{f}_d)$ for $t\in(-\infty,0)$ and $s>s_0(\mathbf z)$. Then
$$
\partial_{t}u_{\mathbf z}(t,s)=\partial_{t}v(t,\mathbf z+s\mathbf{f}_d),\qquad \partial_su_{\mathbf z}(t,s)=D_{\mathbf{f}_d}v(t,\mathbf z+s\mathbf{f}_d).
$$
Hence, under the transformation
$x_1=\sqrt{d} t$, $x_2=\sqrt{d}(s-s_0(\mathbf z))$, we have
$$
\partial_{t}v(t,\mathbf z+s\mathbf{f}_d)=e^{-\sqrt{d} s_0(\mathbf z)}\sqrt{d} \partial_{x_1}w_{\mathbf z}(x_1,x_2),
\qquad
D_{\mathbf{f}_d}v(t,\mathbf z+s\mathbf{f}_d)=e^{-\sqrt{d} s_0(\mathbf z)}\sqrt{d} \partial_{x_2}w_{\mathbf z}(x_1,x_2),
$$
and
$$
\big(|\partial_{t} v(t, \mathbf y)|^2+|D_{\mathbf{f}_d}v(t, \mathbf y)|^2\big)^{p/2}
=e^{-p\sqrt{d} s_0(\mathbf z)}d^{p/2}\big(|\partial_{x_1}w_{\mathbf z}(x_1,x_2)|^2+|\partial_{x_2}w_{\mathbf z}(x_1,x_2)|^2\big)^{p/2}.
$$
Combining this with $\mathrm{d} t\dd s=d^{-1}\dd x_1\dd x_2$,
we have
\[\begin{aligned}
&\int_{\Omega_d}\left(|\partial_{t} v(t, \mathbf y)|^2+|D_{\mathbf{f}_d}v(t, \mathbf y)|^2\right)^{p/2}\dd t\dd \mathbf y \\
={}&\int_{ \mathbf{f}_d^{\perp}}\int_{-\infty}^{0}\int^{\infty}_{s_0(\mathbf z)}\left(|\partial_{t} v(t,\mathbf z+s\mathbf{f}_d)|^2+|D_{\mathbf{f}_d}v(t, \mathbf z+s\mathbf{f}_d)|^2\right)^{p/2}\dd s\dd t\dd \mathbf z\\
\geq {}&\int_{ \mathbf{f}_d^{\perp}}\int_{ R_{\mathbf z}}\left(|\partial_{t} v(t,\mathbf z+s\mathbf{f}_d)|^2+|D_{\mathbf{f}_d}v(t, \mathbf z+s\mathbf{f}_d)|^2\right)^{p/2}\dd t\dd s\dd \mathbf z\\
= {}&d^{p/2-1}\int_{ \mathbf{f}_d^{\perp}}e^{-p\sqrt{d} s_0(\mathbf z)}\int_{  R_0}\left(|\partial_{x_1}w_{\mathbf z}(x_1,x_2)|^2+|\partial_{x_2}w_{\mathbf z}(x_1,x_2)|^2\right)^{p/2}\dd x_1\dd x_2\dd \mathbf z\\
\geq {}&d^{p/2-1}\eta_p\int_{ \mathbf{f}_d^{\perp}}e^{-p\sqrt{d} s_0(\mathbf z)} \dd \mathbf z=\eta_p \,d^{(p-1)/2}p^{1-d},
\end{aligned}\]
 where the last equality follows from
\[
 p^{-d}=\int_{(0,\infty)^d}e^{-p\sum_{i=1}^{d} y_i}\dd \mathbf y=\int_{ \mathbf{f}_d^{\perp}}\int^{\infty}_{s_0(\mathbf z)}e^{-p\sqrt{d} s}
 \dd s\dd \mathbf z=\frac{1}{p\sqrt{d}}\int_{\mathbf{f}_d^\perp}e^{-p\sqrt{d} s_0(\mathbf z)}\dd \mathbf z.
\]
This completes the proof.
\end{proof}

Define
\[
 c_p^{(2)}\triangleq 2^{\max\{p/2-1,0\}}.
\]

\begin{lemma}\label{lem:coordface}
Under the assumptions of Proposition \ref{prop:diagHS},
\[
 \int_{\Omega_d}\left(|\partial_{t} v(t, \mathbf y)|^p+\sum_{i=1}^d|\partial_{y_i}v(t, \mathbf y)|^p\right)\dd t\dd \mathbf y
 \ge \kappa_p\,d^{\theta_p}p^{-d},
\]
where
\[
 \kappa_p\triangleq \frac{\eta_p p}{c_p^{(2)}}>0,
 \qquad
 \theta_p\triangleq\min\left\{\frac{p-1}{2},\frac{1}{2}\right\}.
\]
\end{lemma}
\begin{proof}
	By Lemma~\ref{nver1}, for every $p\in[1,\infty)$,
\[
|D_{\mathbf{f}_d}v|^p=\left|\frac{1}{\sqrt{d}}\sum_{i=1}^d\partial_{y_i}v\right|^p
\le d^{p/2-1}\sum_{i=1}^d|\partial_{y_i}v|^p.
\]
Hence using the inequality $(a^2+b^2)^{p/2}\le c_p^{(2)}(a^p+b^p)$ for $a,b\ge0$, we have
\[
\bigl(|\partial_t v|^2+|D_{\mathbf{f}_d}v|^2\bigr)^{p/2}
\le c_p^{(2)}\max\{1,d^{p/2-1}\}
\left(|\partial_t v|^p+\sum_{i=1}^d|\partial_{y_i}v|^p\right).
\]
Combining this with Proposition \ref{prop:diagHS} gives
\[
 \int_{\Omega_d}\left(|\partial_{t} v(t, \mathbf y)|^p+\sum_{i=1}^d|\partial_{y_i}v(t, \mathbf y)|^p\right)\dd t\dd \mathbf y
 \ge \frac{\eta_pp}{c_p^{(2)}}
 d^{(p-1)/2}\min\{1,d^{1-p/2}\}p^{-d},
\]
where the exponent on $d$ is $(p-1)/2$ for $1\le p\le2$ and $1/2$ for $p\ge2$.
\end{proof}

The higher-order coordinate estimate is obtained by applying Lemma \ref{lem:coordface}.

\begin{proposition}\label{prop:coordhigher}
Let $m,d\in\dbN$ and $1\le p<\infty$. Suppose $F=\widetilde F|_{\Omega_d}$, where $\widetilde F\in W_{\mathrm{loc}}^{m,p}((-\infty,a)\times(0,\infty)^d)$ for some $a>0$ such that
\[
\widetilde F(t,\mathbf y)=e^{-t-\sum_{i=1}^d y_i},\qquad (t,\mathbf y)\in(0,a)\times(0,\infty)^d.
\]
Then
\[
 \sum_{|\alpha|=m}\int_{\Omega_d}|\partial^\alpha F(t,\mathbf y)|^p\dd t\dd \mathbf y
 \ge \frac{\kappa_p}{m}
 \binom{d+m-2}{m-1}d^{\theta_p}p^{-d}.
\]
Here $\alpha\in\mathbb{N}_0^{d+1}$ is taken in the variables $(t,\mathbf y)$.
\end{proposition}
\begin{proof}
For each $\beta\in\mathbb N_0^d$ with $|\beta|=m-1$, set
$W_\beta=\partial_{\mathbf y}^\beta F$. Then the trace of $W_\beta$ at $t=0$ is $(-1)^{m-1}g_d.$
By Lemma \ref{lem:coordface},
\[
 \int_{\Omega_d}\left(|\partial_{t}W_\beta(t,\mathbf y)|^p+
 \sum_{i=1}^d|\partial_{y_i}W_\beta(t,\mathbf y)|^p\right)\dd t\dd \mathbf y
 \ge \kappa_p d^{\theta_p}p^{-d}.
\]
There are $\binom{d+m-2}{m-1}$ such $\beta$.

Write $\alpha=(\alpha_0,\alpha_1,\ldots,\alpha_d)$, where $\alpha_0$ corresponds to $t$. Then
\[
\sum_{|\beta|=m-1}\left(|\partial_tW_\beta|^p+\sum_{i=1}^d|\partial_{y_i}W_\beta|^p\right)
=\sum_{|\alpha|=m}n_\alpha|\partial^\alpha F|^p,\]
where
\[
m\ge n_\alpha=
\begin{cases}
\#\{i\in\{1,\ldots,d\}:\alpha_i>0\},&\alpha_0=0,\\
1,&\alpha_0=1,\\
0,&\alpha_0\ge2.
\end{cases}
\]
Therefore
\[
 \sum_{\beta\in\mathbb N_0^d,\,|\beta|=m-1}\int\left(|\partial_tW_\beta(t,\mathbf y)|^p+
 \sum_{i=1}^{d}|\partial_{y_i}W_\beta(t,\mathbf y)|^p\right)\dd t\dd \mathbf y
 \le m\sum_{|\alpha|=m}\int|\partial^\alpha F(t,\mathbf y)|^p\dd t\dd \mathbf y,
\]
which proves the assertion.
\end{proof}

Fix $N\ge2$, and let
\[
 P_N\triangleq\prod_{i=1}^N\mathcal N_{a_i},\qquad
 \mathbb{D}^N_2\triangleq(-1,1)^N,\qquad \rho(\mathbf x)\triangleq\prod_{i=1}^N\frac1{\sqrt{2\pi}a_i}e^{-\frac{x_i^2}{2a_i^2}}.
\]
Set $\mathbf c\triangleq(1,\dots,1)$, $\mathsf{D}_a\triangleq\operatorname{diag}(a_1,\dots,a_N)$, and
$\lambda_N \triangleq \prod_{i=1}^N a_i$.  For $r>0$, use the change of variables
\[
 \mathbf x=\mathbf c-r\mathsf{D}_a\mathbf y.
\]
Then $\mathbb{D}^N_2$ becomes
\[
 Q_{N,r}\triangleq\prod_{i=1}^N\left(0,\frac2{ra_i}\right)
\]
and
\[
w_r(\mathbf y) \triangleq \frac{\rho(\mathbf c-r\mathsf{D}_a\mathbf y)}{\rho(\mathbf c)}=\exp\left(r\sum_{i=1}^N\frac{y_i}{a_i}-\frac{r^2}{2}\|\mathbf y\|^2\right).
\]
For later counting set
\[
 M_{N,k}\triangleq\#\{\alpha\in\mathbb N_0^N:|\alpha|=k\}
 =\binom{N+k-1}{k}.
\]

Define $u_r$ on $\mathbb{D}^N_2$ by
\[
 u_r(\mathbf c-r\mathsf{D}_a\mathbf y)\triangleq g_N(\mathbf y)= e^{-\sum_{i=1}^Ny_i}.
\]
Every ordinary derivative of $g_N$ is $\pm g_N$.  Moreover
\[
 \mathbf{a}^\alpha \partial_{\mathbf x}^\alpha u_r(\mathbf c-r\mathsf{D}_a\mathbf y)=(-1)^{|\alpha|}r^{-|\alpha|}\partial_{\mathbf y}^\alpha g_N(\mathbf y)
 =\pm r^{-|\alpha|}g_N(\mathbf y).
\]
Consequently
\[
 \|u_r\|_{W^{m,p} (\mathbb{D}^N_2, P_N)}^p
 =\rho(\mathbf c)\lambda_Nr^{N-mp}J_{r,p}
 \sum_{k=0}^mM_{N,k}r^{(m-k)p},
\]
where
\[
 J_{r,p}\triangleq\int_{Q_{N,r}}e^{-p\sum_i y_i}w_r(\mathbf y)\dd \mathbf y.
\]

\begin{lemma}\label{lem:J}
For fixed $N$ and $1\le p<\infty$, $J_{r,p}\to p^{-N}$ as $r\downarrow0$.  Therefore
\[
 \frac{\|u_r\|_{W^{m,p}(\mathbb{D}^N_2, P_N)}^p}
 {\rho(\mathbf c)\lambda_Nr^{N-mp}}
 \longrightarrow M_{N,m}p^{-N}.
\]
\end{lemma}
\begin{proof}
For every $\mathbf y\in(0,\infty)^N$,
\[
\mathbf{1}_{Q_{N,r}}(\mathbf y)e^{-p\sum_i y_i}w_r(\mathbf y)\longrightarrow e^{-p\sum_i y_i}.
\]
If $0<r\le\frac p2\min_{1\le i\le N}a_i$, then
\[
0\le\mathbf1_{Q_{N,r}}(\mathbf y)e^{-p\sum_i y_i}w_r(\mathbf y)
\le e^{-\sum_i(p-r/a_i)y_i}\le e^{-(p/2)\sum_i y_i}.
\]
The dominated convergence theorem gives $J_{r,p}\to p^{-N}$.
Moreover,
\[
\sum_{k=0}^mM_{N,k}r^{(m-k)p}\longrightarrow M_{N,m},
\]
which proves the assertion.
\end{proof}

Let $T:W^{m,p} (\mathbb{D}^N_2, P_N)\to W^{m,p} (\dbR^N, P_N)$ be any extension map and
set $\tilde{u}_r=Tu_r$, $v_r(\mathbf y)=\tilde{u}_r(\mathbf c-r\mathsf{D}_a\mathbf y)$.  For $j=1,\dots,N$ define
\[
 U_j\triangleq\{\mathbf x=(x_1,\ldots,x_N):x_j>1,\ -1<x_i<1\ \text{for }i\in\{1,\ldots,N\}\setminus\{ j\}\}.
\]
These sets are pairwise disjoint.  In $\mathbf y$ variables,
\[
 \Omega_{j,r}\triangleq\left\{\mathbf y=(y_1,\ldots,y_N):y_j<0,\ 0<y_i<\frac2{ra_i}\ \text{for }i\in\{1,\ldots,N\}\setminus\{ j\}\right\}.
\]
For every $\alpha\in\mathbb{N}_0^{N}$ with $|\alpha|=m$,
\[
 |\mathbf{a}^\alpha \partial_{\mathbf x}^\alpha \tilde{u}_r(\mathbf c-r\mathsf{D}_a\mathbf y)|^p
 =r^{-mp}|\partial_{\mathbf y}^\alpha v_r(\mathbf y)|^p,
\]
and
$$
\dd P_N(\mathbf x)=\rho(\mathbf c)\lambda_Nr^{N}w_r(\mathbf y)\dd \mathbf y.
$$
Hence
\[
 \frac1{\rho(\mathbf c)\lambda_Nr^{N-mp}}
 \sum_{|\alpha|=m}\int_{U_j}|\mathbf{a}^\alpha \partial_{\mathbf x}^\alpha \tilde{u}_r|^p\dd P_N
 =\int_{\Omega_{j,r}}\sum_{|\alpha|=m}|\partial_{\mathbf y}^\alpha v_r|^p w_r(\mathbf y)\dd \mathbf y.
\]
Because the Gaussian density is locally bounded above and below away from zero,
$\tilde{u}_r\in W^{m,p}_{\mathrm{loc}}(\dbR^N)$ in the ordinary finite-dimensional sense.
Since $v_r\in W_{\mathrm{loc}}^{m,p}(\mathbb R^N)$ and $v_r=g_N$ on $Q_{N,r}$,
\[
\operatorname{Tr}^-\partial_{\mathbf y}^{\alpha}v_r=
\operatorname{Tr}^+\partial_{\mathbf y}^{\alpha}v_r
=(-1)^{|\alpha|}e^{-\sum_{i\ne j}y_i},\qquad |\alpha|\le m-1.
\]
Here $\operatorname{Tr}^+$ (\resp $\operatorname{Tr}^-$) denotes the trace
from $y_j>0$ (\resp $y_j<0$) to $y_j=0$, respectively.

\begin{lemma} \label{lem:slab}
Let $d=N-1$.  For every $j$,
\[
 \liminf_{r\downarrow0}
 \int_{\Omega_{j,r}}\sum_{|\alpha|=m}|\partial_{\mathbf y}^\alpha v_r|^p w_r(\mathbf y)\dd \mathbf y
 \ge \frac{\kappa_p}{m}
 \binom{N+m-3}{m-1}(N-1)^{\theta_p}p^{-(N-1)}.
\]
\end{lemma}
\begin{proof}
After relabelling, we may assume $j=N$. Write
$\mathbf y=(\mathbf y',t)$ with $\mathbf y'=(y_1,\ldots,y_{N-1})$ and $t=y_N$.
Let $s_0(\mathbf z)$ be as in the proof of Proposition~\ref{prop:diagHS} and set
\[
G_r(\mathbf y)\triangleq\sum_{|\alpha|=m}|\partial_{\mathbf y}^\alpha v_r(\mathbf y)|^p.
\]
For $R>0$, define
\[
\mathcal Z_R\triangleq\{\mathbf z\in\mathbf{f}_d^\perp:\|\mathbf z\|<R\},
\]
\[
E_R\triangleq\left\{
(\mathbf z+s\mathbf{f}_d,t):
\mathbf z\in\mathcal Z_R,\;
-d^{-1/2}<t<0,\;s_0(\mathbf z)<s<s_0(\mathbf z)+d^{-1/2}\right\}.
\]
Since
\[
E_R\subset (0,2R+d^{-1})^d\times(-d^{-1/2},0),
\]
for each fixed $R$ and all sufficiently small $r$, we have
\[
E_R\subset \Omega_{N,r},\qquad
\inf_{E_R}w_r\longrightarrow1 \quad\text{as }r\downarrow0.
\]

For every $\beta\in\mathbb N_0^d$ with $|\beta|=m-1$,
the equality $v_r=g_N$ a.e. on $Q_{N,r}$ gives
\[
\operatorname{Tr}_{t=0}\partial_{\mathbf y'}^\beta v_r
=(-1)^{m-1}g_d.
\]
By the proof of Proposition~\ref{prop:coordhigher},
\[
\int_{E_R}G_r\dd\mathbf y
\ge C_R\triangleq
\frac{\kappa_p}{mp}
\binom{d+m-2}{m-1}
\frac{d^{(p-2)/2}}{\max\{1,d^{p/2-1}\}}
\int_{\mathcal Z_R}
e^{-p\sqrt{d} s_0(\mathbf z)}\dd\mathbf z.
\]
Hence
\[
\int_{\Omega_{N,r}}G_rw_r\dd\mathbf y
\ge
\left(\inf_{E_R}w_r\right)C_R,
\qquad
\liminf_{r\downarrow0}
\int_{\Omega_{N,r}}G_rw_r\dd\mathbf y\ge C_R.
\]
Finally, since $\mathbf{1}_{\mathcal Z_R}\to\mathbf{1}_{\mathbf{f}_d^\perp}$ and
\[
\int_{\mathbf{f}_d^\perp}
e^{-p\sqrt{d} s_0(\mathbf z)}\dd\mathbf z
=\sqrt{d} p^{1-d},
\]
the monotone convergence theorem yields
\[
C_R\uparrow
\frac{\kappa_p}{m}
\binom{d+m-2}{m-1}d^{\theta_p}p^{-d}.
\]
Letting $R\to\infty$ and recalling $d=N-1$ proves the result.
\end{proof}

\begin{theorem} \label{thm:cofinite}
Fix $N\ge2$, $m\ge1$, and $p\in[1,\infty)$.  Suppose
\[
 T:W^{m,p} (\mathbb{D}^N_2, P_N)\longrightarrow W^{m,p} (\dbR^N, P_N)
\]
satisfies $Tu=u$ a.e. on $\mathbb{D}^N_2$ and
\[
 \|Tu\|_{W^{m,p} (\dbR^N, P_N)}
 \le C_{N,m,p}\|u\|_{W^{m,p}(\mathbb{D}^N_2, P_N)}
\]
for every $u$.  Then
\begin{equation}\label{20260913for1}
 (C_{N,m,p})^p\ge
 1+\frac{\eta_p p^2}{c_p^{(2)}}\frac{N(N-1)}{(N+m-1)(N+m-2)}(N-1)^{\theta_p}
\end{equation}
with $\theta_p\triangleq\min\left\{\frac{p-1}{2},\frac{1}{2}\right\}$.
\end{theorem}
\begin{proof}
By Lemmas \ref{lem:J} and \ref{lem:slab},
\begin{equation*}
 \liminf_{r\downarrow0}\frac{\|\tilde{u}_r\|_{W^{m,p} (\dbR^N, P_N)}^p}
 {\rho(\mathbf c)\lambda_Nr^{N-mp}}
 \ge M_{N,m}p^{-N}+N\frac{\kappa_p}{m}
 \binom{N+m-3}{m-1}(N-1)^{\theta_p}p^{-(N-1)}.
\end{equation*}
On the other hand, the extension bound and Lemma~\ref{lem:J} give
\[
\limsup_{r\downarrow0}
\frac{\|\tilde{u}_r\|_{W^{m,p}(\mathbb R^N,P_N)}^p}
{\rho(\mathbf c)\lambda_Nr^{N-mp}}
\le C_{N,m,p}^pM_{N,m}p^{-N}.
\]
Thus
\[
 (C_{N,m,p})^p\ge1+
 \frac{\kappa_ppN}{m}
 \frac{\binom{N+m-3}{m-1}}{\binom{N+m-1}{m}}
 (N-1)^{\theta_p}.
\]
The elementary identity
\[
 \frac{N}{m}
 \frac{\binom{N+m-3}{m-1}}{\binom{N+m-1}{m}}
 =\frac{N(N-1)}{(N+m-1)(N+m-2)}
\]
and $\kappa_pp=\eta_p p^2/c_p^{(2)}$ yield $\eqref{20260913for1}$.
\end{proof}

Write
\[
 P=P_N\times P^{\widehat{1,\ldots,N}},\qquad
 \mathbb{D}^\infty_2=\mathbb{D}^N_2\times\Pi^{N},\]
\[ \Pi^{N}=\{\mathbf y=(y_j)_{j\in\mathbb N\setminus\{1,\ldots,N\}}\in\ell^2(\mathbb{N}\setminus\{1,\ldots,N\}):|y_j|<1,\;\forall j>N\}.
\]
Set $q_N\triangleq  P^{\widehat{1,\ldots,N}}(\Pi^N)$. By \cite[Proposition~1.25, p.~21]{Da06} or \cite[Lemma 2.1, p. 9]{WYZ}, we have $q_N>0$.

For $u\in W^{m,p} (\mathbb{D}^N_2, P_N)$ define
\[
 (L_Nu)(\mathbf x)\triangleq u(x_1,\dots,x_N),\qquad \mathbf x=(x_i)_{i\in\mathbb N}\in\mathbb{D}^\infty_2.
\]
By Definition~\ref{def:Sob1} and Fubini's theorem,
\begin{equation}\label{eq:lift-co}
 \|L_Nu\|_{W^{m,p} (\mathbb{D}^\infty_2, P)}^p
 =q_N\|u\|_{W^{m,p} (\mathbb{D}^N_2, P_N)}^p.
\end{equation}

For $f\in L^1(\ell^2,P)$, define
\[
 (\mathcal{A}_Nf)(\mathbf z)\triangleq\frac{1}{q_N}\int_{\Pi^{N}}f(\mathbf z,\mathbf y)\dd P^{\widehat{1,\ldots,N}}(\mathbf y),
 \qquad \mathbf z\in\dbR^N.
\]
By a similar argument to the proof of \eqref{zhouwenhui}, we have
\begin{equation*}
 \partial^\alpha(\mathcal{A}_NF)=\mathcal{A}_N(\partial^\alpha F)
\end{equation*}
for any $F\in W^{m,p}(\ell^2,P)$ and $\alpha\in\mathbb{N}_0^{(\mathbb{N})}$ with $|\alpha|\le m$ and $\alpha_j=0$ for $j>N$.

\begin{lemma}\label{lem:avg-co}
For every $p\in[1,\infty)$, 
\[
 \|\mathcal{A}_NF\|_{W^{m,p} (\dbR^N, P_N)}\le q_N^{-1/p}\|F\|_{W^{m,p} (\ell^2, P)}.
\]
\end{lemma}
\begin{proof}
 For any $\mathbf z\in\dbR^N$, H\"older's inequality gives
\[\begin{aligned}
 |\mathbf{a}^\alpha \partial^\alpha(\mathcal{A}_NF)(\mathbf z)|^p
 &=\left| \frac{1}{q_N}\int_{\Pi^{N}}\mathbf{a}^\alpha \partial^\alpha F(\mathbf z,\mathbf y)\dd P^{\widehat{1,\ldots,N}}(\mathbf y)\right|^p\\
 &\le\frac{1}{q_N}\int_{\Pi^{N}}|\mathbf{a}^\alpha \partial^\alpha F(\mathbf z,\mathbf y)|^p\dd P^{\widehat{1,\ldots,N}}(\mathbf y),
\end{aligned}\]
from which the result follows.
\end{proof}

\begin{proposition} \label{prop:reverse-co}
Suppose there exist $C\in(0,\infty)$ and an extension map
\[
 E:W^{m,p} (\mathbb{D}^\infty_2, P)\to W^{m,p}(\ell^2, P)
\]
with $Ef=f$ on $\mathbb{D}^\infty_2$ and $\|Ef\|_{W^{m,p}(\ell^2, P)}\le C\|f\|_{W^{m,p} (\mathbb{D}^\infty_2, P)}$.  Then for every $N$ there exists
\[
 T_N:W^{m,p} (\mathbb{D}^N_2, P_N)\to W^{m,p}(\dbR^N, P_N)
\]
with $T_Nu=u$ on $\mathbb{D}^N_2$ and $\|T_Nu\|_{W^{m,p}(\dbR^N, P_N)}\le C\|u\|_{W^{m,p} (\mathbb{D}^N_2, P_N)}$.
\end{proposition}
\begin{proof}
Set $T_Nu\triangleq \mathcal{A}_N(E(L_Nu))$.  Lemma \ref{lem:avg-co} and \eqref{eq:lift-co} give
\[
 \|T_Nu\|_{W^{m,p}(\dbR^N, P_N)}\le q_N^{-1/p}C\|L_Nu\|_{W^{m,p} (\mathbb{D}^\infty_2, P)}=C\|u\|_{W^{m,p} (\mathbb{D}^N_2, P_N)}.
\]
Since $E(L_Nu)=L_Nu$ almost everywhere on
$\mathbb{D}^N_2\times\Pi^{N}$, Fubini's theorem implies $T_Nu(\mathbf z)=u(\mathbf z)$ for almost every $\mathbf z\in\mathbb{D}^N_2$.  
\end{proof}

\subsubsection{Proof of non-extendability}
\begin{proof}[Proof of Theorem \ref{thm:co-main}]
Assume a bounded extension map $E$ exists with constant $C\in(0,\infty)$.  By Proposition~\ref{prop:reverse-co} and Theorem~\ref{thm:cofinite},
\[
 C^p\ge 1+\frac{\eta_p p^2}{c_p^{(2)}}\frac{N(N-1)}{(N+m-1)(N+m-2)}(N-1)^{\theta_p}
\]
for every $N\ge 2$.  Since $p>1$, $\theta_p>0$, and the right-hand side tends to $+\infty$.  This is a
contradiction.
\end{proof}

\begin{proof}[Proof of Theorem \ref{thm:HS-main}]
Suppose that $\mathscr E$ exists with bound $C$.
By Fubini's theorem and Jensen's inequality, the
$\mathscr W^{m,p}$ analogues of \eqref{eq:lift-co} and
Lemma~\ref{lem:avg-co} hold. Thus, for every $N\ge2$,
$T_N\triangleq\mathcal A_N\mathscr E L_N$ satisfies
\[
T_Nu=u\quad\text{a.e. on }\mathbb D_2^N,\qquad
\|T_Nu\|_{\mathscr W^{m,p}(\mathbb R^N,P_N)}
\le C\|u\|_{\mathscr W^{m,p}(\mathbb D_2^N,P_N)}.
\]

Use $u_r$, $\tilde{u}_r=T_Nu_r$ and
$v_r(\mathbf y)=\tilde{u}_r(\mathbf c-r\mathsf D_a\mathbf y)$
as in Section~\ref{subsubsec:pre}.
Lemma~\ref{lem:J} gives
\[
\frac{\|u_r\|_{\mathscr W^{m,p}(\mathbb D_2^N,P_N)}}
     {(\rho(\mathbf c)\lambda_Nr^{N-mp})^{1/p}}
=
J_{r,p}^{1/p}
\left[
r^m+
\left(\sum_{k=1}^mM_{N,k}r^{2(m-k)}\right)^{1/2}
\right]
\longrightarrow M_{N,m}^{1/2}p^{-N/p}.
\]
By the proofs of Proposition~\ref{prop:diagHS} and Lemma~\ref{lem:slab},
\[
\liminf_{r\downarrow0}
\int_{\Omega_{j,r}}\left(
\sum_{|\alpha|=m}
|\partial_{\mathbf y}^{\alpha}v_r(\mathbf y)|^2
\right)^{p/2}w_r(\mathbf{y})\dd\mathbf y
\ge
\frac{\eta_p}{(m!)^{p/2}}
(N-1)^{(mp-1)/2}p^{2-N}.
\]
Hence by an argument similar to the proof of Theorem~\ref{thm:cofinite}, we have
\[
C^p\ge
1+\frac{\eta_pp^2N(N-1)^{(mp-1)/2}}
        {(m!)^{p/2}M_{N,m}^{p/2}}.
\]
Finally, for fixed $m$ and $p$,
\[
M_{N,m}\sim\frac{N^m}{m!},\qquad
\frac{N(N-1)^{(mp-1)/2}}
     {(m!)^{p/2}M_{N,m}^{p/2}}
\sim\sqrt N\longrightarrow\infty.
\]
This contradicts $C<\infty$ for every $p\in[1,\infty)$.
\end{proof}

\subsubsection{$W^{1,1}$ extension}\label{subsubsec:W11}
Throughout this subsection, we assume $A\triangleq\sum_{j=1}^{\infty}a_j<\infty$ for simplicity.
Set $\log0\triangleq -\infty$ and define $\tau(\mathbf x)\triangleq\sup\limits_{j\in\mathbb N}\frac{\log|x_j|}{a_j}$ for $\mathbf x=(x_i)_{i\in\mathbb N}\in\ell^2$.
\begin{lemma} \label{lem:gauge}
If $\mathbf x\ne\mathbf 0$, then $\tau(\mathbf x)$ is finite and attained by some $j\in\dbN$; $\tau(\mathbf 0)=-\infty$. Moreover: (i) $\mathbf x\in \mathbb{D}^\infty_2$ if and only if $\tau(\mathbf x)<0$; (ii) $\mathbf x\in\partial \mathbb{D}^\infty_2$ if and only if $\tau(\mathbf x)=0$; (iii)  $\mathbf x\in\ell^2\setminus\overline{\mathbb{D}^\infty_2}$ if and only if $\tau(\mathbf x)>0$; (iv) $P(\partial \mathbb{D}^\infty_2)=0$, and
$$
P\left(\bigcup_{i\neq j}\left\{\mathbf x=(x_j)_{j\in\mathbb N}\in\ell^2:\frac{\log|x_i|}{a_i}=\frac{\log|x_j|}{a_j}=\tau(\mathbf x)\right\}\right)=0.
$$
\end{lemma}
\begin{proof}
If $\mathbf x=(x_j)_{j\in\mathbb N}\ne\mathbf 0$, then $x_j\to0$, so $|x_j|\le1/2$ for all sufficiently large $j$. Since $a_j\to0$, $\log|x_j|/a_j\le-\log2/a_j\to-\infty$. Thus the supremum is a finite maximum over finitely many coordinates. The first three assertions follow. Since
$$
\partial \mathbb{D}^\infty_2\subset\bigcup_{j=1}^{\infty}\{\mathbf x=(x_j)_{j\in\mathbb N}\in\ell^2:|x_j|=1\},
$$
we have $P(\partial \mathbb{D}^\infty_2)=0$. For any $i\ne j$,
\begin{equation*}\begin{aligned}
&P\left(\{\mathbf x=(x_j)_{j\in\mathbb N}\in\ell^2:\log|x_i|/a_i=\log|x_j|/a_j=\tau(\mathbf x)\}\right)\\
\leq{} &P\left(\{\mathbf x=(x_j)_{j\in\mathbb N}\in\ell^2:\log|x_i|/a_i=\log|x_j|/a_j \}\right)=0,
\end{aligned}\end{equation*}
 and there are countably many pairs. The last assertion follows.
\end{proof}
For $\mathbf x=(x_j)_{j\in\mathbb N}\in\ell^2\setminus\{\mathbf 0\}$ and $i\in\mathbb N$, define $R_i(\mathbf x)\triangleq e^{-2a_i\tau(\mathbf x)}x_i$, $R(\mathbf x)\triangleq (R_j(\mathbf x))_{j\in\mathbb N}$, and $R(\mathbf 0)\triangleq\mathbf 0$. Define
\[S(\mathbf x)\triangleq\begin{cases}\mathbf x,&\tau(\mathbf x)\le0,\\R(\mathbf x),&\tau(\mathbf x)>0,\end{cases}\]
and $S_i(\mathbf x)\triangleq \langle S(\mathbf x),\mathbf e_i\rangle,\,i\in\mathbb N$.
\begin{lemma} \label{lem:geom}
If $\tau(\mathbf x)>0$, then $R(\mathbf x)\in \mathbb{D}^\infty_2$ and $\tau(R\mathbf x)=-\tau(\mathbf x)$, $R(R\mathbf x)=\mathbf x$. More generally, $R$ is an involution on $\ell^2\setminus\{\mathbf 0\}$.
\end{lemma}
\begin{proof}
Let $t=\tau(\mathbf x)>0$. Then for every $i\in\mathbb N$, we have $|x_i|\le e^{a_it}$, with equality for some $i=k$. Hence $|R_i(\mathbf x)|=e^{-2a_it}|x_i|\le e^{-a_it}<1$. Also
\[\frac{\log|R_i(\mathbf x)|}{a_i}=\frac{\log|x_i|}{a_i}-2t\le-t,\]
with equality at $i=k$, so $\tau(R\mathbf x)=-t$. Hence $R(R\mathbf x)=\mathbf x$. Applying the same formula again gives $R(R\mathbf y)=\mathbf y$ for every $\mathbf y\in \ell^2\setminus\{\mathbf 0\}$.
\end{proof}
Define
\[\tau^N(\mathbf x)\triangleq \max_{ j =1,\ldots, N}\frac{\log|x_j|}{a_j} ,\qquad N\in\mathbb N,\, \mathbf x=(x_j)_{j\in\mathbb N}\in\ell^2,\]
and let $S^{(N)}\mathbf x\triangleq \mathbf x$ if $\tau^N(\mathbf x)\le0$, while for $\tau^N(\mathbf x)>0$ set
\[( S^{(N)}\mathbf x )_i \triangleq\begin{cases}e^{-2a_i\tau^N(\mathbf x)}x_i,&i\le N,\\x_i,&i>N,\end{cases}\]
and $S^{(N)}\mathbf x\triangleq\Big(( S^{(N)}\mathbf x  )_i\Big)_{i\in\mathbb N}\in\ell^2$.
\begin{lemma} \label{lem:finiteconv}
For every $\mathbf x\in\ell^2$, $S^{(N)}\mathbf x\to S\mathbf x$ in $\ell^2$.
\end{lemma}
\begin{proof}
By Lemma~\ref{lem:gauge},
$\tau^N(\mathbf x)=\tau(\mathbf x)$ for all sufficiently large $N$.
Put $t=\max\{\tau(\mathbf x),0\}$. Then
\[
\|S^{(N)}\mathbf x-S\mathbf x\|^2
=\sum_{i>N}(1-e^{-2a_it})^2|x_i|^2\le\sum_{i>N}|x_i|^2\longrightarrow0,
\]
which completes the proof.
\end{proof}

Set $c_0\triangleq 4\sqrt{\frac{2}{3}}$.
\begin{lemma} \label{lem:sinh}
For every $u>0$, $\sinh(2u)/u^2\ge c_0$. Consequently, for all $a,t>0$,
\[\frac{\sinh(2at)}{a^2}\ge c_0t^2.\]
\end{lemma}
\begin{proof}
The positive-term series gives $\sinh(2u)\ge2u+(2u)^3/3!=2u+(4/3)u^3$. Hence
\[\frac{\sinh(2u)}{u^2}\ge\frac2u+\frac43u\ge2\sqrt{\frac83}=4\sqrt{\frac23},\]
which completes the proof.
\end{proof}
Fix $N$ and write $A_N=\sum_{i=1}^Na_i$. For $\mathbf x=(x_1,\ldots,x_N)\in\mathbb R^N$, set
\[\tau_N(\mathbf x)\triangleq \max_{1\le i\le N}\frac{\log|x_i|}{a_i},\qquad (R_N\mathbf x)_i\triangleq e^{-2a_i\tau_N(\mathbf x)}x_i,\qquad R_N\mathbf x\triangleq \Big((R_N\mathbf x)_i\Big)_{i\in\{1,\ldots, N\}},\]
with the convention $R_N\mathbf0=\mathbf0$. By arguments similar to the proofs of Lemmas \ref{lem:gauge} and \ref{lem:geom}, we obtain the following properties.
\begin{lemma} \label{20260906lem1}
If $\mathbf x\in \mathbb R^N\setminus\{\mathbf 0\}$, then $\tau_N(\mathbf x)$ is finite; $\tau_N(\mathbf 0)=-\infty$. Moreover: (i) $\mathbf x\in  \mathbb{D}^N_2$ if and only if $\tau_N(\mathbf x)<0$; (ii) $\mathbf x\in\partial  \mathbb{D}^N_2$ if and only if $\tau_N(\mathbf x)=0$; (iii) $\mathbf x\in\mathbb R^N\setminus  \overline{\mathbb{D}^N_2}$ if and only if $\tau_N(\mathbf x)>0$; (iv) $P_N(\partial  \mathbb{D}^N_2)=0$, and
$$
P_N\left(\bigcup_{i\neq j}\left\{\mathbf x=(x_1,\ldots,x_N)\in\mathbb R^N:\frac{\log|x_i|}{a_i}=\frac{\log|x_j|}{a_j}=\tau_N(\mathbf x)\right\}\right)=0.
$$
 If $\tau_N(\mathbf x)>0$, then $R_N\mathbf x\in  \mathbb{D}^N_2$ and $\tau_N(R_N\mathbf x)=-\tau_N(\mathbf x)$, $R_N(R_N\mathbf x)=\mathbf x$. More generally, $R_N$ is an involution on $\mathbb R^N\setminus\{\mathbf 0\}$.
\end{lemma}
\begin{proposition} \label{prop:finite}
For every nonnegative Borel function $h$ on $ \mathbb{D}^N_2$, it holds that
\[\int_{\mathbb R^N\setminus  \mathbb{D}^N_2}h(R_N\mathbf x)\,\mathrm d P_N(\mathbf x)\le \exp\!\left(\frac{A_N^2}{c_0}\right)\int_{  \mathbb{D}^N_2}h(\mathbf y)\,\mathrm dP_N(\mathbf y).\]

\end{proposition}
\begin{proof}
For any $k\in\{1,\ldots,N\}$ and sign $\varepsilon\in\{\pm1\}$, define
\begin{eqnarray*}\label{20260906for1}
E_{k,\varepsilon}\triangleq \left\{\mathbf x=(x_1,\ldots,x_N)\in \mathbb R^N\setminus\{\mathbf 0\}:\frac{\log|x_k|}{a_k}=\tau_N(\mathbf x)>\max_{i\in \{1,\ldots, N\}\setminus\{k\}}\frac{\log|x_i|}{a_i},\quad x_k\varepsilon>0\right\},
\end{eqnarray*}
which is obviously an open subset of $\mathbb R^N$.

Fix $k\in\{1,\ldots,N\}$ and $\varepsilon\in\{\pm1\}$. For any $t\in\mathbb R$ and $\mathbf z=(z_i)_{i\in \{1,\ldots,N\}\setminus \{k\}}\in(-1,1)^{N-1}$, set $z_k=\varepsilon$ and define
\[\Psi_{k,\varepsilon}(t,\mathbf z)_i\triangleq e^{a_it}z_i,\qquad i=1,\ldots,N,\qquad \Psi_{k,\varepsilon}(t,\mathbf z)\triangleq (\Psi_{k,\varepsilon}(t,\mathbf z)_1,\ldots,\Psi_{k,\varepsilon}(t,\mathbf z)_N) .\]
Then $\Psi_{k,\varepsilon}$ is a homeomorphism from $\mathbb R\times (-1,1)^{N-1}$ onto $E_{k,\varepsilon}$ such that
\[\Psi_{k,\varepsilon}\left([0,\infty)\times (-1,1)^{N-1}\right)=E_{k,\varepsilon}\cap \left(\mathbb R^N\setminus \mathbb{D}^N_2\right),\quad \Psi_{k,\varepsilon}\left((-\infty,0)\times (-1,1)^{N-1}\right)=E_{k,\varepsilon}\cap \mathbb{D}^N_2,\]
\[\tau_N(\Psi_{k,\varepsilon}(t,\mathbf z))=t,\quad R_N\Psi_{k,\varepsilon}(t,\mathbf z)=\Psi_{k,\varepsilon}(-t,\mathbf z).\]
The absolute Jacobian of the transformation $(t,\mathbf z )\mapsto\Psi_{k,\varepsilon}(t,\mathbf z)$ is $a_ke^{A_Nt}$. Let
\[W_{k,N}(t,\mathbf z)\triangleq \frac{a_ke^{A_Nt}}{(2\pi)^{N/2}\lambda_N}\exp\!\left(-\frac{1}{2}\sum_{i=1}^N\frac{e^{2a_it}z_i^2}{a_i^2}\right),\quad  t\in\mathbb R,\;\mathbf z=(z_i)_{i\in \{1,\ldots,N\}\setminus \{k\}}\in(-1,1)^{N-1}.\]
For $t>0$, we have
\[\frac{W_{k,N}(t,\mathbf z)}{W_{k,N}(-t,\mathbf z)}
=\exp\!\left(2A_Nt-\sum_{i=1}^N\frac{\sinh(2a_it)}{a_i^2}z_i^2\right)
\le\exp(2A_Nt-c_0t^2)\le\exp\!\left(\frac{A_N^2}{c_0}\right),\]
where the first inequality follows from $z_k^2=1$ and Lemma~\ref{lem:sinh}, and the last inequality is from
$$2A_Nt-c_0t^2 -\frac{A_N^2}{c_0}=-c_0\left(t-\frac{A_N}{c_0}\right)^2\leq 0.$$
Therefore, by Lemma \ref{20260906lem1}, we have
\[\begin{aligned}
&\int_{\mathbb R^N\setminus  \mathbb{D}^N_2}h(R_N\mathbf x)\,\mathrm d P_N(\mathbf x) \\
={}&\sum_{k\in\{1,\ldots,N\},\,\varepsilon\in\{\pm1\}}\int_{E_{k,\varepsilon}\cap \left(\mathbb R^N\setminus  \mathbb{D}^N_2\right)}h(R_N\mathbf x)\,\mathrm d P_N(\mathbf x)\\
={}&\sum_{k\in\{1,\ldots,N\},\,\varepsilon\in\{\pm1\}}\int_{[0,\infty)\times (-1,1)^{N-1}}h(R_N\Psi_{k,\varepsilon}(t,\mathbf z))W_{k,N}(t,\mathbf z)\,\mathrm d t\,\mathrm d \mathbf z\\
={}&\sum_{k\in\{1,\ldots,N\},\,\varepsilon\in\{\pm1\}}\int_{[0,\infty)\times (-1,1)^{N-1}}h( \Psi_{k,\varepsilon}(-t,\mathbf z))\frac{W_{k,N}(t,\mathbf z)}{W_{k,N}(-t,\mathbf z)}W_{k,N}(-t,\mathbf z)\,\mathrm d t\,\mathrm d \mathbf z\\
\leq{}&\exp\!\left(\frac{A_N^2}{c_0}\right)\sum_{k\in\{1,\ldots,N\},\,\varepsilon\in\{\pm1\}}\int_{[0,\infty)\times (-1,1)^{N-1}}h( \Psi_{k,\varepsilon}(-t,\mathbf z)) W_{k,N}(-t,\mathbf z)\,\mathrm d t\,\mathrm d \mathbf z\\
={}&\exp\!\left(\frac{A_N^2}{c_0}\right)\sum_{k\in\{1,\ldots,N\},\,\varepsilon\in\{\pm1\}}\int_{E_{k,\varepsilon}\cap  \mathbb{D}^N_2}h(\mathbf y)\,\mathrm d P_N(\mathbf y)\\
={}&\exp\!\left(\frac{A_N^2}{c_0}\right) \int_{  \mathbb{D}^N_2}h(\mathbf y)\,\mathrm d P_N(\mathbf y).
\end{aligned}\]
This completes the proof.
\end{proof}

\begin{theorem} \label{thm:push}
Let
\[\Lambda_A\triangleq \exp\!\left(\frac{A^2}{c_0}\right).\]
For every nonnegative $h\in C_b(\ell^2)$, it holds that
\[\int_{\ell^2\setminus \overline{\mathbb{D}^\infty_2}}h(R\mathbf x)\,\mathrm d P(\mathbf x)\le\Lambda_A\int_{\mathbb{D}^\infty_2}h(\mathbf y)\,\mathrm dP(\mathbf y).\]
\end{theorem}
\begin{proof}
Set
\[
D_N\triangleq
\{\mathbf x\in\ell^2:|x_i|<1,\ 1\le i\le N\}=\mathbb{D}^N_2\times\ell^2(\dbN\setminus\{1,\dots,N\}).
\]
By Fubini's theorem and Proposition~\ref{prop:finite},
\[
\int_{\ell^2\setminus D_N}h(S^{(N)}\mathbf x)\dd P(\mathbf x)
\le e^{A_N^2/c_0}\int_{D_N}h\dd P
\le\Lambda_A\int_{D_N}h\dd P.
\]
Then it follows from Lemma~\ref{lem:finiteconv}, $P(\partial\mathbb D_2^\infty)=0,$ and the dominated convergence theorem that
\[
\int_{\ell^2\setminus \overline{\mathbb{D}^\infty_2}}h(R\mathbf x)\dd P(\mathbf x)
\le\Lambda_A\int_{\mathbb D_2^\infty}h\dd P,
\]
which proves the assertion.
\end{proof}

Let $f\in\mathscr C_b^{\infty}$, say $f(\mathbf y)=\varphi(y_1,\ldots,y_n)$ for every $\mathbf y=(y_j)_{j\in\mathbb N}\in\ell^2$ with $n\in\mathbb N$ and $\varphi\in C^\infty_b(\mathbb R^n)$. Define
\[(E_0f)(\mathbf x)\triangleq \varphi(S_1(\mathbf x),\ldots,S_n(\mathbf x)),\qquad \mathbf x\in\ell^2.\]
Since for every $\mathbf x\in \mathbb{D}^\infty_2$, $S\mathbf x=\mathbf x$, $E_0f=f$ on $\mathbb{D}^\infty_2$. Define
\[G_f(\mathbf y)\triangleq\sum_{i=1}^na_i|\partial_{x_i}\varphi(y_1,\ldots,y_n)|,\qquad \mathbf y=(y_j)_{j\in\mathbb N}\in\ell^2.\]
\begin{lemma} \label{lem:chain}
Let $\mathbf x=(x_i)_{i\in\mathbb N}\in \ell^2$ and $k\in\mathbb N$. If $\tau(\mathbf x)>0$ and
$$
\frac{\log|x_k|}{a_k}=\tau (\mathbf x)>\max_{i\in \mathbb N\setminus\{k\}}\frac{\log|x_i|}{a_i},
$$
then
\[\partial_{x_j}\tau(\mathbf x)=\frac{\delta_{jk}}{a_kx_k},\qquad
\partial_{x_j}S_i(\mathbf x)=e^{-2a_i\tau(\mathbf x)}\bigl(\delta_{ij}-2a_ix_i\partial_{x_j}\tau(\mathbf x)\bigr),\qquad j\in\mathbb N,\]
and
\begin{equation}\label{20260917for1}
\sum_{j=1}^{\infty}a_j|\partial_{x_j}(E_0f)(\mathbf x)|\le3G_f(R\mathbf x).
\end{equation}
If $\tau(\mathbf x)<0$, then
\begin{equation}\label{20260917for2}
\sum_{j=1}^{\infty}a_j|\partial_{x_j}(E_0f)(\mathbf x)|=G_f(\mathbf x).
\end{equation}
\end{lemma}
\begin{proof}
Since
$$
\frac{\log|x_k|}{a_k}=\tau (\mathbf x)>\max_{i\in \mathbb N\setminus\{k\}}\frac{\log|x_i|}{a_i},
$$
holds locally, the first two formulas follow from the ordinary chain rule.
\[\begin{aligned}
\sum_{j=1}^{\infty}a_j|\partial_{x_j}(E_0f)(\mathbf x)|
&=\sum_{j=1}^{\infty}a_j\left|\sum_{i=1}^{n}\partial_{x_i}\varphi(S_1(\mathbf x),\ldots,S_n(\mathbf x))\partial_{x_j}S_i(\mathbf x)\right|\\
&=\sum_{j=1}^{\infty}a_j\left|\sum_{i=1}^{n}\partial_{x_i}\varphi(S_1(\mathbf x),\ldots,S_n(\mathbf x))e^{-2a_i\tau(\mathbf x)}\left(\delta_{ij}-2a_ix_i\frac{\delta_{jk}}{a_kx_k}\right)\right|\\
&\leq \sum_{i=1}^{n}a_i |\partial_{x_i}\varphi(S_1(\mathbf x),\ldots,S_n(\mathbf x))|\\
&\hphantom{{}={}}+2\sum_{i=1}^{n}e^{-2a_i\tau(\mathbf x)}\frac{|x_i|}{|x_k|} a_i |\partial_{x_i}\varphi(S_1(\mathbf x),\ldots,S_n(\mathbf x))|.
\end{aligned}\]
For every $i=1,\ldots,n, $ since $|x_i|\le e^{a_i\tau(\mathbf x)}$ and $|x_k|=e^{a_k\tau(\mathbf x)}$,
\[e^{-2a_i\tau(\mathbf x)}\frac{|x_i|}{|x_k|}\le e^{-(a_i+a_k)\tau(\mathbf x)}\le1.\]
Thus we obtain \eqref{20260917for1}.

If $\tau(\mathbf x)<0$, then $S (\mathbf x)=\mathbf x $ and hence \eqref{20260917for2} holds.
\end{proof}

\begin{lemma} \label{lem:ac}
For every $j\in\mathbb N$ and $ \mathbf x^{\widehat j}=(x_i)_{i\in\mathbb N\setminus \{j\}}\in\ell^2(\mathbb N\setminus\{j\})$ ,  the map
\[x_j\mapsto E_0f(x_1,\ldots,x_{j-1},x_j,x_{j+1},\ldots),\qquad x_j\in\mathbb R,\]
is absolutely continuous on every bounded interval. Moreover, the ordinary coordinate derivative $\partial_{x_j}(E_0f)$ is given $P$-almost everywhere
by Lemma~\ref{lem:chain} and coincides with the weak derivative.
\end{lemma}
\begin{proof}
Put
\[c\triangleq\max\left\{0,\sup_{i\in\mathbb N\setminus \{j\}}\frac{\log|x_i|}{a_i}\right\}.\]

For every $i=1,\ldots,n,$ let
$$
\widehat{S}_i(x_j)\triangleq S_i(\mathbf x),
$$
where $\mathbf x=(x_k)_{k\in\mathbb N},\,x_j\in\mathbb R.$

Then
\[
\widehat{S}_i(x_j) =
\begin{cases}
e^{-2a_ic}x_i,&|x_j|\leq e^{a_jc},\\
 |x_j|^{-\frac{2a_i}{a_j}}x_i,&|x_j|>e^{a_jc}.
\end{cases}
\]
These functions are continuous and piecewise $C^1$,
so $E_0f$ is absolutely continuous on bounded coordinate intervals.
Since $E_0f$ and its coordinate derivatives are bounded, we can apply
Lemma~\ref{20260917slices}, after relabelling coordinates,
to identify the coordinate derivatives with the weak derivatives.
\end{proof}
\begin{theorem} \label{thm:core}
	Assume $A=\sum_{j=1}^\infty a_j<\infty$. Then $\mathbb{D}^\infty_2$ is a $W^{1,1}$-extension domain.
\end{theorem}
\begin{proof}
	Let $f\in\mathscr C_b^{\infty}$.
It follows from Theorem~\ref{thm:push} that
\[\int_{\ell^2}|E_0f|\,\mathrm d P\le(1+\Lambda_A)\int_{\mathbb{D}^\infty_2}|f|\,\mathrm d P.\]
By Theorem~\ref{thm:push}, Lemmas~\ref{lem:chain}--\ref{lem:ac} and Tonelli's theorem,
\[\sum_{j=1}^{\infty}a_j\int_{\ell^2}|\partial_{x_j}(E_0f)|\,\mathrm d P
\le\int_{\mathbb{D}^\infty_2}G_f\,\mathrm d P+3\int_{\ell^2\setminus\overline{\mathbb{D}^\infty_2}}G_f(R\mathbf x)\,\mathrm d P(\mathbf x)
\le(1+3\Lambda_A)\int_{\mathbb{D}^\infty_2}G_f\,\mathrm d P.\]
Hence for every $f\in\mathscr C_b^{\infty}$,
\[\|E_0f\|_{W^{1,1}(\ell^2,P)}\le(1+3\Lambda_A)\|f\|_{W^{1,1}(\mathbb{D}^\infty_2,P)},\]
and the result follows from Theorem~\ref{zhuhanshuguanghuabijin} combined with Example~\ref{ex:Hilbert polydisc}.
\end{proof}

\begin{corollary}
For every $N\in\mathbb N$, the same reflection construction gives
\[E_N:W^{1,1}(\mathbb D_2^N,P_N)\to W^{1,1}(\mathbb R^N,P_N),\qquad
\|E_N\|\le1+3\exp(A_N^2/c_0)\le1+3\exp(A^2/c_0).\]
\end{corollary}

\section{Trace properties}\label{sec:tr}
In this section, we study the trace properties of the domain $O$.

\subsection{Domains with uniformly regular boundaries}

The following definition is motivated by \cite[Definition~13.11]{Leo17}.
\begin{definition}\label{def:unc1}
   We say $\partial O$ is \emph{uniformly of class $C^1$} if there exist constants $L>0$, $m\in\dbN$, and an open cover $\{U_i\}_{i\in\L}$ of $\partial O$ with $\L\subset\dbN$ satisfying the following conditions:
   \begin{enumerate}[label=(\arabic*)]
      \item For each $i\in\L$, there exist $i^*\in\dbN$, $r_i>0$, $\varepsilon_i\in\{1,-1\}$, an open subset $V_i$ of $\ell^2(\mathbb{N}\setminus\{i^*\})$, and a function $\psi_i\in C_{\ell^2(\dbN\setminus\{i^*\})}^1(V_i)$ such that
      \begin{enumerate}[label=(\arabic{enumi}\alph*)]
         \item $\min\{a_{i^*},r_i\}\ge L^{-1}$;
         \item $U_i=\{\mathbf{x}=(x_j)_{j\in\dbN}\in\ell^2:\mathbf{x}'\triangleq(x_k)_{k\neq i^*}\in V_i,\,|\varepsilon_ix_{i^*}-\psi_i(\mathbf{x}')|<r_i\}$, and
         \[U_i\cap O=\{\mathbf{x}\in U_i:\varepsilon_ix_{i^*}>\psi_i(\mathbf{x}')\},\qquad U_i\cap\partial O=\{\mathbf{x}\in U_i:\varepsilon_ix_{i^*}=\psi_i(\mathbf{x}')\};\]
         \item for every $\mathbf{x}'\in V_i$,
		 \[r_i\psi_i(\mathbf{x}')\le L,\qquad \sqrt{\sum_{k\neq i^*}|\partial_{x_k}\psi_i(\mathbf{x}')|^2}\le L.\]
      \end{enumerate}
      \item No point of $O$ is contained in more than $m$ members of $\{U_i\}_{i\in\L}$.
   \end{enumerate}
\end{definition}

We first give some examples of $O$ with $\partial O$ uniformly of class $C^1$.

\begin{example}\label{ex:aff}
Let $\mathbf{b}=(b_i)_{i\in\dbN}\in\ell^2\setminus\{\mathbf{0}\}$ with finite support and let $r\in\dbR$. Then the boundary of
\[O=\left\{(x_j)_{j\in\dbN}\in\ell^2:\sum_{j=1}^{\infty}b_jx_j>r\right\}\]
is uniformly of class $C^1$. Indeed, for each $i\in\supp(\mathbf{b})$, set $i^*=i$, $r_i=1$, $\varepsilon_i=\sgn(b_i)$, and
\[\begin{aligned}
   V_i&=\left\{(x_k)_{k\neq i}\in\ell^2(\dbN\setminus\{i\}):\psi_i((x_k)_{k\neq i})\triangleq|b_i|^{-1}\left(r-\sum_{k\neq i}b_kx_k\right)<1+\frac{r}{\sum_{j=1}^{\infty}|b_j|}\right\},\\
   U_i&=\{\mathbf{x}=(x_j)_{j\in\dbN}\in\ell^2:\mathbf{x}'\triangleq(x_k)_{k\neq i}\in V_i,\,|\varepsilon_ix_{i}-\psi_i(\mathbf{x}')|<1\}.
\end{aligned}
\]
Clearly Definition~\ref{def:unc1} (1b) is satisfied. For any $(x_j)_{j\in\dbN}\in\partial O$, we have
\[\sum_{i\in\supp(\mathbf{b})}|b_i|\varepsilon_ix_i=r,\]
which implies
\[\varepsilon_ix_i=\psi_i((x_k)_{k\neq i})\le\frac{r}{\sum_{j=1}^{\infty}|b_j|}<1+\frac{r}{\sum_{j=1}^{\infty}|b_j|}\]
for some $i\in\supp(\mathbf{b})$. Thus $\{U_i\}_{i\in\supp(\mathbf{b})}$ covers $\partial O$.
One can take $m=|\supp(\mathbf{b})|$,
\[L=\max\left\{1,\,1+\frac{r}{\sum_{j=1}^{\infty}|b_j|},\,\max_{i\in\supp(\mathbf{b})}a_i^{-1},\max_{i\in\supp(\mathbf{b})}|b_i|^{-1}\|(b_k)_{k\neq i}\|_{\ell^2(\dbN\setminus\{i\})}\right\},\]
and verify the remaining conditions in Definition~\ref{def:unc1} readily.
\end{example}

\begin{remark}
   The trace theorem for arbitrary half-spaces will be proved by Proposition~\ref{prop:linear} and Theorem~\ref{thm:tr}.
\end{remark}

\begin{example}
   Fix $i_0\in\dbN$ and choose $\psi\in C_{\ell^2(\dbN\setminus\{i_0\})}^1
(\ell^2(\dbN\setminus\{i_0\}))$ such that
   \[\sup_{\ell^2(\dbN\setminus\{i_0\})}\psi<\infty,\qquad\sup_{\ell^2(\dbN\setminus\{i_0\})}\sqrt{\sum_{k\neq i_0}|\partial_{x_k}\psi_i|^2}<\infty.\]
   Then for any $\varepsilon_0\in\{1,-1\}$, the boundary of
   \[O=\{(x_j)_{j\in\dbN}\in\ell^2:\varepsilon_0 x_{i_0}>\psi((x_k)_{k\neq i_0})\}\]
   is uniformly of class $C^1$. Indeed, one can choose
   \[L=\max\left\{1,\,\sup_{\ell^2(\dbN\setminus\{i_0\})}\psi,\,\sup_{\ell^2(\dbN\setminus\{i_0\})}\sqrt{\sum_{k\neq i_0}|\partial_{x_k}\psi_i|^2},\,a_{i_0}^{-1}\right\},\qquad m=1,\]
   and $\{U_i\}_{i\in\L}=\{U\}$, where
   \[U=\{\mathbf{x}=(x_j)_{j\in\dbN}\in\ell^2:\mathbf{x}'\triangleq(x_k)_{k\neq i_0}\in \ell^2(\dbN\setminus\{i_0\}),\,|\varepsilon_0x_{i_0}-\psi(\mathbf{x}')|<1\}.\]
   A concrete example is
   \[O=\left\{(x_i)_{i\in\dbN}\in\ell^2:x_1>\sum_{j=2}^{\infty}2^{-j}\sin x_j\right\}.\]
\end{example}

To establish the trace theorem for domains whose boundaries
are uniformly of class $C^1$, we need the following lemma.

\begin{lemma}\label{lem:1dim}
   For any $1\le p<\infty$, there exists a constant $C_p>0$ such that for every $a,r>0$ and $u\in C^1((-r,r))$,
\[\frac{1}{\sqrt{2\pi a^2}}\,|u(0)|^p\le \frac{C_p}{\min\{a,r\}}\int_{0}^{r}\left(|u(t)|^p+ a^p|u'(t)|^p\right)\dd \mathcal{N}_a(t).\]
\end{lemma}
\begin{proof}
   By the fundamental theorem of calculus, for $0\le t<b\triangleq\min\{a,r\}$,
   \[u(0)=u(t)-\int_{0}^{t}u'(s)\dd s,\]
   and thus
   \begin{equation}\label{eq:ftc}
   |u(0)|=\frac{1}{b}\int_{0}^{b}\left|u(t)-\int_{0}^{t}u'(s)\dd s\right|\mathrm{d} t\le \frac{1}{b}\int_{0}^{b}|u(t)|\dd t+\int_{0}^{b}|u'(t)|\dd t.
   \end{equation}
   Applying Lemma~\ref{nver1} and H\"older's inequality to \eqref{eq:ftc} gives
   \[\begin{aligned}
   |u(0)|^p&\le 2^{p-1}\left(\frac{1}{b}\int_{0}^{b}|u(t)|\dd t\right)^p+2^{p-1}\left(\int_{0}^{b}|u'(t)|\dd t\right)^p\\
   &\le 2^{p-1}\left(\frac{1}{b^p}\cdot b^{p-1}\int_{0}^{b}|u(t)|^p\dd t+b^{p-1}\int_{0}^{b}|u'(t)|^p\dd t\right)\\
   &=2^{p-1}\left(\frac{1}{b}\int_{0}^{b}|u(t)|^p\dd t+b^{p-1}\int_{0}^{b}|u'(t)|^p\dd t\right).
   \end{aligned}\]
   Combining this with $t<b\le a$, we have
   \[\begin{aligned}
      \frac{1}{\sqrt{2\pi a^2}}\,|u(0)|^p&\le2^{p-1}e^{\frac{1}{2}}\left(\frac{1}{b}\int_{0}^{b}|u(t)|^p\frac{1}{\sqrt{2\pi a^2}}\,e^{-\frac{t^2}{2a^2}}\dd t+b^{p-1}\int_{0}^{b}|u'(t)|^p\frac{1}{\sqrt{2\pi a^2}}\,e^{-\frac{t^2}{2a^2}}\dd t\right)\\
      &\le \frac{2^{p-1}e^{\frac{1}{2}}}{b}\int_{0}^{r}\left(|u(t)|^p+ a^p|u'(t)|^p\right)\dd \mathcal{N}_a(t),
   \end{aligned}\]
   which completes the proof.
\end{proof}

Now we prove our first trace theorem.

\begin{proof}[Proof of Theorem~\ref{thm:trlfb}]
	Let the notation be as in Definition~\ref{def:unc1}.
For any $f\in\mathscr{C}_b^\infty$ and $i\in\L$, we write $\mathbf{x}=(x_j)_{j\in\dbN}$ and $\mathbf{x}'=(x_k)_{i\neq i^*}\in\ell^2(\dbN\setminus\{i^*\})$. We may assume $\varepsilon_i=1$ after changing the variable $\widetilde{x}_{i^*}=-x_{i^*}$ if necessary. Set $t=x_{i^*}-\psi_i(\mathbf{x}')$ and $g(t,\mathbf{x}')=f(\mathbf{x})$. Then $\partial_tg(t,\mathbf{x}')=\partial_{x_{i^*}}f(\mathbf{x})$. By Definition~\ref{20241012def2} and Definition~\ref{def:unc1} (1c),
\begin{equation}\label{eq:intbd}
\begin{aligned}
\int_{U_i \cap \partial O}|f|^p\dd\mu_{\partial O}&=\int_{V_i}\sqrt{1+\|(\partial_{x_j}\psi_i(\mathbf{x}'))_{j\ne i^*}\|^2}\cdot|g(0,\mathbf{x}')|^p\cdot\frac{1}{\sqrt{2\pi a_{i^*}^2}}e^{-\frac{\psi_i(\mathbf{x}')^2}{2a_{i^*}^2}}\dd P^{\widehat{i^*}}(\mathbf{x}')\\
&\le\sqrt{1+L^2}\int_{V_i}\frac{1}{\sqrt{2\pi a_{i^*}^2}}\,|g(0,\mathbf{x}')|^p\cdot e^{-\frac{\psi_i(\mathbf{x}')^2}{2a_{i^*}^2}}\dd P^{\widehat{i^*}}(\mathbf{x}').
\end{aligned}\end{equation}
For fixed $\mathbf{x}'$, viewing $g(t,\mathbf{x}')$ as a function in $t$ and applying Lemma~\ref{lem:1dim}, we have
\[\begin{aligned}
&\hphantom{{}={}}\int_{V_i}\frac{1}{\sqrt{2\pi a_{i^*}^2}}\,|g(0,\mathbf{x}')|^p\cdot e^{-\frac{\psi_i(\mathbf{x}')^2}{2a_{i^*}^2}}\dd P^{\widehat{i^*}}(\mathbf{x}')\\
&\le \frac{C_p}{\min\{a_{i^*},r_i\}}\int_{V_i}\int_{0}^{r_i}\left(|g(t,\mathbf{x}')|^p+ a_{i^*}^p|\partial_t g(t,\mathbf{x}')|^p\right)\dd \mathcal{N}_{a_{i^*}}(t)\cdot e^{-\frac{\psi_i(\mathbf{x}')^2}{2a_{i^*}^2}}\dd P^{\widehat{i^*}}(\mathbf{x}')\\
&\le C_p L\int_{V_i}\int_{0}^{r_i}\left(|g(t,\mathbf{x}')|^p+ a_{i^*}^p|\partial_t g(t,\mathbf{x}')|^p\right)\dd \mathcal{N}_{a_{i^*}}(t)\cdot e^{-\frac{\psi_i(\mathbf{x}')^2}{2a_{i^*}^2}}\dd P^{\widehat{i^*}}(\mathbf{x}')\\
&={C_p}L\int_{U_i\cap O}\left(|f(\mathbf{x})|^p+a_{i^*}^p |\partial_{x_{i^*}}f(\mathbf{x})|^p\right)\cdot \exp\left(\frac{x_{i^*}^2-(x_{i^*}-\psi_i(\mathbf{x}'))^2-\psi_i(\mathbf{x}')^2}{2a_{i^*}^2}\right)\dd P(\mathbf{x})\\
&={C_p}L\int_{U_i\cap O}\left(|f(\mathbf{x})|^p+a_{i^*}^p |\partial_{x_{i^*}}f(\mathbf{x})|^p\right)\cdot \exp\left(\frac{(x_{i^*}-\psi_i(\mathbf{x}'))\psi_i(\mathbf{x}')}{a_{i^*}^2}\right)\dd P(\mathbf{x})\\
&\le {C_p}Le^{L^3}\int_{U_i\cap O}\left(|f(\mathbf{x})|^p+a_{i^*}^p |\partial_{x_{i^*}}f(\mathbf{x})|^p\right)\dd P(\mathbf{x}).
\end{aligned}\]
where the fourth line follows from the change of variables formula.
Combining this with \eqref{eq:intbd}, we have
\[\begin{aligned}
\int_{\partial O}|f|^p\dd\mu_{\partial O}&\le \sum_{i\in\L}\int_{U_i \cap \partial O}|f|^p\dd\mu_{\partial O}\le {C_p}Le^{L^3}\sqrt{1+L^2}\sum_{i\in\L}\|f\|_{W^{1,p}(U_i\cap O,P)}^p\\
&\le m{C_p}Le^{L^3}\sqrt{1+L^2}\|f\|_{W^{1,p}(O,P)}^p,
\end{aligned}\]
where the last inequality follows from Definition~\ref{def:unc1} (2).
Similarly,
\[\int_{\partial O}|f|^p\dd\mu_{\partial O}\le  m{C_p}Le^{L^3}\sqrt{1+L^2}\|f\|_{\mathscr{W}^{1,p}(O,P)}^p.\]
By Theorem~\ref{zhuhanshuguanghuabijin} and its $\mathscr{W}^{m,p}$ version, $\mathscr{C}_b^\infty$ is dense in both $W^{1,p}(O,P)$ and $\mathscr{W}^{1,p}(O,P)$. Thus the result follows.
\end{proof}

\subsection{A more refined approach}\label{subsec:trrefined}


Throughout this subsection, we assume that $O$ has $C^1$ boundary (see the beginning of Section~\ref{subsec:ext}) so that the unit outward normal vector field
\[\nu^{\partial O}(\mathbf{x})=(\nu_i^{\partial O}(\mathbf{x}))_{i\in\dbN},\qquad\mathbf{x}\in\partial O\]
is well-defined. Recall that $\mu_{\partial O}$ denotes the surface measure on $\partial O$ as in Section \ref{subsec:surfmes}. 
For $f\in C_{\mathcal{F}}^1(O)$, we write
\[\nabla f\triangleq(\partial_{x_i} f)_{i\in\dbN}.\]
When $p\in(1,\infty)$, let $p'=p/(p-1)$ be the conjugate of $p$.

\subsubsection{Admissibility}
\begin{definition}\label{def:adm}
   Let $p\in[1,\infty)$. We say $O$ is $(W^{1,p},L^p)$-admissible if $O$ has $C^1$ boundary, and there exist functions $E_n:O\to[0,\infty)$, $F_{n,j}\in C^1(\overline{O})$ ($j=1,\dots,n$) for $n\in\dbN$, 
   together with a constant $M>0$ independent of $n$,
   such that the following conditions hold for all sufficiently large $n$:
   \begin{enumerate}[label=(\arabic*)]
      \item For $\mu_{\partial O}$-almost every $\mathbf{x}\in\partial O$,
      \[\liminf_{n\to\infty}\sum_{j=1}^{n} F_{n,j}(\mathbf{x})\nu_i^{\partial O}(\mathbf{x})\ge M^{-1}.\]
      \item For $P$-almost every $\mathbf{x}\in O$,
      \[\sum_{j=1}^{n}\left(\partial_{x_j}F_{n,j}(\mathbf{x})-\frac{x_j}{a_j^2}F_{n,j}(\mathbf{x})\right)\le M-E_n(\mathbf{x}).\]
      \item For $P$-almost every $\mathbf{x}\in O$,
      \[\sum_{j=1}^{n}\left|\frac{F_{n,j}(\mathbf{x})}{a_j}\right|^{p'}\le M(1+E_n(\mathbf{x}))\quad\text{if }p>1;\qquad\max_{1\le j\le n}\left|\frac{F_{n,j}(\mathbf{x})}{a_j}\right|\le M\quad\text{if }p=1.\]
   \end{enumerate}
   We say  $O$ is $(\mathscr{W}^{1,p},L^p)$-admissible if other conditions remain unchanged, while the condition (3) is replaced by
   \begin{enumerate}
      \item[(3')] For $P$-almost every $\mathbf{x}\in O$,
      \[\sum_{j=1}^{n}\left|\frac{F_{n,j}(\mathbf{x})}{a_j}\right|^{2}\le M(1+E_n(\mathbf{x}))^{2/p'}\quad\text{if }p>1;\qquad\sum_{j=1}^{n}\left|\frac{F_{n,j}(\mathbf{x})}{a_j}\right|^2\le M\quad\text{if }p=1.\]
   \end{enumerate}
\end{definition}

We will soon see that these conditions are closely related to Theorem~\ref{thm:GG}, which is one of the key ingredients in our proof of Theorem~\ref{thm:tr}. In order to avoid checking the integrability assumptions in Theorem~\ref{thm:GG}, we need the following lemma.

\begin{lemma}\label{lem:compact}
   Let $p\in[1,\infty)$. Suppose there is a bounded linear operator
   \[\operatorname{Tr}:W^{1,p}(O,P)\longrightarrow L^p(O,\mu_{\partial O})\]
   such that $\operatorname{Tr}(u)=u|_{\partial O}$ for any $u\in C_{\mathcal{F},b}^1(\ell^2)\cap W^{1,p}(\ell^2,P)$ with $\supp(u)$ compact.
   Then
   \[\operatorname{Tr}(f)=f|_{\partial O},\qquad\forall f\in C_{\mathcal{F},b}^1(\ell^2)\cap W^{1,p}(\ell^2,P).\]
\end{lemma}
\begin{proof}
   The property (ii) of $\mu_{\partial O}$ in Proposition~\ref{prop:radon} gives a sequence of compact sets $\{K_m\subset \partial O\}_{m\in\dbN}$ such that $\mu_{\partial O}(Z_0)=0$, where
   \[Z_0\triangleq\partial O\setminus\bigcup_{m=1}^\infty K_m.\]
   For any $\mathbf{y}=(y_i)_{i\in\dbN}\in\ell^2$, let
   \[t_n(\mathbf{y})\triangleq\sqrt{\sum_{j=n}^{\infty}y_j^2}. \]
   Then $t_n(\mathbf{y})\ge t_{n+1}(\mathbf{y})$ and $t_n(\mathbf{y})\to 0$ as $n\to\infty$.
   Since $t_n(\mathbf{y}-\mathbf{y}')\le \|\mathbf{y}-\mathbf{y}'\|$, the function $t_n:\ell^2\to\dbR$ is continuous. Hence it follows from Dini's theorem that
   \begin{equation}\label{eq:compactclaim}
   \lim_{n\to\infty}\sup_{\mathbf{y}\in K_m}t_n(\mathbf{y})=0.
   \end{equation}

   Define
   \[c_{m,i}\triangleq\sqrt{\sum_{j=i}^{\infty}a_j+\left(\sup_{\mathbf{y}\in K_m}t_i(\mathbf{y})\right)^2}.\]
   Then $\lim_{i\to\infty}c_{m,i}=0$ by \eqref{eq:compactclaim}, and
   \begin{equation}\label{eq:ci}
   \sum_{i=1}^{\infty}\frac{a_i}{c_{m,i}}\le\sum_{i=1}^{\infty}\frac{(c_{m,i}-c_{m,i+1})(c_{m,i}+c_{m,i+1})}{c_{m,i}}\le 2\sum_{n=1}^{\infty}(c_{m,i}-c_{m,i+1})=2c_{m,1}<\infty.
   \end{equation}
   Let
   \[K_{m,n}\triangleq\left\{(x_i)_{i\in\dbN}\in\ell^2:\sum_{i=1}^{\infty}\frac{x_i^2}{c_{m,i}}\le n^2\right\},\]
   which is compact by the proof of Proposition~\ref{230215Th1} with $d_i=c_{m,i}^{-1}$. For every $\mathbf{y}=(y_i)_{i\in\dbN}\in K_m$, we have
   \[\begin{aligned}
   \sum_{i=1}^{\infty}\frac{y_i^2}{c_{m,i}}&\le\sum_{i=1}^{\infty}\frac{y_i^2}{t_i(\mathbf{y})}=\sum_{i=1}^{\infty}\frac{(t_i(\mathbf{y})-t_{i+1}(\mathbf{y}))(t_i(\mathbf{y})+t_{i+1}(\mathbf{y}))}{t_i(\mathbf{y})}\\
   &\le 2\sum_{n=1}^{\infty}(t_i(\mathbf{y})-t_{i+1}(\mathbf{y}))=2t_{1}(\mathbf{y})\le \sup_{K_m}t_1<\infty,
   \end{aligned}\]
   so $K_m\subset K_{m,n_m}$ for some $n_m\in\dbN$.
   By the proof of Proposition~\ref{230215Th1}, there exists a sequence $\{X_{m,n}\}_{n\in\dbN}\subset C_{\mathscr{F}}^\infty(\ell^2)$ such that $\supp(X_{m,n})$ is compact, $\sup_{n\in\dbN}\|X_{m,n}\|_{C_{\mathscr{F}}^{1,1}(O)}<\infty$, and $X_{m,n}(\mathbf{x})=1$ for any $\mathbf{x}\in K_{m,n}$.

   For any $f\in C_{\mathcal{F},b}^1(\ell^2)\cap W^{1,p}(\ell^2,P)$, we have $X_{m,n}\cdot f\in C_{\mathcal{F},b}^1(\ell^2)\cap W^{1,p}(\ell^2,P)$, and by Proposition~\ref{chengjijieduan},
   \[\begin{aligned}
   \|X_{m,n}\cdot f\|_{W^{1,p}(O,P)}&\le C(1,p)\|X_{m,n}\|_{C_{\mathscr{F}}^{1,1}(O)}\cdot\|f\|_{W^{1,p}(O,P)}\\
   &\le C(1,p)\sup_{n\in\dbN}\|X_{m,n}\|_{C_{\mathscr{F}}^{1,1}(O)}\cdot\|f\|_{W^{1,p}(O,P)}.
   \end{aligned}\]
   Combining this with the property 3) in Proposition~\ref{230215Th1} and the Dominated Convergence Theorem, we have
   \[\lim_{n\to\infty}\|X_{m,n}\cdot f-f\|_{W^{1,p}(O,P)}=0.\]
   Since $X_{m,n}\cdot f\in C_{\mathcal{F},b}^\infty(\ell^2)\cap W^{1,p}(\ell^2,P)$ has compact support, our assumption of $\operatorname{Tr}$ implies
   \[\begin{aligned}
   X_{m,n_m}|_{\partial O}\cdot f|_{\partial O}&=(X_{m,n_m}\cdot f)|_{\partial O}=\operatorname{Tr}(X_{m,n_m}\cdot f)\\
   &=\lim_{n\to\infty}\operatorname{Tr}(X_{m,n_m}\cdot X_{m,n}\cdot f)\\
   &=\lim_{n\to\infty} (X_{m,n_m}\cdot X_{m,n}\cdot f)|_{\partial O}\\
   &=X_{m,n_m}|_{\partial O}\cdot \lim_{n\to\infty}(X_{m,n}\cdot f)|_{\partial O}\\
   &=X_{m,n_m}|_{\partial O}\cdot \operatorname{Tr}(f)
   \end{aligned}\]
   in $L^p(\partial O,\mu_{\partial O})$. This, together with the property $K_m\subset X_{m,n_m}^{-1}(\dbR\setminus\{0\})$, yields a measurable set $Z_m\subset K_m$ with $\mu_{\partial O}(Z_m)=0$ such that $\operatorname{Tr}(f)=f|_{\partial O}$ on $K_m\setminus Z_m$. Therefore,
   \[\operatorname{Tr}(f)=f|_{\partial O}\qquad \text{on $\partial O\setminus Z$},\]
   where $Z=Z_0\cup\bigcup_{m=1}^\infty Z_m$ and $\mu_{\partial O}(Z)=0$.
\end{proof}

\begin{remark}
   It is easy to see that Lemma~\ref{lem:compact} also holds if we replace $W^{1,p}(O,P)$ in the statement by $\mathscr{W}^{1,p}(O,P)$.
\end{remark}

Now we are in a position to prove Theorem~\ref{thm:tr}.
\begin{proof}[Proof of Theorem~\ref{thm:tr}]
	We first consider the $W^{1,p}$ part. In fact, we will prove that there exists a constant $C=C(p,O)>0$ such that
	\begin{equation}\label{eq:trst}
	\|f|_{\partial O}\|_{L^p(\partial O,\mu_{\partial O})}\le C\|f\|_{W^{1,p}(O,P)},\qquad \forall f\in C^1_{\mathcal{F},b}(\ell^2)\cap W^{1,p}(\ell^2,P),
	\end{equation}
	which, combined with Theorem~\ref{zhuhanshuguanghuabijin}, proves the desired result.
   Let $E_n$, $F_{n,j}$, and $M$ be as in Definition~\ref{def:adm}. Write $F=(F_i)_{i\in\dbN}$ and
   \[G_n=(G_{n,i})_{i\in\dbN},\qquad G_{n,i}=\begin{cases}
      F_{n,i} &(i\le n)\\
      0 &(i>n)
   \end{cases}.\]
   The conditions (1)--(2) in Definition~\ref{def:adm} imply
   \begin{equation}\label{eq:limdiv}
      \int_{\partial O}|f|^p\dd\mu_{\partial O}\le M\int_{\partial O}|f|^p\liminf_{n\to\infty}\langle G_n,\nu^{\partial O}\rangle\dd\mu_{\partial O}\le\liminf_{n\to\infty}\int_{\partial O}|f|^p \langle G_n,\nu^{\partial O}\rangle\dd\mu_{\partial O},
   \end{equation}
   where the last equality follows from Fatou's lemma.
   By Lemma~\ref{lem:compact}, we may assume that $f$ has compact support.
   Then it follows from Theorem~\ref{thm:GG} that\footnote{If $p=1$, first apply the divergence theorem to $f_\varepsilon=\sqrt{f^2+\varepsilon^2}-\varepsilon$ instead of $|f|$ and then use the Dominated Convergence Theorem with $\varepsilon\to 0$.}
   \begin{equation}\label{eq:trdiv}
   \begin{aligned}
   \int_{\partial O}|f|^p \langle G_n,\nu^{\partial O}\rangle\dd\mu_{\partial O}&=\int_{O}\langle \nabla (|f|^p), G_n\rangle\dd P+\int_{O}|f|^p\Div_{P}G_n\dd P\\
   &=\int_{O}p|f|^{p-1}\operatorname{sgn}(|f|)\langle \nabla f, G_n\rangle\dd P+\int_{O}|f|^p\Div_{P}G_n\dd P,
\end{aligned}\end{equation}
where
\[\Div_P G_n\triangleq \Div G_n-G_n\cdot\left(\frac{x_i}{a_i^2}\right)_{i\in\dbN}=\sum_{j=1}^{n}\left(\partial_{x_j}F_{n,j}-\frac{x_j}{a_j^2}F_{n,j}\right).\]
Thus, combining \eqref{eq:limdiv}, \eqref{eq:trdiv}, and the condition (3) in Definition~\ref{def:adm}, we have
\begin{equation}\label{eq:trint}
\int_{\partial O}|f|^p\dd\mu_{\partial O}\le M\liminf_{n\to\infty}\left(\int_{O}p|f|^{p-1}|\langle \nabla f, G_n\rangle|\dd P+\int_{O}(M-E_n)|f|^p\dd P\right).
\end{equation}

By the condition (4) in Definition~\ref{def:adm}, when $p=1$,
\[|\langle \nabla f, G_n\rangle|\le\left(\sum_{j=1}^{n}|a_j\partial_{x_j}f|\right)\max_{1\le j\le n}\left|\frac{F_{n,j}}{a_j}\right|\le M\sum_{j=1}^{n}|a_j\partial_{x_j}f|,\]
and the result follows easily. Now assume $p>1$. It follows from H\"older's inequality and the condition (4) in Definition~\ref{def:adm} that
\begin{equation}\label{eq:replace}
\begin{aligned}
|\langle \nabla f, G_n\rangle|&\le \left(\sum_{j=1}^{n}|a_j\partial_{x_j}f|^p\right)^{1/p}\left(\sum_{j=1}^{n}\left|\frac{F_{n,j}}{a_j}\right|^{p'}\right)^{1/p'}\\
&\le M^{1/p'}(1+E_n)^{1/p'}\left(\sum_{j=1}^{n}|a_j\partial_{x_j}f|^p\right)^{1/p}.
\end{aligned}\end{equation}
Therefore, by Young's inequality, for every $\delta>0$,
\begin{equation}\label{eq:Young}
p|f|^{p-1}|\langle \nabla f, G_n\rangle|\le M\delta(1+E_n)|f|^p+C_{p,\delta}\sum_{j=1}^{n}|a_j\partial_{x_j}f|^p
\end{equation}
for some constant $C_{p,\delta}>0$ independent of $n$. Choosing $0<\delta<M^{-1}$, we conclude from \eqref{eq:trint}, \eqref{eq:Young}, and the non-negativity of $E_n$ that
\[\begin{aligned}
\left\lVert f|_{\partial O}\right\rVert_{L^p(\partial O,\mu_{\partial O})}^p&=\int_{\partial O}|f|^p\dd\mu_{\partial O}\\
&\le\max\{M^2(1+\delta),MC_{p,\delta}\}\left(\int_{O}|f|^p\dd P+\sum_{i=1}^{\infty}a_i^p\int_{O}|\partial_{x_i}f|^p\dd P\right)\\
&=\max\{M^2(1+\delta),MC_{p,\delta}\}\left\lVert f\right\rVert_{W^{1,p}(O,P)}^p,
\end{aligned}\]
which completes the proof of the $W^{1,p}$ part.

Replacing the estimate of $|\langle \nabla f, G_n\rangle|$ in \eqref{eq:replace} by
\[
|\langle \nabla f, G_n\rangle|\le \left(\sum_{j=1}^{n}|a_j\partial_{x_j}f|^2\right)^{1/2}\left(\sum_{j=1}^{n}\left|\frac{F_{n,j}}{a_j}\right|^{2}\right)^{1/2},
\]
we can also prove the $\mathscr{W}^{1,p}$ part.
\end{proof}


The next lemma gives sufficient conditions for $O$ to be $(W^{1,p},L^p)$-admissible for a wide range of $p$.
\begin{lemma}\label{lem:unifadm}
   Suppose that $O$ satisfies the conditions (1)--(2) in Definition~\ref{def:adm} for all $n\gg0$.
   \begin{enumerate}[label=(\roman*)]
      \item If the following inequality holds for all $n\gg0$ and $P$-almost every $\mathbf{x}\in O$:
      \begin{equation}\label{eq:pge2}
      \sum_{j=1}^{n}\left|\frac{F_{n,j}(\mathbf{x})}{a_j}\right|^{k}\le M(1+E_n(\mathbf{x}))\qquad(k=1,2),
      \end{equation}
   then $O$ is both $(W^{1,p},L^p)$-admissible and $(\mathscr{W}^{1,p},L^p)$-admissible for every $p\in[2,\infty)$.
   \item If the following inequality holds for all $n\gg0$ and $P$-almost every $\mathbf{x}\in O$:
      \[\sum_{j=1}^{n}\left|\frac{F_{n,j}(\mathbf{x})}{a_j}\right|\le M,\]
   then $O$ is both $(W^{1,p},L^p)$-admissible and $(\mathscr{W}^{1,p},L^p)$-admissible for every $p\in[1,\infty)$.
   \end{enumerate}
\end{lemma}
\begin{proof}
For (i), first note that when $p\in[2,\infty)$, $O$ is clearly $(\mathscr{W}^{1,p},L^p)$-admissible, so it remains to show that $O$ is $(W^{1,p},L^p)$-admissible. By the interpolation inequality, for $p\in[2,\infty)$,
\[\left(\sum_{j=1}^{n}\left|\frac{F_{n,j}}{a_j}\right|^{p'}\right)^{1/p'}\le\left(\sum_{j=1}^{n}\left|\frac{F_{n,j}}{a_j}\right|\right)^{1-2/p}\left(\sum_{j=1}^{n}\left|\frac{F_{n,j}}{a_j}\right|^2\right)^{1/p},\]
which proves (i). When $p\in[1,\infty)$, since
\[\sum_{j=1}^{n}\left|\frac{F_{n,j}}{a_j}\right|^{p'}\le\left(\sum_{j=1}^{n}\left|\frac{F_{n,j}}{a_j}\right|\right)^{p'},\qquad\max_{1\le j\le n}\left|\frac{F_{n,j}}{a_j}\right|\le\sum_{j=1}^{n}\left|\frac{F_{n,j}}{a_j}\right|,\]
(ii) follows.
\end{proof}

As a first application, we deal with arbitrary half-spaces.

\begin{proposition}\label{prop:linear}
   Let $(b_i)_{i\in\dbN}\in\ell^2\setminus\{\mathbf{0}\}$ and $b_0\in\dbR$. If
   \[O=\left\{\mathbf{x}=(x_i)_{i\in\dbN}\in\ell^2:\sum_{i=1}^{\infty}b_ix_i+b_0<0\right\},\]
   then $O$ is both $({W}^{1,p},L^p)$-admissible and $(\mathscr{W}^{1,p},L^p)$-admissible for every $p\in[1,\infty)$.
\end{proposition}
\begin{proof}
   Choose a smooth function $\chi:\dbR\to[0,1]$ such that $\chi(t)=0$ for $t\le-1$ and $\chi(t)=1$ for $t\ge 0$. Let $E_n=0$ and
   \[F_{n,j}=\frac{\chi(g_n)a_j^2b_j}{\sum_{i=1}^{\infty}a_i^2b_i^2},\]
   where $g_n(\mathbf{x})=\sum_{j=1}^{n}b_jx_j$.
   Then on $\partial O$,
   \[\sum_{j=1}^{n} F_{n,j}\nu_j^{\partial O}=\frac{\chi(g_n)\sum_{j=1}^{n}a_j^2b_j^2}{\|(b_i)_{i\in\dbN}\|\sum_{i=1}^{\infty}a_i^2b_i^2}\longrightarrow \frac{\chi(0)}{\|(b_i)_{i\in\dbN}\|}=\frac{1}{\|(b_i)_{i\in\dbN}\|}\qquad(n\to\infty).\]
   On $O$, since $g_n\ge -1$ whenever $\chi(g_n)\neq 0$, we have
      \[\sum_{j=1}^{n}\left(\partial_{x_j}F_{n,j}-\frac{x_j}{a_j^2}F_{n,j}\right)=\chi(g_n)\frac{\sum_{j=1}^{n}a_j^2b_j^2}{\sum_{i=1}^{\infty}a_i^2b_i^2}-\frac{g_n\chi(g_n)}{\sum_{i=1}^{\infty}a_i^2b_i^2}\le 1+\frac{1}{\sum_{i=1}^{\infty}a_i^2b_i^2}\]
   and
      \[\sum_{j=1}^{n}\left|\frac{F_{n,j}}{a_j}\right|=\chi(g_n)\frac{\sum_{j=1}^{n}a_jb_j}{\sum_{i=1}^{\infty}a_i^2b_i^2}\le \frac{\|(a_i)_{i\in\dbN}\|\cdot\|(b_i)_{i\in\dbN}\|}{\sum_{i=1}^{\infty}a_i^2b_i^2}.\]
      Hence the result follows from Lemma~\ref{lem:unifadm} (ii).
\end{proof}

Next, for every $p\in[2,\infty)$, we will verify that continuous functions $g:\ell^2\to\dbR$ of the form
\begin{equation}\label{eq:gconv}
g(\mathbf{x})=\sum_{i=1}^{\infty}\varphi_i(x_i),\qquad\forall\mathbf{x}=(x_i)_{i\in\dbN}\in\ell^2
\end{equation}
with $\varphi_i\in C^1(\mathbb{\dbR})$ convex for all $i$ define $({W}^{1,p},L^p)$-admissible (hence $(\mathscr{W}^{1,p},L^p)$-admissible) open sets
\[O_g\triangleq\{\mathbf{x}\in\ell^2:g(\mathbf{x})<0\}.\]
Recall that a function $f:\dbR\to\dbR$ is called convex if for any $x,y\in\dbR$ and $t\in[0,1]$,
\[f(tx+(1-t)y)\le tf(x)+(1-t)f(y).\]
We will use the following elementary properties of convex functions:
\begin{lemma}\label{lem:conv}
   Let $f:\dbR\to\dbR$ be a convex function.
   \begin{enumerate}[label=(\roman*)]
      \item If $f(0)=0$, then for any $x\in\dbR$ and $\lambda\ge1$,
      \[f(\lambda x)\ge \lambda f(x).\]
      \item If $f\in C^1(\dbR)$, then for all $x,y\in\dbR$,
\[f(y)\ge f(x)+f'(x)(y-x).\]
   \end{enumerate}
\end{lemma}

Using the continuity of $g$ and Lemma~\ref{lem:conv}, we obtain the next proposition.

\begin{proposition}\label{prop:convball}
   Let $g\in C(\ell^2)$ such that $O_g$ is nonempty.
   Then there exists $\mathbf{c}=(c_i)_{i\in\dbN}\in\ell^2$ and $R\in(0,\infty)$ such that
   \begin{enumerate}[label=(\arabic*)]
      \item $\supp(\mathbf{c})$ is finite, or equivalently, $c_i=0$ for all $i\gg0$;
      \item $B(\mathbf{c},R)\subset O_g$.
   \end{enumerate}
   If in addition $g$ is of the form in \eqref{eq:gconv} with $\varphi_i\in C^1(\mathbb{\dbR})$ convex, then $g$ is continuously Fr\'echet differentiable and
   \[R\|D g(\mathbf{x})\|\le(\mathbf{x}-\mathbf{c})\cdot D g(\mathbf{x}),\qquad\forall \mathbf{x}\in\partial O_g.\]
\end{proposition}
\begin{proof}
   Since those $\mathbf{c}\in\ell^2$ satisfying (1) are dense, we can find $\mathbf{c}\in O_g$. The continuity of $g$ gives the desired $R\in(0,\infty)$ satisfying (2).

   Next let $g$ be as in \eqref{eq:gconv} with $\varphi_i\in C^1(\mathbb{\dbR})$ convex. We first show that $(\partial_{x_i} g(\mathbf{x}))_{i\in\dbN}=(\varphi_i'(x_i))_{i\in\dbN}\in\ell^2$ for any $\mathbf{x}=(x_i)_{i\in\dbN}\in\ell^2$. By the continuity of $g$, there exists $\delta>0$ such that
   \begin{equation}\label{eq:conti}
   |g(\mathbf{x}+\mathbf{h})-g(\mathbf{x})|\le 1,\qquad\forall\,\mathbf{h}\in B_\delta\subset\ell^2.
   \end{equation}
   By Lemma~\ref{lem:conv} (ii), for every $\mathbf{y}=(y_i)_{i\in\dbN}\in\ell^2$,
   \begin{equation}\label{eq:conv}
      \varphi_i(y_i)\ge \varphi_i(x_i)+(y_i-x_i)\varphi_i'(x_i)\qquad(\forall i\in\dbN).
   \end{equation}
   We may assume $\partial_{x_N}g(\mathbf{x})\neq 0$.
   For any $n\ge N$, summing \eqref{eq:conv} over $i\in\dbN$ with $\mathbf{y}=\mathbf{x}+\mathbf{h}$, where
   \[\mathbf{h}=(h_i)_{i\in\dbN}\in B_\delta,\qquad h_i=\begin{cases}
      \displaystyle\frac{\delta \partial_{x_i}g(\mathbf{x})}{2\|(\partial_{x_j}g(\mathbf{x}))_{j=1}^n\|}&(i\le n)\\[3mm]
      0&(i> n)
   \end{cases},\]
   and using \eqref{eq:conti}, we have
   \[\frac{\delta}{2}\|(\partial_{x_j}g(\mathbf{x}))_{j=1}^n\|\le g(\mathbf{x}+\mathbf{h})-g(\mathbf{x})\le 1,\qquad\forall n\ge N.\]
   Thus, letting $n\to\infty$, we deduce $(\partial_{x_i} g(\mathbf{x}))_{i\in\dbN}\in\ell^2$.

   For any fixed $\mathbf{x}\in\ell^2$, let $\Phi:\ell^2\to\dbR$ be defined by
   \[\Phi(\mathbf{z})\triangleq \sum_{i=1}^{\infty}\Phi_i(z_i),\quad\Phi_i(z_i)\triangleq \varphi_i(x_i+z_i)-\varphi_i(x_i)-z_i\varphi_i'(x_i),\qquad\forall \mathbf{z}=(z_i)_{i\in\dbN}\in\ell^2.\]
   Clearly $\Phi_i:\dbR\to\dbR$ is convex with $\Phi_i(0)=0$. By \eqref{eq:conv},
   \begin{equation}\label{eq:Phi}
   0\le\Phi(\mathbf{z})\le \sum_{i=1}^{\infty}z_i\left(\partial_{x_i}g(\mathbf{x}+\mathbf{z})-\partial_{x_i}g(\mathbf{x})\right).
   \end{equation}
   By definition, $g$ is Fr\'echet differentiable at $\mathbf{x}$ if
   \begin{equation}\label{eq:Fdif}
   \lim_{\|\mathbf{z}\|\to 0^+}\frac{\Phi(\mathbf{z})}{\|\mathbf{z}\|}=0.
   \end{equation}
   Now we prove \eqref{eq:Fdif} by contradiction. Suppose there exist $\varepsilon>0$ and $\mathbf{z}^k=(z^k_i)_{i\in\dbN}\in\ell^2$ with $0<\|\mathbf{z}^k\|<1/k$ for every $k\in\dbN$ such that
   $\Phi(\mathbf{z}^k)\ge\varepsilon\|\mathbf{z}^k\|.$
   We claim that we can find $\mathbf{w}^k=(w^k_i)_{i\in\dbN}\in\ell^2$ with $0<\|\mathbf{w}^k\|<1/k$, $\supp(\mathbf{w}^k)$ finite and $\supp(\mathbf{w}^k)$ mutually disjoint, such that
   \begin{equation}\label{eq:wk}
      \Phi(\mathbf{w}^k)\ge\frac{1}{2}\varepsilon\|\mathbf{w}^k\|.
   \end{equation}
   Let $w_i^1=z^1_i$ for $i\le n$ and $w_i^1=0$ for $i>n$. When $n\gg0$, $\mathbf{w}^1=(w_i^1)_{i\in\dbN}$ satisfies \eqref{eq:wk} because $\Phi$ is continuous. Assume that we have already constructed $\mathbf{w}^1,\dots,\mathbf{w}^k$ as desired and that $\supp(\mathbf{w}^1),\dots,\supp(\mathbf{w}^k)$ are contained in $\{1,\dots,n_k\}$ for $n_k\in\dbN$. Since each $\varphi_i$ is $C^1$, by \eqref{eq:Phi} there exists an integer $K>k$ such that
   \[\sum_{i=1}^{n_k}\Phi_i(t_i)\le \frac{\varepsilon}{2}\|(t_i)_{i=1}^{n_k}\|,\qquad\text{whenever }\|(t_i)_{i=1}^{n_k}\|<\frac{1}{K}.\]
   Set $w^{k+1}_i=0$ for $i\le n_k$ and $w^{k+1}_i=z^K_{i}$ for $i>n_k$. The choice of $\mathbf{w}^{k+1}$ gives $0<\|\mathbf{w}^{k+1}\|<1/K$ and
   \[
   \Phi(\mathbf{w}^{k+1})\ge\Phi(\mathbf{z}^{K})-\sum_{i=1}^{n_k}\Phi_i(z^K_i)\ge\varepsilon\|\mathbf{z}^K\|-\frac{\varepsilon}{2}\|(z^K_i)_{i=1}^{n_k}\|\ge \frac{\varepsilon}{2}\|\mathbf{z}^K\|\ge \frac{\varepsilon}{2}\|\mathbf{w}^{k+1}\|.
   \]
   This proves our claim inductively. Let
   \[\mathbf{w}\triangleq\sum_{k=1}^{\infty}\frac{\mathbf{w}^k}{k\|\mathbf{w}^k\|}=(w_i)_{i\in\dbN}\in\ell^2.\]
   By our choice of $\mathbf{w}^k$ and Lemma~\ref{lem:conv} (ii), we have
   \[\Phi(\mathbf{w})=\sum_{k=1}^{\infty}\Phi\left(\frac{\mathbf{w}^k}{k\|\mathbf{w}^k\|}\right)\ge \sum_{k=1}^{\infty}\frac{\Phi(\mathbf{w}^k)}{k\|\mathbf{w}^k\|}\ge\sum_{k=1}^{\infty}\frac{\varepsilon}{2k}=\infty,\]
   which is impossible. Therefore, \eqref{eq:Fdif} holds and $g$ is Fr\'echet differentiable everywhere.

   To show $g$ is continuously Fr\'echet differentiable, it suffices to prove
   \[\lim_{\|\mathbf{z}\|\to 0}\|Dg(\mathbf{x}+\mathbf{z})-Dg(\mathbf{x})\|=0\]
   for any fixed $\mathbf{x}\in\ell^2$.
   If $Dg(\mathbf{x}+\mathbf{z})-Dg(\mathbf{x})\neq\mathbf{0}$, then
   \[\begin{aligned}
   2\|\mathbf{z}\|\sup_{0<\mathbf{v}\le2\|\mathbf{z}\|}\frac{\Phi(\mathbf{v})}{\|\mathbf{v}\|}&\ge\Phi\left(\mathbf{z}+\frac{\|\mathbf{z}\|(Dg(\mathbf{x}+\mathbf{z})-Dg(\mathbf{x}))}{\|Dg(\mathbf{x}+\mathbf{z})-Dg(\mathbf{x})\|}\right)\\
   &\ge\Phi(\mathbf{z})+\|\mathbf{z}\|\cdot\|Dg(\mathbf{x}+\mathbf{z})-Dg(\mathbf{x})\|\ge \|\mathbf{z}\|\cdot\|Dg(\mathbf{x}+\mathbf{z})-Dg(\mathbf{x})\|,
   \end{aligned}\]
   where the second inequality follows from \eqref{eq:conv}.
   Hence,
   \[\|Dg(\mathbf{x}+\mathbf{z})-Dg(\mathbf{x})\|\le 2\sup_{0<\mathbf{v}\le2\|\mathbf{z}\|}\frac{\Phi(\mathbf{v})}{\|\mathbf{v}\|}\longrightarrow 0\qquad(\|\mathbf{z}\|\to 0).\]

   For the last statement, again by \eqref{eq:conv},
   \begin{equation}\label{eq:ineqconv}
   g(\mathbf{y})\ge g(\mathbf{x})+(\mathbf{y}-\mathbf{x})\cdot D g(\mathbf{x})=(\mathbf{y}-\mathbf{x})\cdot D g(\mathbf{x})\qquad (\forall\mathbf{y}\in\ell^2,\,\mathbf{x}\in\partial O_g).
   \end{equation}
   Setting $\mathbf{y}=\mathbf{c}+r\nu(\mathbf{x})$ in \eqref{eq:ineqconv} for any $r\in(0,R)$, we have
   \[0>g(\mathbf{y})\ge\sum_{i=1}^{\infty}(c_i-x_i)\partial_{x_i}g+r\|D g(\mathbf{x})\|,\]
   so letting $r\to R$ completes the proof.
\end{proof}

For $O_g$ with $g$ as in \eqref{eq:gconv}, we verify the admissibility conditions.

\begin{proposition}\label{prop:conv}
   Let $g\in C(\ell^2)$ as in \eqref{eq:gconv} with $\varphi_i\in C^1(\mathbb{\dbR})$ convex. If $O_g$ is nonempty, then it is both $({W}^{1,p},L^p)$-admissible and $(\mathscr{W}^{1,p},L^p)$-admissible for every $p\in[2,\infty)$.
\end{proposition}
\begin{proof}
	Note that $g\in C^1_{\ell^2}(\ell^2)$ by Proposition~\ref{prop:convball}. If $Dg(\mathbf{x})=\mathbf{0}$ for some $\mathbf{x}\in\partial O_g$, then by the convexity of $\varphi_i$, $0=g(\mathbf{x})=\min_{\ell^2}g$ and thus $O_g=\varnothing$. Assuming $O_g\ne\varnothing$, we deduce that $O_g$ has $C^1$ boundary. By Example~\ref{ex:globalfunction},
	\[\nu_i^{\partial O}(\mathbf{x})=\frac{\partial_{x_i}g(\mathbf{x})}{\sqrt{\sum_{j=1}^{\infty}|\partial_{x_j}g(\mathbf{x})|^2}}=\frac{\partial_{x_i}g(\mathbf{x})}{\|Dg(\mathbf{x})\|}.\]
	We verify the remaining conditions in the following steps.

   \medskip\noindent\textbf{Step 1: Construction of $E_n$ and $F_{n,j}$.}
   Let $\mathbf{c}=(c_i)_{i\in\dbN}$ and $R$ be as in Proposition~\ref{prop:convball}. The idea of construction is to modify $F^0_i=x_i-c_i$ in a way such that all conditions in Definition~\ref{def:adm} are satisfied.
   Suppose $c_i=0$ for all $i>N$, where $N\in\dbN$. Write $A_*\triangleq \sup_{i\in\dbN}\{2a_i\}<\infty$. For each $j\in\dbN$, define
   \[I_j\triangleq\{i>N:2^{-j-1}A_*< a_i^2\le 2^{-j}A_*\}\subset\dbN,\qquad I_{j,n}\triangleq I_j\cap\{1,\dots,n\},\]
   which are finite sets since $\lim_{i\to\infty} a_i=0$ by \eqref{20260424for1}.
   Then $\cup_{j\in\dbN}I_j=\dbN\setminus\{1,\dots,N\}$, $I_{j}\cap I_{j'}=\varnothing$ if $j\neq j'$, and
   \begin{equation}\label{eq:Ij}
   \sum_{j=1}^{\infty}2^{-j}A_*|I_j|\le 2\sum_{i=N+1}^{\infty}a_i^2,
   \end{equation}
   where $|I_j|$ is the cardinality of $I_j$. Since $\sum_{j=1}^{\infty}2^{-j}\cdot j^2<\infty$ and the right-hand side of \eqref{eq:Ij} tends to $0$ as $N\to\infty$, we can choose a large $N$ such that
   \begin{equation}\label{eq:Nlarge}
   16\sum_{|I_j|>0}2^{-j}A_*\left(|I_j|+j^2\right)\le \frac{R^2}{9}.
   \end{equation}
   For every nonempty $I_j$, let
   \[s_j\triangleq 4\sqrt{2^{-j}A_*\left(|I_j|+j^2\right)},\qquad\varepsilon_j\triangleq\frac{2^{-j-1}}{|I_j|+1},\qquad \alpha_j(r)\triangleq\varepsilon_j+(1-\varepsilon_j)h(r/s_j)\quad(\forall r\in\dbR),\]
   where $h\in C^1(\dbR)$ is defined by
   \[h(t)=\begin{cases}
      0&(t\le1)\\[3mm]
      \displaystyle\frac{(t-1)^2}{1+(t-1)^2}&(t>1)\end{cases}.\]
   The following properties of $\alpha_j$ and $h$ follow directly from the definition and will be used later:
   \begin{equation}\label{eq:propalpha}
      (1-\alpha_j(r))r\le\frac{3}{2}s_j,\qquad t h'(t)\le4\sqrt{h(t)}.
   \end{equation}

   For a nonempty set $I\subset\dbN$ and $\mathbf{x}\in\ell^2$, we write $\|\mathbf{x}\|_{\ell^2(I)}\triangleq\|P_I\mathbf{x}\|_{\ell^2(I)}$.
   For any $n\ge N$, define
   \[E_n(\mathbf{x})=\frac{1}{2}\sum_{i=1}^{N}\left(\frac{x_i-c_i}{a_i}\right)^2+\frac{1}{2}\sum_{|I_{j,n}|>0}\alpha_{j}\left(\|\mathbf{x}\|_{\ell^2(I_{j,n})}\right)\sum_{i\in I_{j,n}} \frac{x_i^2}{a_i^2}\]
   and
   \[F_{n,i}(\mathbf{x})=\begin{cases}
      x_i-c_i&(i\le N)\\
      \alpha_{j}\left(\|\mathbf{x}\|_{\ell^2(I_{j,n})}\right)x_i&(i\in I_{j,n})
   \end{cases},\qquad F_{i}(\mathbf{x})=\begin{cases}
      x_i-c_i&(i\le N)\\
      \alpha_{j}\left(\|\mathbf{x}\|_{\ell^2(I_{j})}\right)x_i&(i\in I_{j})
   \end{cases}.\]

   \medskip\noindent\textbf{Step 2: Proof of the conditions (1)--(2) in Definition~\ref{def:adm}.}
   By \eqref{eq:propalpha} and \eqref{eq:Nlarge},
   \[\begin{aligned}
   \|(F_i)_{i\in\dbN}-(\mathbf{x}-\mathbf{c})\|^2&=\sum_{j=1}^{\infty}\sum_{i\in I_j}\left(1-\alpha_{j}\left(\|\mathbf{x}\|_{\ell^2(I_{j})}\right)\right)^2x_i^2\\
   &=\sum_{j=1}^{\infty}\left(1-\alpha_{j}\left(\|\mathbf{x}\|_{\ell^2(I_{j})}\right)\right)^2 \|\mathbf{x}\|_{\ell^2(I_{j})}^2\\
   &\le \frac{9}{4}\sum_{j=1}^{\infty}s_j^2\le \frac{R^2}{4}.
   \end{aligned}\]
   Hence, for any $\mathbf{x}\in \partial O_g$, by Proposition~\ref{prop:convball} and the Cauchy--Schwarz inequality,
   \[\begin{aligned}
   \lim_{n\to\infty}\sum_{i=1}^{n}F_{n,i}(\mathbf{x}) \partial_{x_i}g(\mathbf{x})&=\sum_{i=1}^{\infty}F_{i}(\mathbf{x}) \partial_{x_i}g(\mathbf{x})\\
   &=\sum_{i=1}^{\infty}(x_i-c_i)\partial_{x_i}g(\mathbf{x})+\sum_{i=1}^{\infty}(F_i(\mathbf{x})-x_i+c_i)\partial_{x_i}g(\mathbf{x})\\
   &\ge R\|D g(\mathbf{x})\|-\|(F_i(\mathbf{x}))_{i\in\dbN}-(\mathbf{x}-\mathbf{c})\|\cdot\|D g(\mathbf{x})\|\ge \frac{R}{2}\|D g(\mathbf{x})\|,
   \end{aligned}\]
   which proves the condition (1).

   To prove the condition (2), note that
   \begin{equation}\label{eq:divhead}
   \begin{aligned}
      \sum_{i=1}^{N}\left(\partial_{x_i}F_{n,i}-\frac{x_i}{a_i^2}F_{n,i}\right)&=N-\sum_{i=1}^{N}\left(\frac{x_i-c_i}{a_i}\right)^2-\sum_{i=1}^{N}\frac{x_i-c_i}{a_i}\cdot\frac{c_i}{a_i}\\
      &\le N+\frac{1}{2}\sum_{i=1}^{N}\frac{c_i^2}{a_i^2}-\frac{1}{2}\sum_{i=1}^{N}\left(\frac{x_i-c_i}{a_i}\right)^2,
   \end{aligned}
   \end{equation}
   and for any nonempty $I_{j,n}$,
   \[\begin{aligned}
      S_{j,n}\triangleq\sum_{i\in I_{j,n}}\left(\partial_{x_i}F_{n,i}-\frac{x_i}{a_i^2}F_{n,i}\right)&=|I_j|\alpha_{j}\left(\|\mathbf{x}\|_{\ell^2(I_{j,n})}\right)+\|\mathbf{x}\|_{\ell^2(I_{j,n})}\alpha_{j}'\left(\|\mathbf{x}\|_{\ell^2(I_{j,n})}\right)\\
      &\hphantom{{}={}}-\alpha_{j}\left(\|\mathbf{x}\|_{\ell^2(I_{j,n})}\right)\sum_{i\in I_{j,n}} \frac{x_i^2}{a_i^2}.
   \end{aligned}\]
   If $\|\mathbf{x}\|_{\ell^2(I_{j,n})}\le s_j$, then $\alpha_j=\varepsilon_j$ and
   \[S_{j,n}= |I_j|\varepsilon_j-\alpha_{j}\left(\|\mathbf{x}\|_{\ell^2(I_{j,n})}\right)\sum_{i\in I_{j,n}} \frac{x_i^2}{a_i^2}.\]
   If $\|\mathbf{x}\|_{\ell^2(I_{j,n})}>s_j$, then the by the definition of $I_{j,n}$,
\begin{equation}\label{eq:16Ij}
\sum_{i\in I_{j,n}} \frac{x_i^2}{a_i^2}\ge 2^jA_*^{-1}\|\mathbf{x}\|_{\ell^2(I_{j,n})}^2>2^jA_*^{-1}s_j^2\ge 16|I_j|,
\end{equation}
which implies
\begin{equation}\label{eq:imp}
|I_j|\alpha_{j}\left(\|\mathbf{x}\|_{\ell^2(I_{j,n})}\right)\le \frac{1}{16}\alpha_{j}\left(\|\mathbf{x}\|_{\ell^2(I_{j,n})}\right)\sum_{i\in I_{j,n}} \frac{x_i^2}{a_i^2};
\end{equation}
moreover, it follows from \eqref{eq:propalpha}, $\alpha_j(r)\ge h(r/s_j)$, and \eqref{eq:16Ij} that
\[\begin{aligned}
\|\mathbf{x}\|_{\ell^2(I_{j,n})}\alpha_{j}'\left(\|\mathbf{x}\|_{\ell^2(I_{j,n})}\right)&\le4\sqrt{\alpha_{j}\left(\|\mathbf{x}\|_{\ell^2(I_{j,n})}\right)}\\
&\le32\left(\sum_{i\in I_{j,n}} \frac{x_i^2}{a_i^2}\right)^{-1}+\frac{1}{8}\alpha_{j}\left(\|\mathbf{x}\|_{\ell^2(I_{j,n})}\right)\sum_{i\in I_{j,n}} \frac{x_i^2}{a_i^2}\\
&\le\frac{32A_*}{2^js_j^2}+\frac{1}{8}\alpha_{j}\left(\|\mathbf{x}\|_{\ell^2(I_{j,n})}\right)\sum_{i\in I_{j,n}} \frac{x_i^2}{a_i^2}.
\end{aligned}
\]
Thus in both cases
\begin{equation}\label{eq:divtail}
S_{j,n}\le |I_j|\varepsilon_j+\frac{32A_*}{2^js_j^2}-\frac{1}{2}\alpha_{j}\left(\|\mathbf{x}\|_{\ell^2(I_{j,n})}\right)\sum_{i\in I_{j,n}} \frac{x_i^2}{a_i^2}.
\end{equation}
Since
\[
\sum_{j=1}^\infty |I_j|\varepsilon_j\le1,\qquad \sum_{j=1}^\infty\frac{32A_*}{2^js_j^2}\le2A_*\sum_{j=1}^\infty j^{-2}<\infty,
\]
\eqref{eq:divhead} and \eqref{eq:divtail} yield
\[
\sum_{i=1}^{n}\left(\partial_{x_i}F_{n,i}-\frac{x_i}{a_i^2}F_{n,i}\right)\le M_0-E_n
\]
with $M_0>0$ independent of $n$.

\medskip\noindent\textbf{Step 3: Proof of \eqref{eq:pge2}.}
For $k=2$, since $\alpha_j(r)\le 1$, we have
\[
\sum_{i=1}^n\left|\frac{F_{n,i}}{a_i}\right|^2=\sum_{i=1}^{N}\left(\frac{x_i-c_i}{a_i}\right)^2+\sum_{|I_{j,n}|>0}\alpha_{j}\left(\|\mathbf{x}\|_{\ell^2(I_{j,n})}\right)^2\sum_{i\in I_{j,n}} \frac{x_i^2}{a_i^2}\le 2E_n.
\]
For $k=1$,
\[\begin{aligned}
\sum_{i=1}^n\left|\frac{F_{n,i}}{a_i}\right|&=\sum_{i=1}^{N}\left|\frac{x_i-c_i}{a_i}\right|+\sum_{|I_{j,n}|>0}\alpha_{j}\left(\|\mathbf{x}\|_{\ell^2(I_{j,n})}\right)\sum_{i\in I_{j,n}} \left|\frac{x_i}{a_i}\right|\\
&\le \frac{N}{2}+\frac{1}{2}\sum_{i=1}^{N}\left|\frac{x_i-c_i}{a_i}\right|^2+\frac{1}{2}\sum_{|I_{j,n}|>0}\alpha_{j}\left(\|\mathbf{x}\|_{\ell^2(I_{j,n})}\right)\left(|I_{j}|+\sum_{i\in I_{j,n}} \frac{x_i^2}{a_i^2}\right)\\
&=\frac{N}{2}+E_n+\frac{1}{2}\sum_{|I_{j,n}|>0}|I_j|\alpha_{j}\left(\|\mathbf{x}\|_{\ell^2(I_{j,n})}\right).
\end{aligned}\]
Using
\[
|I_j|\alpha_{j}\left(\|\mathbf{x}\|_{\ell^2(I_{j,n})}\right)=|I_j|\varepsilon_j
\]
when $\|\mathbf{x}\|_{\ell^2(I_{j,n})}\le s_j$, and \eqref{eq:imp} when $\|\mathbf{x}\|_{\ell^2(I_{j,n})}>s_j$, we have
\[\sum_{i=1}^n\left|\frac{F_{n,i}}{a_i}\right|\le\frac{N}{2}+\frac{17}{16}E_n+\sum_{j=1}^{\infty}|I_j|\varepsilon_j\le\frac{N+1}{2}+\frac{17}{16}E_n.\]
Take
\[M=\max\left\{\frac{2}{R},\,M_0,\,\frac{N+1}{2},\,\frac{17}{16}\right\},\]
and then the result follows from Lemma~\ref{lem:unifadm} (i).
\end{proof}


\subsubsection{Further discussion on the case of balls}\label{sec:ball}
Note that the proof of Proposition~\ref{prop:conv} does not work for $p\in[1,2)$. We will show that when $O=B_r$, neither of the operators in Problem~\ref{prob:tr} exists for any $p\in[1,2)$.

Clearly $B_r$ can be written as $O_g$ for $g$ as in Proposition~\ref{prop:conv}.
Let $B^n_r\triangleq\{(x_1,\dots,x_n)\in\dbR^n:\sum_{i=1}^{n}x_i^2<r^2\}\subset\dbR^n$ for each $n\in\dbN$ and $r>0$. Recall that $\mathcal{N}^n$ denotes the product measure $\prod_{i=1}^n\mathcal{N}_{a_i}$. We denote by $\mu_n$ the surface measure on $B^n_r$ as in \eqref{eq:surfmfin}.
We first study the relation between the integral of a cylindrical function on $B_r$ and integrals on $B^n_r$.

\begin{lemma}\label{lem:coarea}
   For every integer $m>n$, let $\rho_{m,n}\triangleq \rho_{n+1}*\dots * \rho_{m}$ be the convolution of
   \[\rho_j\colon\dbR\to\dbR,\quad \rho_j(t)=\begin{cases}
   \left(\sqrt{2\pi a_j^2 t}\right)^{-1}e^{-\frac{t}{2a_j^2}} &(t>0)\\
   0 &(t\le 0)
   \end{cases} \qquad(j=n+1,\dots,m).\]
   If $f$ is a bounded Borel function on a neighborhood of $\overline{B^n_r}$ in $\dbR^n$, then viewing $f$ as a cylindrical function on $\overline{B^m_r}$, we have
   \[
   \int_{\partial B^m_r}f \dd \mu_m=2r\int_{B^n_r}f(y)\cdot\rho_{m,n}(r^2-|y|^2)\dd \mathcal{N}^n.
   \]
\end{lemma}
\begin{proof}
   Let $X_j\sim N(0,a_j^2)$ $(j=n+1,\dots,m)$ be independent random variables. Note that $\rho_j$ is the density of $X_j^2$ and $\rho_{m,n}$ is the density of $\sum_{j={n+1}}^{m}X_j^2$ (cf.~\cite[Theorem~6.1.2, p.~145]{Chu74}). Write $(x_1,\dots,x_m)=(y,z)$, where $y=(y_1,\dots,y_n)\in\dbR^n$ and $z=(z_1,\dots,z_{m-n})\in\dbR^{m-n}$. Then by Fubini's theorem,
   \begin{equation}\label{eq:fub}
      \int_{\partial B^m_r} f\dd \mu_m=\int_{B^n_r}f(y)\cdot\frac{r}{\sqrt{r^2-|y|^2}}\left(\int_{|z|^2=r^2-|y|^2}\prod_{j=n+1}^{m}\frac{1}{\sqrt{2\pi a_j^2}} e^{-\frac{z_j^2}{2a_j^2}}\,\mathrm{d}S_z\,\right)\mathrm{d}P_y,
   \end{equation}
   where $\mathrm{d}S_z$ is the Euclidean surface measure on $\{|z|^2=r^2-|y|^2\}\subset\dbR^{m-n}$. The change of variables $z_1=\pm\sqrt{t-\sum_{i=2}^{m}z_i^2}$ yields
   \begin{equation}\label{eq:coar}
   \begin{aligned}
   \int_{0}^{r^2}\rho_{m,n}(t)\dd t&=\dbP\left(\sum_{j=n+1}^{m}X_j^2<r^2\right)=\int_{B_r^{m-n}}\prod_{j=n+1}^{m}\frac{1}{\sqrt{2\pi a_j^2}} e^{-\frac{z_j^2}{2a_j^2}}\dd z\\
   &=\int_{0}^{r^2}\int_{|z|^2=t}\frac{1}{2\sqrt{t}}\prod_{j=n+1}^{m}\frac{1}{\sqrt{2\pi a_j^2}} e^{-\frac{z_j^2}{2a_j^2}}\dd S_z \dd t\qquad(\forall r>0),
   \end{aligned}\end{equation}
   which implies
   \begin{equation}\label{eq:rhomn}
   \rho_{m,n}(t)=\frac{1}{2\sqrt{t}}\int_{|z|^2=t}\prod_{j=n+1}^{m}\frac{1}{\sqrt{2\pi a_j^2}} e^{-\frac{z_j^2}{2a_j^2}}\dd S_z\qquad(\forall t>0).
   \end{equation}
   The result follows from \eqref{eq:fub} and \eqref{eq:rhomn}.
\end{proof}

\begin{proposition}\label{prop:lim}
   Let $f$ be as in Lemma~\ref{lem:coarea}. Then
   \begin{equation}\label{eq:lim}
   \lim_{m\to\infty}\int_{B^m_r}f \dd \mathcal{N}^m= \int_{B_r} f \dd P\quad\text{and}\quad\lim_{m\to\infty}\int_{\partial B^m_r}f \dd \mu_m= \int_{\partial B_r} f \dd \mu_{\partial B_r}.
   \end{equation}
\end{proposition}
\begin{proof}
   The first equality in \eqref{eq:lim} follows from the Dominated Convergence Theorem applied to the functions $f\cdot \mathbf{1}_{B_r^{m}\times\ell^2(\dbN\setminus\{1,\dots,m\})}$ on $\ell^2$ since
   \[\int_{B^m_r}f \dd \mathcal{N}^m=\int_{B_r^{m}\times\ell^2(\dbN\setminus\{1,\dots,m\})}f\dd P\]
   and $\mathbf{1}_{B_r^{m}\times\ell^2(\dbN\setminus\{1,\dots,m\})}\to \mathbf{1}_{B_r}$ almost everywhere as $m\to\infty$.

   With the same notation as that in Lemma~\ref{lem:coarea}, the Fourier transforms of $\rho_j$ and $\rho_{m,n}$ are
   \[\hat{\rho}_j(t)=\left(1-2a_j^2t\sqrt{-1}\right)^{-\frac{1}{2}},\quad\hat{\rho}_{m,n}(t)=\prod_{j=n+1}^{m}\left(1-2a_j^2t\sqrt{-1}\right)^{-\frac{1}{2}},\]
   respectively. Since $\sum_{i=1}^{\infty} a_i^2<\infty$ and
   \[0\ge\log\left(\prod_{j=m}^{m'}\left|\left(1-2a_j^2t\sqrt{-1}\right)^{-\frac{1}{2}}\right|\right)=-\frac{1}{4}\sum_{j=m}^{m'}\log(1+4a_j^4t^2)\ge-\frac{1}{4}\sum_{j=m}^{m'}4a_j^4t^2,\]
   the function $\hat{\rho}_{m,n}(t)$ converges pointwise to $\hat{\rho}_n(t)\triangleq \prod_{j=n+1}^{\infty}\left(1-2a_j^2t\sqrt{-1}\right)^{-\frac{1}{2}}$ as $m\to\infty$. By the Fourier inversion theorem and the Dominated Convergence Theorem, $\rho_{m,n}$ converges uniformly to a continuous function $\rho_n$.

   For $\mathbf{x}=(x_i)_{i=1}^\infty\in \overline{B_r}$, write $\mathbf{x}=(y,\mathbf{z})$, where $y\in\dbR^n$ and $\mathbf{z}=(z_i)_{i=1}^\infty\in\ell^2(\dbN\setminus\{1,\dots,n\})$. Letting $m\to\infty$ in \eqref{eq:coar} and using the first equality in \eqref{eq:lim}, we have
   \[\int_{0}^{r^2}\rho_{n}(t)\dd t=\int_{B_r}\mathrm{d}P_{\mathbf{z}}=\int_{0}^{r^2}\int_{\|\mathbf{z}\|^2=t}\frac{1}{2\sqrt{t}}\dd \mu_{S_\mathbf{z}} \dd t\qquad(\forall r>0),\]
   where the second equality follows from the change of variables $z_1=\pm\sqrt{t-\sum_{i=2}^{\infty}z_i^2}$. Hence
   \begin{equation}\label{eq:rhon}
      \rho_{n}(t)=\frac{1}{2\sqrt{t}}\int_{\|\mathbf{z}\|^2=t}\mathrm{d}\mu_{S_\mathbf{z}}\qquad(\forall t>0).
   \end{equation}
   By \eqref{eq:surfmes}, Fubini's theorem, and the coordinate given by $x_{n+1}=\pm\sqrt{r^2-\sum_{i\neq n+1}x_i^2}$,
   \[\int_{\partial B_r} f\dd \mu=\int_{B^n_r}f(y)\cdot\frac{r}{\sqrt{r^2-|y|^2}}\left(\int_{\|\mathbf{z}\|^2=r^2-|y|^2}\mathrm{d}\mu_{S_\mathbf{z}}\,\right)\mathrm{d}P_y.\]
   Combining this with \eqref{eq:rhon} and Lemma~\ref{lem:coarea}, we have
   \[\begin{aligned}
   \int_{\partial B_r} f \dd \mu&=2r\int_{B^n_r}f(y)\cdot\rho_n(r^2-|y|^2)\dd P_y\\
   &=\lim_{m\to\infty}2r\int_{B^n_r}f(y)\cdot\rho_{m,n}(r^2-|y|^2)\dd P_y\\
   &=\lim_{m\to\infty}\int_{\partial B^m_r}f\dd \mu_m,
   \end{aligned}\]
   which completes the proof of \eqref{eq:lim}.
\end{proof}

Fix $\lambda>0$ such that $p\lambda>r$. For every $n\in\dbN$, define
\[u_{n}(x)=v_n(x_n)=\exp\left(\frac{\lambda(x_n-r)}{a_n^2}\right),\qquad\forall x\in \overline{B^n_r},\]
and view $u_{n}(x)$ as a cylindrical function on $\overline{B^m_r}$ for any $m\ge n$.
Then $0<u_n\le 1$ on $\overline{B^m_r}$.
Denote by $h_{m,n}$ the convolution of the following $m-1$ functions:
\[\rho_k\colon\dbR\to\dbR,\quad \rho_k(t)=\begin{cases}
   \left(\sqrt{2\pi a_k^2 t}\right)^{-1}e^{-\frac{t}{2a_k^2}} &(t>0)\\
   0 &(t\le 0)
   \end{cases} \qquad(1\le k\le m, \, k\neq n).\]

\begin{lemma}\label{lem:denbd}
	We have the following properties:
	\begin{enumerate}[label=(\roman*)]
		\item There exists $C_h>0$, independent of $m$ and $n$, such that $\|h_{m,n}\|_{L^\infty(\dbR)}\le C_h$ for all $n\ge 3$ and $m\ge n$.
	\item For every compact interval $K\subset(0,\infty)$, there exists $c_K>0$, independent of $m$ and $n$, such that $\inf_{K}h_{m,n}\ge c_K$ for all $n\ge 3$ and $m\ge n$.
	\end{enumerate}
\end{lemma}
\begin{proof}
	Using $\|u*v\|_{L^\infty(\dbR)}\le \|u\|_{L^\infty(\dbR)}\cdot \|v\|_{L^1(\dbR)}$ and
	\[(\rho_1*\rho_2)(s)\le\frac{1}{2\pi a_1a_2}\int_{0}^{s}\frac{\mathrm{d}t}{\sqrt{t(s-t)}}=\frac{1}{2a_1a_2},\]
	we obtain (i). Since $h_{m,n}$ is strictly positive and continuous on $(0,\infty)$, the assertion (ii) follows immediately.
\end{proof}

By an argument similar to the proof of Lemma~\ref{lem:coarea}, we have
\begin{equation}\label{eq:unbdry}
	\begin{aligned}
   \int_{\partial B^m_r}|u_n|^p\dd\mu_m&=2r\int_{-r}^{r}|v_n(t)|^p\frac{1}{\sqrt{2\pi a_n^2}}e^{-\frac{t^2}{2a_n^2}}h_{m,n}(r^2-t^2)\dd t\\
   &=2r\int_{-r}^{r}\frac{1}{\sqrt{2\pi a_n^2}}\exp\left(\frac{G(t)}{2a_n^2}\right)h_{m,n}(r^2-t^2)\dd t,
   \end{aligned}
\end{equation}
where $G(t)=2p\lambda(t-r)-t^2$.

\begin{lemma}\label{lem:surface-mass}
Fix $2^{-1/2}r<c_1<c_2<r$.  There exists a constant $C>0$, independent of $m$ and
$n$, such that for all sufficiently large $n$ and $m\ge n$,
\begin{equation}\label{eq:uplwbd}
 C^{-1}a_n\exp\left(\frac{G(c_2)}{2a_n^2}\right)\leq\int_{\partial B^m_r}|u_n|^p\dd\mu_m\leq Ca_n\exp\left(-\frac{r^2}{2a_n^2}\right)
\end{equation}
and
\begin{equation}\label{eq:concentration}
 \int_{\left\{x\in \partial B^m_r\,:\,x_n\geq c_1\right\}}|u_n|^p\dd\mu_m\geq\frac{1}{2}\int_{\partial B^m_r}|u_n|^p\dd\mu_m.
\end{equation}
\end{lemma}

\begin{proof}
By Lemma~\ref{lem:denbd} (i) and \eqref{eq:unbdry} ,
\[
\int_{\partial B^m_r}|u_n|^p\dd\mu_m\leq rC_h\sqrt{\frac{2}{\pi}}\int_{-r}^r\frac{1}{a_n}
 \exp\left(\frac{G(t)}{2a_n^2}\right)\dd t.
\]
Combining this with
\[
-\frac{r^2}{2a_n^2}-G(t)=\frac{1}{2a_n^2}(r-t)\left(2p\lambda-{r-t}\right)\geq\frac{(p\lambda-r)(r-t)}{a_n^2},\qquad t\in[-r,r]
\]
proves the upper bound in \eqref{eq:uplwbd}.

For the lower bound in \eqref{eq:uplwbd}, choose a compact interval $K\subset(0,\infty)$ containing
$r^2-c_2^2$.  For all large $n$, the condition
$t\in[c_2,c_2+a_n^2]$ ensures $r^2-t^2\in K$ and $t<r$. Thus, Lemma~\ref{lem:denbd}~(ii), \eqref{eq:unbdry} , and the monotonicity of $G(t)$ give
\[
 \int_{\partial B^m_r}|u_n|^p\dd\mu_m\geq r\tilde c_K\sqrt{\frac{2}{\pi}}\int_{c_2}^{c_2+a_n^2}\frac{1}{a_n}
 \exp\left(\frac{G(t)}{2a_n^2}\right)\dd t
 \geq C^{-1}a_n \exp\left(\frac{G(c_2)}{2a_n^2}\right),
\]
which completes the proof of \eqref{eq:uplwbd}.
Similarly to the proof of the upper bound in \eqref{eq:uplwbd}, we have
\[
 \int_{\left\{x\in \partial B^m_r\,:\,x_n<c_1\right\}}|u_n|^p\dd\mu_m
 \leq Ca_n \exp\left(\frac{G(c_1)}{2a_n^2}\right).
\]
Since $G(c_1)<G(c_2)$,
\[\frac{\int_{\left\{x\in \partial B^m_r\,:\,x_n< c_1\right\}}|u_n|^p\dd\mu_m}{\int_{\partial B^m_r}|u_n|^p\dd\mu_m}\longrightarrow 0\qquad(n\to\infty).\]
This proves \eqref{eq:concentration}.
\end{proof}

Write
\[H_{m,n}(s)\triangleq\int_{0}^{s}h_{m,n}(t)\dd t=\int_{\left\{\sum_{k\neq n}x_k^2<s\right\}} \prod_{1\le k\le m,\,k\neq n}\frac{1}{\sqrt{2\pi a_k^2}}e^{-\frac{x_k^2}{2a_k^2}}\,dx_1\dots\widehat{dx_n}\dots dx_m.\]
Then by Fubini's theorem,
\begin{equation}\label{eq:un}
\begin{aligned}
	\int_{B^m_r}|u_n|^p\dd\mathcal{N}^m&=\int_{-r}^{r}|v_n(t)|^p\frac{1}{\sqrt{2\pi a_n^2}}e^{-\frac{t^2}{2a_n^2}}H_{m,n}(r^2-t^2)\dd t\\
	  &=\int_{-r}^{r}\frac{1}{\sqrt{2\pi a_n^2}}\exp\left(\frac{G(t)}{2a_n^2}\right)H_{m,n}(r^2-t^2)\dd t.
\end{aligned}
\end{equation}

\begin{proposition}\label{prop:surfaceineq}
In the setting of Lemma~\ref{lem:surface-mass}, there exists $C>0$, independent of $n$, such that for all sufficiently large $n$,
\begin{equation*}
 \|u_n\|_{W^{1,p}(B_r,P)}+\|u_n\|_{\mathscr{W}^{1,p}(B_r,P)}\leq C a_n^{(2-p)/p}\|u_n\|_{L^p(\partial B_r,\mu_{\partial B_r})}.
\end{equation*}
In particular, when $p\in[1,2)$, neither of the operators in Problem~\ref{prob:tr} exists.
\end{proposition}

\begin{proof}
	We use the following definition of $\lVert \cdot \rVert_{\mathscr{W}^{m,p}( {O},P)}$ which is equivalent to the norm in Definition~\ref{def:Sob2}:
	$$\| f \|_{\mathscr{W}^{m,p}( {O},P)} =\left(\int_{{O}} \left| f \right|^{p} \mathrm{d} P+\int_{O}\left(\sum_{0<|\alpha|\le m}
   \left|\mathbf{a}^{\alpha}\partial^{\alpha}f\right|^{2}\right)^{\frac{p}{2}}
  \mathrm{d}P\right)^{\frac{1}{p}},\qquad\forall f\in \mathscr{W}^{m,p}( {O},P).$$
	Since the only nonzero first order partial derivative of $u_n$ is $\partial_{x_n}u_n=\lambda u_n/a_n^2$, we have
	\[\|u_n\|_{W^{1,p}(B_r,P)}=\|u_n\|_{\mathscr{W}^{1,p}(B_r,P)}=\left[\left(1+\lambda^pa_n^{-p}\right)\int_{B_r}|u_n|^p\dd P\right]^{1/p}.\]
By Proposition~\ref{prop:lim}, it suffices to find a constant $C>0$, independent of $m$ and $n$, such that for all sufficiently large $n$ and $m\ge n$,
	\begin{equation}\label{eq:IM}
 \int_{B^m_r}|u_n|^p\dd\mathcal{N}^m\leq Ca_n^2\int_{\partial B^m_r}|u_n|^p\dd\mu_m.
\end{equation}
Split \eqref{eq:un} as
\[\int_{B^m_r}|u_n|^p\dd\mathcal{N}^m=I_-+I_+,\]
where $I_-$ (\resp $I_+$) is the integral from $-r$ to $c_1$ (\resp from $c_1$ to $r$).
On $[-r,c_1]$, by an argument similar to that for the upper bound of \eqref{eq:uplwbd},
\[
 I_-\leq Ca_n \exp\left(\frac{G(c_1)}{2a_n^2}\right)\leq C^2 \exp\left(\frac{G(c_1)-G(c_2)}{2a_n^2}\right)\int_{\partial B^m_r}|u_n|^p\dd\mu_m\le \widetilde{C}a_n^2\int_{\partial B^m_r}|u_n|^p\dd\mu_m,
\]
where the second inequality follows from the lower bound in \eqref{eq:uplwbd} and the last inequality follows from $G(c_1)-G(c_2)<0$.

For $I_+$, note that almost everywhere,
\[
 \left(H_{m,n}(r^2-t^2)\right)'=-2t h_{m,n}(r^2-t^2).
\]
Write
\[
 M_+=2r\int_{c_1}^{r}F(t)h_{m,n}(r^2-t^2)\dd t,\qquad F(t)=\frac{1}{\sqrt{2\pi a_n^2}}\exp\left(\frac{G(t)}{2a_n^2}\right).\]
Then integration by parts yields
\begin{equation}
 M_+=\frac{rF(t)}{t}H_{m,n}(r^2-c_1^2)
 +r\int_{c_1}^{r}
 \left(\frac{F(t)}{t}\right)'H_{m,n}(r^2-t^2)\dd t.
 \label{eq:ibp}
\end{equation}
Furthermore,
\[
 \left(\frac{F(t)}{t}\right)'
 =\frac{F(t)}{t}
 \left(\frac{p\lambda-t}{a_n^2}-\frac{1}{t}\right)
 \geq\frac{p\lambda-r}{2a_n^2}F(t)
\]
on $[c_1,r]$ for all large $n$. This, together with \eqref{eq:ibp} and \eqref{eq:unbdry}  implies
\[I_+\leq Ca_n^2M_+\leq Ca_n^2\int_{\partial B^m_r}|u_n|^p\dd\mu_m,\]
which proves \eqref{eq:IM}.
\end{proof}

\begin{remark}\label{rmk:surfmes}
   If one uses the surface measure $\varrho$ of Feyel and de~La~Pradelle \cite{FdLP92} instead of $\mu_{\partial B_r}$ in our infinite-dimensional setting, then the two operators in Problem~\ref{prob:tr} exist for all $p\in[1,\infty)$. Indeed,
   \[\mathrm{d}\varrho=\frac{\sqrt{\sum_{i=1}^{\infty}a_i^2x_i^2}}{r}\mathrm{d}\mu_{\partial B_r}\qquad \text{on }\partial B_r.\]
   Theorem~\ref{thm:tr} and Lemma~\ref{lem:unifadm} (ii) can be applied to $\varrho$ after suitable modifications.
   Moreover, in \cite[Proposition~5.7]{FdLP92}, the trace problem for $B_r$ is reduced to the inequality \cite[(5.9)]{FdLP92} which holds if $B_r$ is a $\mathscr{W}^{1,p}$-extension domain and thus proved by Theorem~\ref{thm:ext} combined with Proposition~\ref{prop:ellipsoid}.
\end{remark}

\subsection{Examples of infinite surface measures}\label{subsec:infsurf}

In this section, we provide examples of bounded open sets $O$ with $C^\infty$ boundaries such that $\mu_{\partial O}(\partial O)=+\infty$. In particular, for such open sets, the operators in Problem~\ref{prob:tr} do not exist even when $p=1$.

Define the linear operator $\mathcal Q$ on $\ell^2$ by
\[
\mathcal{Q}\mathbf{e}_i=a_i^2\mathbf{e}_i,\qquad i\in\mathbb{N}.
\]
Then $\mathcal{Q}$ is positive, injective, and trace class.

Decompose the space as
\[
\ell^2=\text{span}\{\mathbf{e}_1\}\oplus Y,
\qquad
Y=\mathbf{e}_1^\perp.
\]
Write
$\mathcal{Q}_0=\mathcal{Q}|_Y$. For every $\mathbf{x}=t\mathbf{e}_1+\mathbf{y}$, where $t\in\mathbb{R}$ and $\mathbf{y}\in Y$,  write $\mathbf{x}=(t,\mathbf{y})$. Then
\[
P=\mathcal{N}_{a_1}\times P^{\widehat{1}},
\]
where $P^{\widehat{1}}$ denotes the product measure obtained by omitting the first components; i.e., it is the restriction of the product measure $\prod\limits_{j\in\mathbb{N}\setminus\{1\}}\mathcal{N}_{a_j}$ to the space
$
\left(Y,\mathscr{B}\big(Y\big)\right).
$
We write
\[
G_a(s)\triangleq\frac{1}{\sqrt{2\pi a^2}}\exp\left(-\frac{s^2}{2a^2}\right),\qquad a\in(0,\infty),\,s\in\mathbb{R}.
\]


Consider
\[\Gamma_f(V)=\{(f(\mathbf{y}),\mathbf{y}):\mathbf{y}\in V\},\]
where $V$ is a nonempty open subset of $Y$ and $f\in C^1_{Y}(V)$. Clearly $\Gamma_f(V)$ is a $C^1$-differentiable surface of $\ell^2$ with codimension $1$ as in Definition~\ref{20241121def1}. By Definition~\ref{20241012def2},
\begin{equation}\label{20260813for1}
\mu_{\Gamma_f(V)}\left(\Gamma_f(B)\right)=\int_BG_{a_1}(f(\mathbf{y}))\sqrt{1+\left\| Df(\mathbf{y})\right\|_Y^2}\,\mathrm{d}P^{\widehat{1}}(\mathbf{y}),
\end{equation}
where the equality is allowed to take the value $\infty$, $B$ is Borel subset of $V$ and
$$
\left\| Df(\mathbf{y})\right\|_Y^2=\sum_{i\neq 1}|\partial_{x_i}f(\mathbf{y})|^2.
$$


\subsubsection{A general construction}

For any $\mathbf{x}\in \ell^2$, $\mathbf{y}\in Y$ and $r\in(0,\infty)$, let
$$
B_Y(\mathbf{y},r)\triangleq \left\{\mathbf{y}'\in Y:\|\mathbf{y}'-\mathbf{y}\|<r\right\}.
$$
Fix $b,r_0>0$ with $b>\sqrt{2}r_0$ and
\[R_0\triangleq b+\frac{r_0}{2}<1.\]
For $n\in \mathbb{N}$, define
\[
\mathbf{c}_n=b\mathbf{e}_{n+1}\in Y,\qquad
U_n=B_Y\left(\mathbf{c}_n,r_0\right),\qquad
V_n=B_Y\left(\mathbf{c}_n,\frac{r_0}{2}\right).
\]

Since
\[
\|\mathbf{c}_n-\mathbf{c}_m\|=\sqrt{2}b>2r_0\qquad (n\neq m),
\]
the closed balls $\overline{U_n}$ are pairwise separated by a
positive distance. Moreover,
\[
U_n\subset B_Y(\mathbf{0},b+r_0),\qquad
V_n\subset B_Y\left(\mathbf{0},b+\frac{r_0}{2}\right)\subset B_Y(\mathbf{0},1).
\]
By \cite[Lemma 2.1, p. 9]{WYZ}, since every $V_n$ is a nonempty open set in $Y$, we have
\[
p_n\triangleq P^{\widehat{1}}(V_n)>0.
\]

Choose $\eta\in C_c^\infty([0,\infty))$
such that
\[
0\leq \eta\leq 1,\qquad
\eta(s)=1 \quad (0\leq s\leq 1/2),\qquad
\eta(s)=0 \quad (s\geq 1).
\]
Define
\begin{equation*}
g(\mathbf{y})\triangleq\sum_{n=1}^{\infty}\eta\left(\frac{\|\mathbf{y}-\mathbf{c}_n\|^2}{r_0^2}\right)\xi_n(y_{n+1}),\qquad \forall\mathbf{y}=(y_j)_{j=2}^\infty\in Y,
\end{equation*}
where $\xi_n\in C^\infty(\dbR)$ will be specified later.
Since $g^{-1}(\mathbb{R}\setminus\{0\})\subset \bigcup_{n=1}^{\infty}\overline{U_n}$ and $\{\overline{U_n}\}_{n=1}^{\infty}$ are pairwise separated, in a
neighborhood of each point the sum contains at most one nonzero
term. Therefore $g$ is well-defined and $g\in C_Y^\infty(Y)$.

Define
\[
\Phi:\text{span}\{\mathbf{e}_1\}\oplus Y\longrightarrow\text{span}\{\mathbf{e}_1\}\oplus Y,\qquad
\Phi(s,\mathbf{y})=(s+g(\mathbf{y}),\mathbf{y}).
\]
It is a $C^\infty$ diffeomorphism, with inverse
\[
\Phi^{-1}(t,\mathbf{y})=(t-g(\mathbf{y}),\mathbf{y}).
\]
Set
\[
M\triangleq\frac{\max_{j>1}\{a_j\}\cdot R_0}{\sqrt{1-R_0^2}}.
\]

Let
\begin{equation*}
{O}\triangleq\Phi(B_1)=\left\{(t,\mathbf{y})\in\text{span}\{\mathbf{e}_1\}\oplus Y:(t-g(\mathbf{y}))^2+\|\mathbf{y}\|_Y^2<1\right\}.
\end{equation*}

\begin{theorem}\label{20260813thm1}
Suppose that the following conditions hold:
\begin{enumerate}[label=(\arabic*)]
   \item there exists $L>0$ such that $|\xi_n(y_{n+1})|\le L$ for all $n\in\dbN$ and $\mathbf{y}=(y_j)_{j=2}^\infty\in U_n$;
   \item there exists a sequence $\{d_n\}_{n\in\dbN}\subset\dbR$ with $\sum_{n=1}^{\infty}d_n=+\infty$ such that
   \begin{equation}\label{eq:infsurf}
   a_{n+1}\int_{V_n}|\xi_n'(y_{n+1})|\dd P^{\widehat{1}}(\mathbf{y})\ge d_n+Mp_n,\qquad\forall n\in\dbN.
   \end{equation}
\end{enumerate}
Then ${O}$ is a bounded open subset of $\ell^2$ with $C^\infty$ boundary, and
\[
\mu_{\partial{O}}(\partial{O})=+\infty.
\]
\end{theorem}

\begin{proof}
Since $\Phi$ is a diffeomorphism, ${O}$ is open, and
\[
\partial{O}=\Phi\left(\partial B_1\right)=\left\{(t,\mathbf{y})\in\text{span}\{\mathbf{e}_1\}\oplus Y:(t-g(\mathbf{y}))^2+\|\mathbf{y}\|_Y^2=1\right\}.
\]
Therefore $\partial{O}$ is a $C^\infty$ differentiable surface.
The construction of $g$ and the condition~(1) imply
$|g(\mathbf y)|\leq L$ for every $\mathbf y\in Y$.
If $(t,\mathbf y)\in O$, then $\|\mathbf y\|_Y<1$ and
\[
|t|<|g(\mathbf y)|+\sqrt{1-\|\mathbf y\|_Y^2}\leq L+1.
\]
Thus ${O}$ is bounded.

It suffices to consider only the upper half of the boundary. Let
\[
h(\mathbf{y})\triangleq\sqrt{1-\|\mathbf{y}\|_Y^2},\qquad
f(\mathbf{y})\triangleq g(\mathbf{y})+h(\mathbf{y})\qquad(\|\mathbf{y}\|_Y<1,\,\mathbf{y}\in Y).
\]
Then
\[
\Gamma_+\triangleq\{(f(\mathbf{y}),\mathbf{y}):\|\mathbf{y}\|_Y<1,\,\mathbf{y}\in Y\}\subset\partial{O}.
\]

For $\mathbf{y}\in V_n$, we have
\[
\|\mathbf{y}\|_Y< \|\mathbf{c}_n\|+\frac{r_0}{2} =R_0,
\]
and
\[
Dh(\mathbf{y})=-\frac{\mathbf{y}}{\sqrt{1-\|\mathbf{y}\|_Y^2}}.
\]
It follows from
\[
\left\|\mathcal{Q}_0^{1/2}\right\|_{\mathcal{L}(Y;Y)}=\max_{j>1}\{a_j\}
\]
that
\begin{equation*}
\left\|\mathcal{Q}_0^{1/2}Dh(\mathbf{y})\right\|\leq\frac{\max_{j>1}\{a_j\}\cdot R_0}{\sqrt{1-R_0^2}}=M,\qquad
\forall\mathbf{y}\in V_n.
\end{equation*}
Combining this with the triangle inequality and \eqref{eq:infsurf}, we obtain
\[\begin{aligned}
\int_{V_n}\left\|\mathcal{Q}_0^{1/2}Df(\mathbf{y})\right\|\,\mathrm{d}P^{\widehat{1}}(\mathbf{y})
&\geq\int_{V_n}\left\|\mathcal{Q}_0^{1/2}Dg(\mathbf{y})\right\|\,\mathrm{d}P^{\widehat{1}}(\mathbf{y})-\int_{V_n}\left\|\mathcal{Q}_0^{1/2}Dh(\mathbf{y})\right\|\,\mathrm{d}P^{\widehat{1}}(\mathbf{y})\\
&\geq a_{n+1}\int_{V_n}|\xi_n'(y_{n+1})|\dd P^{\widehat{1}}(\mathbf{y})-Mp_n\\
&\geq d_n.
\end{aligned}\]

On the other hand, since $|g|\leq L$ and $0<h\leq 1$,
we have $|f(\mathbf{y})|\leq 1+L$ for any $\mathbf{y}\in V_n$. Set
\[
a\triangleq\max_{j>1}\{a_j\},\qquad c_*\triangleq\min_{|s|\leq 1+L}G_{a_1}(s)>0.
\]
Since the sets $\{V_n\}_{n=1}^{\infty}$ are pairwise disjoint, using the graph-area
 formula~\eqref{20260813for1}, 
we obtain
\[\begin{aligned}\mu_{\partial{O}}(\partial{O})&\geq\mu_{\partial{O}}\left(\left\{(f(\mathbf{y}),\mathbf{y}):\mathbf{y}\in\bigcup_{n\geq 1}V_n\right\}\right)\\
&=\sum_{n=1}^{\infty}\int_{V_n}G_{a_1}(f(\mathbf{y}))\sqrt{1+\left\| Df(\mathbf{y})\right\|^2}\,\mathrm{d}P^{\widehat{1}}(\mathbf{y})\\
&\geq \frac{1}{a}\sum_{n=1}^{\infty}\int_{V_n}G_{a_1}(f(\mathbf{y}))\sqrt{a^2+\left\|\mathcal{Q}_0^{1/2}Df(\mathbf{y})\right\|^2}\,\mathrm{d}P^{\widehat{1}}(\mathbf{y})\\
&\geq\frac{ c_* }{a}\sum_{n=1}^{\infty}\int_{V_n}\left\|\mathcal{Q}_0^{1/2}Df(\mathbf{y})\right\|\,\mathrm{d}P^{\widehat{1}}(\mathbf{y})\\
&\geq\frac{ c_* }{a}\sum_{n=1}^{\infty}d_n=+\infty.
\end{aligned}\]
This proves Theorem \ref{20260813thm1}.
\end{proof}

\subsubsection{An explicit example}

Define
\[
L_n(\mathbf{y})\triangleq \langle \mathbf{y}-\mathbf{c}_n,\mathbf{e}_{n+1}\rangle_{\ell^2},\qquad \mathbf{y}\in Y,\,\,n\in\mathbb{N}.
\]
The following fact will be used to prove the existence of $\xi_n$ satisfying the conditions in Theorem~\ref{20260813thm1}.

\begin{lemma}\label{20260813lem1}
For every fixed $n$,
\begin{equation}\label{eq:perav}
\lim_{k\to\infty}\int_{V_n}\left|\cos(kL_n(\mathbf{y}))\right|\,\mathrm{d}P^{\widehat{1}}(\mathbf{y})=\frac{2}{\pi}p_n.
\end{equation}
\end{lemma}
\begin{proof}
Write $Y=\operatorname{span}\{\mathbf e_{n+1}\}\oplus Z$, where $\operatorname{span}\{\mathbf e_{n+1}\}\perp Z$,
so every $\mathbf y\in Y$ can be written uniquely as
\[
\mathbf y=s \mathbf e_{n+1}+\mathbf z,\qquad \mathbf z\in Z.
\]
Then $L_n(\mathbf y)=s-b$.


Now consider the finite measure $\nu_n$ on $\dbR$ defined by
\[
\nu_n(B)\triangleq P^{\widehat 1}\bigl(V_n\cap L_n^{-1}(B)\bigr),\qquad B\in \mathscr B(\mathbb R).
\]
Let $P^{\widehat{1,n+1}}$ denote the product Gaussian measure on $Z$
obtained by omitting the first and ($n+1$)-th components of $P$.
Since $s=t+b$ when $t=L_n(\mathbf y)$,
\[
\nu_n(B)=\int_B G_{a_{n+1}}\!\left(t+b\right)\left(\int_{Z}\mathbf{1}_{V_n}\left(\left(t+b\right)\mathbf e_{n+1}+\mathbf z\right)\,\mathrm{d}P^{\widehat{1,n+1}}(\mathbf z)\right)\,\mathrm{d}t,
\]
and the density is explicitly
\[
w_n(t)\triangleq G_{a_{n+1}}\!\left(t+b\right)\,q_n(t)
\]
where
\[
q_n(t)\triangleq P^{\widehat{1,n+1}}\!\left(\left\{\mathbf z\in Z:\left(t+b\right)\mathbf e_{n+1}+\mathbf z\in V_n\right\}\right).
\]
Observe that $0\le q_n(t)\le1.$ In fact, since $V_n=B_Y(\mathbf c_n,r_0/2)$,
the condition
\[
\left(t+b\right)\mathbf e_{n+1}+\mathbf z\in V_n
\]
is equivalent to
\[
\|t \mathbf e_{n+1}+\mathbf z\|^2=t^2+\|\mathbf z\|^2<(r_0/2)^2.
\]
More explicitly,
\[
q_n(t)=
\begin{cases}
P^{\widehat{1,n+1}}\!\left(\left\{\mathbf z\in Z:
\|\mathbf z\|<\sqrt{(r_0/2)^2-t^2}\right\}\right),
& |t|<r_0/2,\\[4pt]
0,
& |t|\ge r_0/2.
\end{cases}
\]

In particular,
\[
0\leq w_n(t)\leq G_{a_{n+1}}(t+b), \qquad t\in\mathbb R,
\]
so $w_n\in L^1(\mathbb R)$.
Moreover,
\begin{equation}\label{20260827for1}
\int_{\mathbb R}w_n(t)\,\mathrm{d}t=P^{\widehat 1}(V_n)=p_n.
\end{equation}
By the change of variables formula (\cite[Exercise 1.4.37, p. 85]{Tao11}), we obtain
\[
\int_{V_n}|\cos(kL_n(\mathbf y))|\,\mathrm{d}P^{\widehat 1}(\mathbf y)
=\int_{\mathbb R}|\cos(kt)|\,w_n(t)\,\mathrm{d}t.
\]
Since $w_n\in L^1(\mathbb R)$ and
$w_n\geq0$, it suffices to prove that
\[
\lim_{k\to\infty}
\int_{\mathbb R}|\cos(kt)|w_n(t)\,\mathrm{d}t
=
\frac{2}{\pi}\int_{\mathbb R}w_n(t)\,\mathrm{d}t.
\]

Set $F(t)\triangleq|\!\cos t|$ for $t\in\mathbb R$.
Then $F$ is a continuous $\pi$-periodic function with mean value
\[
\frac{1}{\pi}\int_0^\pi F(t)\,\mathrm{d}t
=
\frac{2}{\pi}.
\]
By the Fej\'er theorem (\cite[Theorem 4.25, p. 91]{Rud87}), the Ces\`aro means of the Fourier series of
$F$ form a sequence of trigonometric polynomials
\[
T_N(t)
=
\frac{2}{\pi}
+
\sum_{m=1}^{N}
\left(
a_{m,N}\cos(2mt)+b_{m,N}\sin(2mt)
\right),\qquad N\in\mathbb{N},\,t\in\mathbb{R},
\]
which converges uniformly to $F$ on $\mathbb R$. In particular,
\[
\|F-T_N\|_{L^\infty(\mathbb R)}
\longrightarrow0
\qquad\text{as }N\to\infty.
\]

Fix $\varepsilon>0$. Choose $N$ sufficiently large such that
$\|F-T_N\|_{L^\infty(\mathbb R)}<\varepsilon$.
Then, for every $k>0$,
\begin{align*}
\left|\int_{\mathbb R}F(kt)w_n(t)\,\mathrm{d}t-\int_{\mathbb R}T_N(kt)w_n(t)\,\mathrm{d}t\right|\leq
\|F-T_N\|_{L^\infty(\mathbb R)}
\int_{\mathbb R}w_n(t)\,\mathrm{d}t
\leq
\varepsilon p_n.
\end{align*}

For the trigonometric polynomial $T_N$, we have
\[\begin{aligned}
\int_{\mathbb R}T_N(kt)w_n(t)\,\mathrm{d}t
={}&\frac{2}{\pi}\int_{\mathbb R}w_n(t)\,\mathrm{d}t
+\sum_{m=1}^{N}a_{m,N}\int_{\mathbb R}\cos(2mkt)w_n(t)\,\mathrm{d}t\\
&+\sum_{m=1}^{N}b_{m,N}\int_{\mathbb R}\sin(2mkt)w_n(t)\,\mathrm{d}t.
\end{aligned}\]
Since $w_n\in L^1(\mathbb R)$, the Riemann--Lebesgue lemma (\cite[p. 103]{Rud87}) yields
\[
\int_{\mathbb R}e^{2\sqrt{-1}mkt}w_n(t)\,\mathrm{d}t
\longrightarrow0,
\qquad\text{as }k\to\infty,
\]
for every $m\geq1$. Hence
\[
\lim_{k\to\infty}\int_{\mathbb R}\cos(2mkt)w_n(t)\,\mathrm{d}t
=\lim_{k\to\infty}\int_{\mathbb R}\sin(2mkt)w_n(t)\,\mathrm{d}t=0,
\]
which implies
\[
\lim_{k\to\infty}
\int_{\mathbb R}T_N(kt)w_n(t)\,\mathrm{d}t
=\frac{2}{\pi}\int_{\mathbb R}w_n(t)\,\mathrm{d}t.
\]

Consequently,
\[\begin{aligned}
\limsup_{k\to\infty}
\left|\int_{\mathbb R}F(kt)w_n(t)\,\mathrm{d}t
-\frac{2}{\pi}\int_{\mathbb R}w_n(t)\,\mathrm{d}t\right|
\leq\varepsilon p_n.
\end{aligned}\]
Since $\varepsilon>0$ is arbitrary, it follows that
\[
\lim_{k\to\infty}
\int_{\mathbb R}|\cos(kt)|w_n(t)\,\mathrm{d}t
=
\frac{2}{\pi}\int_{\mathbb R}w_n(t)\,\mathrm{d}t.
\]
By \eqref{20260827for1}, we obtain
\[
\lim_{k\to\infty}
\int_{\mathbb R}|\cos(kt)|w_n(t)\,\mathrm{d}t
=
\frac{2}{\pi}p_n.
\]
This completes the proof of Lemma \ref{20260813lem1}
\end{proof}

By Lemma \ref{20260813lem1}, for every $n\in\mathbb{N}$, we can choose positive integers $k_n$ such that
\begin{equation*}\label{eq:exeq1}
\int_{V_n}\left|\cos(k_nL_n(\mathbf{y}))\right|\,\mathrm{d}P^{\widehat{1}}(\mathbf{y})\geq\frac{p_n}{\pi},
\end{equation*}
and
\begin{equation*}\label{eq:exeq2}
\frac{a_{n+1} k_n \,p_n}{\pi}\geq 1+Mp_n.
\end{equation*}

Then we can take
\[\xi_n(t)=\sin\left(k_n\left(t-b\right)\right).\]
Since $|\xi_n(t)|\leq1$ for all $t\in\mathbb R$, the condition~(1) of Theorem~\ref{20260813thm1} holds with $L=1$. Moreover,
\[
\begin{aligned}
a_{n+1}\int_{V_n}|\xi_n'(y_{n+1})|\,\mathrm dP^{\widehat1}(\mathbf y)
&=a_{n+1} k_n\int_{V_n}|\cos(k_nL_n(\mathbf y))|\,\mathrm dP^{\widehat1}(\mathbf y)\\
&\geq \frac{a_{n+1} k_np_n}{\pi}\geq 1+Mp_n.
\end{aligned}
\]
Thus the condition~(2) holds with $d_n=1$ for every $n$.
Theorem~\ref{20260813thm1} therefore yields a bounded open set
$O$ with $C^\infty$ boundary and
$\mu_{\partial O}(\partial O)=+\infty$.

\section{Equivalence of Sobolev norms}\label{sec:norm}

Throughout this section\footnote{The material in this section is a rewritten version of that in arXiv:2608.24702}, unless stated otherwise, $p\in[1,\infty)$.

Fix a nonempty open convex set $\Omega\subset\dbR^n$.  By a slight abuse of notation, we denote by $\mathbf e_1,\dots,\mathbf e_n$ the standard basis of $\dbR^n$. Let $\mathbf e_1^*,\dots,\mathbf e_n^*$ be the corresponding dual basis.  Put
\[
\Ls{0}(\dbR^n)\triangleq\dbR,\qquad
\langle c,d\rangle_0\triangleq cd,\qquad
\|c\|_0\triangleq |c|\qquad(c,d\in\Ls{0}(\dbR^n)).
\]
For $k\in\dbN$, let $\Ls{k}(\dbR^n)$ be the vector space of symmetric $k$-linear forms on $\dbR^n$.  If
\[
\varphi=
\sum_{i_1,\dots,i_k=1}^n
c_{i_1\dots i_k}\,
\mathbf e_{i_1}^*\otimes\dots\otimes\mathbf e_{i_k}^*,
\]
and
\[
\psi=
\sum_{i_1,\dots,i_k=1}^n
 d_{i_1\dots i_k}\,
\mathbf e_{i_1}^*\otimes\dots\otimes\mathbf e_{i_k}^*,
\]
define
\[
\langle \varphi,\psi\rangle_k
\triangleq
\sum_{i_1,\dots,i_k=1}^n
c_{i_1\dots i_k}d_{i_1\dots i_k},
\qquad
\|\varphi\|_k
\triangleq
\left(
\sum_{i_1,\dots,i_k=1}^n c_{i_1\dots i_k}^2
\right)^{1/2}.
\]
Thus $\Ls{k}(\dbR^n)$ is a finite-dimensional Hilbert space.  We write $\mathcal P(\dbR^n)$ for the set of all real polynomials on $\dbR^n$.

For $k\in\dbN$ and $(i_1,\dots,i_k)\in\{1,\dots,n\}^k$, we use
\[
\partial_{i_1\dots i_k}\varphi
\triangleq
\frac{\partial^k \varphi}{\partial x_{i_1}\dots\partial x_{i_k}}.
\]
The standard Gaussian measure on $\dbR^n$ is
\begin{equation}\label{eq:standard Gaussian}
\mathrm{d}\gamma_n(\mathbf x)
\triangleq
(2\pi)^{-n/2}e^{-|\mathbf x|^2/2}\dd\mathbf x.
\end{equation}
For a sufficiently smooth real-valued function $f$ on $\dbR^n$, define
\[
\mathcal{J}_0(f)
\triangleq
\left(\int_{\Omega}|f|^p\dd\gamma_n\right)^{1/p},
\]
and, for $k\in\dbN$,
\[
\mathcal{J}_k(f)
\triangleq
\left[
\int_{\Omega}
\left(
\sum_{i_1,\dots,i_k=1}^n
|\partial_{i_1\dots i_k}f|^2
\right)^{p/2}
 \dd\gamma_n
\right]^{1/p}.
\]
Equivalently, $\mathcal{J}_k(f)=\|\,\|D^kf\|_k\,\|_{L^p(\Omega,\gamma_n)}$.

\subsection{Dimension-free Sobolev estimates on convex  domains}\label{subsec:normdimfree}
Throughout this subsection, $D^0f=f$, and for $k\in\dbN$ the symbol $D^kf$ denotes the $k$-th Fr\'echet derivative. 
For $f\in C^{k+1}(\dbR^n)$, let
\[\partial_iD^kf(\mathbf{x})\triangleq D(D^kf)(\mathbf{x})\,\mathbf e_i=\left.\frac{d}{dt}\right|_{t=0}D^kf(\mathbf{x}+t\mathbf{e}_i)\qquad (\mathbf{x}\in\dbR^n),\]
regarded as an element of $\Ls{k}(\dbR^n)$.

\begin{lemma}\label{qi}
Let $k\in\dbN$ and $f\in C^{k+1}(\dbR^n)$. Then for every $\mathbf{x}\in\dbR^n$,
\begin{equation*}
\|D^{k+1}f(\mathbf{x})\|_{k+1}
=
\left(\sum_{i=1}^{n}\|\partial_iD^kf(\mathbf{x})\|_k^2\right)^{1/2}.
\end{equation*}
\end{lemma}

\begin{proof}
Writing
\[
D^kf(\mathbf{x})
=
\sum_{i_1,\dots,i_k=1}^n
\partial_{i_1\dots i_k}f(\mathbf{x})\,
\mathbf e_{i_1}^*\otimes\dots\otimes\mathbf e_{i_k}^*,
\]
we obtain
\[
\begin{aligned}
\sum_{i=1}^n\|\partial_iD^kf(\mathbf{x})\|_k^2
&=\sum_{i=1}^n\sum_{i_1,\dots,i_k=1}^n
\left|\partial_i\partial_{i_1\dots i_k}f(\mathbf{x})\right|^2\\
&=\|D^{k+1}f(\mathbf{x})\|_{k+1}^2.
\end{aligned}
\]
Taking square roots proves Lemma~\ref{qi}.
\end{proof}

The next result transfers the vector-valued Gaussian Poincar\'e inequality to an open convex  domain.  Its constant is independent of the dimension.

\begin{theorem}\label{poincare}
Let $(\mathcal{H},\langle\cdot,\cdot\rangle_\mathcal{H})$ be a finite-dimensional real Hilbert space.  Suppose that $F\in C_{\dbR^n}^\infty(\Omega;\mathcal{H})$, and that $\|F\|_\mathcal{H}$ and
$\| \partial_{i} F \|_{\mathcal{H}}$
have at most polynomial growth for all $i\in \left\{ 1,2,\dots ,n \right\}$.
Then
\begin{equation}\label{1}
\int_{\Omega}
\left\lVert
F-
\frac{1}{\gamma_n(\Omega)}\int_{\Omega}F\dd\gamma_n\right\rVert_\mathcal{H}^p\dd\gamma_n
\le
k_p\int_{\Omega}\left(\sum_{i=1}^n\|\partial_iF\|_\mathcal{H}^2\right)^{p/2}\dd\gamma_n,
\end{equation}
where
\begin{equation}
k_p\triangleq
\begin{cases}
(p-1)^{p/2},&p\in[2,\infty),\\[2mm]
(\pi/2)^p,&p\in[1,2).
\end{cases}
\label{2}
\end{equation}
\end{theorem}

\begin{proof}
Let
\[
\dd\mu_n
\triangleq
\frac{\mathbf{1}_{\Omega}}{\gamma_n(\Omega)}\dd\gamma_n.
\]
Equivalently,
\[
\dd\mu_n=e^{-V_n}\dd\gamma_n,
\qquad
V_n\triangleq-\log\mathbf{1}_{\Omega}+\log\gamma_n(\Omega).
\]
Since $\Omega$ is convex, $V_n:\dbR^n\to\dbR\cup\{+\infty\}$ is convex. By Caffarelli's contraction theorem \cite[Theorem~11]{Caf00}, in the formulation of \cite[Theorem~2.2, p.~440]{GozlanJuillet}, there exists a $1$-Lipschitz transport map $T:\dbR^n\to\dbR^n$ such that
\[
T_{\#}\gamma_n=\mu_n,
\]
where $T_{\#}$ denotes the pushforward of measure under $T$.

We apply the vector-valued Gaussian Poincar\'e inequality \cite[Theorem~2.7, p.~11]{AMR22} to $F\circ T$.  The composition $F\circ T$ is interpreted $\gamma_n$-almost everywhere since $T_{\#}\gamma_n=\mu_n$ and $\mu_n$ is concentrated on $\Omega$.  A standard Sobolev approximation gives
\[
\int_{\dbR^n}\left\|F\circ T-\int_{\mathbb{R}^{n}}F\circ T\dd\gamma_n\right\|_\mathcal{H}^p
\dd\gamma_n
\le k_p\int_{\dbR^n}\left(\sum_{i=1}^n\|\partial_i(F\circ T)\|_\mathcal{H}^2\right)^{p/2}
\dd\gamma_n.
\]
Since $T$ is $1$-Lipschitz, Rademacher's theorem gives $\|DT\|_{\mathrm{op}}\le1$ almost everywhere.  The Sobolev chain rule therefore yields
\[
\left(
\sum_{i=1}^n\|\partial_i(F\circ T)\|_\mathcal{H}^2
\right)^{1/2}
\le
\left(
\sum_{i=1}^n\|(\partial_iF)\circ T\|_\mathcal{H}^2
\right)^{1/2}
\quad\text{a.e.}
\]
Consequently,
\[
\int_{\dbR^n}
\left\|
F\circ T-
\int_{\dbR^n}F\circ T\dd\gamma_n
\right\|_\mathcal{H}^p
\dd\gamma_n
\le
k_p
\int_{\dbR^n}
\left(
\sum_{i=1}^n
\|(\partial_iF)\circ T\|_\mathcal{H}^2
\right)^{p/2}
\dd\gamma_n.
\]

Using $T_{\#}\gamma_n=\mu_n$ on both sides gives
\[
\int_{\Omega}
\left\|
F-\int_{\Omega}F\dd\mu_n
\right\|_\mathcal{H}^p
\dd\mu_n
\le
k_p
\int_{\Omega}
\left(
\sum_{i=1}^n\|\partial_iF\|_\mathcal{H}^2
\right)^{p/2}
\dd\mu_n.
\]
Multiplying by $\gamma_n(\Omega)$ and using the definition of $\mu_n$ gives \eqref{1}.  This proves Theorem~\ref{poincare}.
\end{proof}

\begin{theorem}\label{bijin}
Let $r\in\mathbb{N}_0$ and $f\in C_{\mathrm{pol}}^{\infty}(\mathbb{R}^{n})$. Then there exists $Q\in \mathcal{P} \left( \mathbb{R}^{n} \right)$ such that $\deg Q\le r$ and, for $j=0,1,\dots ,r$,
\begin{equation}\label{233}
	\mathcal{J}_{j}\left( f-Q \right) \le  k_{p}^{\frac{r+1-j}{p}}\mathcal{J}_{r+1}\left( f \right),
\end{equation}
where $k_{p}$ is defined in \eqref{2}.
\end{theorem}
\begin{proof}
If $r=0$, take
\[T_0\triangleq\frac{1}{\gamma_n(\Omega)}\int_{\Omega}f\dd\gamma_n,
\qquad Q\triangleq T_0.\]
Then \eqref{233} is exactly Theorem~\ref{poincare} applied to the scalar function $f$.  Hence we may assume $r\ge1$.

We construct the coefficients of the polynomial recursively.  First, define
\[T_{r}\triangleq \frac{1}{\gamma_{n} (\Omega)} \int_{\Omega} D^{r}f\dd \gamma_{n} \in \mathcal{L}_{s}^{r} \left( \mathbb{R}^{n} \right).\]
Suppose next that $T_r,T_{r-1},\dots,T_k$ have already been defined for some $1\le k\le r$.
For $j\le r$ and $\mathbf{x}\in\dbR^n$, denote by $T_j(\mathbf{x},\dots,\mathbf{x})$ the evaluation of $T_j$ at $(\mathbf{x},\dots,\mathbf{x})\in(\dbR^{n})^j$, viewed as a homogeneous polynomial of degree $j$ in $\mathbf{x}$.
Define
\[T_{k-1}\triangleq \frac{1}{\gamma_{n} (\Omega)} \int_{\Omega} D^{k-1}\left( f-\frac{T_{r}\left( \mathbf{x} ,\dots ,\mathbf{x} \right)}{r!} -\dots -\frac{T_{k}\left( \mathbf{x} ,\dots ,\mathbf{x} \right)}{k!} \right) \dd \gamma_{n} \in \mathcal{L}_{s}^{k-1} \left( \mathbb{R}^{n} \right).\]
For $k=1,2,\dots ,r$, it follows from Lemma \ref{qi} and Theorem \ref{poincare} that
\[\begin{aligned}
	&\int_{\Omega} \left\| D^{k-1}\left( f-\frac{T_{r}\left( \mathbf{x} ,\dots ,\mathbf{x} \right)}{r!} -\dots -\frac{T_{k}\left( \mathbf{x} ,\dots ,\mathbf{x} \right)}{k!} \right) -T_{k-1} \right\|_{k-1}^{p} \dd \gamma_{n}\\
	\le {}& k_{p}\int_{\Omega} \left( \sqrt{\sum_{i=1}^{n} \left\| \partial_{i} \left[ D^{k-1}\left( f-\frac{T_{r}\left( \mathbf{x} ,\dots ,\mathbf{x} \right)}{r!} -\dots -\frac{T_{k}\left( \mathbf{x} ,\dots ,\mathbf{x} \right)}{k!} \right) \right] \right\|_{k-1}^{2}} \right)^{p} \dd \gamma_{n}\\
	={}& k_{p}\int_{\Omega} \left\| D^{k}\left( f-\frac{T_{r}\left( \mathbf{x} ,\dots ,\mathbf{x} \right)}{r!} -\dots -\frac{T_{k}\left( \mathbf{x} ,\dots ,\mathbf{x} \right)}{k!} \right) \right\|_{k}^{p} \dd \gamma_{n}\\
	={}& k_{p}\int_{\Omega} \left\| D^{k}\left( f-\frac{T_{r}\left( \mathbf{x} ,\dots ,\mathbf{x} \right)}{r!} -\dots -\frac{T_{k+1}\left( \mathbf{x} ,\dots ,\mathbf{x} \right)}{\left( k+1 \right) !} \right) -T_{k} \right\|_{k}^{p} \dd \gamma_{n} .
\end{aligned}\]		
Let
\[Q\triangleq T_{0}+\sum_{j=1}^{r} \frac{T_{j}\left( \mathbf{x} ,\mathbf{x}, \dots ,\mathbf{x} \right)}{j!}\in \mathcal{P} \left( \mathbb{R}^{n} \right).\]
Then $\deg Q\le r$.  For $j=0,1,2,\dots,r$, iterating the preceding inequality gives
\[\begin{aligned}
	&\hphantom{{}={}}\mathcal{J}_{j}\left( f-Q \right) =\left(\int_{\Omega} \left\| D^{j}\left( f-Q \right) \right\|_{j}^{p} \dd \gamma_{n}\right)^{1/p}\\
	&  =\left(\int_{\Omega} \left\| D^{j}\left( f-\sum_{k=j}^{r} \frac{T_{k}\left( \mathbf{x} ,\dots ,\mathbf{x} \right)}{k!} \right) \right\|_{j}^{p} \dd \gamma_{n}\right)^{1/p}\\
	& \le \left(k_{p}\right)^{1/p} \left(\int_{\Omega} \left\| D^{j+1}\left( f-\sum_{k=j+1}^{r} \frac{T_{k}\left( \mathbf{x} ,\dots ,\mathbf{x} \right)}{k!} \right) \right\|_{j+1}^{p} \dd \gamma_{n}\right)^{1/p}\\
	& \  \  \vdots\\
	& \le \left( \left(k_{p}\right)^{1/p} \right)^{r-j} \left(\int_{\Omega} \left\| D^{r}\left( f-\frac{T_{r}\left( \mathbf{x} ,\dots ,\mathbf{x} \right)}{r!} \right) \right\|_{r}^{p} \dd \gamma_{n}\right)^{1/p}\\
	&  =\left( \left(k_{p}\right)^{1/p} \right)^{r-j} \left(\int_{\Omega} \left\| D^{r}f-T_{r} \right\|_{r}^{p} \dd \gamma_{n}\right)^{1/p}\\
	&  \le \left( \left(k_{p}\right)^{1/p} \right)^{r+1-j} \mathcal{J}_{r+1}\left( f \right) .\end{aligned}\]
This proves Theorem~\ref{bijin}.
\end{proof}

\begin{theorem}\label{weishuguanguji}
	Let $r,k\in \mathbb{N}$. Then there exists $C_{p,r,k}>0$ depending only on $p,r,k$, such that for every $Q\in \mathcal{P} \left( \mathbb{R}^{n} \right)$ of degree at most $r$, we have
	\begin{equation}\label{1fff13f}
	\int_{\Omega} \left\| D^{k}Q \right\|_{k}^{p} \dd \gamma_{n} \le \frac{C_{p,r,k}}{[\gamma_n(\Omega)]^{1+pr}} \int_{\Omega} \left| Q \right|^{p} \dd \gamma_{n}.
	\end{equation}
\end{theorem}
\begin{proof}
If $Q\equiv0$, the estimate is immediate.  We may therefore assume $Q\not\equiv0$.  If $k>r$, then $D^kQ\equiv0$, so the estimate again holds (for instance with $C_{p,r,k}=1$).  It remains to consider $1\le k\le r$.

\medskip\noindent\textbf{Step 1: Hermite expansion and the $L^2$ derivative estimate.}

By \cite[Proposition~5.48, p.~135]{AS17},
\begin{equation}\label{13f1f3}
	\int_{\mathbb{R}^{n}} \left| Q \right|^{q} \dd \gamma_{n} \le \left( q-1 \right)^{\frac{rq}{2}} \left( \int_{\mathbb{R}^{n}} \left| Q \right|^{2} \dd \gamma_{n} \right)^{\frac{q}{2}},\qquad\forall q\ge 2.
\end{equation}
	
	Let
	\[H_{m}(t) \triangleq \frac{\left( -1 \right)^{m}}{\sqrt{m!}} e^{\frac{t^{2}}{2}}\frac{d^{m}}{dt^{m}} e^{-\frac{t^{2}}{2}},\qquad m\in\dbN_0,\]
	and
	\[H_{\alpha}\left( x_{1},x_{2},\dots ,x_{n} \right) =\prod_{j=1}^{n} H_{\alpha_{j}}\left( x_{j} \right) ,\qquad\alpha = \left( \alpha_{j} \right) \in \mathbb{N}_{0}^{n}.\]
	We know that $\left\{ H_{\alpha} \right\}_{\alpha \in \mathbb{N}_{0}^{n}}$ is a complete orthonormal basis of $L^{2}\left( \mathbb{R}^{n} ,\gamma_{n} \right)$.  We adopt the convention $H_\beta=0$ whenever $\beta\in\mathbb Z^n$ has at least one negative component.  With this convention,
	\begin{equation}\label{1f1vvv}
		\partial_{x_i} H_{\alpha}=\sqrt{\alpha_{i}} H_{\alpha -\mathbf{e}_{i}},\qquad \alpha \in \mathbb{N}_{0}^{n},\;i=1,2,\dots ,n.
	\end{equation}
	Consider $\left| \alpha \right| \triangleq \sum\limits_{i=1}^{n} \alpha_{i}$ and the expansion
	\[Q=\sum_{\left| \alpha \right| \le r} c_{\alpha}H_{\alpha},\]
	where $c_{\alpha}\in\mathbb R$ for $|\alpha|\le r$. We have
	\[\int_{\mathbb{R}^{n}} \left| Q \right|^{2} \dd \gamma_{n} =\sum_{\left| \alpha \right| \le r} \left| c_{\alpha} \right|^{2}.\]
	For the fixed $k\in\{1,\dots,r\}$ and $(i_1,i_2,\dots,i_k)\in\{1,2,\dots,n\}^k$, let $\beta_{j}$ denote the number of occurrences of $j$ among $i_{1},i_{2},\dots ,i_{k}$. Using \eqref{1f1vvv}, we obtain
	\[\int_{\mathbb{R}^{n}} \left| \partial_{i_{1}i_{2}\dots i_{k}} Q \right|^{2} \dd \gamma_{n} =\sum_{\left| \alpha \right| \le r} \left| c_{\alpha} \right|^{2} \cdot \left[ \prod_{j=1}^{n} \alpha_{j} \left( \alpha_{j} -1 \right) \cdots \left( \alpha_{j} -\beta_{j} +1 \right) \right],\]
	where, when $\beta_{j} =0$, the corresponding product is understood to be $1$.
	
	Note the identity
	\[
	\sum_{i_{1},i_{2},\dots,i_{k}=1}^{n}\,
	\prod_{j=1}^{n}
	\alpha_{j}(\alpha_{j}-1)\cdots(\alpha_{j}-\beta_{j}+1)
	=
	|\alpha|(|\alpha|-1)\cdots(|\alpha|-k+1).
	\]
	Indeed, regard the $j$-th box as containing $\alpha_j$ distinct balls.
	For a fixed ordered sequence
	$(i_1,i_2,\dots,i_k)\in\{1,2,\dots,n\}^k$,
	the quantity
	\[
	\prod_{j=1}^{n}
	\alpha_{j}(\alpha_{j}-1)\cdots(\alpha_{j}-\beta_{j}+1)
	\]
	is the number of ordered selections of $k$ distinct balls such that the
	$\ell$-th selected ball is taken from the $i_\ell$-th box.
	Therefore, summing over all ordered sequences
	$(i_1,i_2,\dots,i_k)$ counts each ordered selection of $k$ distinct balls
	from the total of $|\alpha|$ balls exactly once, and the total number of such
	selections is
	\[
	|\alpha|(|\alpha|-1)\cdots(|\alpha|-k+1).
	\]

We now have
	\[\begin{aligned}
		\sum_{i_{1},i_{2},\dots, i_{k}=1}^{n} \int_{\mathbb{R}^{n}} \left| \partial_{i_{1}i_{2}\dots i_{k}} Q \right|^{2} \dd \gamma_{n}&=\sum_{\left| \alpha \right| \le r} \left[ \left| c_{\alpha} \right|^{2} \left| \alpha \right| \cdot \left( \left| \alpha \right| -1 \right) \cdots \left( \left| \alpha \right| -k+1 \right) \right]\\
		&\le r\left( r-1 \right) \cdots \left( r-k+1 \right) \sum_{\left| \alpha \right| \le r} \left| c_{\alpha} \right|^{2}\\
		&=r\left( r-1 \right) \cdots \left( r-k+1 \right) \int_{\mathbb{R}^{n}} \left| Q \right|^{2} \dd \gamma_{n}.\end{aligned}\]
	Hence, for $k=1,2,\dots ,r$, we obtain the following basic estimate:
	\begin{equation}\label{1f1f3}
		\int_{\mathbb{R}^{n}} \left\| D^{k}Q \right\|_{k}^{2} \dd \gamma_{n} \le r\left( r-1 \right) \cdots \left( r-k+1 \right) \int_{\mathbb{R}^{n}} \left| Q \right|^{2} \dd \gamma_{n}.
	\end{equation}
	
\medskip\noindent\textbf{Step 2: Passage from $L^2$ to $L^p$.}

When $p\ge 2$ and $k=1,2,\dots,r$, we have
\[\begin{aligned}
	\left( \int_{\mathbb{R}^{n}} \left\| D^{k}Q \right\|_{k}^{p} \dd \gamma_{n} \right)^{\frac{2}{p}}&=\left( \int_{\mathbb{R}^{n}} \left( \sum_{i_{1},i_{2},\dots ,i_{k}=1}^{n} \left| \partial_{i_{1}i_{2}\dots i_{k}} Q \right|^{2} \right)^{\frac{p}{2}} \dd \gamma_{n} \right)^{\frac{2}{p}}\\
	&\le \sum_{i_{1},i_{2},\dots ,i_{k}=1}^{n} \left( \int_{\mathbb{R}^{n}} \left( \left| \partial_{i_{1}i_{2}\dots i_{k}} Q \right|^{2} \right)^{\frac{p}{2}} \dd \gamma_{n} \right)^{\frac{2}{p}}\\
	&=\sum_{i_{1},i_{2},\dots ,i_{k}=1}^{n} \left( \int_{\mathbb{R}^{n}} \left| \partial_{i_{1}i_{2}\dots i_{k}} Q \right|^{p} \dd \gamma_{n} \right)^{\frac{2}{p}}\\
	&\le \sum_{i_{1},i_{2},\dots ,i_{k}=1}^{n} \left( p-1 \right)^{r-k} \int_{\mathbb{R}^{n}} \left| \partial_{i_{1}i_{2}\dots i_{k}} Q \right|^{2} \dd \gamma_{n}\\
	&=\left( p-1 \right)^{r-k} \int_{\mathbb{R}^{n}} \left\| D^{k}Q \right\|_{k}^{2} \dd \gamma_{n}\\ &\le \left( p-1 \right)^{r-k} r\left( r-1 \right) \cdots \left( r-k+1 \right) \int_{\mathbb{R}^{n}} \left| Q \right|^{2} \dd \gamma_{n}\\
	&\le \left( p-1 \right)^{r-k} r\left( r-1 \right) \cdots \left( r-k+1 \right) \left( \int_{\mathbb{R}^{n}} \left| Q \right|^{p} \dd \gamma_{n} \right)^{\frac{2}{p}} ,\end{aligned}\]
where the first inequality follows from $\frac{p}{2} \ge 1$ and Minkowski's inequality, the second inequality follows from \eqref{13f1f3}, the third inequality follows from \eqref{1f1f3}, and the last inequality follows from H\"older's inequality. Therefore, for $p\in \left[ 2,+\infty \right)$ and $k=1,2,\dots ,r$, we obtain the estimate
\begin{equation}\label{11}
	\int_{\mathbb{R}^{n}} \left\| D^{k}Q \right\|_{k}^{p} \dd \gamma_{n} \le \left( p-1 \right)^{\frac{r-k}{2} p} \left( r\left( r-1 \right) \cdots \left( r-k+1 \right) \right)^{\frac{p}{2}} \int_{\mathbb{R}^{n}} \left| Q \right|^{p} \dd \gamma_{n}.
\end{equation}	

When $1\le p < 2$ and $k=1,2,\dots,r$, it follows first from \eqref{13f1f3} that
\[\int_{\mathbb{R}^{n}} \left\vert Q \right\vert^{4} \dd \gamma_{n} \le 9^{r}\left( \int_{\mathbb{R}^{n}} \left| Q \right|^{2} \dd \gamma_{n} \right)^{2}.\]
Then, by the Paley--Zygmund inequality (see, for instance, \cite[Proposition~III.3, p.~23]{LQ17}), we have
\begin{equation*}
		\gamma_{n} \left( \left\{ \mathbf{x} \in \mathbb{R}^{n} :\left| Q \right| \ge \sqrt{\frac{1}{2} \int_{\mathbb{R}^{n}} \left| Q \right|^{2} \dd \gamma_{n}} \right\} \right)\ge \frac{\left( \int_{\mathbb{R}^{n}} \left| Q \right|^{2} \dd \gamma_{n} \right)^{2}}{4\int_{\mathbb{R}^{n}} \left\vert Q \right\vert^{4} \dd \gamma_{n}} \ge \frac{1}{4\cdot 9^{r}},
\end{equation*}
which, together with H\"older's inequality, implies
\begin{equation}\label{1ff133}
\left( \int_{\mathbb{R}^{n}} \left| Q \right|^{p} \dd \gamma_{n} \right)^{\frac{1}{p}} \ge \int_{\mathbb{R}^{n}} \left| Q \right| \dd \gamma_{n} \ge \sqrt{\frac{1}{2} \int_{\mathbb{R}^{n}} \left| Q \right|^{2} \dd \gamma_{n}} \cdot \frac{1}{4\cdot 9^{r}}.
\end{equation}
Therefore, for $p\in \left[ 1,2 \right)$ and $k=1,2,\dots ,r$, we have
\begin{equation}\label{12}
\begin{aligned}
	&\hphantom{{}={}}\!\int_{\mathbb{R}^{n}} \left\| D^{k}Q \right\|_{k}^{p} \dd \gamma_{n} \le \left( \int_{\mathbb{R}^{n}} \left\| D^{k}Q \right\|_{k}^{2} \dd \gamma_{n} \right)^{\frac{p}{2}}\\
	&  \le \left[ r\left( r-1 \right) \cdots \left( r-k+1 \right) \right]^{\frac{p}{2}} \left( \int_{\mathbb{R}^{n}} \left\vert Q \right\vert^{2} \dd \gamma_{n} \right)^{\frac{p}{2}}\\
	&  \le \left( 4\sqrt{2} \right)^{p} \cdot 3^{2pr}\cdot \left[ r\left( r-1 \right) \cdots \left( r-k+1 \right) \right]^{\frac{p}{2}} \int_{\mathbb{R}^{n}} \left\vert Q \right\vert^{p} \dd \gamma_{n},\end{aligned}
\end{equation}
where the first inequality follows from H\"older's inequality, the second inequality follows from \eqref{1f1f3}, and the third inequality follows from \eqref{1ff133}.

\medskip\noindent\textbf{Step 3: Localization from $\dbR^n$ to $\Omega$.}
	
The Carbery--Wright inequality \cite[Theorem~8, p.~244]{CW}, applied with degree bound $d=r$ and parameter $q=rp$, yields an absolute constant $C>0$ such that, for every $t>0$,
\[\gamma_{n} \left\{ \mathbf{x} \in \mathbb{R}^{n} :\left| Q( \mathbf{x} ) \right| \le t \right\} \le Crpt^{\frac{1}{r}} \left( \int_{\mathbb{R}^{n}} \left| Q \right|^{p} \dd \gamma_{n} \right)^{-\frac{1}{rp}}.\]
Choosing
\[t_0 =\left( \frac{\gamma_{{}n} (\Omega)}{2Crp} \right)^{r} \left(\int_{\mathbb{R}^{n}} \left| Q \right|^{p} \dd \gamma_{n}\right)^{1/p}\]
and substituting it into the preceding inequality gives
\begin{equation}\label{13f13ff}
	\gamma_{n} \left\{ \mathbf{x} \in \mathbb{R}^{n} :\left| Q(\mathbf{x}) \right| \le t_{0} \right\} \le \frac{\gamma_{{}n} (\Omega)}{2}.
\end{equation}
Now, by \eqref{13f13ff},
\[\begin{aligned}
	\gamma_{n} \left( \Omega \bigcap \left\{ \mathbf{x} \in \mathbb{R}^{n} :\left| Q(\mathbf{x}) \right| >t_{0} \right\} \right)& =\gamma_{n} (\Omega) -\gamma_{n} \left( \Omega \bigcap \left\{ \mathbf{x} \in \mathbb{R}^{n} :\left| Q(\mathbf{x}) \right| \le t_{0} \right\} \right)\\
	&  \ge \gamma_{n} (\Omega) -\frac{\gamma_{n} (\Omega)}{2} =\frac{\gamma_{n} (\Omega)}{2} .\end{aligned}\]
Therefore,
\[\begin{aligned}\int_{\Omega} \left| Q \right|^{p} \dd \gamma_{n}&\ge \int_{\Omega \bigcap \left\{ \mathbf{x} \in \mathbb{R}^{n} :\left| Q \right| >t_{0} \right\}} \left| Q \right|^{p} \dd \gamma_{n} \ge \frac{\gamma_{{}n} (\Omega)}{2} t_{0}^{p}\\ &=\frac{\gamma_{{}n} (\Omega)}{2} \left( \frac{\gamma_{{}n} (\Omega)}{2Crp} \right)^{pr} \int_{\mathbb{R}^{n}} \left| Q \right|^{p} \dd \gamma_{n} ,\end{aligned}\]
that is,
\begin{equation}\label{1f33gfg}
\int_{\mathbb{R}^{n}} \left| Q \right|^{p} \dd \gamma_{n} \le \left( \frac{2}{\gamma_{n} (\Omega)} \right) \left( \frac{2Crp}{\gamma_{n} (\Omega)} \right)^{pr} \int_{\Omega} \left| Q \right|^{p} \dd \gamma_{n}.
\end{equation}

\medskip\noindent\textbf{Step 4: Conclusion.}

Combining \eqref{11}, \eqref{12}, and \eqref{1f33gfg}, we obtain a constant $C_{p,r,k}>0$ such that \eqref{1fff13f} holds. This completes the proof of Theorem~\ref{weishuguanguji}.
\end{proof}
\begin{theorem}\label{zyao}
Let $r\in\mathbb{N}$ and $f\in C_{\mathrm{pol}}^{\infty}(\mathbb{R}^{n})$. Then
	\begin{equation*}
	\begin{aligned}
\sum_{j=1}^{r}\mathcal{J}_j(f)\le
\sum_{j=1}^{r}
\max\left\{\frac{C_{p,r,j}^{1/p}}{[\gamma_n(\Omega)]^{1/p+r}},\frac{C_{p,r,j}^{1/p}}{[\gamma_n(\Omega)]^{1/p+r}}k_p^{(r+1)/p}+k_p^{(r+1-j)/p}\right\}\bigl(\mathcal{J}_0(f)+\mathcal{J}_{r+1}(f)\bigr).
\end{aligned}
\end{equation*}
\end{theorem}
\begin{proof}
Let $Q$ be the polynomial given by Theorem~\ref{bijin}.  For each $j=1,2,\dots,r$, Theorems~\ref{bijin} and~\ref{weishuguanguji} give
	\[\begin{aligned}\mathcal{J}_{j}\left( f \right)&\le \mathcal{J}_{j}\left( f-Q \right) +\mathcal{J}_{j}\left( Q \right) \le k_{p}^{\frac{r+1-j}{p}}\mathcal{J}_{r+1}\left( f \right) +\frac{C_{p,r,j}^{\frac{1}{p}}}{[\gamma_n(\Omega)]^{\frac{1}{p}+r}} \left(\int_{\Omega} \left| Q \right|^{p} \dd \gamma_{n}\right)^{1/p}\\ &\le k_{p}^{\frac{r+1-j}{p}}\mathcal{J}_{r+1}\left( f \right) +\frac{C_{p,r,j}^{\frac{1}{p}}}{[\gamma_n(\Omega)]^{\frac{1}{p}+r}} \left[ \mathcal{J}_{0}\left( f \right) +\mathcal{J}_{0}\left( f-Q \right) \right]\\ &\le k_{p}^{\frac{r+1-j}{p}}\mathcal{J}_{r+1}\left( f \right) +\frac{C_{p,r,j}^{\frac{1}{p}}}{[\gamma_n(\Omega)]^{\frac{1}{p}+r}} \left[ \mathcal{J}_{0}\left( f \right) +k_{p}^{\frac{r+1}{p}}\mathcal{J}_{r+1}\left( f \right) \right]\\ &\le \max \left\{ \frac{C_{p,r,j}^{\frac{1}{p}}}{[\gamma_n(\Omega)]^{\frac{1}{p}+r}} ,\frac{C_{p,r,j}^{\frac{1}{p}}}{[\gamma_n(\Omega)]^{\frac{1}{p}+r}} k_{p}^{\frac{r+1}{p}}+k_{p}^{\frac{r+1-j}{p}} \right\} \left( \mathcal{J}_{0}\left( f \right) +\mathcal{J}_{r+1}\left( f \right) \right) .\end{aligned}\]
	Summing this over $j=1,2,\dots,r$ completes the proof of Theorem~\ref{zyao}.
\end{proof}

\subsection{Sobolev norm equivalence on open convex subsets of \texorpdfstring{$\ell^{2}$}{l2}}
Fix a nonempty open convex set ${O}\subset\ell^2$.  For $n\in\dbN$, let
\[
\iota_n:\mathbb{R}^n\longrightarrow\ell^2,
\qquad
\iota_n(x_1,\dots,x_n)
\triangleq
(x_1,\dots,x_n,0,0,\dots)
\]
be the canonical coordinate embedding.  Define
\[
O_n
\triangleq
\left\{(x_1,\dots,x_n)\in\mathbb{R}^n:
\iota_n(a_1x_1,\dots,a_nx_n)\in{O}
\right\}.
\]
Because ${O}$ is open and nonempty and the finite-coordinate truncations converge in $\ell^2$, $O_n$ is nonempty for all sufficiently large $n$. It follow directly from the definition that ${O}_n$ is open and convex.
For later use, set
\[
E_n\triangleq\iota_n^{-1}({O})
=\left\{(x_1,\dots,x_n)\in\mathbb{R}^n:\iota_n(x_1,\dots,x_n)\in{O}\right\}.
\]
For a sufficiently smooth function $G$ on $\mathbb R^n$, let $\mathcal{J}_{j,n}(G)$ denote the quantity $\mathcal{J}_j(G)$ from Section~\ref{subsec:normdimfree} with the finite-dimensional domain $\Omega$ replaced by $O_n$.  

\begin{lemma}\label{jihejixian}
	We have
	\begin{equation}\label{13f13f13g}
		\lim_{n\rightarrow \infty} \gamma_{n} \left( O_n \right) =P({O}) >0.
	\end{equation}
\end{lemma}
\begin{proof}
	Since every nondegenerate Gaussian measure on a separable Hilbert space is full, we have $P({O}) >0$ (see, e.g., \cite[Proposition~1.25, p.~21]{Da06}, \cite[Lemma 2.1, p. 9]{WYZ}).
For $\mathbf{x}=(x_i)_{i\in\dbN}\in\ell^2$, set $\pi_n\mathbf{x}=(x_1,\dots,x_n,0,0,\dots)$. Since $\pi_n\mathbf{x}\to\mathbf{x}$ in $\ell^2$ as $n\to\infty$, openness of ${O}$ implies that the following indicator functions converge pointwise to $1$ on ${O}$ and to $0$ on $\ell^2\setminus\overline{O}$:
	\[\lim_{n\rightarrow \infty} \mathbf{1}_{E_n \times \ell^{2} \left( \mathbb{N} \setminus \left\{ 1,2,\dots ,n \right\} \right)} =\begin{cases}1,&\mathbf{x} \in {O}\\ 0,&\mathbf{x} \in \ell^{2} \setminus \overline{O}\end{cases}.\]
	Moreover, by the discussion following \cite[Corollary~9, p.~1074]{BBL}, the boundary of a convex set with nonempty interior has Gaussian measure zero. Hence $P\left( \partial {O} \right) =0$. Therefore, by the dominated convergence theorem and the change of variables $y_i=a_ix_i$ ($i=1,2,\dots,n$), we obtain
	\[\begin{aligned}\lim_{n\rightarrow \infty} \gamma_{n} \left( O_n \right)&=\lim_{n\rightarrow \infty} \left( \int_{O_n} 1\dd\gamma_{n} \right)\\ &=\lim_{n\rightarrow \infty} \left( \int_{E_n \times \ell^{2} \left( \mathbb{N} \setminus \left\{ 1,2,\dots ,n \right\} \right)} 1\dd P \right)\\ &=\int_{\ell^{2}} \lim_{n\rightarrow \infty} \mathbf{1}_{E_n \times \ell^{2} \left( \mathbb{N} \setminus \left\{ 1,2,\dots ,n \right\} \right)} \dd P\\ &=\int_{{O}} 1\dd P=P({O}) >0,\end{aligned}\]
	This completes the proof of Lemma~\ref{jihejixian}.
\end{proof}

Using the $\mathscr{W}^{m,p}$ version of Theorem~\ref{zhuhanshuguanghuabijin}, one obtains the following density statement.
\begin{theorem}\label{choumixingdingli}
Let $m\in\dbN$.  The restrictions to ${O}$ of functions in $\mathscr C_{\mathrm{pol}}^\infty$ are dense in $\mathscr{W}^{m,p}({O},P)$.
	\end{theorem}

	Now we prove Theorem~\ref{thm:norm}.

\begin{proof}[Proof of Theorem~\ref{thm:norm}]
	

By Theorem \ref{choumixingdingli} and Remark~\ref{rmk:anothernorm} (which states for $W^{1,p}$ but the $\mathscr{W}^{1,p}$ case is similar), it suffices to show that there exists a constant $C>0$, depending only on ${O}$, $p$, and $m$, such that for every $f\in C_{\mathrm{pol}}^{\infty}\left( \mathbb{R}^{N} \right)$ ($N\in \mathbb{N}$),
\begin{equation}\label{1ff313g}
\begin{aligned}
&\hphantom{{}={}}\!\left(\int_{{O}}|f|^p\dd P\right)^{1/p}
+\sum_{k=1}^{m}
\left[
\int_{{O}}
\left(\sum_{i_1,\dots,i_k=1}^{N}\left|a_{i_1}\dots a_{i_k}\partial_{i_1\dots i_k}f\right|^2
\right)^{p/2}\dd P
\right]^{1/p}\\
&\le
C\left\{
\left(\int_{{O}}|f|^p\dd P\right)^{1/p}
+
\left[
\int_{{O}}
\left(
\sum_{i_1,\dots,i_m=1}^{N}
\left|a_{i_1}\dots a_{i_m}\partial_{i_1\dots i_m}f\right|^2
\right)^{p/2}\dd P
\right]^{1/p}
\right\}.
\end{aligned}
\end{equation}
Consider
\[\begin{aligned}C&\triangleq 1+\sum_{j=1}^{m-1} \max \left\{ \frac{C_{p,m-1,j}^{\frac{1}{p}}}{[P({O})]^{\frac{1}{p}+m-1}} ,\frac{C_{p,m-1,j}^{\frac{1}{p}}}{[P({O})]^{\frac{1}{p}+m-1}} k_{p}^{\frac{m}{p}}+k_{p}^{\frac{m-j}{p}} \right\} ,\\ C_{n}&\triangleq 1+\sum_{j=1}^{m-1} \max \left\{ \frac{C_{p,m-1,j}^{\frac{1}{p}}}{[\gamma_n(O_n)]^{\frac{1}{p}+m-1}} ,\frac{C_{p,m-1,j}^{\frac{1}{p}}}{[\gamma_n(O_n)]^{\frac{1}{p}+m-1}} k_{p}^{\frac{m}{p}}+k_{p}^{\frac{m-j}{p}} \right\} ,\qquad n\in\dbN.\end{aligned} \]
By Lemma \ref{jihejixian}, we obtain
\begin{equation}\label{1f3ggg}
	\lim_{n\rightarrow \infty} C_{n}=C.
\end{equation}
For sufficiently large $n>N$ with $\gamma_n(O_n)>0$, regard $f$ as a function on $\dbR^n$ independent of the last $n-N$ variables and set $F_n(x_1,\dots,x_n)\triangleq f(a_1x_1,\dots,a_Nx_N)$.  By Theorem~\ref{zyao}, we have
\begin{equation}\label{1f31f1f13}
	\sum_{j=0}^{m} \mathcal{J}_{j,n}\left( F_n \right) \le C_{n}\left( \mathcal{J}_{0,n}\left( F_n \right) +\mathcal{J}_{m,n}\left( F_n \right) \right).
\end{equation}
By an argument similar to the proof of Lemma \ref{jihejixian}, the dominated convergence theorem yields
\[
\begin{aligned}
\lim_{n\to\infty}\mathcal{J}_{0,n}(F_n)
&=\lim_{n\to\infty}
\left(\int_{O_n}|F_n|^p\dd\gamma_n\right)^{1/p}\\
&=\lim_{n\to\infty}
\left(\int_{E_n}|f|^p\dd\mathcal N^n\right)^{1/p}\\
&=\left(\int_{{O}}|f|^p\dd P\right)^{1/p},
\end{aligned}
\]
and for each $j=1,2,\dots,m$,
\[
\begin{aligned}
\lim_{n\to\infty}\mathcal{J}_{j,n}(F_n)
&=\lim_{n\to\infty}
\left[
\int_{O_n}
\left(
\sum_{i_1,\dots,i_j=1}^{N}
|\partial_{i_1\dots i_j}F_n|^2
\right)^{p/2}\dd\gamma_n
\right]^{1/p}\\
&=\lim_{n\to\infty}
\left[
\int_{E_n}
\left(
\sum_{i_1,\dots,i_j=1}^{N}
a_{i_1}^2\dots a_{i_j}^2
|\partial_{i_1\dots i_j}f|^2
\right)^{p/2}\dd\mathcal N^n
\right]^{1/p}\\
&=
\left[
\int_{{O}}
\left(
\sum_{i_1,\dots,i_j=1}^{N}
a_{i_1}^2\dots a_{i_j}^2
|\partial_{i_1\dots i_j}f|^2
\right)^{p/2}\dd P
\right]^{1/p}.
\end{aligned}
\]
Combining \eqref{1f3ggg} with the preceding limits and letting $n\to\infty$ in \eqref{1f31f1f13} gives \eqref{1ff313g}.  This completes the proof of Theorem~\ref{thm:norm}.
\end{proof}

For $1\le p<2$, the use of $\lVert\cdot\rVert_{\mathscr{W}^{m,p}({O},P)}$ in Theorem~\ref{thm:norm} is essential.  If one instead uses $\lVert\cdot\rVert_{{W}^{m,p}({O},P)}$, an analogue of Theorem~\ref{thm:norm} fails, as the following proposition shows.
\begin{proposition}\label{1f3g}
For $p\in \left[ 1,2 \right)$ and any integer $m\ge 2$, there does not exist a constant $C > 0$ such that, for every $f\in \mathscr C_{\mathrm{pol}}^{\infty}$,
	\begin{equation}\label{13f311gg}
		\begin{aligned}
			&\hphantom{{}={}}\!\int_{\ell^{2}} \left| f \right|^{p} \dd P+\sum_{k=1}^{m} \sum_{i_{1},i_{2},\dots ,i_{k}=1}^{\infty} a_{i_{1}}^{p}a_{i_{2}}^{p}\dots a_{i_{k}}^{p}\int_{\ell^{2}} \left| \partial_{i_{1}i_{2}\dots i_{k}} f \right|^{p} \dd P\\
			&\le C\left[ \int_{\ell^{2}} \left| f \right|^{p} \dd P+\sum_{i_{1},i_{2},\dots ,i_{m}=1}^{\infty} a_{i_{1}}^{p}a_{i_{2}}^{p}\dots a_{i_{m}}^{p}\int_{\ell^{2}} \left| \partial_{i_{1}i_{2}\dots i_{m}} f \right|^{p} \dd P \right].
		\end{aligned}
	\end{equation}
\end{proposition}
\begin{proof}
	Suppose that there exists a constant $C > 0$ such that \eqref{13f311gg} holds for every $f\in \mathscr C_{\mathrm{pol}}^{\infty}$. Then for every $n\in\dbN$, $f=\sum_{i=1}^{n} \frac{x_{i}}{a_{i}}$ should satisfy inequality \eqref{13f311gg}.
	
	Let $\Gamma$ denote the gamma function, i.e.,
	\[\Gamma(s)=\int_{0}^{\infty}t^{s-1}e^{-t}\dd t,\qquad s>0.\]
	A direct computation shows that inequality \eqref{13f311gg} becomes
	\[\frac{2^{\frac{p}{2}}\Gamma \left( \frac{p+1}{2} \right)}{\sqrt{\pi}} n^{\frac{p}{2}}+n\le C\frac{2^{\frac{p}{2}}\Gamma \left( \frac{p+1}{2} \right)}{\sqrt{\pi}} n^{\frac{p}{2}},\qquad \forall n\in\dbN,\]
	which is clearly impossible. This contradiction proves Proposition~\ref{1f3g}.
\end{proof}

\section{Orlicz embeddings and logarithmic Sobolev inequalities}
\subsection{Orlicz spaces}
	
	Let $(X,\nu)$ be a finite measure space. A Young function is a convex, nondecreasing function
	$\Phi:[0,\infty)\to[0,\infty)$ such that $\Phi(0)=0$ and
	$\Phi(t)\to\infty$ as $t\to\infty$. The associated Orlicz space
	$L^\Phi(X,\nu)$ is equipped with the Luxemburg norm
	\begin{equation}\label{eq:Lux}
		\|u\|_{L^\Phi(X,\nu)}=\inf\left\{\lambda>0:\int_X\Phi\left(\frac{|u|}{\lambda}\right)\dd\nu\le1\right\}.
	\end{equation}
	See \cite{BS88,RR} for the basic theory.
	
	For $1\le p<\infty$ and $q\ge0$, set
	\[\Phi_{p,q}(t)=t^p[\log(e+t)]^q,\qquad t\ge0,\]
	and write
	\[L^p\log^q L(X,\nu)\triangleq L^{\Phi_{p,q}}(X,\nu).\]
	
	\begin{lemma}\label{lem:Nfunction}
		The function $\Phi_{p,q}$ is a Young function satisfying
		\[
		p\Phi_{p,q}(t)\le t\Phi_{p,q}'(t)\le(p+q)\Phi_{p,q}(t),\qquad \forall t>0.
		\]
		Moreover, for $a\ge1$ and $t\ge0$,
		\begin{equation}\label{eq:Phi scaling}
			a^p\Phi_{p,q}(t)\le\Phi_{p,q}(at)\le a^{p+q}\Phi_{p,q}(t),\qquad \forall a\ge1,\; t\ge0.
		\end{equation}
		In particular, $\Phi_{p,q}$ satisfies the global
		$\Delta_2$-condition:
		\[\Phi_{p,q}(2t)\le2^{p+q}\Phi_{p,q}(t),\qquad \forall t\ge0.\]
	\end{lemma}
	
	\begin{proof}
		For $t>0$,
		\[\frac{t\Phi_{p,q}'(t)}{\Phi_{p,q}(t)}
		=p+\frac{qt}{(e+t)\log(e+t)}\in[p,p+q].
		\]
		Hence $\Phi_{p,q}$ is nondecreasing, and
		\[p\log a=p\int_{t}^{at}\frac{1}{s}\dd s\le\int_{t}^{at}\frac{\Phi_{p,q}'(s)}{\Phi_{p,q}(s)}\dd s=\log\left(\frac{\Phi_{p,q}(at)}{\Phi_{p,q}(t)}\right)\le (p+q)\log a,\]
		which proves \eqref{eq:Phi scaling}.
		
		To verify the convexity of $\Phi_{p,q}$, write $l=\log(e+t)$. Then
		\[
		\Phi_{p,q}''(t)={}p(p-1)t^{p-2}l^q+\frac{qt^{p-1}l^{q-1}}{e+t}\left(2p-\frac{t}{e+t}+(q-1)\frac{t}{(e+t)l}\right).
		\]
		Since $l\ge1$, $p\ge 1$ and $q\ge0$,
		\[
		2p-\frac{t}{e+t}+(q-1)\frac{t}{(e+t)l}
		\ge 2p-\frac{t}{e+t}-\frac{t}{(e+t)l}
		\ge 0.
		\]
		Thus $\Phi_{p,q}$ is convex. The remaining properties follow
		from $\Phi_{p,q}(0)=0$ and $\Phi_{p,q}(t)\ge t^p$.
	\end{proof}
	
	Since $\Phi_{p,q}(t)\ge t^p$,
	\[\|u\|_{L^p(X,\nu)}\le\|u\|_{L^{\Phi_{p,q}}(X,\nu)}.
	\]
	Thus there is a continuous embedding
	\[
	L^p\log^q L(X,\nu)\hookrightarrow L^p(X,\nu).
	\]
	By \eqref{eq:Phi scaling},
	\[u\in L^p\log^q L(X,\nu)\;\Longleftrightarrow\;\int_{X}\Phi_{p,q}(|u|)\dd\nu<\infty\;\Longleftrightarrow\;\int_{X}\Phi_{p,q}(|au|)\dd\nu<\infty\text{ for every $a>0$}.\]
	
	Moreover, we have the following result which will be used in the sequel.
	\begin{lemma}\label{20261002lem1}
		If $\mathcal B$ is a bounded set in $L^{\Phi_{p,q}}(X,\nu)
		=L^p\log^qL(X,\nu)$ with respect to the Luxemburg norm, then
		\[
		\sup_{u\in\mathcal B}
		\int_X |u|^p\bigl[\log(e+|u|)\bigr]^q\,\mathrm d\nu
		<\infty.
		\]
	\end{lemma}
	
	\begin{proof}
		Since $\mathcal B$ is bounded in $L^{\Phi_{p,q}}(X,\nu)$, there exists
		$0<M<\infty$ such that
		\[
		\|u\|_{L^{\Phi_{p,q}}(X,\nu)}\le M
		\qquad\text{for all }u\in\mathcal B.
		\]
		Hence, by the definition of the Luxemburg norm,
		\begin{equation}\label{20261002for4}
			\int_X
			\Phi_{p,q} \left(\frac{|u|}{2M}\right)\,\mathrm d\nu
			\le 1
			\qquad\text{for all }u\in\mathcal B.
		\end{equation}
		Moreover, by \eqref{eq:Phi scaling}, for every fixed $c>0$ there exists a constant $C_c<\infty$
		such that
		\[
		\Phi_{p,q} (ct)\le C_c\,\Phi_{p,q} (t),\qquad t\ge 0.
		\]
		Therefore, taking $c=2M$,
		\[
		\Phi_{p,q} (|u|)\le C_{2M}\,
		\Phi_{p,q} \left(\frac{|u|}{2M}\right).
		\]
		Integrating and using \eqref{20261002for4} gives
		\[
		\int_X \Phi_{p,q} (|u|)\,\mathrm d\nu\le C_{2M}.
		\]
		Thus
		\[
		\sup_{u\in\mathcal B}
		\int_X |u|^p\bigl[\log(e+|u|)\bigr]^q\,\mathrm d\nu
		\le C_{2M}<\infty,
		\]
		which completes the proof.
	\end{proof}
\subsection{Orlicz embeddings}\label{subsec:embed}
	
	Throughout this subsection we assume that ${O}$ is a nonempty open subset of $\ell^2$,
	\begin{equation}
		\mathbf a=(a_i)_{i\in \mathbb N}\in \ell^2,\qquad \mathbf b=(b_i)_{i\in\mathbb N}\in\ell^{\infty},\qquad a_k>0,\quad\quad b_k>0,\quad k\in\mathbb N.
		\label{eq:assumptions}
	\end{equation}
	Let $\mu_{O}(E)\triangleq P({O})^{-1}P(E)$ for every $E\in \mathscr{B}({O})$, and define
	\begin{equation}\label{eq:r_M_delta}
		M_{\mathbf a,\mathbf b}\triangleq\sup_{i\in \mathbb{N}}\frac{a_i}{b_i},\qquad
		h(\mathbf{b}/\mathbf{a})\triangleq\liminf_{k\to\infty}\frac{b_k}{a_k},\qquad r_k\triangleq\frac{b_k}{a_k},
		\qquad k\in\mathbb N.
	\end{equation}
	Then $M_{\mathbf a,\mathbf b}<\infty$ is equivalent to $h(\mathbf{b}/\mathbf{a})>0$; $\lim\limits_{k\to\infty}r_k=\infty$ is equivalent to $\limsup_{k\to\infty}\frac{a_k}{b_k}=0$.
	
	For $u\in\operatorname{span}\mathscr{C}_c^\infty$, define
	\begin{equation}
		D_{\mathbf{b}} u\triangleq\left(b_k\frac{\partial u}{\partial x_k}  \right)_{k\in \mathbb{N}}.
		\label{eq:grad_ops}
	\end{equation}
	Here $D_{\mathbf{b}}$ is the gradient operator weighted by $\mathbf b$.
	Recall that $\operatorname{span}\mathscr{C}_c^\infty$ is dense in $L^p(O,P)$. Following the proof of \cite[Proposition~3.2]{LZZ23}, one can show that for any open set $O\subset \ell^2$ and $1\le p<\infty$, the operator
	\[
	D_{\mathbf{b}}:\operatorname{span}\mathscr{C}_c^\infty\subset L^p(O,P)\to L^p(O,P;\ell^2)
	\]
	is closable.

	Similar to the definition on the whole space in \cite{LZZ23}, we define the corresponding space on $O$ as follows.
	\begin{definition}
		\label{def:Wb1p}
		The space $W_{\mathbf{b}}^{1,p}({O},\mu_{O})$ is the domain of the closure of $D_{\mathbf{b}}$ in $L^p({O},\mu_{O})$, equipped with the  norm
		\[
		\|u\|_{W_{\mathbf{b}}^{1,p}({O},\mu_{O})}
		\triangleq\Bigl(\|u\|_{L^p({O},\mu_{O})}^p+\|D_{\mathbf{b}} u\|_{L^p({O},\mu_{O};\ell^2)}^p\Bigr)^{1/p},\qquad u\in W_{\mathbf{b}}^{1,p}({O},\mu_{O}),
		\]
		where
		\[
		\|D_{\mathbf{b}} u\|_{L^p({O},\mu_{O};\ell^2)} \triangleq  \left(\int_{{O}}\left(\sum_{k=1}^{\infty}\left|b_k\frac{\partial u(\mathbf x)}{\partial x_k} \right|^2\right)^{\frac{p}{2}}\,\mathrm d \mu_{{O}}(\mathbf x)\right)^{\frac{1}{p}}.
		\]
		Equivalently, it is the completion of $\operatorname{span}\mathscr{C}_c^\infty$ under this graph norm.
	\end{definition}
	

	For two normed spaces $A$ and $B$, we write $A\hookrightarrow B$ (\resp $A \Subset B$) when $A$ is continuously (\resp compactly)
	embedded into $B$.
	Luo, Zhang and Zhao~\cite{LZZ23} investigated weighted Gaussian Sobolev spaces on $\ell^r$.
	They characterised continuous and compact logarithmic Sobolev embeddings on $\ell^r$ ($1\le r<\infty$) for positive weights $\mathbf a\in\ell^r$ and $\mathbf b\in\ell^\infty$.
	
	For $1\le p<\infty$ and $r=2$, Theorem~4.2 in \cite{LZZ23} gives
	\begin{equation}
		W_{\mathbf{b}}^{1,p}(\ell^2,P)\hookrightarrow L^p\log^{p/2}L(\ell^2,P)
		\iff h(\mathbf{b}/\mathbf{a})>0.
		\label{eq:full_cont}
	\end{equation}
	For $1<p<\infty$, $r=2$ and $0\le q<p/2$, Theorem~5.2 in \cite{LZZ23} gives
	\begin{equation}
		W_{\mathbf{b}}^{1,p}(\ell^2,P)\Subset L^p\log^q L(\ell^2,P)
		\iff h(\mathbf{b}/\mathbf{a})=\infty.
		\label{eq:full_comp}
	\end{equation}
	Remark~5.3 in~\cite{LZZ23} leaves the compact embedding for $p=1$ unproved and lists the critical exponent $q=p/2$ as an open problem.
	
	In this subsection, we prove the following result.

	\begin{theorem}\label{thm:embd}
		If ${O}$ is a convex set, then for every $1\le p<\infty$, the following conclusions hold.
		\begin{enumerate}
			\item[\upshape(i)] There is a continuous embedding
			\begin{equation}\label{eq:main_cont}
				W_{\mathbf{b}}^{1,p}({O},\mu_{O})\hookrightarrow  L^p\log^{p/2}L({O},\mu_{O})
			\end{equation}
			if and only if $h(\mathbf{b}/\mathbf{a})>0$.
			When this condition holds, there exists a constant $C_p$ depending only on $p$ such that
			\[
			\|u\|_{L^p\log^{p/2}L({O},\mu_{O})}
			\le C_p\big(\|u\|_{L^p({O},\mu_{O})}
			+M_{\mathbf a,\mathbf b}\|D_{\mathbf{b}} u\|_{L^p({O},\mu_{O};\ell^2)}\big),
			\]
			for any $u\in  W_{\mathbf{b}}^{1,p}({O},\mu_{O})$.
			
			\item[\upshape(ii)] For every $0\le q<p/2$, there is a compact embedding
			\[
			W_{\mathbf{b}}^{1,p}({O},\mu_{O})\Subset L^p\log^q L({O},\mu_{O})
			\]
			if and only if $\lim_{k\to\infty}r_k=\infty.$
		\end{enumerate}
	\end{theorem}

	Note that Theorem~\ref{thm:embd} implies Theorem~\ref{thm:log}.
	Indeed, if $\nu=P|_{O}$, then
\[
\|u\|_{W_{\mathbf b}^{1,p}({O},\mu_O)}
=P(O)^{-1/p}\|u\|_{W_{\mathbf b}^{1,p}(O,\nu)}.
\]
	By the convexity of the Young function $\Phi$,
	\[
	\|u\|_{L^\Phi({O},\nu)}\le\|u\|_{L^\Phi({O}, \mu_{O})}\le P({O})^{-1}\|u\|_{L^\Phi({O},\nu)},\qquad u\in L^\Phi({O},\nu).
	\]
	Therefore the necessary and sufficient conditions in the following theorem hold also for the unnormalised restricted measure $P|_{O}$, though constants in quantitative estimates may depend on $P({O})$.

	Let $\{\textbf{e}_k\}_{k=1}^{\infty}$ be the standard orthonormal basis of $\ell^2$, and for each $n\in\mathbb N,$ we define
	$$
	P_n\mathbf x\triangleq\sum_{k=1}^n x_k \textbf{e}_k,\qquad \mathbf x=\sum_{k=1}^{\infty} x_k \textbf{e}_k\in\ell^2,
	$$
	and
	\begin{equation}\label{20261003for1}
		K_n\triangleq\left\{\mathbf y=(y_1,\ldots,y_n)\in\mathbb{R}^n\,:\,\sum_{k=1}^n a_k y_k \textbf{e}_k\in{O}\right\},\qquad
		\tilde{O}_n\triangleq\{\mathbf{x}\in \ell^2\,:\,P_n \mathbf{x}\in{O}\}.
	\end{equation}
	Combining Lemma \ref{jihejixian} and its proof, we have the following result.
	
	\begin{lemma}
		\label{lemma:zero_bd_sec}
		If ${O}$ is a convex set,
		then $P(\partial{O})=0$, each $K_n$ is convex and open, and nonempty for all sufficiently large $n$.
		Moreover,
		\begin{equation}
			\mathbf{1}_{\tilde{O}_n}\to\mathbf{1}_{O}
			\quad\text{$P$-almost everywhere and in $L^1(\ell^2,P)$},\qquad
			\lim_{n\to\infty}\gamma_n(K_n)= P({O}).
		\end{equation}
	\end{lemma}

	
	The following inequality is well-known (see \cite[Theorem~2.2, p.~176]{Pis86}); for the reader's convenience we provide a proof here.
	\begin{lemma}[Gaussian $L^p$-Poincar\'e inequality]\label{lemma:Gaussian_L^p_Poincare}
		Let $1\leq p<\infty$. There exists a constant $K_p>0$ depending
		only on $p$ such that, for every $n\in\mathbb N$ and
		$f\in W^{1,p}(\mathbb R^n,\mathcal N^n)$,
		\[
		\left\|f-\int_{\mathbb R^n}f\,\mathrm{d}\mathcal N^n\right\|_{L^p(\mathbb R^n,\mathcal N^n)}
		\leq K_p\left\|
		\left(\sum_{k=1}^n a_k^2\left| \frac{\partial f}{\partial x_k}\right|^2
		\right)^{1/2}\right\|_{L^p(\mathbb R^n,\mathcal N^n)}.
		\]
		The constant $K_p$ is independent of $n$ and of
		$a_1,\ldots,a_n$.
	\end{lemma}
	
	\begin{proof}
		We first consider the standard Gaussian measure $\gamma_n$.
		It suffices to prove the estimate for
		$f\in C_c^\infty(\mathbb R^n)$. Set
		\[
		\bar f=\int_{\mathbb R^n}f\,\mathrm{d}\gamma_n.
		\]
		By Jensen's inequality,
		\[
		\begin{aligned}
			\int_{\mathbb R^n}|f(\mathbf x)-\bar f|^p\,
			\mathrm{d}\gamma_n(\mathbf x)
			&=
			\int_{\mathbb R^n}
			\left|
			\int_{\mathbb R^n}
			\bigl(f(\mathbf x)-f(\mathbf y)\bigr)\,
			\mathrm{d}\gamma_n(\mathbf y)
			\right|^p
			\mathrm{d}\gamma_n(\mathbf x)
			\\
			&\leq
			\int_{\mathbb R^n}\int_{\mathbb R^n}
			|f(\mathbf x)-f(\mathbf y)|^p\,
			\mathrm{d}\gamma_n(\mathbf y)
			\mathrm{d}\gamma_n(\mathbf x).
		\end{aligned}
		\]
		For $0\leq t\leq \pi/2$, set
		\[
		\mathbf z_t\triangleq\cos t\mathbf x+\sin t\mathbf y,
		\qquad
		\mathbf w_t\triangleq-\sin t\mathbf x+\cos t\mathbf y.
		\]
		Since $\mathbf z_0=\mathbf x$, $\mathbf z_{\pi/2}=\mathbf y$, and
		\[
		\frac{\mathrm{d}}{\mathrm{d}t}\mathbf z_t=\mathbf w_t,
		\]
		the fundamental theorem of calculus gives
		\[
		f(\mathbf y)-f(\mathbf x)
		=
		\int_0^{\pi/2}
		\nabla f(\mathbf z_t)\cdot \mathbf w_t\,\mathrm{d}t.
		\]
		Hence, by H\"older's inequality,
		\[
		|f(\mathbf y)-f(\mathbf x)|^p
		\leq
		\left(\frac{\pi}{2}\right)^{p-1}
		\int_0^{\pi/2}
		|\nabla f(\mathbf z_t)\cdot \mathbf w_t|^p\,\mathrm{d}t.
		\]
		For $p=1$, the same estimate follows directly.
		
		For fixed $t$, the change of variables
		\[
		\mathbf z=\cos t\mathbf x+\sin t\mathbf y,
		\qquad
		\mathbf w=-\sin t\mathbf x+\cos t\mathbf y,
		\]
		is an orthogonal transformation of $\mathbb R^{2n}$. Thus
		\[
		|\mathbf z|^2+|\mathbf w|^2=|\mathbf x|^2+|\mathbf y|^2,
		\qquad
		\mathrm{d}\mathbf x\,\mathrm{d}\mathbf y
		=
		\mathrm{d}\mathbf z\,\mathrm{d}\mathbf w,
		\]
		and therefore
		\[
		\mathrm{d}\gamma_n(\mathbf x)
		\mathrm{d}\gamma_n(\mathbf y)
		=
		\mathrm{d}\gamma_n(\mathbf z)
		\mathrm{d}\gamma_n(\mathbf w).
		\]
		Consequently,
		\[
		\int_{\mathbb R^n}\int_{\mathbb R^n}
			|\nabla f(\mathbf z_t)\cdot\mathbf  w_t|^p\,
			\mathrm{d}\gamma_n(\mathbf y)
			\mathrm{d}\gamma_n(\mathbf x)
			=
			\int_{\mathbb R^n}\int_{\mathbb R^n}
			|\nabla f(\mathbf z)\cdot \mathbf w|^p\,
			\mathrm{d}\gamma_n(\mathbf w)
			\mathrm{d}\gamma_n(\mathbf z).
		\]
		By the orthogonal invariance of the standard Gaussian density, for
		every $\mathbf v\in\mathbb R^n$,
		\[
		\int_{\mathbb R^n}
		|\mathbf v\cdot \mathbf w|^p\,\mathrm{d}\gamma_n(\mathbf w)
		=
		m_p\|\mathbf v\|^p,
		\]
		where
		\[
		m_p
		\triangleq
		\int_{\mathbb R}
		|x|^p\,\mathrm{d}\mathcal N_1(x).
		\]
		It follows that
		\[
		\begin{aligned}
			\int_{\mathbb R^n}|f-\bar f|^p\,
			\mathrm{d}\gamma_n
			&\leq
			\left(\frac{\pi}{2}\right)^{p-1}
			\int_0^{\pi/2}
			\int_{\mathbb R^n}\int_{\mathbb R^n}
			|\nabla f(\mathbf z)\cdot \mathbf w|^p\,
			\mathrm{d}\gamma_n(\mathbf w)
			\mathrm{d}\gamma_n(\mathbf z)
			\,\mathrm{d}t
			\\
			&=
			\left(\frac{\pi}{2}\right)^p
			m_p
			\int_{\mathbb R^n}
			|\nabla f(\mathbf z)|^p\,
			\mathrm{d}\gamma_n(\mathbf z).
		\end{aligned}
		\]
		Therefore
		\[
		\left\|
		f-\int_{\mathbb R^n}f\,\mathrm{d}\gamma_n
		\right\|_{L^p(\mathbb R^n,\gamma_n)}
		\leq
		K_p
		\|\nabla f\|_{L^p(\mathbb R^n,\gamma_n;\mathbb R^n)},
		\]
		where one may take
		\[
		K_p
		=
		\frac{\pi}{2}m_p^{1/p}.
		\]
		
		We now pass to the Gaussian measure
		$\mathcal N^n=\prod_{k=1}^n\mathcal N_{a_k}.$
		Make the change of variables
		$x_k=a_ky_k$ ($k=1,\ldots,n$),
		and define
		\[
		F(\mathbf y)\triangleq f(a_1y_1,\ldots,a_ny_n),\qquad \mathbf y=(y_1,\ldots,y_n)\in\mathbb R^n.
		\]
		Then
		\[
		\int_{\mathbb R^n}f\,\mathrm{d}\mathcal N^n
		=
		\int_{\mathbb R^n}F\,\mathrm{d}\gamma_n,
		\]
		and
		\[
	\frac{\partial  F(\mathbf y)}{\partial y_k}
		=
		a_k\frac{\partial f}{\partial x_k}(a_1y_1,\ldots,a_ny_n).
		\]
		Hence
		\[
		|\nabla F(\mathbf y)|
		=
		\left(
		\sum_{k=1}^n
		a_k^2
		\left|\frac{\partial f}{\partial x_k}(a_1y_1,\ldots,a_ny_n)\right|^2
		\right)^{1/2}.
		\]
		Applying the preceding estimate to $F$ and changing variables back
		gives
		\[
		\left\|
		f-\int_{\mathbb R^n}f\,\mathrm{d}\mathcal N^n
		\right\|_{L^p(\mathbb R^n,\mathcal N^n)}
		\leq
		K_p
		\left\|
		\left(
		\sum_{k=1}^n
		a_k^2\left|\frac{\partial f}{\partial x_k}\right|^2
		\right)^{1/2}
		\right\|_{L^p(\mathbb R^n,\mathcal N^n)}.
		\]
		The general case follows by density.
	\end{proof}
	
	For every $n\in\mathbb N$ and $f\in L^1(\ell^2, P)$, recall \eqref{20250920def1} for the definition of $f_n$ and its fundamental properties in Proposition \ref{Reduce dimension}.
	
	\begin{proposition}\label{pro:A_N_compactness}
		If $M_{\mathbf a,\mathbf b}<\infty$, then for every $n\in\mathbb N$ and $u\in W_{\mathbf{b}}^{1,p}(\ell^2,P)$,
		\begin{equation}\label{eq:A_N_compactness}
			\|u-u_n\|_{L^p(\ell^2,P)}
			\le K_p\left(\sup_{i>n}\frac{a_i}{b_i}\right)\|D_{\mathbf{b}} u\|_{L^p(P;\ell^2)}.
		\end{equation}
		Moreover, the mapping  $u\mapsto u_n$ maps $W_{\mathbf{b}}^{1,p}(\ell^2,P)$ boundedly into
		$W_{\mathbf{b}_n}^{1,p}(\mathbb R^n,\mathcal N^n)$, where
		$\mathbf{b}_n=(b_1,\dots,b_n)$.
	\end{proposition}
	\begin{proof}
		It suffices to consider $u\in\operatorname{span}\mathscr{C}_c^\infty$. By Lemma \ref{lemma:Gaussian_L^p_Poincare}
		and Fubini's theorem,
		\[
		\|u-u_n\|_{L^p(\ell^2,P)}\le
		K_p\left\|
		\left(\sum_{i>n}a_i^2\left|\frac{\partial u}{\partial x_i}\right|^2
		\right)^{1/2}\right\|_{L^p(\ell^2,P)}
		\le
		K_p\left(\sup_{i>n}\frac{a_i}{b_i}\right)\|D_{\mathbf{b}}u\|_{L^p(\ell^2,P;\ell^2)}.
		\]
		Moreover, for $k\le n$,
		\[
		b_k\frac{\partial u_n}{\partial x_k}
		=
		\left(b_k\frac{\partial u}{\partial x_k}\right)_n.
		\]
		Jensen's inequality therefore shows that the mapping $u\mapsto u_n$ is bounded
		from $W_{\mathbf{b}}^{1,p}(\ell^2,P)$ into
		$W_{\mathbf{b}_n}^{1,p}(\mathbb R^n,\mathcal N^n)$.
		The general case follows from the density of $\operatorname{span}\mathscr{C}_c^\infty$ in $W_{\mathbf{b}}^{1,p}(\ell^2,P)$.
	\end{proof}
	
	The following estimate is a basic tool used in~\cite{Led88,LZZ23}. We omit the proof here.
\begin{lemma}\label{lemma:Gaussian_Log_Sob_Embedding}
	For each $1\le p<\infty$ there exists a positive constant $C_p$ independent of $n$ such that for all $f\in W^{1,p}(\mathbb R^n,\gamma_n)$,
	\[
	\|f\|_{L^p\log^{p/2}L(\mathbb R^n,\gamma_n)}
	\le C_p\bigl(\|f\|_{L^p(\mathbb R^n,\gamma_n)}+\|\nabla f\|_{L^p(\mathbb R^n,\gamma_n;\mathbb R^n)}\bigr).
	\]
\end{lemma}
	\begin{lemma}\label{lemma:Finite_Gaussian_compactness}
		For every
		$1\le p<\infty$ and $n\in\mathbb N$,
		\[
		W_{\mathbf{b}_n}^{1,p}(\mathbb R^n,\mathcal N^n)
		\Subset
		L^p(\mathbb R^n,\mathcal N^n).
		\]
	\end{lemma}
	
	\begin{proof}
		Let $\{v_j\}_{j=1}^{\infty}$ be bounded in
		$W_{\mathbf{b}_n}^{1,p}(\mathbb R^n,\mathcal N^n)$.
		On every bounded ball, the density of $\mathcal N^n$ is bounded above
		and below by positive constants, and the weights $b_1,\ldots,b_n$ are
		positive. Hence $\{v_j\}_{j=1}^{\infty}$ is bounded in the usual local
		$W^{1,p}(\mathbb R^n)$ space.
		
		For every $R\in(0,\infty)$, denote the open ball in $\mathbb R^n$ with center at $\mathbf 0$ and radius $R$  by $B_R$. By the Rellich--Kondrachov theorem \cite[Theorem~6.3, p.~168]{AF},
		\[
		W^{1,p}(B_R) \Subset L^p(B_R),
		\qquad 1\le p<\infty.
		\]
		Indeed, for $1<p<n$, this follows from
		\[
		W^{1,p}(B_R) \Subset L^q(B_R),
		\qquad
		1\le q<p^*\triangleq\frac{np}{n-p},
		\]
		since $p<p^*$. For $p=n$, one has
		\[
		W^{1,n}(B_R) \Subset L^q(B_R),
		\qquad
		1\le q<\infty.
		\]
		For $p>n$, Morrey's inequality and the Arzel\`a--Ascoli theorem yield
		\[
		W^{1,p}(B_R)
		\hookrightarrow  C^{0,\alpha}(\overline{B_R})
		\Subset C(\overline{B_R})
		\hookrightarrow L^p(B_R),
		\qquad
		\alpha=1-\frac{n}{p}>0.
		\]
		The case $p=1$ follows directly from the Rellich--Kondrachov theorem.
		Therefore, by a diagonal argument, after passing to a subsequence,
		$v_j$ converges in $L^p(B_R)$ .
		
		On the other hand, Lemma~\ref{lemma:Gaussian_Log_Sob_Embedding}, together with the change of variables
		$x_k=a_k y_k$ for $k=1,\ldots,n$, yields the required logarithmic integrability estimate.
		Indeed, setting
		\[
		f_j(\mathbf y)\triangleq v_j(a_1y_1,\ldots,a_ny_n),\qquad j\in\mathbb N,\,\mathbf y=(y_1,\ldots,y_n)\in\mathbb R^n,
		\]
		we have
		\[
		\|f_j\|_{L^p(\mathbb R^n,\gamma_n )}=\|v_j\|_{L^p(\mathbb R^n,\mathcal N^n)},
		\]
		and
		\[
		|\nabla f_j|=\left(
		\sum_{k=1}^n a_k^2 \left| \frac{\partial v_j}{\partial x_k}\right|^2
		\right)^{1/2}
		\le\left(\max_{1\le k\le n}\frac{a_k}{b_k}\right) |D_{\mathbf{b}_n}v_j|.
		\]
		Hence Lemma~\ref{lemma:Gaussian_Log_Sob_Embedding} implies
		\[
		\sup_{j\in\mathbb N}
		\|v_j\|_{L^p\log^{p/2}L(\mathbb R^n,\mathcal N^n)}<\infty.
		\]
		By  Lemma \ref{20261002lem1},
			\[
			\sup_{j\in \mathbb{N}}
			\int_{\mathbb R^n}
			|v_j|^p
			\bigl[\log(e+|v_j|)\bigr]^{p/2}
			\,\mathrm{d}\mathcal N^n
			<\infty.
			\]
		Thus $\{|v_j|^p\}_{j=1}^{\infty}$ is uniformly integrable. Indeed, for every measurable
		set $E\subset\mathbb R^n$ and every $T>0$,
		\begin{equation}\label{20261002for1}
			\sup_{j\in \mathbb{N}}\int_E |v_j|^p\,\mathrm{d}\mathcal N^n
			\le
			T^p\mathcal N^n(E)
			+
			\frac{C}{[\log(e+T)]^{p/2}}.
		\end{equation}
		Taking $E=\mathbb R^n\setminus B_R$, first choosing $T$ large and then
		$R$ large, we obtain
		\[
		\sup_{j\in \mathbb{N}}
		\int_{\mathbb R^n\setminus B_R}
		|v_j|^p\,\mathrm{d}\mathcal N^n
		\longrightarrow 0
		\qquad (R\to\infty).
		\]
		Combining this with the local $L^p$ convergence
		shows that $v_j$ converges in $L^p(\mathbb R^n,\mathcal N^n)$.
	\end{proof}

	\begin{theorem}\label{thm:compact_embed_full_endpoint}
		If $\lim\limits_{k\to\infty}r_k=\infty$, then, for every $1\le p<\infty$,
		\[
		W_{\mathbf{b}}^{1,p}(\ell^2,P)
		\Subset
		L^p(\ell^2,P).
		\]
	\end{theorem}
	
	\begin{proof}
		Fix $n\in\mathbb N$. By Proposition \ref{pro:A_N_compactness},
		the mapping $u\mapsto u_n$ maps
		$W_{\mathbf{b}}^{1,p}(\ell^2,P)$ boundedly into
		$W_{\mathbf{b}_n}^{1,p}(\mathbb R^n,\mathcal N^n).$
		By Lemma~\ref{lemma:Finite_Gaussian_compactness}, the latter space is compactly embedded into
		$L^p(\mathbb R^n,\mathcal N^n)$. Since $L^p(\mathbb R^n,\mathcal N^n)$ can be viewed as a closed subspace of $L^p(\ell^2,P)$ which depends only on the first $n$ variables, it follows that
		\[
		u\mapsto u_n:
		W_{\mathbf{b}}^{1,p}(\ell^2,P)\longrightarrow L^p(\ell^2,P),
		\]
		is compact.
		
		Since $\lim\limits_{k\to\infty}r_k=\infty$ is equivalent to $\limsup_{k\to\infty}\frac{a_k}{b_k}=\lim\limits_{k\to\infty}\sup_{i>k}\frac{a_i}{b_i}=0$, the right-hand side of \eqref{eq:A_N_compactness} tends to zero.
		Hence the embedding operator is the operator-norm limit of the compact
		operators and is therefore compact.
	\end{proof}

	We use the following form of Caffarelli's contraction theorem~\cite{Caf00,FGP20}:
	if $K\subset\mathbb R^n$ is a nonempty convex open set and
	$$
	\nu_K\triangleq\gamma_n(K)^{-1}\textbf 1_K\gamma_n,
	$$
	then there exists a $1$-Lipschitz map $T:\mathbb R^n\to\mathbb R^n$ such that
	\[
	T_\#\gamma_n=\nu_K,
	\]
	where $(T_\#\gamma_n)(B)\triangleq\gamma_n(T^{-1}B)$ for every $B\in\mathscr {B}(\mathbb R^n)$. This version is a consequence of Theorem~1 in~\cite{FGP20}.

	\begin{proposition}\label{pro:continuous_convex}
		For any nonempty convex open set ${O}\subset \ell^2$,
		\[
		\|u\|_{L^p\log^{p/2}L({O},\mu_{O})}
		\le C_p\bigl(\|u\|_{L^p({O},\mu_{O})}+\|D_{\mathbf{a}} u\|_{L^p({O},\mu_{O};\ell^2)}\bigr),
		\qquad u\in W_{\mathbf a}^{1,p}({O},\mu_{O}).
		\]

	\begin{proof}
		Start with a cylindrical function
		\[
		u(\mathbf{x})=\varphi(x_1,\ldots,x_m),\qquad
		\varphi\in C_b^\infty(\mathbb R^m).
		\]
		For $n\geq m$ sufficiently large, by Lemma \ref{lemma:zero_bd_sec} the set $K_n$ given at \eqref{20261003for1} is nonempty open convex. We define
		\[
		F_n(\mathbf y)
		\triangleq
		\varphi(a_1y_1,\ldots,a_my_m),
		\qquad
		\mathbf y=(y_1,\ldots,y_n)\in\mathbb R^n,
		\]
		and set
		\[
		\nu_{K_n}
		\triangleq
		\gamma_n(K_n)^{-1}\textbf 1_{K_n}\gamma_n.
		\]
		By Caffarelli's contraction theorem, there exists
		a $1$-Lipschitz map $T_n:\mathbb R^n\longrightarrow \mathbb R^n$
		such that $(T_n)_\#\gamma_n=\nu_{K_n}.$
	
		Applying Lemma~\ref{lemma:Gaussian_Log_Sob_Embedding} to
		$F_n\circ T_n$ gives
		\[
		\begin{aligned}
			\|F_n\circ T_n\|_{L^p\log^{p/2}L(\mathbb R^n,\gamma_n)}
			\leq C_p\Big(
			&\|F_n\circ T_n\|_{L^p(\mathbb R^n,\gamma_n)}
		 +
			\|\nabla(F_n\circ T_n)\|_
			{L^p(\mathbb R^n,\gamma_n;\mathbb R^n)}
			\Big),
		\end{aligned}
		\]
		where $C_p$ is independent of $n$.
		Since
		$
		(T_n)_\#\gamma_n=\nu_{K_n},
	$
		we have
		\[
		\|F_n\circ T_n\|_{L^p\log^{p/2}L(\mathbb R^n,\gamma_n)}
		=
		\|F_n\|_{L^p\log^{p/2}L(\mathbb R^n, \nu_{K_n})}
		\]
		and
		\[
		\|F_n\circ T_n\|_{L^p(\mathbb R^n,\gamma_n)}
		=
		\|F_n\|_{L^p(\mathbb R^n,\nu_{K_n})}.
		\]
		Moreover, by the chain rule and the $1$-Lipschitz property of $T_n$,
		\[
		\|\nabla(F_n\circ T_n)(\mathbf y)\|
		\leq
		\|\nabla F_n(T_n(\mathbf y))\|
		\]
		for almost every $\mathbf y\in\mathbb R^n$. Hence
		\[
		\|\nabla(F_n\circ T_n)\|_
		{L^p(\mathbb R^n,\gamma_n;\mathbb R^n)}
		\leq
		\|\nabla F_n\|_{L^p(\mathbb R^n,\nu_{K_n};\mathbb R^n)}.
		\]
		It follows that
		\begin{equation}\label{20261003for2}
		\|F_n\|_{L^p\log^{p/2}L(\mathbb R^n, \nu_{K_n})}
		\leq
		C_p\left(
		\|F_n\|_{L^p(\mathbb R^n,\nu_{K_n})}
		+
		\|\nabla F_n\|_{L^p(\mathbb R^n,\nu_{K_n};\mathbb R^n)}
		\right).
		\end{equation}
		
		We next identify the terms in this inequality with integrals on
		$\ell^2$. By the definition of $K_n$ in
		\eqref{20261003for1} and the change of variables
		\[
		x_k=a_ky_k,
		\qquad k=1,\ldots, n,
		\]
		we have
	$
		\gamma_n(K_n)=P(\tilde{O}_n).
$
		Since $u$ depends only on the first $m$ coordinates, Fubini's theorem
		gives
		\[
		\|F_n\|_{L^p(\mathbb R^n,\nu_{K_n})}^p
		=
		\frac{1}{P(\tilde{O}_n)}
		\int_{\tilde{O}_n}
		|u|^p\,\mathrm{d}P.
		\]
		Furthermore,
		\[
	 \frac{\partial F_n(\mathbf y)}{\partial y_k}
		=
		\begin{cases}
			a_k\displaystyle\frac{\partial u}{\partial x_k} (a_1y_1,\ldots,a_my_m),
			& 1\leq k\leq m,\\
			0,
			& m<k\leq n.
		\end{cases}
		\]
		Therefore
		\[
		\|\nabla F_n(\mathbf y)\|^2
		=
		\sum_{k=1}^m
		a_k^2
		\left|\frac{\partial u}{\partial x_k} (a_1y_1,\ldots,a_my_m)\right|^2.
		\]
		Since $\frac{\partial u}{\partial x_k} =0$ for $k>m$, the Euclidean gradient of $F_n$
		corresponds exactly to $D_{\mathbf{a}}u$. Consequently,
		\[
		\|\nabla F_n\|_{L^p(\mathbb R^n,\nu_{K_n};\mathbb R^n)}^p
		=
		\frac{1}{P(\tilde{O}_n)}
		\int_{\tilde{O}_n}
		\|D_{\mathbf{a}}u\|^p\,\mathrm{d}P.
		\]
		Similarly, for every $\lambda>0$,
		\[
		\int_{K_n}
		\Phi_{p,p/2}\left(\frac{|F_n|}{\lambda}\right)
		\,\mathrm{d}\nu_{K_n}
		=
		\frac{1}{P(\tilde{O}_n)}
		\int_{\tilde{O}_n}
		\Phi_{p,p/2}\left(\frac{|u|}{\lambda}\right)
		\,\mathrm{d}P.
		\]
		
		By Lemma~\ref{lemma:zero_bd_sec},
		\[
		\textbf{1}_{\tilde{O}_n}
		\longrightarrow
		\textbf{1}_{O}
		\quad\text{$P$-almost everywhere},
		\qquad
		P(\tilde{O}_n)\longrightarrow P({O}).
		\]
		Since both $u$ and $D_{\mathbf{a}}u$ are bounded, dominated
		convergence gives
		\[
		\|F_n\|_{L^p(\mathbb R^n,\nu_{K_n})}
		\longrightarrow
		\|u\|_{L^p({O},\mu_{O})}
		\]
		and
		\[
		\|\nabla F_n\|_{L^p(\mathbb R^n,\nu_{K_n};\mathbb R^n)}
		\longrightarrow
		\|D_{\mathbf{a}}u\|_{L^p({O},\mu_{O};\ell^2)}.
		\]
		Set
		\[
		A
		=
		C_p\left(
		\|u\|_{L^p({O},\mu_{O})}
		+
		\|D_{\mathbf{a}}u\|_{L^p({O},\mu_{O};\ell^2)}
		\right).
		\]
		Fix $\varepsilon>0$. For all sufficiently large $n$,
		\[
		C_p\left(
		\|F_n\|_{L^p(\mathbb R^n,\nu_{K_n})}
		+
		\|\nabla F_n\|_{L^p(\mathbb R^n,\nu_{K_n};\mathbb R^n)}
		\right)
		<
		A+\varepsilon.
		\]
Combining \eqref{20261003for2} and the definition of the Luxemburg norm,
		\[
		\int_{K_n}
		\Phi_{p,p/2}\left(
		\frac{|F_n|}{A+\varepsilon}
		\right)
		\,\mathrm{d}\nu_{K_n}
		\leq 1.
		\]
		Using the preceding change-of-variables identity, we obtain
		\[
		\frac{1}{P(\tilde{O}_n)}
		\int_{\tilde{O}_n}
		\Phi_{p,p/2}\left(
		\frac{|u|}{A+\varepsilon}
		\right)
		\,\mathrm{d}P
		\leq 1.
		\]
		Since $u$ is bounded, the dominated convergence theorem and
		Lemma~\ref{lemma:zero_bd_sec} yield
		\[
		\int_{O}
		\Phi_{p,p/2}\left(
		\frac{|u|}{A+\varepsilon}
		\right)
		\,\mathrm{d}\mu_{O}
		\leq 1.
		\]
		Therefore
		\[
		\|u\|_{L^p\log^{p/2}L({O},\mu_{O})}
		\leq A+\varepsilon.
		\]
		Letting $\varepsilon\to0$ gives
		\[
		\|u\|_{L^p\log^{p/2}L({O},\mu_{O})}
		\leq
		C_p\left(
		\|u\|_{L^p({O},\mu_{O})}
		+
		\|D_{\mathbf{a}}u\|_{L^p({O},\mu_{O};\ell^2)}
		\right)
		\]
		for cylindrical functions.
		
		Now take an approximating sequence
		$\{u_j\}_{j=1}^{\infty}\subset\operatorname{span}\mathscr{C}_c^\infty $ in the domain graph norm.
		Applying the preceding estimate to $u_j-u_k$, we see that
		$\{u_j\}_{j=1}^{\infty}$ is a Cauchy sequence in
		$L^p\log^{p/2}L({O},\mu_{O})$.
		Since this space embeds continuously into $L^p({O},\mu_{O})$,
		its limit agrees with $u$. Passing to the limit yields the desired
		estimate for every
		$u\in W_{\mathbf{a}}^{1,p}({O},\mu_{O})$.
	\end{proof}

	\end{proposition}
	
	\begin{proof}[Proof of the ``if'' part in Theorem~\ref{thm:embd} (i)]
		If $h(\mathbf{b}/\mathbf{a})>0$, then $M_{\mathbf a,\mathbf b}<\infty$, and
		\[
		\|D_{\mathbf{a}} u\| \le M_{\mathbf a,\mathbf b}\,\|D_{\mathbf{b}} u\|,\qquad u\in\operatorname{span}\mathscr{C}_c^\infty.
		\]
		By closure this gives a continuous embedding $W_{\mathbf{b}}^{1,p}({O},\mu_{O})\hookrightarrow W_{\mathbf{a}}^{1,p}({O},\mu_{O})$.
		Combined with Proposition~\ref{pro:continuous_convex} we obtain the desired Orlicz bound.
	\end{proof}

	\begin{proposition}
		If $h(\mathbf{b}/\mathbf{a})=0$, then
		\[
		W_{\mathbf{b}}^{1,p}({O},\mu_{O})\not\hookrightarrow L^p\log^{p/2}L({O},\mu_{O}).
		\]
	\end{proposition}
	\begin{proof}
		Pass to a subsequence still indexed by $n$ with $r_n=b_n/a_n\to 0$.
		Choose $\mathbf{x}^*=(x_k^*)_{k\in\mathbb N}$ with $\supp(\mathbf{x}^*)$ finite and $R>0$ such that $B(\mathbf{x}^*,R)\subset{O}$.
		Choose $n\in\mathbb N$ so that $x_n^*=0$.
		Define
		\[
		s_n\triangleq\min\left\{\frac1{r_n},\frac{R}{4p a_n}\right\},\qquad
		\mathbf{h}_n\triangleq p a_n s_n \textbf{e}_n.
		\]
		Since $a_n\to 0$ and $r_n\to 0$, $s_n\to\infty$; at the same time $r_n s_n\le 1$ and $\|\mathbf{h}_n\|\le R/4$.
		Set
		\[
		f_n(\mathbf{x})\triangleq\exp\Bigl(s_n\frac{x_n}{a_n}-\frac p2 s_n^2\Bigr),\qquad \mathbf x=(x_i)_{i\in\mathbb N}\in\ell^2.
		\]
		
		The translation identity used below follows directly from the one-dimensional measure $\mathcal N_{a_n}$ and the product structure of $P$.  Since
		\[
		f_n(\mathbf{x})^p\,\dd\mathcal N_{a_n}(x_n)
		=\frac{1}{\sqrt{2\pi}a_n}
		\exp\!\left(-\frac{(x_n-pa_ns_n)^2}{2a_n^2}\right)\dd x_n,
		\]
		Fubini's theorem and the substitution $y_n=x_n-pa_ns_n$ give, for every non-negative measurable function $\Psi$ on $\ell^2$,
		\begin{equation}\label{20261001for1}
			\int_{\ell^2} f_n(\mathbf{x})^p\Psi(\mathbf{x})\,\dd P(\mathbf{x})
			=\int_{\ell^2} \Psi(\mathbf{x}+\mathbf{h}_n)\,\dd P(\mathbf{x}).
		\end{equation}
		Moreover $D_{\mathbf{b}} f_n=r_n s_n f_n \textbf{e}_n$, hence
		\[
		\|f_n\|_{L^p({O},\mu_{O})}^p\le P({O})^{-1},\qquad
		\|D_{\mathbf{b}} f_n\|_{L^p({O},\mu_{O};{\ell^2})}^p\le P({O})^{-1}.
		\]
		
		Let
		\[
		A_n\triangleq B(\mathbf{x}^*,R/2)\cap\{\mathbf{x}=(x_i)_{i\in\mathbb N}\in\ell^2:x_n\ge 0\},\qquad
		\beta\triangleq P\bigl(B(\mathbf{x}^*,R/2)\bigr)>0.
		\]
		Reflection in the $n$-th coordinate preserves the measure and the ball, so $P(A_n)=\beta/2$.
		For every $\mathbf{x}\in A_n$, we have
		\begin{equation}\label{20261001for2}
			\mathbf{x}+\mathbf h_n\in{O},\qquad \log f_n(\mathbf{x}+\mathbf{h}_n)=s_n\frac{x_n}{a_n}+\frac p2 s_n^2
			\ge \frac p2 s_n^2.
		\end{equation}
		
		Fix arbitrary $\lambda>0$.
		For sufficiently large $n$, by \eqref{20261001for1} and \eqref{20261001for2}, we have
		\[
		\begin{aligned}
			\int_{O}\Phi_{p,p/2}\!\left(\frac{f_n}{\lambda}\right)\dd\mu_{O}
			&= \frac{1}{P({O})\lambda^p}\int_{\ell^2}|f_n|^p\left[\log\Bigl(e+\frac{f_n }{\lambda}\Bigr)\right]^{p/2}\mathbf 1_{{O}}\dd P \\
			&=\frac{1}{P({O})\lambda^p}\int_{\ell^2} \left[\log\left(e+\frac{f_n(\mathbf{x}+\mathbf{h}_n) }{\lambda}\right)\right]^{p/2}\mathbf 1_{{O}}(\mathbf{x}+\mathbf{h}_n)\dd P(\mathbf{x}) \\
			&\ge \frac{1}{P({O})\lambda^p}
			\int_{A_n}\left[\log\left(e+\frac{f_n(\mathbf{x}+\mathbf{h}_n)}{\lambda}\right)\right]^{p/2}\dd P(\mathbf{x})\\
			&\ge \frac{\beta}{2 P({O})\lambda^p}
			\left[\log\left(e+\frac{ps_n^2}{2\lambda}\right)\right]^{p/2}\longrightarrow\infty.
		\end{aligned}
		\]
		Since $\lambda>0$ is arbitrary, we have
		$\|f_n\|_{L^p\log^{p/2}L({O},\mu_{O})}\to\infty$, whereas $\{f_n\}_{n=1}^{\infty}$ is bounded in
		$W_{\mathbf{b}}^{1,p}({O},\mu_{O})$. This contradicts the continuity of the
		embedding.
	\end{proof}

We need the following result.
	\begin{lemma}\label{lemma:Lipschitz_multiplier}

		\begin{enumerate}
			\item[\upshape(i)]
			Every bounded Lipschitz function
			$\eta:\ell^2\to\mathbb R$ belongs to
			$W_{\mathbf{b}}^{1,p}(\ell^2,P)$ for every $1\leq p<\infty$, and
			\[
			\|D_{\mathbf{b}}\eta\|
			\leq\left(\sup_{k\in\mathbb N}b_k\right)\cdot\operatorname{Lip}(\eta)\qquad
			P\text{-almost everywhere}.
			\]
			
			\item[\upshape(ii)]
		If $0\leq\eta\leq1$ is a
			Lipschitz function on $\ell^2$ such that
		$
			\eta=0
			\,\text{on }\ell^2\setminus{O},
	$
			then, for every
			$u\in W_{\mathbf{b}}^{1,p}({O},\mu_{O})$, the zero-extension
			\[
			E_\eta u(\mathbf{x})
			\triangleq
			\begin{cases}
				\eta(\mathbf{x})u(\mathbf{x}),
				& \mathbf{x}\in{O},\\
				0,
				& \mathbf{x}\in\ell^2\setminus{O},
			\end{cases}
			\]
			belongs to $W_{\mathbf{b}}^{1,p}(\ell^2,P)$ and satisfies
			\[
			D_{\mathbf{b}}(E_\eta u)
			=
			\mathbf{1}_{O}
			\bigl(
			\eta D_{\mathbf{b}}u
			+
			uD_{\mathbf{b}}\eta
			\bigr)
			\qquad
			P\text{-almost everywhere}.
			\]
			In particular,
			\[
			E_\eta:
			W_{\mathbf{b}}^{1,p}({O},\mu_{O})
			\longrightarrow
			W_{\mathbf{b}}^{1,p}(\ell^2,P)
			\]
			is a bounded linear operator.
		\end{enumerate}
	\end{lemma}
	
	\begin{proof}
		Set
		\[
		L\triangleq\operatorname{Lip}(\eta),
		\qquad
		M\triangleq\sup_{\mathbf x\in\ell^2}|\eta(\mathbf x)|,
		\qquad
		B\triangleq\sup_{k\in\mathbb N}b_k.
		\]
		
		We first prove \textup{(i)}.
		For $n\in\mathbb N$, set
		\[
		\eta^{(n)}(\mathbf{x})\triangleq\eta(P_n\mathbf{x}),\qquad \mathbf x\in\ell^2.
		\]
		Since $P_n\mathbf{x}\to\mathbf{x}$ in $\ell^2$ and $\eta$ is Lipschitz,
		\[
		\eta^{(n)}(\mathbf{x})\to\eta(\mathbf{x})
		\qquad
		\text{for every }\mathbf{x}\in\ell^2.
		\]
		Moreover,
		\[
		|\eta^{(n)}|\leq M.
		\]
		Hence, by the dominated convergence theorem,
		\[
		\eta^{(n)}\to\eta
		\qquad\text{in }L^2(\ell^2,P).
		\]
		
		Identify $\operatorname{span}\{\textbf{e}_1,\ldots,\textbf{e}_n\}$
		with $\mathbb R^n$ and define
		\[
		g_n(y_1,\ldots,y_n)
		\triangleq
		\eta\left(\sum_{k=1}^n y_k\textbf{e}_k\right),\qquad (y_1,\ldots,y_n)\in\mathbb R^n.
		\]
		Then
		\[
		\operatorname{Lip}(g_n)\leq L,
		\qquad
		\sup_{\mathbf y\in\mathbb R^n}|g_n(\mathbf y)|\leq M.
		\]
		Let $\rho_\varepsilon$ be a standard mollifier on $\mathbb R^n$ and set
		\[
		g_{n,\varepsilon}\triangleq g_n*\rho_\varepsilon,\qquad n\in\mathbb N,\,\varepsilon\in(0,\infty).
		\]
		Then
		\[
	\sup_{\mathbf y\in\mathbb R^n}|g_{n,\varepsilon}(\mathbf y)|
		\leq M,
		\qquad
		\|\nabla g_{n,\varepsilon}\|
		\leq L
		\quad\text{on }\mathbb R^n.
		\]
		Furthermore,
		\[
		g_{n,\varepsilon}\to g_n
		\qquad
		\text{locally uniformly as }\varepsilon\to0.
		\]
		Choosing $\varepsilon_n\to0$ sufficiently fast and, if necessary,
		multiplying by finite-dimensional smooth cut-off functions, we obtain
		smooth cylindrical functions $\{\eta_n\}_{n=1}^{\infty}\subset\operatorname{span}\mathscr C_c^\infty$ such that
		\[
		\eta_n\to\eta
		\qquad\text{in }L^2(\ell^2,P),
		\]
		and
		\[
		\|\eta_n\|_{L^\infty(P)}\leq M+o(1),
		\qquad
		\|D_{\mathbf{b}}\eta_n\|
		\leq BL+o(1)
		\quad
		P\text{-almost everywhere}\,\,\text{as}\,n\to\infty.
		\]
		Indeed, before the cut-off modification,
		\[
		\begin{aligned}
			\|D_{\mathbf{b}}\eta_n\|^2
			&=
			\sum_{k=1}^n
			b_k^2\left| \frac{\partial \eta_n}{\partial x_k}\right|^2
			\leq
			B^2
			\sum_{k=1}^n
			\left| \frac{\partial \eta_n}{\partial x_k}\right|^2
			\leq B^2L^2.
		\end{aligned}
		\]
		The finite-dimensional cut-offs can be chosen with gradients tending
		uniformly to zero, so the same estimate holds up to an $o(1)$ term.
		
		Thus $\{D_{\mathbf{b}}\eta_n\}_{n=1}^{\infty}$ is bounded in
		$L^2(\ell^2,P;\ell^2)$. Passing to a subsequence, there exists
		$G\in L^2(\ell^2,P;\ell^2)$ such that
		\[
		D_{\mathbf{b}}\eta_n\rightharpoonup G
		\qquad
		\text{weakly in }L^2(\ell^2,P;\ell^2).
		\]
		The graph of the closure of $D_{\mathbf{b}}$ is a closed linear
		subspace of
		\[
		L^2(\ell^2,P)\times L^2(\ell^2,P;\ell^2),
		\]
		and is therefore weakly closed. Since
		$\eta_n\to\eta$ strongly, hence weakly, in $L^2(\ell^2,P)$, it follows that
		\[
		\eta\in W_{\mathbf{b}}^{1,2}(\ell^2,P),
		\qquad
		D_{\mathbf{b}}\eta=G.
		\]
		The pointwise bound passes to the weak limit and gives
		\[
		\|D_{\mathbf{b}}\eta\|
		\leq BL
		\qquad
		P\text{-almost everywhere}.
		\]
		
		By Mazur's lemma, there exist
\[
\zeta_j\in\operatorname{conv}\{\eta_n:n\ge j\}\subset\operatorname{span}\mathscr C_c^\infty,
\]
where $\operatorname{conv}\{\cdot\}$ denotes the convex hull,
such that
\[
\|\zeta_j-\eta\|_{L^2(\ell^2,P)}+\|D_{\mathbf b}\zeta_j-D_{\mathbf b}\eta\|_{L^2(\ell^2,P;\ell^2)}\longrightarrow0,
\]
with
\[
\sup_j\left(\|\zeta_j\|_{L^\infty(\ell^2,P)}+\|D_{\mathbf b}\zeta_j\|_{L^\infty(\ell^2,P;\ell^2)}\right)<\infty.
\]
Using
\[
\|v\|_{L^p(\ell^2,P)}\le
\begin{cases}
\|v\|_{L^2(\ell^2,P)},&1\le p\le2,\\
\|v\|_{L^\infty(\ell^2,P)}^{1-2/p}\|v\|_{L^2(\ell^2,P)}^{2/p},&p>2,
\end{cases}
\]
we obtain
\[
\|\zeta_j-\eta\|_{W_{\mathbf b}^{1,p}(\ell^2,P)}\longrightarrow0,\qquad 1\le p<\infty.
\]
This proves \textup{(i)}.
		
		We next prove \textup{(ii)}. We first record the locality property
		\[
		D_{\mathbf{b}}\eta=0
		\qquad
		P\text{-almost everywhere on }\{\eta=0\}.
		\]
		Fix $k\in\mathbb N$ and write
$P=\mathcal N_{a_k}\times P^{\widehat k}$.
For each $\mathbf x_{\mathbb N\setminus\{k\}}\in\ell^2(\mathbb N\setminus\{k\})$, the function
\[
t\longmapsto\eta(\mathbf x_{\mathbb N\setminus\{k\}}+t\mathbf e_k)
\]
is Lipschitz on $\mathbb{R}$. Hence
\[
\frac{\mathrm d}{\mathrm dt}
\eta(\mathbf x_{\mathbb N\setminus\{k\}}+t\mathbf e_k)=0
\quad\text{for }\mathcal N_{a_k}\text{-a.e. }t
\text{ such that }
\eta(\mathbf x_{\mathbb N\setminus\{k\}}+t\mathbf e_k)=0.
\]
By Fubini's theorem,
		\[
	\frac{\partial \eta}{\partial x_k}=0
		\qquad
		P\text{-almost everywhere on }\{\eta=0\}.
		\]
		Since $k$ is arbitrary,
		\[
		D_{\mathbf{b}}\eta=0
		\qquad
		P\text{-almost everywhere on }\{\eta=0\}.
		\]
		In particular, because $\eta=0$ on $\ell^2\setminus{O}$,
		\[
		D_{\mathbf{b}}\eta=0
		\qquad
		P\text{-almost everywhere on }\ell^2\setminus{O}.
		\]
		
		The approximation in \textup{(i)} also gives the usual product rule
		with a smooth cylindrical function $v$:
		\[
		D_{\mathbf{b}}(\eta v)
		=
		\eta D_{\mathbf{b}}v
		+
		vD_{\mathbf{b}}\eta.
		\]
		Indeed, apply the classical product rule to the smooth cylindrical
		approximations of $\eta$ and pass to the graph-norm limit.
		
		Now let $u\in W_{\mathbf{b}}^{1,p}({O},\mu_{O}).$
		By Definition~\ref{def:Wb1p}, choose
		$\{u_j\}_{j=1}^{\infty}\subset\operatorname{span}\mathscr C_c^\infty$
		such that
		\[
		u_j\to u \quad\text{in }L^p({O},\mu_{O}),\qquad
		D_{\mathbf{b}}u_j\to D_{\mathbf{b}}u\quad\text{in }L^p({O},\mu_{O};\ell^2).
		\]
		For $v=u_j-u_k$, since $\eta=0$ outside ${O}$ and
		$0\leq\eta\leq1$,
		\[
		\|\eta v\|_{L^p(\ell^2,P)}
		\leq
		P({O})^{1/p}
		\|v\|_{L^p({O},\mu_{O})}.
		\]
		By \textup{(i)}, $\|D_{\mathbf{b}}\eta\|\leq BL$
		$P$-almost everywhere. Using the product rule and the locality
		property,
		\[
		D_{\mathbf{b}}(\eta v)
		=
		\textbf{1}_{O}
		\bigl(
		\eta D_{\mathbf{b}}v
		+
		vD_{\mathbf{b}}\eta
		\bigr),
		\]
		and therefore
		\[
		\begin{aligned}
			\|D_{\mathbf{b}}(\eta v)\|_{L^p(\ell^2,P;\ell^2)}
			\leq
			P({O})^{1/p}\Bigl(
			&
			\|D_{\mathbf{b}}v\|_{L^p({O},\mu_{O};\ell^2)}
			 +
			BL\|v\|_{L^p({O},\mu_{O})}
			\Bigr).
		\end{aligned}
		\]
		Consequently, $\{\eta u_j\}_{j=1}^{\infty}$ is a Cauchy sequence in $W_{\mathbf{b}}^{1,p}(\ell^2,P)$.
		
		Moreover,
		\[
		\|\eta u_j-E_\eta u\|_{L^p(\ell^2,P)}
		\leq
		P({O})^{1/p}
		\|u_j-u\|_{L^p({O},\mu_{O})}
		\longrightarrow0.
		\]
		Hence the graph-norm limit of $\{\eta u_j\}_{j=1}^{\infty}$ is $E_\eta u$. Passing to
		the limit in the preceding product formula yields
		\[
		D_{\mathbf{b}}(E_\eta u)
		=
		\textbf{1}_{O}
		\bigl(
		\eta D_{\mathbf{b}}u
		+
		uD_{\mathbf{b}}\eta
		\bigr).
		\]
		Finally,
		\[
		\|E_\eta u\|_{L^p(\ell^2,P)}
		\leq
		P({O})^{1/p}\|u\|_{L^p({O},\mu_{O})}
		\]
		and
		\[
		\|D_{\mathbf{b}}(E_\eta u)\|_{L^p(\ell^2,P;\ell^2)}
		\leq
		P({O})^{1/p}
		\left(
		\|D_{\mathbf{b}}u\|_{L^p({O},\mu_{O};\ell^2)}
		+
		BL\|u\|_{L^p({O},\mu_{O})}
		\right).
		\]
		Thus $E_\eta$ is bounded.
	\end{proof}
	\begin{proposition}\label{pro:compactness_Sob_Lp}
		If ${O}$ is a convex set and $\lim\limits_{k\to\infty}r_k=\infty$, then
		\[
		W_{\mathbf{b}}^{1,p}({O},\mu_{O})\Subset L^p({O},\mu_{O}),\qquad 1\le p<\infty.
		\]
		
		\begin{proof}
			When ${O}=\ell^2$, this follows from
			Theorem~\ref{thm:compact_embed_full_endpoint}.
			Assume ${O}\neq \ell^2$.
			Let
			\[
			d_{O}(\mathbf{x})\triangleq \operatorname{dist}(\mathbf{x},\ell^2\setminus{O}),\qquad \mathbf x\in\ell^2.
			\]
			This is a non-negative $1$-Lipschitz function, positive on ${O}$ and zero on the complement of ${O}$.
			Choose a Lipschitz function $\theta:[0,\infty)\to[0,1]$ with $\theta=0$ on $[0,1]$ and $\theta=1$ on $[2,\infty)$.
			Set
			\[
			\eta_m(\mathbf{x})\triangleq\theta\bigl(m\,d_{O}(\mathbf{x})\bigr),\quad m\in\mathbb N,\,\mathbf x\in\ell^2.
			\]
			By Lemma \ref{lemma:Lipschitz_multiplier}, $E_{\eta_m}:W_{\mathbf{b}}^{1,p}({O},\mu_{O})\to W_{\mathbf{b}}^{1,p}(\ell^2,P)$ is bounded.
			Combined with Theorem~\ref{thm:compact_embed_full_endpoint} and boundedness of the restriction map to $L^p({O},\mu_{O})$, the operator
			\[
			\mathcal{K}_m :W_{\mathbf{b}}^{1,p}({O},\mu_{O})\longrightarrow L^p({O},\mu_{O}),\qquad u\longmapsto \eta_m u
			\]
			is compact.
			
			Since $\lim\limits_{k\to\infty}r_k=\infty$, we have $h(\mathbf{b}/\mathbf{a})>0$, so the already proved continuous embedding of Theorem~\ref{thm:embd} (i) and Lemma \ref{20261002lem1} give, for the unit ball $\mathcal B$ of $W_{\mathbf{b}}^{1,p}({O},\mu_{O})$,
			\[
			\sup_{u\in\mathcal B}\int_{O} |u|^p\bigl[\log(e+|u|)\bigr]^{p/2}\dd\mu_{O}\le A<\infty.
			\]
			Let
			\[
			E_m=\{\mathbf{x}\in{O}\,\big|\,d_{O}(\mathbf{x})<2/m\}.
			\]
			Then $E_m\rightarrow\varnothing$, and hence $\mu_{O}(E_m)\to 0$.
			The function $1-\eta_m$ can be non-zero only on $E_m$.
			Using the same uniform integrability estimate as \eqref{20261002for1},
			\[
			\sup_{u\in\mathcal B}\|u-\mathcal{K}_m u\|_{L^p({O},\mu_{O})}^p
			\le T^p\mu_{O}(E_m)+\frac{A}{[\log(e+T)]^{p/2}}.
			\]
			First send $m\to\infty$, then $T\to\infty$; this yields $K_m\to I$ in operator norm.
			Hence the domain embedding is also a norm limit of compact operators, therefore compact.
		\end{proof}
	\end{proposition}

	\begin{proof}[Proof of the ``if'' part in Theorem~\ref{thm:embd} (ii)]
		Let $\{u_j\}_{j=1}^{\infty}\subset W_{\mathbf b}^{1,p}(O,\mu_O)$ be a bounded sequence.
		Proposition~\ref{pro:compactness_Sob_Lp} yields an $L^p({O},\mu_{O})$-convergent subsequence $\{u_{j_k}\}_{k=1}^\infty$.
		By the continuous embedding~\eqref{eq:main_cont}, $\{u_{j_k}\}_{k=1}^\infty$ is bounded in
		$L^p\log^{p/2}L({O},\mu_{O})$. Since convergence in $L^p$ together
		with boundedness in $L^p\log^{p/2}L({O},\mu_{O})$ implies convergence in
		$L^p\log^qL({O},\mu_{O})$ for every $q<p/2$ (by the standard truncation argument as in \eqref{20261002for1}),
		the desired compactness follows.
	\end{proof}

	\begin{proposition}\label{pro:noncompact_Sob_Lp}
		If $h(\mathbf{b}/\mathbf{a})<\infty$, then the embedding
		\[
		W_{\mathbf{b}}^{1,p}({O},\mu_{O})
		\hookrightarrow
		L^1({O},\mu_{O})
		\]
		is not compact.
	\end{proposition}
	
	\begin{proof}
		Choose a subsequence $\{k_j\}_{j=1}^\infty$ such that $r_{k_j}\le M<\infty$, for every $j\in\mathbb N$ and some $M\in(0,\infty)$, and set
		\[
		u_j(\mathbf{x})
		\triangleq
		\sin\left(\frac{x_{k_j}}{a_{k_j}}\right),\qquad \mathbf x=(x_i)_{i\in\mathbb N}\in\ell^2,\,j\in\mathbb N.
		\]
		Then
		\[
		|u_j|\le1,
		\qquad
		D_{\mathbf{b}}u_j
		=
		r_{k_j}
		\cos\left(\frac{x_{k_j}}{a_{k_j}}\right)\mathbf e_{k_j},\qquad j\in\mathbb N,
		\]
		so $\{u_j\}_{j=1}^\infty$ is bounded in $W_{\mathbf{b}}^{1,p}({O},\mu_{O})$.
		
		Let
		\[
		d_0
		\triangleq
		\int_{\mathbb R^2}
		|\sin s-\sin t|\,
		\mathrm{d}\mathcal N_1(s)\,
		\mathrm{d}\mathcal N_1(t)
		>0.
		\]
		Set $F_n\triangleq(\mathbf 1_{O})_n$ for every $n\in\mathbb N$. Then by Proposition \ref{Reduce dimension} we have
		$F_n\longrightarrow \mathbf 1_{O}$
		in $L^1(\ell^2,P)$.
		
		For $k_j,k_l>n$, $j\ne l$, the product decomposition of $P$ and
		Fubini's theorem give
		\[
		\int_{\ell^2}
		F_n|u_j-u_l|\,\mathrm{d}P
		=
		P({O})d_0.
		\]
		Hence
		\[
		\left|
		\|u_j-u_l\|_{L^1({O},\mu_{O})}-d_0
		\right|
		\le
		\frac{2}{P({O})}
		\|F_n-\mathbf 1_{O}\|_{L^1(\ell^2,P)}
		\longrightarrow0.
		\]
		Thus, for all sufficiently large $j,l$ with $j\neq l$,
		\[
		\|u_j-u_l\|_{L^1({O},\mu_{O})}
		\ge\frac{d_0}{2}>0.
		\]
		Therefore $\{u_j\}_{j=1}^\infty$ has no $L^1({O},\mu_{O})$-Cauchy subsequence, and the
		embedding is not compact.
	\end{proof}
	
	\begin{proof}[Proof of the ``only if'' part in Theorem~\ref{thm:embd} (ii)]
		If $r_k$ does not tend to infinity, there exists a bounded subsequence, and hence $h(\mathbf{b}/\mathbf{a})<\infty$.
		By Proposition~\ref{pro:noncompact_Sob_Lp} the embedding into $L^1({O},\mu_{O})$ fails to be compact.
		Since $\mu_{O}({O})=1$, H\"older's inequality gives the chain of continuous embeddings
		\[
		L^p\log^q L({O},\mu_{O})\hookrightarrow L^p({O},\mu_{O})\hookrightarrow L^1({O},\mu_{O}).
		\]
		Therefore embedding into $L^p\log^q L({O},\mu_{O})$ cannot be compact either.
	\end{proof}
	
	\begin{theorem}\label{20261002thm1}
		There exist $\mathbf a$ and $\mathbf b$ satisfying \eqref{eq:assumptions} with
		$\lim\limits_{k\to\infty}r_k=\infty$ and a convex open set
		${O}\subsetneq\ell^2$ such that
		\begin{equation*}
			W_{\mathbf{b}}^{1,p}({O},\mu_{O})\hookrightarrow L^p\log^{p/2}L({O},\mu_{O}),
		\end{equation*}
		but this embedding is not compact. More precisely, one may take
		\[
		a_k=2^{-k},\qquad b_k=1,\qquad k\in\mathbb N,
		\]
		and
		${O}=\{\mathbf{x}=(x_i)_{i\in\mathbb N}\in\ell^2:x_1<0\}.$
	\end{theorem}
	
	\begin{proof}
		Since ${O}=\{\mathbf{x}\in\ell^2:x_1<0\}$ depends only on the
		first coordinate, the product decomposition of $P$ gives, for every
		non-negative measurable function $F$ on $\mathbb R$,
		\[
		\int_{O} F\left(\frac{x_2}{a_2}\right)\,\mathrm{d}\mu_{O}(\mathbf x)
		=
		\int_{\mathbb R}F(x)\,\mathrm{d}\mathcal N_1(x),
		\qquad
		\mathbf x=(x_i)_{i\in\mathbb N}\in \ell^2.
		\]
		
		For $t\in [1,\infty)$, define
		\[
		u_t(\mathbf{x})
		\triangleq
		t^{-1}
		\exp\left(
		\frac{tx_2}{a_2}-\frac p2t^2
		\right),\qquad \mathbf x=(x_i)_{i\in\mathbb N}\in \ell^2.
		\]
		It follows that
		$u_t\in W_{\mathbf{b}}^{1,p}({O},\mu_{O}).$
		We have
		\[
		\int_{\mathbb R}
		\exp\left(ptx-\frac{p^2}{2}t^2\right)\,\mathrm{d}\mathcal N_1(x)
		=\frac{1}{\sqrt{2\pi}}\int_{\mathbb R}
		\exp\left(-\frac{(x-pt)^2}{2} \right)\,\mathrm{d} x =1,
		\]
		and hence
		$
		\|u_t\|_{L^p({O},\mu_{O})}
		=
		t^{-1}.
		$
		Moreover,
		\[
		D_{\mathbf{b}}u_t
		=
		r_2t\,u_t\,\textbf{e}_2,
		\]
		so
		\[
		\|D_{\mathbf{b}}u_t\|_{L^p({O},\mu_{O};\ell^2)}
		=
		r_2.
		\]
		Thus $\{u_t:t\in [1,\infty)\}$ is bounded in
		$W_{\mathbf{b}}^{1,p}({O},\mu_{O})$ and
		\[
		u_t\longrightarrow0,
		\qquad\text{in }L^p({O},\mu_{O}),\qquad \text{as}\quad t\to\infty.
		\]
		
		More generally, for every non-negative Borel measurable function $F$ on $\mathbb R$,
		completing the square and making the substitution $y=x-pt$ yield
		\begin{equation}\label{20261002for2}
			\int_{O}
			|u_t(\mathbf x)|^pF\left(\frac{x_2}{a_2}\right)\,\mathrm{d}\mu_{O}(\mathbf x)
			=
			t^{-p}
			\int_{\mathbb R}
			F(y+pt)\,\mathrm{d}\mathcal N_1(y).
		\end{equation}
		Fix $\lambda>0$. We obtain
		\[
		\begin{aligned}
			\int_{O}
			\Phi_{p,p/2}
			\left(
			\frac{u_t}{\lambda}
			\right)
			\,\mathrm{d}\mu_{O}
			&=\frac{1}{\lambda^p}\int_{O}
			\left|
			u_t
			\right|^p\left[
			\log\left(
			e+
			\left|
			\frac{u_t}{\lambda}
			\right|
			\right)
			\right]^{p/2}
			\,\mathrm{d}\mu_{O}\\
			&=
			\frac{1}{\lambda^p}t^{-p}
			\int_{\mathbb R}
			\left[
			\log\left(
			e+
			\frac{1}{\lambda t}
			e^{ty+\frac p2t^2}
			\right)
			\right]^{p/2}
			\,\mathrm{d}\mathcal N_1(y),
		\end{aligned}
		\]
		where the second equality follows from \eqref{20261002for2}. For $y\ge0$ and all sufficiently large $t$,
		\begin{equation}\label{20261002for3}
			\log\left(
			e+
			\frac{1}{\lambda t}
			e^{ty+\frac p2t^2}
			\right)
			\ge
			\frac p4t^2.
		\end{equation}
		Since $\mathcal N_1([0,\infty))=\frac{1}{2}$, it follows that
		\begin{equation}\label{eq:notconvzero}
			\liminf_{t\to\infty}
		\int_{O}\Phi_{p,p/2}\left(\frac{u_t}{\lambda}\right)\,\mathrm{d}\mu_{O}
		\ge\frac 12\left(\frac {p}{4\lambda^2}\right)^{p/2}>0.
		\end{equation}		
		
		Now take $t_n\triangleq n$ for every $n\in\mathbb N$. Then $\{u_{t_n}\}_{n=1}^{\infty}$ is bounded in
		$W_{\mathbf{b}}^{1,p}({O},\mu_{O})$ and converges to zero in
		$L^p({O},\mu_{O})$, whereas \eqref{eq:notconvzero} shows that
		$\inf_n\|u_{t_n}\|_{L^{\Phi_{p,p/2}}({O},\mu_{{O}})}>0$.
		If a subsequence of $\{u_{t_n}\}_{n=1}^{\infty}$ were convergent in
		$L^p\log^{p/2}L({O},\mu_{O})$, then, since this space embeds
		continuously into $L^p({O},\mu_{O})$, its limit would necessarily be
		zero, a contradiction. Therefore the embedding from $W_{\mathbf{b}}^{1,p}({O},\mu_{O})$ to $L^p\log^{p/2}L({O},\mu_{O})$ is not
		compact.
		
		Its continuity follows from Theorem~\ref{thm:embd} (i).
	\end{proof}

	
	\begin{proposition}\label{prop:sharp}
		Let $O$, $\mathbf a$ and $\mathbf b$ be as in Theorem~\ref{20261002thm1}. Then:
		\begin{enumerate}
			\item[\upshape(i)] For every $q>p/2$, there is no continuous embedding
			$W_{\mathbf{b}}^{1,p}({O},\mu_{O})\hookrightarrow L^p\log^q L({O},\mu_{O})$.
			\item[\upshape(ii)] For every $s>p$, there exists $v\in W_{\mathbf{b}}^{1,p}({O},\mu_{O})$ such that $v\notin L^s({O},\mu_{O})$.
		\end{enumerate}
	\end{proposition}	
		\begin{proof}
			Part (i): use the same family $\{u_t:t\in [1,\infty)\}$.
			For fixed $\lambda>0$ and sufficiently large $t$,
			\[
			\begin{aligned}
				\int_{{O}}\Phi_{p,q}\!\left(\frac{u_t}{\lambda}\right)\dd\mu_{O}
				&=\frac{1}{\lambda^p}\int_{O}
				\left|
				u_t
				\right|^p\left[
				\log\left(
				e+
				\left|
				\frac{u_t}{\lambda}
				\right|
				\right)
				\right]^{q}
				\,\mathrm{d}\mu_{O}\\
				&=
				\frac{1}{\lambda^p}t^{-p}
				\int_{\mathbb R}
				\left[
				\log\left(
				e+
				\frac{1}{\lambda t}
				e^{ty+\frac p2t^2}
				\right)
				\right]^{q}
				\,\mathrm{d}\mathcal N_1(y)\\
				&\ge
				\frac{1}{\lambda^p}t^{-p}
				\int_{0}^{\infty}
				\left[
				\log\left(
				e+
				\frac{1}{\lambda t}
				e^{ty+\frac p2t^2}
				\right)
				\right]^{q}
				\,\mathrm{d}\mathcal N_1(y)\\
				&\ge \frac{p^q}{2\lambda^p4^q}t^{2q-p},
			\end{aligned}
			\]
			whose right side tends to $\infty$ as $t\to\infty$. Here the second equality follows from \eqref{20261002for2} and the last inequality is from \eqref{20261002for3} and $\mathcal N_1([0,\infty))=1/2$. Hence
			$$
			\lim_{t\to\infty}\|u_t\|_{L^{\Phi_{p,q}}({O},\mu_{{O}})}=\infty
			$$
			and $\{u_t:t\in [1,\infty)\}$ is bounded in
			$W_{\mathbf{b}}^{1,p}({O},\mu_{O})$, so continuous embedding fails.
			
			Part (ii): take
			\[
			v(\mathbf{x})\triangleq\exp\!\left(\frac{x_2^2}{2sa_2^2}\right), \qquad \mathbf x=(x_i)_{i\in\mathbb N}\in \ell^2.
			\]
			Since
			\[
			\int_{{O}}|v|^p \dd\mu_{O}=\frac1{\sqrt{2\pi}}\int_{\mathbb R}e^{(\frac{p}{s}-1)x^2/2} \dd x<\infty,\qquad
			\|D_{\mathbf{b}} v\|=\frac{r_2}{sa_2}|x_2|\exp\!\left(\frac{x_2^2}{2sa_2^2}\right)\in L^p({O},\mu_{O}),
			\]
			we have $v\in W_{\mathbf{b}}^{1,p}({O},\mu_{O})$. However
			\[
			\int_{{O}}|v|^s \dd\mu_{O}
			=\frac1{\sqrt{2\pi}}\int_{\mathbb R}e^{x^2/2}e^{-x^2/2}\dd x=\infty.
			\]
			Hence (ii) follows.
		\end{proof}
	For the remainder of this subsection, let $O\subset\ell^2$ be an arbitrary nonempty open set.
	\begin{lemma}\label{lem:approx}
		For every $\varepsilon>0$, there are an integer $N\geq1$ and a Borel set $C$ depending only on the first $N$ variables such that
		\[
		P({O}\mathbin{\triangle}C)<\varepsilon.
		\]
	\end{lemma}
	\begin{proof}
		By a similar argument as in Sections 2.2 and 2.3 of \cite{YZ}, we can obtain the above result.
	\end{proof}
	We define
	$$
	Z_j(\mathbf x)\triangleq \frac{x_j}{a_j},\qquad \mathbf x=(x_i)_{i\in\mathbb N}\in \ell^2,\, j\in\mathbb N.
	$$
	\begin{lemma}\label{lem:measure-sep}
		There are constants $c>0$ and $N_0\geq1$ such that for any distinct $j,k>N_0$,
		\[
		P\left({O}\cap\{\mathbf x\in\ell^2:|\sin Z_j(\mathbf x)-\sin Z_k(\mathbf x)|>1/2\}\right)\geq c.
		\]
	\end{lemma}
	
	\begin{proof}
		Let $Z,Z'$ be independent standard normal random variables on $\ell^2$ with respect to $P$ and put
		\[
		q\triangleq P(\{ \mathbf x\in\ell^2:|\sin Z(\mathbf x)-\sin Z'(\mathbf x)|>1/2\}).
		\]
		Then
		$$
		q=\int_{\Omega}\frac{1}{2\pi}e^{-\frac{x_1^2+x_2^2}{2}}\,\mathrm d x_1\mathrm d x_2,
		$$
		where $\Omega=\{(x_1,x_2)\in\mathbb R^2:|\sin x_1-\sin x_2|>1/2\}$. Since $\Omega$ is a nonempty open subset of $\mathbb R^2$, we have $q>0$.

		Write $m=P({O})q>0$. Choose $\varepsilon>0$ so small that $2\varepsilon<m/2$, and use Lemma~\ref{lem:approx} to find a cylinder Borel set $C$ with $P({O}\triangle C)<\varepsilon$, where $C$ depends only on the first $N$ variables. For distinct $j,k>N$, the pair $(Z_j,Z_k)$ is independent of $C$, and consequently
		\[
		P\bigl(C\cap\{\mathbf x\in\ell^2:|\sin Z_j(\mathbf x)-\sin Z_k(\mathbf x)|>1/2\}\bigr)=P(C)q.
		\]
		Let $E_{j,k}\triangleq\{\mathbf x\in\ell^2:|\sin Z_j(\mathbf x)-\sin Z_k(\mathbf x)|>1/2\}$, then
		\begin{align*}
			\left|P({O}\cap E_{j,k})-P({O})q\right|
			&\leq |P({O}\cap E_{j,k})-P(C\cap E_{j,k})|
			+q|P(C)-P({O})|\\
			&\leq (1+q)P({O}\triangle C)
			<2\varepsilon<m/2.
		\end{align*}
		It follows that $P({O}\cap E_{j,k})>m/2$. The result holds with $c=m/2$ and $N_0=N$.
	\end{proof}
	\begin{theorem} \label{20260928thm1}
		Let $O\subset\ell^2$ be any nonempty open set. For every $1\leq p<\infty$, and every Young function $\Phi$ as in \eqref{eq:Lux}, there exists a sequence
		\[\{u_j\}\subset W^{1,p}(O,P)\cap\mathscr W^{1,p}(O,P)\]
		such that
		\[\sup_j\left(\|u_j\|_{W^{1,p}(O,P)}+\|u_j\|_{\mathscr W^{1,p}(O,P)}\right)<\infty,\qquad\inf_{j\ne k}\|u_j-u_k\|_{L^\Phi(O,P)}>0.\]
		Consequently, neither $W^{1,p}(O,P)$ nor $\mathscr W^{1,p}(O,P)$ can be compactly embedded into $L^\Phi(O,P)$.
	\end{theorem}
	\begin{proof}
		For each $j\in\mathbb N$, define
		\[
		u_j(\mathbf x)\triangleq\sin Z_j(\mathbf x)=\sin\!\left(\frac{x_j}{a_j}\right),\qquad \mathbf x=(x_i)_{i\in\mathbb N}\in{O}.
		\]
		Their partial derivatives are
		\[
		\frac{\partial u_j}{\partial x_k}=\begin{cases}
			a_j^{-1}\cos Z_j,&k=j,\\
			0,&k\ne j.
		\end{cases}
		\]
		These are also the weak derivatives of $u_j$: Gaussian integration by parts in the $k$-th variable gives, for every test function $\varphi\in C_{0,F^{\infty}}^{\infty}({O})$ from the weak-derivative definition,
		\[
		\int_{{O}}u_j\left(\frac{\partial \varphi}{\partial x_k}-\frac{x_k}{a_k^2}\varphi\right)\mathrm d P
		=-\int_{{O}}\frac{\partial u_j}{\partial x_k}\varphi\mathrm d P.
		\]
		Obviously, we have $\{u_j\}_{j=1}^{\infty}\subset W^{1,p}({O},P)\cap \mathscr W^{1,p}({O},P)$.
		
		Since $|u_j|\leq1$ for every $j\in\mathbb N$, and
		\[
		\left(\sum_{k=1}^{\infty}a_k^2\left|\frac{\partial u_j}{\partial x_k}\right|^2\right)^{1/2}
		=|\cos Z_j|\leq1,
		\]
		we have
		\[
		\|u_j\|_{\mathscr W^{1,p}({O},P)}\leq 2P({O})^{1/p}.
		\]
		Meanwhile,
		\[
		\sum_{k=1}^{\infty}\int_{{O}}a_k^p\left|  \frac{\partial u_j}{\partial x_k}\right|^p\,\mathrm dP
		=\int_{{O}}|\cos Z_j|^p\,\mathrm dP\leq P({O}),
		\]
		which implies $\|u_j\|_{W^{1,p}({O},P)}\leq2P({O})^{1/p}$.  Thus $\{u_j\}_{j=1}^{\infty}$ is bounded in both Sobolev spaces.
		
		By Lemma~\ref{lem:measure-sep}, after discarding finitely many terms and relabelling, we have for every distinct pair $j,k$ in the remaining sequence
		\begin{eqnarray*}
			P\bigl({O}\cap\{\mathbf x\in \ell^2:|u_j(\mathbf x)-u_k(\mathbf x)|>1/2\}\bigr)
			\geq c.
		\end{eqnarray*}
		Choose $s>0$ so large that $c\Phi(s)>1$, which is possible because $\Phi(s)\to\infty$, and set $\lambda_0\triangleq(2s)^{-1}$. Then for any such $j\ne k$, monotonicity of $\Phi$ gives
		\begin{align*}
			\int_{{O}}\Phi\!\left(\frac{|u_j-u_k|}{\lambda_0}\right)\dd P
			&\geq  P\bigl({O}\cap\{\mathbf x\in \ell^2:|u_j(\mathbf x)-u_k(\mathbf x)|>1/2\}\bigr)\,\Phi(s)\\
			&\geq c\Phi(s)>1.
		\end{align*}
		By the definition of the Luxemburg norm, this implies
		\[
		\|u_j-u_k\|_{L^\Phi({O},P)}\geq\lambda_0.
		\]
		Thus the sequence has no Cauchy subsequence in $L^\Phi({O},P)$, and in particular no convergent subsequence. A compact embedding would send every bounded sequence to a sequence with a convergent subsequence in $L^\Phi$. This contradiction proves all assertions.
	\end{proof}

\subsection{Logarithmic trace inequalities}
In this subsection, we use the results in the previous subsection and Section~\ref{sec:tr} to prove logarithmic trace inequalities.
\begin{theorem}\label{thm:logtrlfb}
	Fix $p\in[1,\infty)$. Let $O$ be either a convex open set or a $\mathscr{W}^{1,p}$-extension domain.
   If $\partial O$ is uniformly of class $C^1$ as in Definition~\ref{def:unc1}, then there exists a constant $C>0$ such that
   \[
   \|f|_{\partial O}\|_{L^{\Phi_{p,(p-1)/2}}(\partial O,\mu_{\partial O})}\le C\|{f}\|_{\mathscr{W}^{1,p}(O,P)},\qquad\forall {f}\in\mathscr{C}_b^\infty.
   \]
   In particular, $\mathscr{T}\!\mathrm{r}(\mathscr{W}^{1,p}(O,P))\subset L^{\Phi_{p,(p-1)/2}}(\partial O,\mu_{\partial O})=L^p\log^{(p-1)/2}L(\partial O,\mu_{\partial O})$.
\end{theorem}
\begin{proof}
	When $p=1$, $L^{\Phi_{p,(p-1)/2}}(\partial O,\mu_{\partial O})=L^p(\partial O,\mu_{\partial O})$ and the proof is easy. We now assume $p>1$. Take any ${f}\in\mathscr{C}_b^\infty$ with $\|f\|_{\mathscr{W}^{1,p}(O,P)}\le 1$. Let the notation be as in Definition~\ref{def:unc1}.
For any $i\in\L$, we write $\mathbf{x}=(x_j)_{j\in\dbN}$ and $\mathbf{x}'=(x_k)_{k\neq i^*}\in\ell^2(\dbN\setminus\{i^*\})$. We may assume $\varepsilon_i=1$ after changing the variable $\widetilde{x}_{i^*}=-x_{i^*}$ if necessary. Set $t=x_{i^*}-\psi_i(\mathbf{x}')$ and $g(t,\mathbf{x}')=f(\mathbf{x})$. Then $\partial_tg(t,\mathbf{x}')=\partial_{x_{i^*}}f(\mathbf{x})$. For simplicity, we write $F_p=\Phi_{p,(p-1)/2}$.

By Definition~\ref{20241012def2} and Definition~\ref{def:unc1} (1c),
\begin{equation}\label{eq:intbd'}
\begin{aligned}
&\hphantom{{}={}}\!\int_{U_i \cap \partial O}F_p(|f|)\dd\mu_{\partial O}\\
&=\int_{V_i}\sqrt{1+\|(\partial_{x_j}\psi_i(\mathbf{x}'))_{j\ne i^*}\|^2}\cdot F_p(|g(0,\mathbf{x}')|)\cdot\frac{1}{\sqrt{2\pi a_{i^*}^2}}e^{-\frac{\psi_i(\mathbf{x}')^2}{2a_{i^*}^2}}\dd P^{\widehat{i^*}}(\mathbf{x}')\\
&\le\sqrt{1+L^2}\int_{V_i}\frac{1}{\sqrt{2\pi a_{i^*}^2}}\,F_p(|g(0,\mathbf{x}')|)\cdot e^{-\frac{\psi_i(\mathbf{x}')^2}{2a_{i^*}^2}}\dd P^{\widehat{i^*}}(\mathbf{x}').
\end{aligned}\end{equation}
For fixed $\mathbf{x}'$, view $g(t,\mathbf{x}')$ as a function in $t$. Note that $F_p(|t|)$ is a $C^1$ function on $\dbR$, so $F_p(|g(t,\mathbf{x}')|)$ is a $C^1$ function in $t\in(-r_i,r_i)$. Applying Lemma~\ref{lem:1dim}, we have
\[\begin{aligned}
	\frac{1}{\sqrt{2\pi a_{i^*}^2}}\,F_p(|g(0,\mathbf{x}')|&\le \frac{C_1}{\min\{a_{i^*},r_i\}}\int_{0}^{r_i}\left(F_p(|g(t,\mathbf{x}')|)+ a_{i^*}|\partial_t F_p(|g(t,\mathbf{x}')|\right)\dd \mathcal{N}_{a_{i^*}}(t)\\
	&\le {C_1L}\int_{0}^{r_i}\left(F_p(|g(t,\mathbf{x}')|)+ a_{i^*}F_p'(|g(t,\mathbf{x}')|)\cdot|\partial_t g(t,\mathbf{x}')|\right)\dd \mathcal{N}_{a_{i^*}}(t).
\end{aligned}\]
Hence
\begin{equation}\label{eq:logsum}
\begin{aligned}
&\hphantom{{}={}}\int_{V_i}\frac{1}{\sqrt{2\pi a_{i^*}^2}}\,F_p(|g(0,\mathbf{x}')|\cdot e^{-\frac{\psi_i(\mathbf{x}')^2}{2a_{i^*}^2}}\dd P^{\widehat{i^*}}(\mathbf{x}')\\
&\le C_1L\int_{V_i}\int_{0}^{r_i}\left(F_p(|g(t,\mathbf{x}')|)+ a_{i^*}F_p'(|g(t,\mathbf{x}')|)\cdot|\partial_t g(t,\mathbf{x}')|\right)\dd \mathcal{N}_{a_{i^*}}(t)\,e^{-\frac{\psi_i(\mathbf{x}')^2}{2a_{i^*}^2}}\dd P^{\widehat{i^*}}(\mathbf{x}')\\
&={C_1}L\int_{U_i\cap O}\left(F_p(|f|)+a_{i^*} F_p'(|f|)\cdot|\partial_{x_{i^*}}f|\right)\cdot \exp\left(\frac{x_{i^*}^2-(x_{i^*}-\psi_i(\mathbf{x}'))^2-\psi_i(\mathbf{x}')^2}{2a_{i^*}^2}\right)\dd P(\mathbf{x})\\
&={C_1}L\int_{U_i\cap O}\left(F_p(|f|)+a_{i^*} F_p'(|f|)\cdot|\partial_{x_{i^*}}f|\right)\cdot \exp\left(\frac{(x_{i^*}-\psi_i(\mathbf{x}'))\psi_i(\mathbf{x}')}{a_{i^*}^2}\right)\dd P(\mathbf{x})\\
&\le {C_1}Le^{L^3}\int_{U_i\cap O}\left(F_p(|f|)+a_{i^*} |F_p'(|f|)|\cdot|\partial_{x_{i^*}}f|\right)\dd P.
\end{aligned}
\end{equation}
where the third line follows from the change of variables formula.

By Lemma~\ref{lem:Nfunction} and Young's inequality,
\begin{equation}\label{eq:logsum1}
\begin{aligned}
&\hphantom{{}={}}\!\int_{U_i\cap O}a_{i^*} F_p'(|f|)\cdot|\partial_{x_{i^*}}f|\dd P\\
&\le\frac{\widetilde{C}_p}{p'}\int_{U_i\cap O}|f|^p\left[\log(e+|f|)\right]^{\frac{p}{2}}\dd P+\frac{\widetilde{C}_p}{p}\int_{U_i\cap O}a_{i^*}^p|\partial_{x_{i^*}}f|^p\dd P\\
&\le \frac{\widetilde{C}_p}{p'}\int_{U_i\cap O}|f|^p\left[\log(e+|f|)\right]^{\frac{p}{2}}\dd P+\frac{\widetilde{C}_p}{p}\|f\|^p_{\mathscr{W}^{1,p}(U_i\cap O,P)}.
\end{aligned}
\end{equation}
Moreover, it is easy to see
\begin{equation}\label{eq:logsum2}
\int_{U_i\cap O}F_p(|f|)\dd P\le \int_{U_i\cap O}|f|^p\left[\log(e+|f|)\right]^{\frac{p}{2}}\dd P.
\end{equation}
Combining \eqref{eq:logsum}--\eqref{eq:logsum2} with \eqref{eq:intbd'}, we have
\[\begin{aligned}
&\hphantom{{}={}}\!\int_{\partial O}F_p(|f|)\dd\mu_{\partial O}\le \sum_{i\in\L}\int_{U_i \cap \partial O}F_p(|f|)\dd\mu_{\partial O}\\
&\le {C_1}Le^{L^3}\sqrt{1+L^2}\left[\left(\frac{\widetilde{C}_p}{p'}+1\right)\sum_{i\in\L}\int_{U_i\cap O}|f|^p\left[\log(e+|f|)\right]^{\frac{p}{2}}\dd P+\frac{\widetilde{C}_p}{p}\sum_{i\in\L}\|f\|_{\mathscr{W}^{1,p}(U_i\cap O,P)}^p\right]\\
&\le m{C_1}Le^{L^3}\sqrt{1+L^2}\left[\left(\frac{\widetilde{C}_p}{p'}+1\right)\int_{O}|f|^p\left[\log(e+|f|)\right]^{\frac{p}{2}}\dd P+\frac{\widetilde{C}_p}{p}\|f\|_{\mathscr{W}^{1,p}(O,P)}^p\right],
\end{aligned}\]
where the last inequality follows from Definition~\ref{def:unc1} (2).
Thus the result follows from Theorem~\ref{thm:embd} (i) with $\mathbf{b}=\mathbf{a}$ if $O$ is convex, or Lemma~\ref{lemma:Gaussian_Log_Sob_Embedding} (or Theorem~\ref{thm:embd} (i) for $O=\ell^2$) if $O$ is a $\mathscr{W}^{1,p}$-extension domain.
\end{proof}

\begin{theorem}\label{thm:logtr}
Let $O=O_g$ as in Proposition~\ref{prop:conv}. For any $p\in[2,\infty)$, there is a constant
$C=C(p,O)>0$ such that for every $u\in\mathscr{C}_b^\infty$ with $|u\|_{\mathscr W^{1,p}(O,P)}>0$,
\begin{equation}\label{eq:positive}
 \int_{\partial O}
 \left(\frac{\big|u|_{\partial O}\big|}{\|u\|_{\mathscr{W}^{1,p}(O,P)}}\right)^p
 \left[\log\left(e+\frac{\big|u|_{\partial O}\big|}{\|u\|_{\mathscr{W}^{1,p}(O,P)}}\right)\right]^{\frac{p-2}{2}}\dd\mu_{\partial O}\leq C.
\end{equation}
In particular, $\mathscr{T}\!\mathrm{r}(\mathscr{W}^{1,p}(O,P))\subset L^{\Phi_{p,(p-2)/2}}(\partial O,\mu_{\partial O})=L^p\log^{(p-2)/2}L(\partial O,\mu_{\partial O})$.
\end{theorem}

\begin{proof}
If $p=2$, this is proved by Proposition~\ref{prop:conv} combined with Theorem~\ref{thm:tr}. Now we assume $p>2$. Write
\[
 F(t)=|t|^{\frac{p}{2}}[\log(e+|t|)]^{\frac{p-2}{4}}.
\]
Since $p>2$, $F\in C^1(\dbR)$, $F'(0)=0$, and as in Lemma~\ref{lem:Nfunction},
\begin{equation}\label{eq:Fderivative}
 |F'(t)|\leq C_p|t|^{\frac{p-2}{2}}[\log(e+|t|)]^{\frac{p-2}{4}}
\end{equation}
for some constant $C_p>0$.

Let $u\in\mathscr{C}^\infty_b$ with $\|u\|_{\mathscr{W}^{1,p}(O,P)}\leq1$. Then
\[
 |F(u)|^2=|u|^p[\log(e+|u|)]^{\frac{p-2}{2}},
\]
so Theorem~\ref{thm:embd} (i) implies $\|F(u)\|_{L^2(O,P)}\le C_0$ for a constant $C_0>0$ independent of $u$. It follows from the chain rule, \eqref{eq:Fderivative}, and H\"older's inequality with exponents
$p/2$ and $p/(p-2)$ that
\begin{align*}
 \int_O\sum_{i=1}^{\infty}a_i^2|\partial_{x_i}F(u)|^2\dd P
 &\leq C_p\int_O\left(\sum_{i=1}^{\infty}a_i^2|\partial_{x_i}u|^2\right)|u|^{p-2}
 [\log(e+|u|)]^{\frac{p-2}{2}}\dd P\\
 &\leq C_p\left(\int_O\left(\sum_{i=1}^{\infty}a_i^2|\partial_{x_i}u|^2\right)^{\frac{p}{2}}\dd P\right)^{\frac{2}{p}}\left(\int_O|u|^p
 [\log(e+|u|)]^{\frac{p}{2}}\dd P\right)^{\frac{p-2}{p}}.
\end{align*}
It follows from $\|u\|_{\mathscr{W}^{1,p}(O,P)}\leq1$ and Theorem~\ref{thm:embd} (i) that
\[
 \|F(u)\|_{\mathscr{W}^{1,2}(O,P)}\leq C.
\]
By the $\mathscr{W}^{1,p}$ version of \eqref{eq:trst} for $p=2$ and Proposition~\ref{prop:conv},
\[
 \int_{\partial O}\big|u|_{\partial O}\big|^p
 \left[\log\left(e+\big|u|_{\partial O}\big|\right)\right]^{\frac{p-2}{2}}\dd\mu_{\partial O}
 =\int_{\partial O}\big|F(u)|_{\partial O}\big|^2\dd\mu_{\partial O}\leq \widetilde{C}.
\]
For an arbitrary $u\in\mathscr{C}_b^\infty$ with $\|u\|_{\mathscr W^{1,p}(O,P)}>0$, replacing $u$ by $u/\|u\|_{\mathscr{W}^{1,p}(O,P)}$, we obtain \eqref{eq:positive}.
\end{proof}

\subsection{Examples}\label{subsec:ex}
\subsubsection{No integrability improvement in the interior}
This subsection is devoted to proving Theorem~\ref{thm:logmain}.

We need the following elementary facts (see the proof of Lemma~\ref{jihejixian} and the references therein).

\begin{lemma}\label{lem:facts}
The measure $P$ satisfies the following properties:
\begin{enumerate}[label=(\roman*)]
\item every nonempty open ball has positive $P$-measure;
\item every line segment parallel to a coordinate axis has
      $P$-measure zero;
\item for every $\mathbf z\in\ell^2$ and $r>0$,
      $P(\partial B(\mathbf z,r))=0$.
\end{enumerate}
\end{lemma}

Fix $0<r<1/10$ and put $c=\sqrt{1-r^2}$. Define
\begin{equation*}
       U_j\triangleq B(4\mathbf e_j,r),\qquad \mu_j\triangleq P(U_j),\qquad j\in\mathbb N.
\end{equation*}
The balls $\{U_j\}_{j=1}^{\infty}$ are pairwise disjoint.  Lemma~\ref{lem:facts}
shows that $\mu_j>0$, and $\sum_{j=1}^{\infty}\mu_j<\infty$ gives
\begin{equation*}
    \lim_{j\to\infty}\mu_j=0.
\end{equation*}

For $\varepsilon>0$, write every $\mathbf x\in\ell^2$ uniquely as
$\mathbf x=t\mathbf e_j+\mathbf y$, where $t=x_j$ and $\mathbf y\perp \mathbf e_j$. Set
\begin{equation*}
 D_j(\varepsilon)\triangleq\left\{t\mathbf e_j+\mathbf y\in\ell^2:
   c<t<4+r,\quad\mathbf  y\perp\mathbf  e_j,\quad
   \|\mathbf y\|^2<h_\varepsilon(t)\right\},
\end{equation*}
where $h_\varepsilon\in C^\infty([c,4+r])$ will be specified later. Let
\[
 E_j(\varepsilon)\triangleq D_j(\varepsilon)\cap\{\mathbf x\in\ell^2:\langle \mathbf x,\mathbf e_j\rangle >2\},\qquad
 G_j(\varepsilon)=D_j(\varepsilon)\cap\{\mathbf x\in\ell^2:2<\langle \mathbf x,\mathbf e_j\rangle <3\}.
\]

\begin{assumption}\label{assu:hdelta}
Assume there are positive numbers $\varepsilon_j\downarrow0$ such that for all $j\in\dbN$, the following conditions hold:
\begin{enumerate}[label=(\arabic*)]
	\item $h_{\varepsilon_j}(t)\ge 1-t^2$ for $t\in(c,1)$;
	\item $\sup_{[c,4+r]}h_{\varepsilon_j}<c^2$;
	\item the function \begin{equation}\label{eq:Hj}
 H_j(t)=
 \begin{cases}
  1-t^2,&t\leq c,\\
  h_{\varepsilon_j}(t),&c\leq t\leq 4+r,\\
  r^2-(t-4)^2,&t\geq 4+r
 \end{cases}
\end{equation}
is $C^\infty$ on $\dbR$ and satisfies $\{t\in\dbR:H_j(t)=0\}=\{-1,4+r\}$;
\item $U_j\subset E_j(\varepsilon_j)$, $P(E_j(\varepsilon_j))\leq(1+2^{-j})\mu_j$, and $P(G_j(\varepsilon_j))\leq2^{-j}\mu_j$.
\end{enumerate}
\end{assumption}
Later in in Proposition~\ref{prop:hepsilon}, we will give a construction of $h_\varepsilon$ such that Assumption~\ref{assu:hdelta} holds.

We now define
\begin{equation}\label{eq:logO}
          O\triangleq B_1\cup\bigcup_{j=1}^\infty D_j(\varepsilon_j).
\end{equation}

\begin{proposition}\label{prop:geometry}
Under Assumption~\ref{assu:hdelta}, the set $O$ in \eqref{eq:logO} is a bounded, path-connected open
set with locally $C^\infty$ boundary. 
\end{proposition}

\begin{proof}
Each $D_j(\varepsilon_j)$ is open and clearly path-connected.  It meets $B_1$ by Assumption~\ref{assu:hdelta} (1). Hence the union $O$ is
path-connected.  For $\mathbf x=t\mathbf e_j+\mathbf y\in D_j(\varepsilon_j)$, we have
\[
 \|\mathbf x\|^2=t^2+\|\mathbf y\|^2<(4+r)^2+\sup_{[c,4+r]}h_{\varepsilon_j}<(4+r)^2+c^2,
\]
so $O$ is bounded.

For $i\ne j$, the closures $\overline{D_i(\varepsilon_i)}$ and $\overline{D_j(\varepsilon_j)}$ are disjoint. Indeed, if $\mathbf x=(x_k)_{k\in\mathbb N}\in \overline{D_i(\varepsilon_i)}\cap\overline{D_j(\varepsilon_j)}$, then
\[
c^2\le x_j^2
\le\sum_{k\ne i}x_k^2
\le\sup_{[c,4+r]}h_{\varepsilon_i}
<c^2,
\]
a contradiction. In particular, this proves the pairwise disjointness of the sets
$\{E_j(\varepsilon_j)\}_{j=1}^{\infty}$.

For every $\mathbf z\in\ell^2$, choose $N$ such that
\[
\sup_{j>N}|z_j|<c/2.
\]
Then
\[
B(\mathbf z,c/2)\cap
\bigcup_{j>N}\overline{D_j(\varepsilon_j)}
=\varnothing.
\]
Thus the family
$\{\overline{D_j(\varepsilon_j)}\}_{j\in\mathbb N}$
is locally finite.

It remains to show that $O$ has $C^\infty$ boundary. For $\mathbf x=t\mathbf e_j+\mathbf y$, $\mathbf y\perp\mathbf  e_j$, define
\begin{equation*}
             \rho_j(\mathbf x)\triangleq \|\mathbf y\|^2-H_j(t),
\end{equation*}
where $H_j$ is as in \eqref{eq:Hj}.
Since $  h_{\varepsilon_j}(t)\geq 1-t^2$ for $t\in(c,1)$ by Assumption~\ref{assu:hdelta} (1), direct inspection of \eqref{eq:Hj} gives
\begin{equation}\label{eq:local}
             B_1\cup D_j(\varepsilon_j)=\{\rho_j<0\}.
\end{equation}
Now let $\mathbf z\in\partial{O}$. Since the sets $\overline{D_j(\varepsilon_j)}$ are pairwise disjoint and locally finite, we can choose a neighborhood $V$ of $\mathbf z$ meeting at most one of them.  If $V$ does not meet $\overline{D_j(\varepsilon_j)}$ for all $j$, then
${O}\cap V=B_1\cap V$ and $\rho_0(\mathbf x)=\|\mathbf x\|^2-1$ is a defining
function.  If it meets $\overline{D_j(\varepsilon_j)}$, then \eqref{eq:local} gives
\[
       {O}\cap V=\{\rho_j<0\}\cap V,
       \qquad \partial{O}\cap V=\{\rho_j=0\}\cap V.
\]
Finally,
\[
       D\rho_j(t\mathbf e_j+\mathbf y)=2\mathbf y-H_j'(t)\mathbf e_j.
\]
If $H_j(t)>0$, then $\mathbf y\neq\mathbf 0$.
Since $H_j'(-1)=2$ and $H_j'(4+r)=-2r$, it follows from Assumption~\ref{assu:hdelta} (3) that $D\rho_j\ne\mathbf{0}$ on $\partial O\cap V$.
Likewise, $D\rho_0=2\mathbf x\neq\mathbf 0$ on $\partial B_1$. Thus the result follows.
\end{proof}

Choose $\eta\in C^\infty(\dbR)$ such that
\begin{equation*}
 0\leq\eta\leq1,\qquad
 \eta(t)=0 \text{ for }\,t\leq2,\qquad
 \eta(t)=1 \text{ for }\,t\geq3,
\end{equation*}
and put $C_\eta=\sup_{\dbR}|\eta'|$.  For $p\in[1,\infty)$, define
\begin{equation}\label{eq:fj}
        f_{j,p}(\mathbf x)\triangleq \mu_j^{-1/p}\eta(x_j),\qquad \mathbf x=(x_k)_{k\in\mathbb N}\in O.
\end{equation}
Each $f_{j,p}$ is the restriction of a bounded smooth cylindrical function, and $f_{j,p}^{-1}(\dbR\setminus\{0\})\subset E_j(\varepsilon_j)$. Its only possibly
nonzero partial derivative is the $j$-th one, and that derivative
vanishes outside $G_j(\varepsilon_j)$. Then Assumption~\ref{assu:hdelta} (4) gives
\begin{align*}
 \int_{O}|f_{j,p}|^p\dd P
  &\leq \mu_j^{-1}P(E_j(\varepsilon_j))\leq1+2^{-j},\\
 \int_{O}|a_j \partial_{x_j}f_{j,p}|^p\dd P
  &\leq \mu_j^{-1}\left(\sup_{i\in\dbN}a_i\right)^pC_\eta^pP(G_j(\varepsilon_j))
    \leq \left(\sup_{i\in\dbN}a_i\right)^pC_\eta^p2^{-j}.
\end{align*}
Consequently,
\begin{equation}\label{eq:fj bounded}
 \sup_j\left(\|f_{j,p}\|_{W^{1,p}({O},P)}+\|f_{j,p}\|_{\mathscr W^{1,p}({O},P)}\right)
  \leq C_p<\infty.
\end{equation}

Now we prove Theorem~\ref{thm:logmain}, first under Assumption~\ref{assu:hdelta}.
\begin{proof}[Proof of Theorem~\ref{thm:logmain} under Assumption~\ref{assu:hdelta}]
	Let $O$ be as in \eqref{eq:logO}. On $U_j$, the function $f_{j,p}$ is identically $\mu_j^{-1/p}$.
Since $U_j\subset E_j(\varepsilon_j)\subset O$ by Assumption~\ref{assu:hdelta}~(4) and $\mu_j^{-1/p}\to\infty$ as $j\to\infty$, for every fixed $\lambda>0$,
\begin{align*}
 \int_{O}\Phi\left(\frac{|f_{j,p}|}{\lambda}\right)\dd P
 &\geq\mu_j\Phi\left(\frac{\mu_j^{-1/p}}\lambda\right)\\
 &=\lambda^{-p}\,
   \frac{\Phi\left(\mu_j^{-1/p}/\lambda\right)}{\left(\mu_j^{-1/p}/\lambda\right)^p}
   \longrightarrow\infty.
\end{align*}
It follows directly from the definition \eqref{eq:Lux} that
$\|f_{j,p}\|_{L^\Phi(O,P)}\to\infty$.  This makes both
continuous embeddings impossible because $\sup_j\|f_{j,p}\|_{W^{1,p}({O},P)}$ and $\sup_j\|f_{j,p}\|_{\mathscr{W}^{1,p}({O},P)}$ are bounded by \eqref{eq:fj bounded}.
\end{proof}

The proof of Theorem~\ref{thm:logmain} will be completed by the next proposition.

\begin{proposition}\label{prop:hepsilon}
	There exist $h_{\varepsilon}\in C^\infty([c,4+r])$ and $\varepsilon_j\downarrow0$ satisfying Assumption~\ref{assu:hdelta}.
\end{proposition}

\begin{proof}
	Let
\[
 \theta(s)=\begin{cases}
 0,&s\leq0,\\[2mm]
 \displaystyle\frac{e^{-1/s}}{e^{-1/s}+e^{-1/(1-s)}},&0<s<1,\\[3mm]
 1,&s\geq1.
 \end{cases}
\]
This is a nondecreasing $C^\infty$ function. Let
$\gamma\triangleq(1-c)/2$, and define $$
h_{\varepsilon}(t)\triangleq\left\{\begin{array}{ll}
1-t^2+\left(r^2-(1-t^2)\right)
      \theta\left(\frac{t-c}{\gamma}\right),\quad&t\in[c,c+\gamma],\\
r^2 ,\quad&t\in[c+\gamma,1],\\
r^2+(\varepsilon^2-r^2)\theta(t-1),\quad&t\in[1,2],\\
\varepsilon^2,\quad&t\in[2,3].\end{array}\right.
$$

To define $h_{\varepsilon}(t)$ for $t\in[3,4+r]$, we choose an even, convex $C^\infty$ function $\chi_\tau$ such that $\chi_\tau(s)=|s|$ for $|s|\geq\tau$ and
\[
 |s|\leq\chi_\tau(s)\leq|s|+\tau,\qquad \forall s\in\mathbb R.
\]
Let
\[
 M_\tau(u,v)\triangleq \frac{u+v+\chi_\tau(u-v)}2,\qquad u,v\in\mathbb R.
\]
Then $M_\tau$ is a smooth function,
$$
\max\{u,v\}\leq M_\tau(u,v)\leq \max\{u,v\}+\tau/2,
$$
and
$$
 M_\tau(u,v)= \max\{u,v\},\qquad\text{whenever $|u-v|\geq\tau$.}
$$
When $3\varepsilon^2<2r^2$, let $t_{\varepsilon}=4-\sqrt{r^2-3\varepsilon^2/2}\in(3,4)$, which is a solution of $r^2-(t-4)^2=3\varepsilon^2/2$, and define
$$
h_{\varepsilon}(t)\triangleq\left\{\begin{array}{ll}
 M_{\varepsilon^2/2}(\varepsilon^2,r^2-(t-4)^2),\quad& t\in [3,t_\varepsilon],\\
r^2-(t-4)^2 ,& t\in [t_\varepsilon,4+r].\end{array}\right.
$$
Then for $0<\varepsilon\ll 1$, we have $0<h_\varepsilon(t)\le r^2$ for any $t\in(c,4+r)$. It is easy to check that the conditions (1)--(3) in Assumption~\ref{assu:hdelta} hold for all $0<\varepsilon_j\ll 1$.

Since $h_{\varepsilon}(t)\ge r^2-(t-4)^2$, we have $U_j\subset E_j(\varepsilon)$ for $0<\varepsilon\ll 1$. For each fixed $j$, let
\[K_j^{(s)}=\{\mathbf{x}\in\ell^2:\operatorname{dist}\left(\mathbf{x},\overline{U_j}\cup[2,4-r]\mathbf{e}_j\right)\le s\}.\]
For any $\mathbf{z}=z_j\mathbf{e}_j+\mathbf{y}\in E_j(\varepsilon)$, we have $2<z_j<4+r$ and $\|\mathbf{y}\|^2<h_\varepsilon(z_j)$. If $z_j\ge t_\varepsilon$, then $h_\varepsilon(z_j)=r^2-(z_j-4)^2$ and thus $\mathbf{z}\in\overline{U_j}$. If $2<z_j<t_\varepsilon$, then $z_j\mathbf{e}_j\in\overline{U_j}\cup[2,4-r]\mathbf{e}_j$ and
\[\|\mathbf{z}-z_j\mathbf{e}_j\|=\|\mathbf{y}\|<\sqrt{h_{\varepsilon}(z_j)}\le 2\varepsilon.\]
Therefore, $U_j\subset E_j(\varepsilon)\subset K_j^{(2\varepsilon)}$ for all $0<\varepsilon\ll 1$. Note that $\mu_j= P(U_j)=\lim_{\varepsilon\to 0} P(K_j^{(2\varepsilon)})$ by Lemma~\ref{lem:facts}, so $\lim_{\varepsilon\to 0} P(E_j(\varepsilon))=\mu_j$. Similarly, one can show that $G_j(\varepsilon)\subset L_j^{(\varepsilon)}$ and $\lim_{\varepsilon\to 0}P(G_j(\varepsilon))=\lim_{\varepsilon\to 0}P(L_j^{(\varepsilon)})=0$, where
\[L_j^{(\varepsilon)}=\{\mathbf{x}\in\ell^2:\operatorname{dist}\left(\mathbf{x},[2,3]\mathbf{e}_j\right)\le \varepsilon\}.\]
Since $\mu_j>0$ by Lemma~\ref{lem:facts}, we can choose $\varepsilon_j\downarrow 0$ inductively such that all conditions in Assumption~\ref{assu:hdelta} hold.
\end{proof}

\begin{remark}
	Using $f_{j,p}$ in \eqref{eq:fj}, we can construct $u_p\in W^{1,p}(O,P)\cap\mathscr{W}^{1,p}(O,P)$ such that $u_p\not\in L^p\log^qL(O,P)$ for any $q>0$.
\end{remark}

\subsubsection{No integrability improvement on the boundary}
\begin{theorem}\label{thm:logtrace2}
Let $\Phi(t)=t^2\Psi(t)$ be a Young function with $\Psi$ non-decreasing. Suppose that $\Phi$ satisfies \eqref{eq:strictPhi} for $p=2$. Then there exists $u\in W^{1,2}(B_r,P)\cap\mathscr{W}^{1,2}(B_r,P)$ such that $\operatorname{Tr}(u)=\mathscr{T}\!\mathrm{r}(u)\not\in L^{\Phi}(\partial B_r,\mu_{\partial B_r})$.
\end{theorem}
\begin{proof}
	Fix $\lambda>r/2$ and $r/\sqrt{2}<c_1<c_2<r$ such that
	\[b\triangleq\frac{r^2}{4}-\lambda(r-c_1)>0.\]
	By \eqref{eq:strictPhi}, $\lim_{t\to\infty}\Psi(t)=\infty$.
	Choose a strictly increasing sequence $\{n_i\}_{i\in\dbN}$ of positive integers such that the conclusions of Lemmas~\ref{lem:surface-mass} and \ref{prop:surfaceineq} hold for all $n=n_i$, and
	\begin{equation*}
	a_{n_i}^{-2}\ge i^2,\qquad 4^{-i}\Psi\left(\exp\left(\frac{b}{4a_{n_i}^2}\right)\right)\ge 2,\qquad \forall i\in\dbN.
	\end{equation*}
	Take $u_{n_i}$ in Lemma~\ref{prop:surfaceineq} and write
	\[w_i=\frac{2^{-i}u_{n_i}}{\|u_{n_i}|_{\partial B_r}\|_{L^2(\partial B_r,\mu_{\partial B_r})}}.\]
	By Lemma~\ref{prop:surfaceineq},
	\[\|w_i\|_{W^{1,2}(B_r,P)}+\|w_i\|_{\mathscr{W}^{1,2}(B_r,P)}\le 2^{-i}C.\]
	Therefore, $u\triangleq\sum_{i=1}^\infty w_i$ converges in both $W^{1,2}(B_r,P)$ and $\mathscr{W}^{1,2}(B_r,P)$. Since
	\[\|w_i|_{\partial B_r}\|_{L^{2}(\partial B_r,\mu_{\partial B_r})}=2^{-i},\]
	$\sum_{i=1}^{\infty}w_i|_{\partial B_r}$ converges in $L^{2}(\partial B_r,\mu_{\partial B_r})$ and thus $\operatorname{Tr}(u)=\mathscr{T}\!\mathrm{r}(u)=\sum_{i=1}^{\infty}w_i|_{\partial B_r}$.

	Let
\[
 E_i=\{\mathbf{x}=(x_j)_{j\in\dbN}\in \partial B_r:x_{n_i}\geq c_1\}.
\]
These sets are pairwise disjoint since $c_1>r/\sqrt{2}$. Proposition~\ref{prop:lim} and \eqref{eq:concentration} give
\begin{equation}\label{eq:Ei}
 \int_{E_i}w_i^2\dd\mu_{\partial B_r}\geq\frac{4^{-i}}{2}.
\end{equation}

The upper bound in \eqref{eq:uplwbd} combined with Proposition~\ref{prop:lim} implies that for sufficiently large $i$ and $\mathbf{x}\in E_i$,
\[\begin{aligned}
\log w_i(\mathbf{x})&=-i\log2+\log u_{n_i}-\log\left(\|u_{n_i}|_{\partial B_r}\|_{L^2(\partial B_r,\mu_{\partial B_r})}\right)\\
&\ge-i\log2+\frac{\lambda(x_{n_i}-r)}{a_{n_i}^2}+\frac{r^2}{4a_{n_i}^2}-\frac{1}{2}\log(Ca_{n_i})\\
&\ge\left(\frac{r^2}{4}-\lambda(r-c_1)\right)\frac{1}{2a_{n_i}^2},
\end{aligned}\]
where the last inequality follows from $a_{n_i}^{-2}\ge i^2$. After modifying $\{n_i\}_{i\in\dbN}$, we may assume
\begin{equation} \label{eq:wi}
 w_i(\mathbf{x})\geq\exp\left(\frac{b}{2a_{n_i}^2}\right),\qquad \forall i\in\dbN,\,\mathbf{x}\in E_i.
\end{equation}
Since $w_i\ge0$,
\[
\operatorname{Tr}(u)=\sum_{i=1}^\infty w_i|_{\partial B_r}\ge w_i|_{\partial B_r}
\qquad \mu_{\partial B_r}\text{-a.e.}
\]
Fix $\kappa>0$. By \eqref{eq:wi}, there exists $i_{\kappa}\in\dbN$ such that for all $i\ge i_{\kappa}$,
\[
\frac{w_i(\mathbf x)}{\kappa}\ge\exp\left(\frac{b}{4a_{n_i}^2}\right),
\qquad \mathbf x\in E_i.
\]
Therefore, by the disjointness of $E_i$, the monotonicity of $\Psi$, and \eqref{eq:Ei}, we have
\[
\begin{aligned}
\int_{\partial B_r}
\Phi\left(\frac{|\operatorname{Tr}(u)|}{\kappa}\right)
\,\mathrm d\mu_{\partial B_r}
&\ge
\frac1{\kappa^2}
\sum_{i\ge i_\kappa}
\int_{E_i}w_i^2\Psi(w_i/\kappa)
\,\mathrm d\mu_{\partial B_r}\\
&\ge
\frac1{2\kappa^2}
\sum_{i\ge i_\kappa}
4^{-i}\Psi\left(
\exp\left(\frac{b}{4a_{n_i}^2}\right)
\right)\\
&\ge
\frac1{\kappa^2}\sum_{i\ge i_\kappa}1
=\infty.
\end{aligned}
\]
Since $\kappa>0$ is arbitrary,
$\operatorname{Tr}(u)\notin
L^\Phi(\partial B_r,\mu_{\partial B_r})$.
\end{proof}

\section*{Acknowledgements}
This work is partially supported by the National Key R\&D Program of China under grant No.~2024YFA1013400, and by the New Cornerstone Science Foundation.

This work began around early 2021. AI tool was used to suggest some explicit examples and assist with polishing the exposition of several proofs from the authors' original manuscript. The overall research strategies in this paper were developed independently by the authors.

\end{document}